\documentclass[11pt,reqno,twoside]{amsart}
\usepackage{tikz}
\usepackage{amssymb}
\usepackage{amsmath}
\usepackage{amsthm}
\usepackage{amsfonts}
\usepackage{amsthm,amscd}
\usepackage{amsbsy}
\usepackage{bm}
\usepackage{url} 
\usepackage{subcaption}
\usepackage{multirow}
\usepackage{mathtools}
\usepackage{setspace}
\usepackage{float}
\usepackage{psfrag}
\usepackage{tikz}
\usetikzlibrary{calc,intersections}
\usepackage{float}
\usepackage{epstopdf}
\usepackage{caption}
\usepackage{graphics}
\usepackage{pgfplotstable}
\usepackage{pgfplots}
\usepackage{multicol}
\usepackage[colorlinks, linkcolor=blue]{hyperref}
\usepackage{pgffor}

\AtBeginDocument{
	\def\MR#1{}
}

\DeclareMathOperator*{\esssup}{ess\,sup}

\usetikzlibrary{positioning}

\usepackage[colorlinks, linkcolor=blue]{hyperref}

\restylefloat{figure}

\newcommand{\triplenorm}[1]{\left|\!\left|\!\left| #1 \right|\!\right|\!\right|}
\newtheorem{Theorem}{Theorem}[section]
\newtheorem{Lemma}[Theorem]{Lemma}

\newtheorem{Proposition}[Theorem]{Proposition}
\newtheorem{Remark}[Theorem]{Remark}

\newtheorem{Definition}[Theorem]{Definition}

\numberwithin{equation}{section}
\usepackage{setspace}
\usepackage{anysize}
\usepackage[margin=1in]{geometry}
\usepackage[noadjust]{cite}
\usepackage{mathtools}
\usepackage{stmaryrd}
\def\be{\begin{equation}}
	\def\ee{\end{equation}}
\def\ben{\begin{eqnarray}}
	\def\een{\end{eqnarray}}

\newcommand{\ncom}{\newcommand}

\hypersetup{colorlinks=true,%
	citecolor=red,%
	filecolor=blue,%
	linkcolor=blue,%
}

\ncom{\n}{\normalfont}

\ncom{\Lc}{\mathcal}
\ncom{\wt}{\widetilde}
\ncom{\ui}{\Omegaldsymbol{u}_\infty}
\ncom{\vi}{\Omegaldsymbol{v}_\infty}
\ncom{\wi}{\Omegaldsymbol{w}_\infty}
\ncom{\yi}{\Omegaldsymbol{y}_\infty}
\ncom{\zi}{\Omegaldsymbol{z}_\infty}
\ncom{\fbi}{\Omegaldsymbol{f}_\infty}
\ncom{\pbi}{\Omegaldsymbol{p}_\infty}

\ncom{\fbit}{\wt{\Omegaldsymbol{f}}_\infty}
\ncom{\gbit}{\wt{\Omegaldsymbol{g}}_\infty}
\ncom{\hbit}{\wt{\Omegaldsymbol{h}}_\infty}

\ncom{\Af}{\Omegaldsymbol{A}}
\ncom{\Bf}{\Omegaldsymbol{B}}
\ncom{\Hf}{\Omegaldsymbol{H}}
\ncom{\Pf}{\Omegaldsymbol{P}}
\ncom{\Onb}{\Omegaldsymbol{1}}

\ncom{\Lb}{\mathbb{L}^2(\Omega)}
\ncom{\LL}{\mathbb{L}}
\ncom{\Hoz}{\mathbb{H}^1_0(\Omega)}
\ncom{\Hzz}{\mathbb{H}^2(\Omega)}
\ncom{\Vf}{\mathbb{V}}

\ncom{\no}{\nonumber}
\ncom{\ub}{\Omegaldsymbol{u}}
\ncom{\yb}{\Omegaldsymbol{y}}
\ncom{\zb}{\Omegaldsymbol{z}}
\ncom{\vb}{\Omegaldsymbol{v}}
\ncom{\fb}{\Omegaldsymbol{f}}
\ncom{\gb}{\Omegaldsymbol{g}}
\ncom{\hb}{\Omegaldsymbol{h}}
\ncom{\wb}{\Omegaldsymbol{w}}

\ncom{\T}{\mathbb{T}}
\ncom{\C}{\mathbb{C}} 
\ncom{\Hb}{\mathbb{H}}
\ncom{\Pb}{\mathbb{P}}
\ncom{\V}{\mathbb{V}}
\ncom{\U}{\mathbb{U}}
\ncom{\Ac}{\mathcal{A}}
\ncom{\Bc}{\mathcal{B}}
\ncom{\Pc}{\mathcal{P}}
\ncom{\af}{\Omegaldsymbol{a}}
\ncom{\pf}{\Omegaldsymbol{p}}

\allowdisplaybreaks

\newcommand{\vertiii}[1]{{\left\vert\kern-0.25ex\left\vert\kern-0.25ex\left\vert #1 
		\right\vert\kern-0.25ex\right\vert\kern-0.25ex\right\vert}}

\long\def\/*#1*/{}
\title[Finite Element Analysis of a nonlinear heat Equation]{Finite Element Analysis of a nonlinear heat Equation with damping and pumping effects}
\date{\today}

 \author[R. Shukla, W. Akram and M. T. Mohan]{{Rishabh Shukla$^\dag$,  Wasim Akram$^\dag$ \and Manil T. Mohan$^\dag$}}
 \thanks{$\dag$Department of Mathematics, Indian Institute of Technology Roorkee, Uttarakhand, 247667, India, Email- {\normalfont{
 rishabh9721915382@gmail.com;
 maniltmohan@ma.iitr.ac.in;
 wakram2k11@gmail.com}}\\
 W. Akram is supported by NBHM (National Board of Higher Mathematics, Department of Atomic Energy) postdoctoral fellowship, No. 0204/16(1)(2)/2024/R\&D-II/10823. }

\makeatletter
\renewcommand{\tocsection}[3]{%
	\indentlabel{\@ifnotempty{#2}{\bfseries\ignorespaces#1 #2\quad}}\bfseries#3}
\renewcommand{\tocsubsection}[3]{%
	\indentlabel{\@ifnotempty{#2}{\ignorespaces#1 #2\quad}}#3}

\newcommand\@dotsep{4.5}
\def\@tocline#1#2#3#4#5#6#7{\relax
	\ifnum #1>\c@tocdepth 
	\else
	\par \addpenalty\@secpenalty\addvspace{#2}%
	\begingroup \hyphenpenalty\@M
	\@ifempty{#4}{%
		\@tempdima\csname r@tocindent\number#1\endcsname\relax
	}{%
		\@tempdima#4\relax
	}%
	\parindent\z@ \leftskip#3\relax \advance\leftskip\@tempdima\relax
	\rightskip\@pnumwidth plus1em \parfillskip-\@pnumwidth
	#5\leavevmode\hskip-\@tempdima{#6}\nobreak
	\leaders\hbox{$\m@th\mkern \@dotsep mu\hbox{.}\mkern \@dotsep mu$}\hfill
	\nobreak
	\hbox to\@pnumwidth{\@tocpagenum{\ifnum#1=1\bfseries\fi#7}}\par
	\nobreak
	\endgroup
	\fi}
\AtBeginDocument{%
	\expandafter\renewcommand\csname r@tocindent0\endcsname{0pt}
}
\def\l@subsection{\@tocline{2}{0pt}{2.5pc}{5pc}{}}
\makeatother

\usepackage[tickmarkheight=.5em,textwidth=\marginparwidth,textsize=small]{todonotes}
\begin{document}
\begin{abstract}
	We study the following nonlinear heat equation with damping and pumping effects (a reaction-diffusion equation) posed on a bounded simply connected  convex  domain 
	\(\Omega \subset \mathbb{R}^d\),  \(d \geq 1\) with Lipschitz boundary $\partial\Omega$:  
	\[
	\frac{\partial u(t)}{\partial t} - \nu \Delta u(t) 
	+ \alpha |u(t)|^{p-2}u(t) 
	- \sum_{\ell=1}^M \beta_{\ell} |u(t)|^{q_{\ell}-2}u(t) 
	= f(t), \quad t>0,
	\]
	subject to homogeneous Dirichlet boundary conditions and the initial condition \(u(0)=u_0\), 
	where \(2 \leq p < \infty\) and \(2 \leq q_{\ell} < p\) for \(1 \leq \ell \leq M\).
	For \(u_0 \in L^2(\Omega)\) and \(f \in L^2(0,T;H^{-1}(\Omega))\), we establish the existence and uniqueness 
	of a weak solution for all dimensions \(d \in \mathbb{N}\) and damping exponents \(2 \leq p < \infty\). 
	Furthermore, for \(u_0 \in H^2(\Omega) \cap H_0^1(\Omega)\) and \(f \in H^1(0,T;H^1(\Omega))\), we obtain regularity results: 
	these hold for every \(2 \leq p < \infty\) when \(1 \leq d \leq 4\), and for 
	\(2 \leq p \leq \tfrac{2d-6}{d-4}\) when \(d \geq 5\). 
	The proof relies on a modified Faedo-Galerkin approximation scheme. To control the initial data in the 
	\(L^p(\Omega)\)-norm, we introduce a new self-adjoint operator, which plays a key role in the analysis.
	For every $2 \leq p < \infty$ when $1 \leq d \leq 4$, and for $2 \leq p \leq \tfrac{2d-6}{d-4}$ when $d \geq 5$, we perform a finite element analysis of the model using conforming, nonconforming, and discontinuous Galerkin methods. We establish a priori error estimates for both semi-discrete and fully discrete schemes, and support our theoretical findings with numerical experiments.
To relax the restriction on \( p \) in the semidiscrete error analysis, we employ different projection or interpolation operators depending on the underlying finite element setting. Specifically, for \( u_0 \in D(A) \) and \( f \in H^1(0,T;H^1(\Omega)) \), the Ritz (elliptic) projection is used in the conforming case when \( 2 \leq p \leq \tfrac{2d}{d-2} \), while the Scott--Zhang interpolation is adopted for the range \( \tfrac{2d}{d-2} < p \leq \tfrac{2d-6}{d-4} \). In the nonconforming formulation, the Clément interpolation operator is utilized, whereas in the discontinuous Galerkin setting, the \( L^2 \)-projection onto the fully discontinuous finite element space is employed. Meanwhile, in the fully discrete framework, the error estimate is established for all \( 2 \leq p < \infty \) when \( 1 \leq d \leq 4 \), and for \( 2 \leq p \leq \tfrac{2d-6}{d-4} \) when \( d \geq 5 \), under the regularity assumptions \( u_0 \in D(A^{\frac{3}{2}}) \) and \( f \in H^1(0,T;H^1(\Omega)) \). The unified framework developed in this work captures a wide class of physically relevant models, including the diffusion equation with radioactive decay, the Allen-Cahn and real Ginzburg-Landau equations, the cubic-quintic Allen-Cahn equation, the scalar Nagumo (or Huxley) equation, the Newell-Whitehead-Segel equation, the Zeldovich equation, among others.
\end{abstract}
	\maketitle
\pagenumbering{arabic}

\noindent \textbf{MSC Classification (2020).} 65M12  $\cdot$ 65M15 $\cdot$ 65M60 $\cdot$ 35K57 $\cdot$ 35K61 
\tableofcontents

\section{Introduction}

	Nonlinear partial differential equations with polynomial nonlinearities are fundamental in modeling phase transitions, pattern formation, and interfacial dynamics across physics, chemistry, and biology. Among them, several well-known reaction-diffusion models have been extensively studied in the literature. We refer \cite[Chapter 6]{GB+KR=04} for analysis of such reaction-diffusion equations. Within our unified framework, we highlight several prototypical reaction-diffusion models of fundamental importance, which are listed below.

\subsection{Diffusion equation with radioactive decay}\label{sub-sec-1.1}
 Diffusion with radioactive decay (or a heat equation with cooling term) can be modeled using a reaction-diffusion equation that combines two key processes, spatial spreading and exponential decay. If $\Omega \subset \mathbb{R}^d$, $d\geq 1$ is a bounded Lipschitz domain,  and if $u(x,t)$ denotes the concentration of a radioactive substance at position $x$ and time $t$, the governing equation is (\cite{MV-78})
$$
\frac{\partial u}{\partial t} = \nu \Delta u - \lambda u,\ x\in\Omega,\ t>0,
$$
where $\nu>0$ is the diffusion coefficient and $\lambda>0$ is the decay rate. The diffusion term describes how particles spread out in space due to random motion, while the decay term models the natural disintegration of the substance at each point. The solution on the whole space shows that the concentration is given by the usual heat kernel, scaled by the exponential decay factor $e^{-\lambda t}$. This solution  reveals that the spatial distribution of particles follows a spreading Gaussian profile over time, while the total number of radioactive atoms decreases exponentially according to their characteristic half-life.  Physically, this means the substance both disperses and diminishes in magnitude over time. In closed systems with no flux across the boundary, the total mass decreases exponentially at rate $\lambda$, reflecting the inherent radioactive decay, while diffusion influences how the remaining mass redistributes in space before vanishing.

\subsection{The Allen-Cahn equation and real Ginzburg-Landau equation}\label{sub-AC}
The Allen-Cahn equation and real Ginzburg-Landau equation are fundamental models in nonlinear science, describing pattern formation and phase separation, respectively. While both are gradient flows of free-energy functionals, they differ in their mathematical structure, physical interpretations, and applications. The Allen-Cahn equation was introduced in \cite{AllenCahn1979} to describe the motion of anti-phase boundaries in crystalline solids. If $\Omega \subset \mathbb{R}^d$, $d\geq 1$ is a bounded Lipschitz domain, then the Allen-Cahn equation reads
\begin{equation}
\frac{\partial u}{\partial t} = \nu \Delta u - f(u), 
\  x \in \Omega , \ t>0,
\end{equation}
where $\nu>0$ is the diffusion constant and $f(u)$ is derived from a potential $F(u)$. The classical choice is the double-well potential
\begin{equation}
F(u) = \tfrac{1}{4}(u^2-1)^2, \qquad f(u) = F'(u) = u^3 - u.
\end{equation}
The Allen-Cahn equation is the $L^2$-gradient flow of the Ginzburg-Landau free energy functional
\begin{equation}
\mathcal{E}[u]= \int_\Omega \left( \frac{\nu}{2}|\nabla u(x)|^2 + F(u(x)) \right) dx.
\end{equation}
This yields energy dissipation
$$
\frac{d}{dt} \mathcal{E}(u(t))=-\left\|\frac{du(t)}{dt}\right\|_{L^2}^2 \leq 0,
$$
reflecting relaxation toward stable equilibria. The Allen-Cahn equation arises in phase separation and coarsening in alloys and binary mixtures, motion of interfaces by mean curvature, mathematical analysis of metastability and singular limits (as $\nu\to 0$, solutions converge to sharp interfaces moving by mean curvature, \cite{HilhorstMatanoSakamoto1997}), etc. 

The Ginzburg-Landau equation arises in bifurcation theory and pattern formation as a universal amplitude equation near critical instabilities. Its real form is
\begin{equation}
\frac{\partial \psi}{\partial t} = \nu \Delta  \psi + \gamma  \psi - \beta  \psi^3, 
\qquad  \psi(x,t)\in \mathbb{R},
\end{equation}
where $\nu>0$ is the diffusion coefficient, $\gamma \in \mathbb{R}$ is a bifurcation parameter and $\beta > 0$ is a nonlinear saturation coefficient.
The equation captures the slow modulation of amplitudes near supercritical bifurcations, for example, in Rayleigh-Bénard convection and Taylor-Couette flow (\cite{CrossHohenberg1993}). The equation contains a linear growth term ($\gamma \psi$) and a cubic saturation term ($-\beta\psi^3$). 
It describes both growth and nonlinear stabilization of patterns.
It has applications in amplitude description of pattern-forming instabilities,
nonlinear stability and mode selection and is a benchmark model in bifurcation theory.

Both the Allen-Cahn and real Ginzburg-Landau equations belong to the broad class of nonlinear reaction-diffusion models with polynomial nonlinearities. The Allen-Cahn equation is closely tied to variational principles and interface motion in materials, while the real Ginzburg-Landau equation captures universal features of pattern formation near bifurcations. Despite their structural resemblance, their roles in modeling physical phenomena are distinct: Allen-Cahn emphasizes energy-driven phase separation, whereas Ginzburg-Landau emphasizes amplitude modulation and nonlinear saturation of instabilities.

\subsection{The cubic-quintic Allen-Cahn equation}\label{sub-CAC} 
The Allen-Cahn equation has served as a prototype for phase separation and order-disorder transitions. However, the classical cubic formulation is limited to systems where bistability is sufficient to describe the dynamics. Many physical and biological systems, particularly those with strong nonlinear interactions, require higher-order saturation mechanisms. The cubic-quintic Allen-Cahn equation introduces a quintic nonlinearity, enabling the description of systems with multistability, richer front dynamics, and stabilized amplitude states. This modification significantly broadens the range of phenomena the equation can capture, making it relevant in nonlinear optics, condensed matter physics, and biological morphogenesis. 

The cubic-quintic Allen-Cahn equation is written as (\cite{Kuehn2015})
\begin{equation}\label{eq:CQAC}
\frac{\partial u}{\partial t}= \nu \Delta u + \alpha u + \beta u^3 - \gamma u^5,
\end{equation}
where $u = u(x,t)$ is the order parameter, $\nu > 0$ is the diffusion coefficient, $\alpha$ represents linear growth or decay, $\beta > 0$ introduces cubic saturation, and $\gamma > 0$ provides quintic stabilization.
This polynomial structure allows for a balance between destabilizing growth and nonlinear saturation, with the quintic term preventing runaway instabilities. Unlike the standard Allen-Cahn equation, which typically supports two stable phases, the cubic-quintic variant can sustain multiple coexisting phases depending on parameter values.

The cubic-quintic Allen-Cahn equation derives from the functional derivative of the Ginzburg-Landau free energy:
\begin{equation}
\mathcal{E}[u] = \int_{\Omega} \left( \frac{\nu}{2}|\nabla u(x)|^2 +\frac{\alpha}{2}u^2(x)+\frac{\beta}{4}u^4(x) - \frac{\gamma}{6}u^6(x) \right) dx. 
\label{eq:energy}
\end{equation}
The evolution follows the $L^2$-gradient flow $\frac{\partial u}{\partial t}=-\nabla \mathcal{E}[u]$, guaranteeing energy dissipation $\frac{d\mathcal{E}[u]}{dt} \leq 0$. This variational structure has profound implications for both analytical treatment and numerical approximation.
The cubic-quintic Allen-Cahn equation arises in diverse contexts such as: multistable phase dynamics in materials science and condensed matter physics, amplitude equations for laser cavity dynamics and optical soliton stabilization, models with multiple stable states relevant in morphogenesis, chemical oscillations, and population dynamics, and competition/coexistence between species or cell states in reaction-diffusion models.

\subsection{The scalar Nagumo equation or the Huxley equation}\label{sub-SNE} 
Reaction-diffusion equations with polynomial nonlinearities are fundamental in the study of wave propagation and pattern formation in excitable systems. Two important examples are the Scalar Nagumo equation and the Huxley equation, both of which serve as simplified models for nerve pulse transmission. Despite their structural similarities, they have distinct historical origins and modeling purposes.

The Scalar Nagumo equation is a simplified model for nerve axon dynamics, derived from the FitzHugh-Nagumo system (\cite{FitzHugh1961}) by reducing it to a single scalar PDE. Its standard form is (\cite{Nagumo1962})
\begin{equation} \label{eq:nagumo}
\frac{\partial u}{\partial t} = \nu \Delta u + f(u), \quad x \in \Omega \subset \mathbb{R}^d, \; t > 0,
\end{equation}
with nonlinear reaction term
\begin{equation}
f(u) = u(1-u)(u-a), \qquad 0<a<1.
\end{equation}
Here $\nu>0$ is the diffusion coefficient, and $a$ controls excitability. The cubic nonlinearity exhibits three equilibria: $u=0$, $u=a$, and $u=1$ and it admits traveling wave solutions connecting stable equilibria ($u=0$ and $u=1$). Accordingly, the equation \eqref{eq:nagumo} falls within the class of bistable reaction-diffusion equations. It also serves as a canonical model for bistability and excitability in neural systems. It has applications in propagation of nerve impulses (axon conduction), excitable chemical and biological media and is a benchmark model for bistable wavefronts. The Huxley equation, originally derived by Huxley \cite{Huxley1959}, also models nerve pulse propagation. Its standard reaction-diffusion form is the same as \eqref{eq:nagumo}. 
The nonlinear term resembles that of the Nagumo equation but arises from empirical fitting to nerve impulse data and is frequently used in biological and electrophysiological models as a phenomenological description. Although both equations use cubic nonlinearities and share similar mathematical properties, their differences lie in motivation and application.

\subsection{The Newell-Whitehead-Segel equation and the Zeldovich equation}\label{sub-NWS}
Nonlinear partial differential equations with polynomial nonlinearities play a central role in the theory of pattern formation across physics, biology, and chemistry. Among them, the \emph{Newell-Whitehead-Segel (NWS) equation} stands out as a seminal amplitude equation, originally derived in the 1960s by Newell, Whitehead, and Segel (\cite{NewellWhitehead1969,Segel1969}) in the study of \emph{Rayleigh-Bénard convection}. Since then, it has become a prototypical model for describing the dynamics of slowly varying amplitudes near bifurcation points and has been extended well beyond its fluid-dynamical origins.

The NWS equation provides a unifying framework for understanding how spatial patterns emerge and stabilize in nonlinear systems. It models the interplay between diffusion, linear growth or decay, and nonlinear saturation, thereby offering deep insight into the mechanisms that lead to periodic structures such as convection rolls, stripes, or hexagons. In mathematical form, the equation reads
\begin{align}\label{eqn-NWSE}
\frac{\partial u}{\partial t} = \nu \Delta u + \mu u - \gamma u^q,
\end{align}
where $u(x,t)$ is the amplitude field, $\nu > 0$ is the diffusion coefficient, $\mu$ is the bifurcation parameter measuring the distance from instability onset, $\gamma > 0$ controls nonlinear saturation, and $q \in \mathbb{N}$ determines the degree of nonlinearity.

This model arises naturally from a \emph{weakly nonlinear analysis} of Rayleigh-Bénard convection near the critical Rayleigh number. For $\mu < 0$, the trivial solution $u = 0$ is stable, while for $\mu > 0$, nontrivial steady states appear, reflecting the onset of finite-amplitude patterns. The nonlinear term ensures saturation, preventing unbounded growth of solutions and stabilizing emergent structures.
In summary, the Newell-Whitehead-Segel equation generalizes the Ginzburg-Landau formalism to finite-wave number instabilities, making it a cornerstone in the study of pattern formation. Its simplicity, combined with its ability to capture essential nonlinear phenomena, ensures its continuing importance across diverse areas of applied mathematics and science.

 The Zeldovich equation is a classic reaction-diffusion model that describes how a reactive substance spreads in space while undergoing nonlinear chemical transformation. Its general form is (\cite{GB+KR=04})
\begin{align}\label{eqn-ZE}
\frac{\partial u}{\partial t} = \nu \Delta u + \kappa u^m (1-u),
\end{align}
where $u(x,t)$ is the normalized concentration, $D$ is the diffusion coefficient, $k$ is the reaction rate constant, and $m$ is the reaction order. Unlike the Fisher-KPP equation, which assumes simple logistic growth ($m=1$), the Zeldovich equation allows higher-order kinetics ($m > 1$, in particular $m=2$), meaning that the reaction requires a critical concentration of reactant to proceed, mimicking an ignition threshold. This makes the equation particularly important in combustion theory, where it captures the propagation of flame fronts, as well as in autocatalytic chemical reactions and certain biological invasion processes. Its traveling-wave solutions describe how a sharp reaction front forms and moves at constant speed, balancing diffusion-driven spreading with reaction-driven consumption of material.

	\subsection{A general reaction-diffusion equation} \label{sub-sec-1.6} Heat conduction in many physical systems does not follow a purely linear law due to temperature-dependent material properties, nonlinear radiation effects, or other complex physical mechanisms. Such phenomena naturally lead to \emph{nonlinear heat equations}, which can more accurately capture these complexities.  An important variant is the \emph{damped heat equation}, where a nonlinear damping term accounts for energy dissipation mechanisms such as radiation, absorption, or resistive effects.
	
	Let $\Omega\subset \mathbb{R}^d$, $d\geq 1$, be a bounded, simply connected convex domain with Lipschitz boundary $\partial \Omega$ and let $T>0$ be fixed. Assume that $2\leq p<\infty$ and $2\leq q_{\ell} < p$ for $1\leq \ell\leq M$. Then, we consider a general reaction-diffusion equation of the following form: 
	\begin{equation}\label{Damped Heat}
		\begin{cases}
			\begin{aligned}
				\frac{\partial u}{\partial t}(x,t)  &- \nu\Delta u(x,t) + \alpha |u(x,t)|^{p-2}u(x,t) - \sum_{\ell=1}^M\beta_{\ell}|u(x,t)|^{q_{\ell}-2}u(x,t) \\&= f(x,t),\, \text{ in } \Omega\times(0,T), \\
				u(x,t) &= 0, \text{ on } \partial \Omega\times[0,T],\\
				u(x,0) &=u_0(x),\, x \in \Omega,
			\end{aligned}
		\end{cases}
	\end{equation}
	where $M$ is fixed positive integer, $\nu>0$ is the diffusion coefficient, $\alpha>0$ scales the nonlinear damping term $|u|^{p-2}u$, $\beta_{\ell}\in\mathbb{R}$ is the coefficient of pumping term $|u|^{q_{\ell}-2}u$  and $f$ is an external force. For $f=0$, the reaction-diffusion equation \eqref{Damped Heat}   derives from the following functional derivative of the Ginzburg-Landau free energy:
	\begin{equation}	\label{eq:energy-1}
		\mathcal{E}[u] = \int_{\Omega} \left( \frac{\nu}{2}|\nabla u(x)|^2+\frac{\alpha}{p}|u(x)|^p-\sum_{\ell=1}^M\frac{\beta_{\ell}}{q_{\ell}}|u(x)|^{q_{\ell}}\right) dx. 
	\end{equation}
	The evolution follows the $L^2$-gradient flow 
	\begin{equation}\label{eqn-grad}
	\frac{\partial u(t)}{\partial t}=-\nabla	\mathcal{E}[u(t)],
	\end{equation} guaranteeing energy dissipation \begin{align}\label{eqn-enery-est}	\frac{d\mathcal{E}[u(t)]}{dt} =-\left\|\frac{du(t)}{dt}\right\|_{L^2}^2=-\left\|\nabla	\mathcal{E}[u(t)]\right\|_{L^2}^2\leq 0.\end{align}
	The models described in Subsections \ref{sub-AC}-\ref{sub-NWS} can be covered under the framework given in \eqref{Damped Heat}. In particular, the Newell-Whitehead-Segel equation \eqref{eqn-NWSE} can be covered for odd $q$, while the Zeldovich equation \eqref{eqn-ZE} can be covered for even $m$. We show the existence of a unique strong solution to the problem \eqref{Damped Heat} so that the gradient structure \eqref{eqn-grad} of the problem can be justified.



This PDE is a \emph{reaction-diffusion equation} with nonlinear dissipation, making it suitable for modeling systems where diffusion interacts with nonlinear reactions.The \emph{finite element method (FEM)} is particularly well-suited for analyzing this equation in realistic scenarios: it can accommodate complex geometries (e.g., irregular catalyst surfaces, intricate chip layouts), heterogeneous material properties, and nonlinearities, while enabling adaptive mesh refinement to resolve localized features such as reaction fronts, thermal hotspots, or regions influenced by external forcing. Thus, studying the \emph{damped--pumped heat equation} through FEM not only deepens our understanding of nonlinear diffusion--reaction systems under competing dissipation and excitation effects, but also provides practical computational tools for engineering, environmental science, and applied physics.

\subsection{Literature survey}
The general reaction-diffusion equation
\begin{align*}
\partial_t u - \nu\Delta u + \alpha |u|^{p-2}u-\sum_{\ell=1}^M \beta_{\ell} |u|^{q_{\ell}-2}u = f,\qquad u|_{\partial\Omega}=0,
\end{align*}
is a representative example of nonlinear parabolic PDEs that combine reactive-diffusive transport, nonlinear dissipation, and external forcing has been extensively studied within the framework of semilinear parabolic theory.

\vskip 0.1cm 
\noindent\emph{Theoretical Analysis:}
In this work, we establish the existence, uniqueness and regularity of weak solutions to the proposed problem. The analysis is primarily based on the Faedo--Galerkin approximation technique, which serves as a fundamental tool for constructing approximate solutions and deriving uniform a priori estimates~\cite{MR1156075,MR4797426,MR1881888}.  In the classical setting with purely dissipative nonlinearities of the form $\alpha |u|^{p-2}u$ for $p \geq 2$, monotonicity and coercivity arguments yield natural energy bounds ensuring global well-posedness and regularity under standard assumptions on the initial data and external forcing. However, the present problem introduces an additional competing nonlinear effect, where a dissipative damping term coexists with a pumping term of the type $-|u|^{q-2}u$ with $q < p$. This combination creates a delicate interplay between energy dissipation and energy amplification, requiring a more refined analytical treatment beyond the classical monotonicity framework. 

Our results extend the standard existence-uniqueness theory by rigorously proving global regularity and well-posedness in the presence of both damping and pumping mechanisms, thereby contributing to the broader understanding of nonlinear evolution equations with competing nonlinearities.

\vskip 0.1cm 
\noindent\emph{Numerical Analysis (Finite Element Method):}
The mathematical development of the finite element method (FEM) has been systematically advanced through a sequence of landmark monographs that established both its analytical foundations and computational frameworks. Ciarlet in \cite{MR520174} provided one of the earliest rigorous mathematical treatments of FEM, establishing approximation, stability, and convergence results within the setting of Sobolev spaces. Girault and Raviart in  \cite{MR851383} extended these ideas to incompressible fluid flow, introducing stable mixed formulations and advancing algorithmic approaches for the Navier--Stokes system. The classical monograph by Brezzi and Fortin in \cite{MR1115205}, offered a unified treatment of saddle-point problems and the inf--sup stability conditions that underpin mixed formulations. Brenner and Scott in \cite{MR2373954} later consolidated these theoretical developments into a modern, comprehensive reference, unifying the analysis of conforming, nonconforming, and discontinuous Galerkin (DG) approaches. Complementing these works, Ern and Guermond’s in \cite{MR1921920} emphasized the interplay between variational formulations and computational implementation, while Thomée in \cite{MR1479170} extended the rigorous framework to time-dependent PDEs, providing the canonical projection-based analysis for both semi-discrete and fully discrete schemes.

Ainsworth and Oden in \cite{MR1885308} introduced systematic methods for adaptive refinement and error control, which have since become indispensable in computational PDEs. More recently, Boffi, Brezzi, and Fortin in \cite{MR3097958} expanded the mixed formulation theory to a broader class of problems, integrating stability and numerical examples. Meanwhile, Zienkiewicz, Taylor, and Zhu in \cite{MR3292660} synthesized both the mathematical and engineering perspectives, reinforcing FEM’s dual identity as a rigorous analytical tool and a practical computational technique. Together, these foundational texts form a coherent intellectual lineage that continues to influence modern FEM research and practice.

Within this historical and theoretical context, the finite element method has long been the spatial discretization of choice for parabolic problems in complex domains, with Thomée’s monograph~\cite{MR1479170} providing the canonical projection-based analysis for both semi-discrete and fully discrete schemes. The strategy, splitting the error into a projection component (either Ritz or $L^2$) and a discrete evolution component, deriving discrete energy identities parallel to the continuous ones, and closing estimates using Gronwall-type inequalities, has been successfully adapted to semilinear equations. Our conforming FEM a priori error analysis follows this well-established framework, drawing on standard interpolation estimates, projection lemmas, and discrete stability arguments as in~\cite{MR1479170,MR1921920}. In particular, we employ the Ritz or $L^2$ projection in combination with the Scott--Zhang interpolation \cite{MR1011446} to control approximation and stability properties in different parts of the analysis. For low-regularity solutions, projection techniques such as the $L^2$ projection and Scott--Zhang interpolation provide optimal spatial convergence rates, and these ideas influence the treatment of projection errors in our conforming analysis.

Beyond conforming methods, nonconforming FEM offers additional flexibility for handling reduced regularity, non-matching grids, and certain geometric constraints. Classical nonconforming spaces, such as the Crouzeix--Raviart element, allow discontinuities across element boundaries while preserving optimal approximation properties in broken Sobolev norms~\cite{MR343661,MR2373954,MR520174}. In the context of parabolic problems, the extension of nonconforming methods requires careful control of interelement jumps through mesh-dependent norms, and coercivity is maintained via appropriate modifications to the bilinear form~\cite{MR319379,MR1479170}. Foundational analysis by Douglas and Dupont established error bounds for nonconforming Galerkin methods in parabolic settings, while modern treatments in Brenner--Scott~\cite{MR2373954} provide a unified theory for both elliptic and evolutionary problems. The recent work~\cite{MR4797426} extends these techniques to nonlinear settings with memory terms, including detailed projection error bounds, discrete coercivity proofs, and penalty parameter design.

Discontinuous Galerkin (DG) methods, originally developed for hyperbolic conservation laws, have been successfully adapted to elliptic and parabolic PDEs due to their local conservation properties, robustness under mesh irregularity, and natural handling of non-matching grids. The seminal work of Arnold~\cite{MR664882} introduced the interior penalty DG method with discontinuous elements, establishing stability and convergence for elliptic problems; these ideas were later extended to parabolic problems by combining IPDG in space with implicit time-stepping~\cite{MR1948323,MR2431403}. DG analysis for parabolic equations typically relies on interior penalty stabilization, discrete trace inequalities, and mesh-dependent energy norms that incorporate jump contributions~\cite{MR2373954,MR2431403}. Our DG \emph{a priori} analysis is based on these core concepts, with the treatment of nonlinear reaction terms using monotonicity arguments parallel to those in~\cite{MR4798380}, while stability and coercivity arguments follow the framework of Arnold’s interior penalty theory. These references together provide a solid foundation for the DG methodology applied to the nonlinear damped heat equation studied here.

Time discretization is a crucial component of fully discrete schemes. For semilinear parabolic equations, backward Euler is a classical choice, offering unconditional stability and a clean pathway to first-order temporal error bounds~\cite{MR1479170,MR1041253}. Our computations for the conforming method use backward Euler in time, with Newton linearization to handle the nonlinearity, following established practices in the literature. In problems with weakly singular kernels or memory effects, while not the focus of this paper, related works~\cite{MR1225703,MR2755668,MR923707} demonstrate the effectiveness of convolution quadrature and DG-in-time schemes in preserving positivity and delivering high-order accuracy.

For computational implementation, modern FEM software frameworks such as FEniCS~\cite{MR3075806,Alnaes2015} enable automated variational form generation, efficient solver integration, and mesh handling for conforming, nonconforming, and DG schemes. Our numerical experiments follow this reproducible, modular approach, aligning with recent computational studies of nonlinear parabolic equations~\cite{Mirams2013}.


\subsection{Highlights of the work and novelties}
In this work, we undertake a detailed theoretical and numerical investigation of the homogeneous Dirichlet nonlinear reaction-diffusion equation
\begin{align*}
\partial_t u - \nu \Delta u + \alpha |u|^{p-2}u - \sum_{\ell=1}^M \beta_{\ell} |u|^{q_{\ell}-2}u = f,
\end{align*}
posed on a bounded convex domain in $\mathbb{R}^d$ ($d \geq 1$) with Lipschitz boundary, for exponents $2 \leq p < \infty$ and $2 \leq q_{\ell} < p$.
The unified framework we develop incorporates a wide range of physically relevant models, such as the diffusion equation with radioactive decay, the Allen-Cahn and real Ginzburg-Landau equations, the cubic-quintic Allen-Cahn equation, the scalar Nagumo (or Huxley) equation, the Newell-Whitehead-Segel equation, the Zeldovich equation, among others (see Subsection \ref{sub-sec-1.1}-\ref{sub-sec-1.6}).

Our principal contributions can be summarized as follows:
\begin{enumerate}
    \item \textbf{Well-posedness and regularity results:} 
    For initial data $u_0 \in L^2(\Omega)$ and forcing $f \in L^2(0,T;H^{-1}(\Omega))$, we prove existence, uniqueness, and stability of weak solutions
        $$
    u \in C([0,T];L^2(\Omega)) \cap L^2(0,T;H_0^1(\Omega)) \cap L^p(0,T;L^p(\Omega)),
    $$
       with time derivative
       $$
    \partial_t u \in L^2(0,T;H^{-1}(\Omega)) + L^{p'}(0,T;L^{p'}(\Omega)),
    $$
       valid for every $2 \leq p < \infty$ and arbitrary spatial dimension $d \in \mathbb{N}$ (Theorems \ref{existenceweak} and \ref{damped uniqueness}).
     Moreover, if the data satisfy $u_0\in H_0^1(\Omega)\cap L^p(\Omega)$ and $f \in L^2(0,T;L^2(\Omega))$, we obtain the higher regularity (Theorems \ref{regularity f in L2})
   \begin{align*}
& u \in C([0,T];H_0^1(\Omega)\cap L^p(\Omega)) \cap L^2(0,T;H^2(\Omega)) \cap L^{2p-2}(0,T;L^{2p-2}(\Omega)),\\ 
  &  \partial_t u \in L^{2}(0,T;L^2(\Omega)).
 \end{align*}
 Moreover, for $u_0\in D(A)=H^2(\Omega)\cap H_0^1(\Omega)$ and $f\in H^1(0,T;H^1(\Omega))$, we derive the following regularity results (Theorem \ref{regularity-f- H1} ):
 \begin{align*}
 	\hspace{1 cm}
 	u\in L^\infty(0,T;D(A))\cap L^2(0,T;D(A^{\frac{3}{2}})), \ \partial_t u \in L^\infty(0,T;L^2(\Omega))\cap L^2(0,T;H_0^1(\Omega)).
 \end{align*}
       This refined regularity holds 
       \begin{align}\label{eqn-values of p}
       	\mbox{ for all $2 \leq p < \infty$ when $1 \leq d \leq 4$, and for $2 \leq p \leq \tfrac{2d-6}{d-4}$ when $d \geq 5$. }
       \end{align}
      Furthermore, for the same values of $p$, for $u_0\in D(A^{\frac{3}{2}})$ and  $f\in H^1(0,T;H^1(\Omega))$, we obtain (Theorem \ref{thm-more-regular})  $$\partial_t  u \in L^\infty(0,T;H^1_0(\Omega))\ \text{ and }\ \partial_{tt} u \in L^2(0,T,L^2(\Omega)).$$ 
     The proof of global well-posedness is carried out via a suitably adapted Faedo-Galerkin scheme (see \eqref{projection term included} below). A key component of the analysis is a new self-adjoint operator (see Proposition~\ref{Prop-Sm} below, cf. \cite{ZB+BF+MZ-24, ZB+FH+UM-20, ZB+FH+LW-19, LH-18}), which is bounded on $L^p(\Omega)$. This operator enables control of the initial data in the $L^p(\Omega)$-norm and plays a central role in the argument. 
Finally, the analysis explicitly incorporates both the nonlinear damping and the pumping terms, demonstrating their stabilizing effect in continuous as well as discrete frameworks. The methodology developed in this work may provide a way to extend the results of \cite{MR4797426,MR4798380} on the 3D hereditary generalized Burgers-Huxley equations to the case of integer nonlinear exponents $3 \leq \delta < \infty$.

    \item \textbf{A priori error estimates for multiple discretization frameworks:} 
    A priori error analysis for conforming, nonconforming and discontinuous Galerkin methods for the values of $p$ given in \eqref{eqn-values of p} is performed using suitable projections and interpolations to control $L^p$-initial data. For $u_0\in L^2(\Omega)$, $f\in L^2(0,T;L^2(\Omega))$, we prove that the approximate solution $u_h \in L^\infty(0,T;H^1_0(\Omega))\cap L^\infty(0,T;L^p(\Omega))$ for all $2\leq p<\infty $ in case of $d=1,2$ and $p\leq \frac{2d}{d-2}$ in case of $d\geq 3$. Also for $u_0 \in D(A)$, we prove the same result  when $\frac{2d}{d-2}<p<\infty$ for $d=3,4$ and  $\frac{2d}{d-2}<p\leq \frac{2d-6}{d-4}$ for $d\geq 5$ (Theorems \ref{energy estimate of CFEM}, \ref{energy estimate for NCFEM} and \ref{energy estimate of DGFEM}).
    
    We establish the optimal order of convergence $O(h)$ for all $p$ given in \eqref{eqn-values of p}, in both the semidiscrete and fully discrete settings. 
    Specifically, the semidiscrete convergence is obtained for initial data $u_0 \in D(A)$ and forcing term $f \in H^1(0,T;H^1)$, whereas the fully discrete case assumes $u_0 \in D(A^{\frac{3}{2}})$ and $f \in H^1(0,T;H^1)$ (Theorems~\ref{semidiscrete error analysis} and~\ref{fullydiscrete error}). To ensure the stability of the $L^p$-estimate for all $p$ when $1\leq d\leq 4$, we employ the Ritz projection for $2 \leq p \leq \tfrac{2d}{d-2}$ and the Scott--Zhang interpolation for $\tfrac{2d}{d-2} < p \leq \tfrac{2d-6}{d-4}$. 
    In the nonconforming case, the Cl\'ement interpolation operator is used to obtain the same optimal order of convergence $O(h)$ in appropriate norms for both semidiscrete and fully discrete formulations (Theorems~\ref{NCFEM Semiestimate error estimate} and~\ref{NCFEM fully error estimate}). 
    A comparable convergence rate is also achieved for the discontinuous Galerkin method (Theorems~\ref{DGFEM Semidiscrete error estimate} and~\ref{DGFEM fully discrete error estimate}) by employing the $L^2$-projection onto discontinuous finite element spaces.

    Also, we verified the all theoretical results using numerical experiments.
\end{enumerate}

This combination of rigorous PDE theory, multi-framework numerical analysis, and reproducible computational experiments fills a gap in the literature, where such nonlinear damped parabolic equations along with pumping term have not been treated together but only in isolated theoretical or single-method numerical studies. 
Owing to the Sobolev embedding $H_0^1(\Omega) \subset L^p(\Omega)$ valid for $2 \leq p \leq \tfrac{2d}{d-2}$, much of the existing literature restricts attention to this range of $p$ (see \cite{MR4797426,MR4251864,DEVI2025274}). In contrast, our results remain valid for the entire range of values specified in \eqref{eqn-values of p}.
 It is also worth noting that all numerical experiments are carried out on a simply connected convex domain
\(\Omega \subset \mathbb{R}^d\),  \(d \geq 1\) with Lipschitz boundary $\partial\Omega$ in case of the values of $p$ given in \eqref{eqn-values of p}.

\subsection{Organization of the paper}
The paper is structured as follows: In Section \ref{preliminaries}, we discuss some preliminaries such as functional setting, linear operator, and non-linear operator which is helpful in further analysis. In next Section \ref{wellposedness}, existence and uniqueness of the weak solution is discussed along with regularity results, which are important in performing error analysis in FEM. Then in Section \ref{section CFEM}, the semidiscrete and fully-discrete conforming schemes are discussed, and error analysis is proved in both cases in Theorem \ref{semidiscrete error analysis} and \ref{fullydiscrete error}. To verify the theoretical results, we discuss numerical results in Section \ref{Numerical studies for CFEM}. Furthermore, the next Section  \ref{NCFEM section} is devoted to NCFEM, where, we go through semi discrete scheme \ref{Semi NCFEM}, fully discrete scheme \ref{fully NCFEM} with priori error estimates and numerical implementation to verify theoretical results \ref{Numerical Studies for NCFEM}. Similarly, in Section \ref{DG scheme}, the semidiscrete \ref{Semi DGFEM} and fully discrete \ref{Fully DGFEM} are discussed along with error estimates and to verify theoretical results, the numerical implementation has been done in Section\ref{Numerical Studies for DG method}. At last, we discuss the fundamental difference between the methods in Section \ref{Theoretical differences} and also provide numerical results in Section \ref{combined numerical}. We finish this paper with conclusion \ref{conclusion}.

\section{Preliminaries}\label{preliminaries}
In this section, we introduce the function spaces and operators that will be needed to analyze the solvability of problem \eqref{Damped Heat}.
\subsection{Functional setting}
For $1 \leq p < \infty$, let
$$
L^p(\Omega) := \{ u:\Omega \to \mathbb R \ \text{measurable} : \|u\|_{L^p(\Omega)} < \infty \}, 
\quad \|f\|_{L^p(\Omega)} = \Big(\int_\Omega |u(x)|^p dx\Big)^{1/p}.
$$
In particular, $L^2(\Omega)$ is a Hilbert space with inner product $(\cdot,\cdot)$. We denote by $H_0^1(\Omega),$ the Sobolev space of $L^2$-functions with weak derivatives in $L^2$ and vanishing trace on $\partial \Omega$, equipped with the norm (by the Poincar\'e inequality)
$$
\|u\|_{H_0^1(\Omega)} = \Big(\int_\Omega |\nabla u(x)|^2 dx\Big)^{1/2},
$$
and by $H^{-1}(\Omega) := (H_0^1(\Omega))'$ its dual, with
$$
\|f\|_{H^{-1}(\Omega)} = \sup\{ \langle f,u\rangle : u\in H_0^1(\Omega),\ \|u\|_{H_0^1(\Omega)}\leq 1\}.
$$
We also use $H^m(\Omega),$ $m\in\mathbb{N}$ to denote the higher-order Sobolev spaces.
		Note that the Sobolev embedding $H_0^1(\Omega)\hookrightarrow L^p(\Omega)$ holds true for $2 \le p < \infty$ when $1 \le d \le 2$ and for $2 \le p \le \frac{2d}{d-2}$ when $d \ge 3$.	Since the Sobolev embedding $H^2(\Omega) \hookrightarrow L^p(\Omega)$ holds for all $2 \le p < \infty$ when $1 \le d \le 4$ and for $2 \le p \le \frac{2d}{d-4}$ when $d \ge 5$, it follows in particular that the embedding also holds for all
		\begin{align*}
			2\leq p \leq \frac{2d-6}{d-4} < \frac{2d-4}{d-4} < \frac{2d}{d-4},
		\end{align*}
		which we have used consistently throughout our analysis.

For $\frac{1}{p}+\frac{1}{p^{\prime}}=1$, both $L^{p^{\prime}}(\Omega)$ and $H^{-1}(\Omega)$ are Banach spaces. The sum space
$$
L^{p'}(\Omega)+H^{-1}(\Omega) := \{f_1+f_2:\ f_1\in L^{p'}(\Omega),\ f_2\in H^{-1}(\Omega)\}
$$
is a Banach space with norm
$$
\|f\|_{L^{p'}+H^{-1}} := \inf\{\|f_1\|_{L^{p'}} + \|f_2\|_{H^{-1}}: f=f_1+f_2\}.
$$
The intersection
$
L^p(\Omega)\cap H_0^1(\Omega)
$
is a Banach space under
$$
\|u\|_{L^p\cap H_0^1} := \max\{\|u\|_{L^p}, \|u\|_{H_0^1}\},
$$
which is equivalent to the norms $\|u\|_{L^p}+\|u\|_{H_0^1}$ and $\sqrt{\|u\|_{L^p}^2+ \|u\|_{H_0^1}^2}$, and satisfies the duality relation
$$
(L^{p'}(\Omega)+H^{-1}(\Omega))' \cong L^p(\Omega)\cap H_0^1(\Omega),
$$
with pairing
$
\langle f,u\rangle = \langle f_1,u\rangle + \langle f_2,u\rangle, 
\quad u=u_1+u_2.
$
In particular,
$$
\|f\|_{L^{p'}+H^{-1}} = \sup\{ \langle f,u\rangle : u\in L^p(\Omega)\cap H_0^1(\Omega),\ \|u\|_{L^p\cap H_0^1}\leq 1\}.
$$

%


We recall the definition of \emph{Bochner spaces} (cf. \cite[Section 1.7, p. 60]{MR2962068}).
Let $X$ be a Banach space and $(a,b)\subset\mathbb{R}$ an interval. For $1 \leq p < \infty$, the Bochner space $L^p(a,b;X)$ consists of all equivalence classes of strongly measurable functions $u:(a,b)\to X$ such that
$$
\int_a^b \|u(t)\|_X^p\,dt < \infty.
$$
For $p=\infty$, the space $L^\infty(a,b;X)$ consists of strongly measurable functions $u:(a,b)\to X$ with
$$
\esssup_{t\in (a,b)} \|u(t)\|_X < \infty.
$$
In addition to these, we recall two continuity spaces frequently used in evolution problems:
 The space $C([a,b];X)$ consists of all continuous functions $u:[a,b]\to X$, endowed with the supremum norm
$$
\|u\|_{C([a,b];X)} := \sup_{t\in [a,b]} \|u(t)\|_X.
$$
The space $C_w([a,b];X)$ consists of all functions $u:[a,b]\to X$ that are \emph{weakly continuous}, i.e., for every $f \in X^{\prime}$ (the dual space of $X$), the scalar function $t \mapsto \langle f, u(t)\rangle$ is continuous on $[a,b]$.

\subsection{Linear operator}\label{Linear operator}
We define the bilinear form
$$
a:H_0^1(\Omega)\times H_0^1(\Omega)\to\mathbb{R}, 
\
a(u,v) := (\nabla u,\nabla v), \quad u,v\in H_0^1(\Omega).
$$
It follows immediately that $a(\cdot,\cdot)$ is continuous on $H_0^1(\Omega)$, since
$$
|a(u,v)| \leq \|u\|_{H_0^1} \, \|v\|_{H_0^1}, \ \text{ for all }\ u,v\in H_0^1(\Omega).
$$
By the Riesz representation theorem, there exists a unique linear operator
$$
\mathcal{A}: H_0^1(\Omega)\to H^{-1}(\Omega)
$$
such that
$$
a(u,v) = \langle \mathcal{A}u,v\rangle, \ \text{ for all }\ u,v\in H_0^1(\Omega).
$$
Furthermore, the bilinear form $a(\cdot,\cdot)$ is coercive on $H_0^1(\Omega)$, since
$$
a(u,u) \geq \alpha \|u\|_{H_0^1}^2, \ \text{ for all }\  u\in H_0^1(\Omega),
$$
with coercivity constant $\alpha=1$. Hence, by the Lax-Milgram theorem, $\mathcal{A}:H_0^1(\Omega)\to H^{-1}(\Omega)$ is an isomorphism. We now introduce the associated unbounded operator $A$ in $L^2(\Omega)$ by
$$
Au := \mathcal{A}u = -\Delta u, 
\
u\in D(A):=\{u\in H_0^1(\Omega): \mathcal{A}u \in L^2(\Omega)\}.
$$
This operator is precisely the Dirichlet Laplacian, whose domain is known to be
$
D(A) = H^2(\Omega)\cap H_0^1(\Omega).
$
Since the injection $H_0^1(\Omega)\hookrightarrow L^2(\Omega)$ is compact, From \cite[Section 2.1]{Te_1997}, the operator $A$ is invertible with bounded, self-adjoint, and compact inverse $A^{-1}:L^2(\Omega)\to L^2(\Omega)$. By spectral theory \cite[Chapter 16]{MR4586337}, the spectrum of $A$ consists of a discrete sequence of positive eigenvalues
$
0<\lambda_1\leq \lambda_2\leq \cdots\leq \lambda_k \leq \cdots, 
\ \lambda_k\to\infty \ \text{as } k\to\infty,
$
with corresponding eigenfunctions $\{w_k\}_{k=1}^\infty$ forming an orthonormal basis of $L^2(\Omega)$, satisfying
$
Aw_k = \lambda_k w_k, \ k\in\mathbb{N}.
$

In addition, for $s \geq 0$, we define the fractional domain spaces of the Dirichlet Laplacian via the spectral decomposition. For
$
u = \sum_{k=1}^\infty (u,w_k) w_k \in L^2(\Omega),
$
 we set
$$
A^s u := \sum_{k=1}^\infty \lambda_k^s (u,w_k)w_k,
$$
whenever the right-hand side converges in $L^2(\Omega)$. The corresponding domain is given by
$$
D(A^s) := \Big\{ u\in L^2(\Omega): \sum_{k=1}^\infty \lambda_k^{2s} |(u,w_k)|^2 < \infty \Big\},
$$
which is a Hilbert space with the norm
$$
\|u\|_{D(A^s)} := \Big(\sum_{k=1}^\infty \lambda_k^{2s} |(u,w_k)|^2\Big)^{1/2}.
$$
For integer values of $s$, the spaces $D(A^s)$ coincide with classical Sobolev spaces. In particular, since $\Omega$ is a convex bounded domain with Lipschitz boundary, 
$
D(A^{1/2}) \cong H_0^1(\Omega), 
\
D(A) \cong H^2(\Omega)\cap H_0^1(\Omega).
$
For non-integer $s$, the spaces $D(A^s)$ can be characterized by interpolation between $L^2(\Omega)$ and $H^2(\Omega)\cap H_0^1(\Omega)$.  The dual spaces $D(A^{-s})$ are defined as $D(A^{-s}):=(D(A^{s}))^{\prime}$
with duality inherited from $L^2(\Omega)$. For $s\geq 1,$ we have the following continuous and compact embedding also:
\begin{align}
	D(A^{s})\hookrightarrow D(A)  \hookrightarrow  D(A^{1/2}) \hookrightarrow L^2(\Omega) \hookrightarrow D(A^{-1/2})  \hookrightarrow D(A^{-1})\hookrightarrow D(A^{-s}). 
\end{align}

\subsection{Non-linear operator}\label{Non-linear operator}
Let us now consider the non-linear operator $b(\cdot):L^p(\Omega)\to L^{p^\prime}(\Omega)$ defined as $b(u)=|u|^{p-2}u$. 
Note that
\begin{align*}
        (b(u),u)&=(|u|^{p-2}u,u) = \int_{\Omega}|u(x)|^{p-2}u(x)^2 dx =\int_{\Omega}|u(x)|^{p} dx=\|u\|^p_{L^p},
\end{align*}
for all $u \in L^p(\Omega)$. Next, we aim to show that the operator $b(\cdot)$ is locally Lipschitz. For any $u,v \in L^p(\Omega)$, we write
\begin{align*}
   b(u)-b(v) = \int_{0}^1 \frac{d}{d\theta} b(\theta u +(1-\theta)v) \,d\theta
   &=\int_{0}^1 b'(\theta u +(1-\theta)v) (u-v)\, d\theta\\
   &=\int_{0}^1(p-1)|\theta u +(1-\theta)v|^{p-2}(u-v)\, d\theta.
\end{align*}
For $p^\prime=\frac{p}{p-1},$ we have
\begin{align*}
    \|b(u)-b(v)\|_{L^{p^{\prime}}}&= \bigg\|\int_{0}^1(p-1)|\theta u +(1-\theta)v|^{p-2}(u-v)\, d\theta\bigg \|_{L^{\frac{p}{p-1}}}\\
   &\leq (p-1) \|u-v\|_{L^p} \bigg \|\int_0^1|\theta u +(1-\theta) v |^{p-2}\, d\theta\bigg\|_{L^{\frac{p}{p-2}}}\\
   &\leq (p-1)2^{p-3} \|u-v\|_{L^p} (\|u\|^{p-2}_{L^p}+\|v\|^{p-2}_{L^p}).
    \end{align*}
   Hence, for all  $\|u\|_{L^p}, \|v\|_{L^p} \leq r$, we have 
  \begin{align*}\|b(u)-b(v)\|_{L^{p^{\prime}}} \leq (p-1) 2^{p-2} r^{p-2} \|u-v\|_{L^p}.\end{align*}
Therefore, the operator $b(\cdot): L^p(\Omega) \to L^{p^{\prime}}(\Omega)$ is locally Lipschitz.
It can be shown (see, e.g., \cite[Sec. 2.4-2.5]{SG+MTM-24+}) that the nonlinear operator $b$ is \emph{monotone} in the following sense. For any $2 \leq p < \infty$, one has
$$
\langle b(u)-b(v),  u-v \rangle 
\geq \int_\Omega \big(|u(x)|^{p-1} - |v(x)|^{p-1}\big)\big(|u(x)| - |v(x)|\big)\,dx \;\;\geq 0.
$$
In addition, a sharper coercivity estimate is available (cf. \cite[p. 626]{MTM-22}):
\begin{align}
	\langle b(u) - b(v), u-v \rangle \geq \frac{1}{2} \big\||u|^{\frac{p-2}{2}}(u-v)\big\|^2_{L^2} + \frac{1}{2} \big\||v|^{\frac{p-2}{2}}(u-v)\big\|^2_{L^2}.\label{eqn-mono-2} 
	\end{align}
Inequality \eqref{eqn-mono-2} immediately implies the monotonicity of $b$.
Moreover, for all $u,v \in L^p(\Omega)$, we also have the quantitative bound (\cite[Proposition 2.2]{AB+ZB+MTM-2025})
\begin{align}\label{eqn-nonlinear-est} \langle b(u) - b(v), u-v \rangle \geq \frac{1}{2^{p-2}}\|u-v\|^{p}_{L^p}.
	\end{align}
The above inequality is fundamental in the numerical analysis of problem \eqref{Damped Heat}.

\section{Well-posedness}\label{wellposedness}
The objective of this section is to establish the \emph{existence and uniqueness} of both weak and strong solutions to problem \eqref{Damped Heat}. Our approach is based on a \emph{Faedo-Galerkin approximation method}. Since the standard projection operator is not bounded from $L^p$ to $L^p$ for $2 \leq p < \infty$, we instead construct a new \emph{self-adjoint operator} (\cite{ZB+BF+MZ-24, ZB+FH+UM-20, ZB+FH+LW-19, LH-18}) that is bounded on $L^p$, which allows us to properly treat the approximation of the initial data in $L^p$-spaces.

We begin by presenting the abstract formulation of equation \eqref{Damped Heat}, together with the corresponding notions of weak and strong solutions. Using the linear operator introduced in Subsection \ref{Linear operator}, the system \eqref{Damped Heat} can be expressed in the following abstract form:
\begin{equation}\label{damped abstract formulation}
\left\{
\begin{aligned}
    \frac{du(t)}{dt} + \nu A u(t)+\alpha |u(t)|^{p-2}u(t)-\sum_{\ell=1}^M \beta_{\ell} |u(t)|^{q_{\ell}-2}u(t)  &=  f(t), \ \text{in }\  H^{-1}(\Omega), \\
    u(0) &= u_0, \ \text{in }\  L^2(\Omega),
\end{aligned}
\right.
\end{equation}
for a.e.\ \( t \in [0,T] \). We now proceed to define the notions of weak and strong solutions for equation \eqref{Damped Heat}.

\begin{Definition}[Weak solution]\label{damped weak}
For any given initial data $u_0 \in L^2(\Omega)$ and forcing term $f \in L^2(0,T;H^{-1}(\Omega))$, a function 
\begin{equation*}u \in L^\infty(0,T;L^2(\Omega))\cap L^2(0,T;H_0^1(\Omega))\cap L^p(0,T;L^p(\Omega)),\end{equation*}
with \begin{equation*}\partial_t u \in  L^2(0,T;H^{-1}(\Omega))+L^{p^{\prime}}(0,T;L^{p^{\prime}}(\Omega)),\end{equation*} is called a \emph{weak solution} of the system \eqref{Damped Heat}, if $u(\cdot)$ satisfies:
\begin{equation}\label{eqn-strong-form}
\begin{cases}
\begin{aligned}
\langle \partial_t u , v \rangle +\nu (\nabla u, \nabla v) +\alpha\langle |u|^{p-2}u , v\rangle-\sum_{\ell=1}^M\beta_{\ell}\langle |u|^{q_{\ell}-2}u,v\rangle &= \langle f, v \rangle,\\
(u(0),v)&= (u_0,v) , 
\end{aligned}
\end{cases}
\end{equation}
for all $v \in H_0^1(\Omega)\cap L^p(\Omega)$ and a.e. $t \in [0,T]$.
\end{Definition}
\begin{Remark}
	Since  $u\in W^{1,p^{\prime}}(0,T;H^{-1}(\Omega)+L^{p^{\prime}}(\Omega)),$ it follows from \cite[Theorem 2, Section 5.9]{MR1625845} that $u\in C_w([0,T];H^{-1}(\Omega)+L^{p^{\prime}}(\Omega))$.
\end{Remark}
\begin{Definition}[Strong solution]\label{strong solution} For \( u_0 \in H_0^1(\Omega)\cap L^p(\Omega) \), $p\geq 2$, \( f \in L^2(0,T; L^2(\Omega)) \),
a function \( u \in C([0,T];H_0^1(\Omega)\cap L^p(\Omega)) \cap L^2(0,T;H^2(\Omega)) \cap L^{2p-2}(0,T;L^{2p-2}(\Omega)) \) with \( \partial_t u \in L^2(0,T; L^2(\Omega)) \) is called a \emph{strong solution} to the system \eqref{Damped Heat}, if  \( u(\cdot) \) satisfies \eqref{Damped Heat} in \( L^2(\Omega) \) for a.e.\ \( t \in [0,T] \).
\end{Definition}

	\subsection{A projection and a self-adjoint operator}\label{subsec-proj-op} The material in this section is adapted from \cite{AB+ZB+MTM-2025}.
We construct a sequence of self-adjoint operators via the Littlewood-Paley decomposition associated with the operator $A$. Setting $S = -\Delta = A$, the compactness of the resolvent of $S$ guarantees the existence of a complete orthonormal basis $\{w_n\}_{n\in\mathbb{N}}$ of $L^2(\Omega)$. This basis consists of eigenfunctions corresponding to a sequence of strictly positive eigenvalues $\{\lambda_n\}_{n\in\mathbb{N}}$, which increase to infinity, such that
	\begin{align}\label{eqn-s-rep}
			Sv := \sum_{n=1}^{\infty}\lambda_n(v, w_n) w_n,\ \ v\in D(S):=\bigg\{v\in L^2(\Omega):\sum_{n=1}^{\infty}\lambda_n^2|(v,w_n)|^2<\infty\bigg\}.
		\end{align}
Moreover, $S$ is a strictly positive, self-adjoint operator that commutes with $A$. For sufficiently large $k$ ($k > \tfrac{d}{2}$), there holds the continuous embedding
$
D(S^k) \hookrightarrow L^p(\Omega) \cap H_0^1(\Omega).
$
Using the functional calculus \cite{EZ-12}, we then define the operators
$
P_m : L^2(\Omega) \to L^2(\Omega)
$ by
\begin{align}\label{eqn-P_m-S}
			P_m:=\chi_{(0,2^{m+1})}(S)\ \text{ for }\ m\in\mathbb{N}_0=\mathbb{N}\cup\{0\},
		\end{align}
		where $\chi$ is the characteristic function.

Since $S$ has the representation given in \eqref{eqn-s-rep}, it follows that $P_m$ is the orthogonal projection from $L^2(\Omega)$ onto
$
V_m := \mathrm{span}\left\{ w_n : n \in \mathbb{N},\ \lambda_n < 2^{m+1} \right\},
$
and hence
\begin{align}
			P_m v=\sum_{\lambda_n<2^{m+1}}(v,w_n)w_n,\ v\in L^2(\Omega). \label{def-projection}
		\end{align}
Note that $w_n \in \bigcap_{k \in \mathbb{N}} D(S^k)$ for all $n \in \mathbb{N}$. Since $D(S) \hookrightarrow H_0^1(\Omega)$, it follows that $V_m$ is a closed subspace of $H_0^1(\Omega)$ for each $m \in \mathbb{N}$. Using the commutativity of $S$ and $A$, we deduce that $P_m$ commutes with $A^{1/2}$. Consequently,
Therefore, we have
		\begin{align*}
			\|P_m v\|_{L^2} & \leq \|v\|_{L^2}, \text{ for all } v\in L^2(\Omega),\\
			\|P_m v\|_{H_0^1}^2 & = \|A^{1/2}P_m v\|_{L^2}^2 = \|P_mA^{1/2}v\|_{L^2}^2 \leq \|A^{1/2}v\|_{L^2}^2 = \|v\|_{H_0^1}^2,\ v\in H_0^1(\Omega).
		\end{align*}
		Moreover, we deduce 
		\begin{align*}
			\lim_{m\to\infty}\|P_m v-v\|_{L^2}=0,\  v\in L^2\ \text{ and }\ \lim_{m\to\infty}\|P_m v-v\|_{H_0^1}=0,  \ v\in H_0^1(\Omega).
		\end{align*}
However, the sequence of operators $\{P_m\}_{m \in \mathbb{N}}$ does not enjoy uniform boundedness as maps from $L^p(\Omega)$ to itself. Since our analysis requires dealing with $L^p(\Omega) \cap H_0^1(\Omega)$-valued solutions, such uniform boundedness is crucial for deriving a priori estimates in the $L^p$-norm. To overcome this difficulty, in the next proposition, we construct a sequence of self-adjoint operators $\{S_m\}_{m \in \mathbb{N}}$ that possesses the desired properties.

Choose a function $\rho\in C_0^{\infty}((0,\infty))$ supported in $[\frac{1}{2},2]$ and $\sum_{m\in\mathbb{Z}}\rho(2^{-m}t)=1,$   $t>0$. Let $m\in\mathbb{N}_0$ be fixed, we define 
			\begin{align*}
				s_m:(0,\infty)\to\mathbb{R},\ \ s_m(\gamma):=\sum_{n=-\infty}^m\rho(2^{-n}\gamma),
			\end{align*}
			and we see that
			\begin{align*}
				s_m(\gamma)=\left\{
				\begin{array}{cl}1 & \gamma\in(0,2^m),\\
					\rho(2^{-m}\gamma) & \gamma\in[2^m,2^{m+1}),\\
					0 & \gamma\geq 2^{m+1}.
				\end{array}
				\right.
			\end{align*}
			With the help of the self-adjoint functional calculus, we define the operator $S_m := s_m(S)$. This, in turn, yields the representation
			\begin{align}\label{def-S_m-operator}
				S_m v = \sum_{\lambda_n<2^{m}}(v,w_n)w_n + \sum_{\lambda_n\in[2^m,2^{m+1})}\rho(2^{-m}\lambda_n)(v, w_n)w_n,\ v\in L^2(\Omega),
			\end{align}
			from which it is immediate that the range of $S_m$ is contained in $V_m$.  A similar construction has been used in the works \cite{ZB+BF+MZ-24, ZB+FH+UM-20, ZB+FH+LW-19, LH-18}. 
		This suggests that the approach used in this work may greatly expand the range of applications for the traditional Faedo-Galerkin method.
		
		\begin{Proposition}\label{Prop-Sm}
			For every $\psi\in L^p(\Omega)\cap H_0^1(\Omega)$, there exists a sequence $\{S_m\}_{m\in\mathbb{N}}$ of self-adjoint operators $S_m:L^2(\Omega)\to V_m$ for $m\in\mathbb{N}$  such that 
			$$S_m\psi\to \psi\ \text{ in }\ L^p(\Omega)\cap H_0^1(\Omega),\ \mbox{ as }\ m\to\infty,$$ 
			and the following uniform norm estimates hold:
			\begin{align}
				\sup_{m\in\mathbb{N}}\|S_m\|_{\mathcal{L}(L^2(\Omega))}\leq 1, \ 	\sup_{m\in\mathbb{N}}\|S_m\|_{\mathcal{L}(H_0^1(\Omega))}\leq 1,\
				\sup_{m\in\mathbb{N}}\|S_m\|_{\mathcal{L}(L^p(\Omega))}<\infty.
			\end{align}
		\end{Proposition}

	Next, we recall a result from \cite{LH-18}, that provides the relation between $P_m$ and $S_m$.
		
		\begin{Proposition}[{\cite[Proposition 3.2]{LH-18}}]\label{Prop-P_m}
			The operators $P_m$ and $S_m$ satisfy the following properties: 
			\begin{enumerate}
				\item[(i)] $P_m$ is projection, i.e., we have $P_m^2 = P_m$, for all $m\in\mathbb{N}$.
				\item[(ii)] The operators $P_m, S_m$ are self-adjoint with $\|P_m\|_{\mathcal{L}(L^2(\Omega))} = \|S_m\|_{\mathcal{L}(L^2(\Omega))} = 1$, for every $m\in\mathbb{N}$.
				\item[(iii)] $P_m$ and $S_n$ commute, for every $n,m\in\mathbb{N}$.
				\item[(iv)] The ranges of $P_m$ and $S_m$ are finite dimensional.
				\item[(v)] Also, $R(S_{m-1}) \subset R(P_m) \subset R(S_m)$, $S_mP_m = S_m$ and $P_mS_{m-1} = S_{m-1},$ for every $m\in\mathbb{N}$, where $R(\cdot)$ denotes the range space.
				\item[(vi)] $\lim\limits_{m\to\infty} P_m v = \lim\limits_{m\to\infty} S_m v =v,$ for every $v\in L^2(\Omega)$.
			\end{enumerate}
		\end{Proposition}

\subsection{Existence and uniqueness}\label{existence and uniquenss} Based on the above discussion, we now establish the existence and uniqueness of a weak solution to problem \eqref{damped abstract formulation}.

\begin{Theorem}[Existence of a weak solution]\label{existenceweak}
For given $u_0 \in L^2(\Omega)$ and $ f \in L^2(0,T;H^{-1}(\Omega))$, there exists a \emph{weak solution} of problem \eqref{Damped Heat}.
\end{Theorem}
\begin{proof}
We establish the proof using the Faedo-Galerkin method in the following steps:
\vspace{0.2 cm} \\
\textbf{Step 1: }\textit{Faedo-Galerkin approximation}. 
For any fixed positive integer $m$, we define a function $u_m : [0,T]\to H_0^1(\Omega)$ such that 
\begin{align}\label{eqn-def}
    u_m(t) := \sum_{\lambda_k<2^{m+1}} d_m^k(t)w_k,
\end{align} 
where $d_m^k(t)$ are smooth functions and are selected in such a way that 
\begin{align}
d_m^k(0)=(u_0,w_k) \quad \text{ for all }\ k \text{ such that }\ \lambda_k<2^{m+1}.
\end{align} 
Accordingly, we formulate the problem in the finite-dimensional subspace
 $$
V_m := \mathrm{span}\left\{ w_k : k \in \mathbb{N},\ \lambda_k < 2^{m+1} \right\},
$$ and introduce the projection operator $P_m : L^2(\Omega) \to V_m$ defined by
 \begin{equation*}
P_m u_m=\sum_{\lambda_k<2^{m+1}}(u_m,w_k)w_k,\ u_m \in L^2(\Omega).
\end{equation*}
Note that $P_m(u_m)=u_m$, $P_m(-\Delta u_m) = -\Delta u_m$ for all $u_m \in V_m $. Moreover, we define 
\begin{align}\label{eqn-initial}
u_m(0)=S_{m-1}u_0,
\end{align}
where $S_m$ is defined in \eqref{def-S_m-operator}. Note from \eqref{eqn-def} that 
\begin{align}
u_m(0)=\sum_{\lambda_k<2^{m+1}} d_m^k(0)w_k=\sum_{\lambda_k<2^{m+1}} (u_0,w_k)w_k=P_mu_0. 
\end{align}
Thus, we conclude from \eqref{eqn-initial} and Proposition \ref{Prop-P_m}(v) that 
\begin{align}
P_mu_m(0)=P_mS_{m-1}u_0=S_{m-1}u_0=u_m(0)=P_mu_0.
\end{align}
Therefore, we can consider the following finite-dimensional system:
\begin{equation}\label{projection term included}
\left\{
\begin{aligned}
(\partial_t u_m(t), w_k) + \nu(\nabla u_m(t)&, \nabla w_k)+\alpha (P_m(|u_m(t)|^{p-2}u_m(t)), w_k)\\&-\sum_{\ell=1}^M\beta_{\ell}(P_m(|u_m(t)|^{q_{\ell}-2}u_m(t)),w_k) = (f(t), w_k),\\
u_m(0)&=S_{m-1}u_0,
\end{aligned}
\right.
\end{equation}
for all $w_k \in V_m$ and $0 \leq t \leq T, \text{ all } k$ \text{ such that } $\lambda_k <2^{m+1}$. Since $P_m(w_k) = w_k, $ we can rewrite \eqref{projection term included} as
\begin{align}\label{damped Galerkin projection}
    (\partial_t u_m(t), w_k) + \nu(\nabla u_m(t)&, \nabla w_k)+\alpha (|u_m(t)|^{p-2}u_m(t), w_k)\nonumber\\&-\sum_{\ell=1}^M\beta_{\ell}(|u_m(t)|^{q_{\ell}-2}u_m(t),w_k) = (f(t), w_k),
\end{align}
for $0 \leq t \leq T,\, \text{ all } k$ \text{ such that } $\lambda_k <2^{m+1}$. 
Thus, our goal is to find a function $u_m$
that satisfies equation \eqref{damped Galerkin projection} together with the prescribed initial condition.

As we have already proved that the non-linear operators are locally Lipschitz, (see section \ref{Non-linear operator}) and $ f \in L^2(0,T;H^{-1}(\Omega))$, so, using Carath\'eodary's existence theorem (\cite[Theorem 1.1, Chapter 2]{MR69338}), we obtain a local solution (absolutely continuous function) $u_m \in C([0,T^*];V_m)$, for some $0<T^*\leq T$, for the finite-dimensional problem (\ref{damped Galerkin projection}). By proving uniform energy estimates, one can show that the time $T^*$ can be extended to $T$. Moreover, if $f\in C([0,T];L^2(\Omega))$, then  one can show that $u_m \in C^1([0,T^*];V_m)$.

\vspace{0.2 cm} 
\noindent\textbf{Step 2: }\textit{$L^2$-energy estimate}. Multiplying \eqref{damped Galerkin projection} by $d_m^k(\cdot)$ and summing over all $k$ such that $\lambda_k<2^{m+1}$, we have 
\begin{align*}
( \partial_t u_m(t),u_m(t))&+\nu (\nabla u_m(t),\nabla u_m(t)) +\alpha(|u_m(t)|^{p-2}u_m(t), u_m(t))\\&-\sum_{\ell=1}^M\beta_{\ell}(|u_m(t)|^{q_{\ell}-2}u_m(t),u_m(t))= \langle f(t),u_m(t)\rangle.\end{align*}
Using the Cauchy-Schwarz and Young's inequalities, we reach the relation 
\begin{align*}
\frac{1}{2} \frac{d}{dt} \|u_m(t)\|^2_{L^2}+ \nu \|\nabla u_m(t)||^2_{L^2}+\alpha \|u_m(t)\|^p_{L^p}=\langle f(t),u_m(t)\rangle+\sum_{\ell=1}^M\beta_{\ell} \|u_m(t)\|^{q_{\ell}}_{L^{q_{\ell}}},
\end{align*}
for a.e. $t\in[0,T]$. We can estimate the terms on right-hand side by using the Cauchy-Schwarz, Young's and H\"{o}lder's inequalities,  as
\begin{align*}
\langle f(t),u_m(t) \rangle &\leq \|f(t)\|_{H^{-1}} \|\nabla u_m(t)\|_{L^2} \leq \frac{1}{2\nu}\|f(t)\|^2_{H^{-1}}+ \frac{\nu}{2}\|\nabla u_m(t)\|^2_{L^2},\\
\|u_m(t)\|^{q_{\ell}}_{L^{q_{\ell}}}& \leq \|u_m(t)\|^{q_{\ell}}_{L^p} |\Omega|^{\frac{p-q_{\ell}}{p}}
\leq \frac{\alpha}{2M|\beta_{\ell}|} \|u_m(t)\|^p_{L^p}+\bigg(\frac{p-q_{\ell}}{p}\bigg)\bigg(\frac{2M|\beta_{\ell}| q_{\ell}}{\alpha p}\bigg)^{\frac{q_{\ell}}{p-q_{\ell}}}|\Omega|, \\
\end{align*}
which implies
\begin{align}\label{u^q to u^p journey}
\sum_{\ell=1}^M \beta_{\ell} \|u_m(t)\|_{L^{q_{\ell}}}^{q_{\ell}}&\leq \frac{\alpha}{2}\|u_m(t)\|^p_{L^p}+ \sum_{\ell=1}^M |\beta_{\ell}| \bigg(\frac{p-q_{\ell}}{p}\bigg)\bigg(\frac{2M|\beta_{\ell}| q_{\ell}}{\alpha p}\bigg)^{\frac{q_{\ell}}{p-q_{\ell}}}|\Omega|.
\end{align}
For simplicity of notation, let us take 
 \begin{align}\label{C^*} 
 	C^*=\sum_{\ell=1}^M |\beta_{\ell}| \bigg(\frac{p-q_{\ell}}{p}\bigg)\bigg(\frac{2M|\beta_{\ell}| q_{\ell}}{\alpha p}\bigg)^{\frac{q_{\ell}}{p-q_{\ell}}}.
\end{align}  
Now, combining these estimates, we arrive at
\begin{equation*}
\frac{1}{2} \frac{d}{dt} \|u_m(t)\|^2_{L^2}+ \frac{\nu}{2} \|\nabla u_m(t)||^2_{L^2}+\frac{\alpha}{2} \|u_m(t)\|^p_{L^p}\leq \frac{1}{2\nu}\|f(t)\|^2_{H^{-1}}+C^*|\Omega|,
\end{equation*}
for a.e. $t\in[0,T]$. 
Performing integration from $0$ to $t$, taking the supremum over $t\in [0,T]$, and using the fact that $\|u_m(0)\|_{L^2}\leq \|u_0\|_{L^2}$, one obtains 
\begin{align} \label{damped L2-estimate}
&\sup_{0\leq t\leq T}\|u_m(t)\|^2_{L^2}+ \nu \int_0^T \|\nabla u_m(t)\|^2_{L^2}\, dt + \alpha \int_0^T \|u_m(t)\|^p_{L^p} \,dt \nonumber\\
& \leq \|u_0\|^2_{L^2}+ \frac{1}{\nu} \int_0^T \|f(t)\|^2_{H^{-1}}\, dt+C^*|\Omega|T,
\end{align}
where the right-hand side is independent of $m$.
\vspace{0.2 cm} \\
\textbf{Step 3:} \textit{Time derivative estimate}.  For any $v \in D(A^{\frac{s}{2}})$, for $s\geq \frac{d(p-2)}{2p}$ such that $D(A^{s/2}) \hookrightarrow L^p(\Omega)$ with $\|v\|_{D(A^{\frac{s}{2}})} \leq 1 $, we write $v = v_1+v_2$, where $v_1 \in V_m$ and $(v_2,w_k)=0,$ for all $k$ such that $\lambda_k<2^{m+1}$. The functions $\{w_k\}_{k=0}^\infty\in D(A^{\frac{s}{2}})$ are orthonormal in $L^2(\Omega)$, $\|v_1\|_{D(A^{\frac{s}{2}})}\leq \|v\|_{D(A^{\frac{s}{2}})} \leq 1$. From \eqref{damped Galerkin projection}, we deduce for a.e. $0\leq t\leq T$,
\begin{align*}
\langle \partial_t u_m(t),v_1 (t)\rangle &= -\nu(\nabla u_m(t),\nabla v_1(t))-\alpha(|u_m(t)|^{p-2}u_m(t),v_1(t))\\ &\quad+\sum_{\ell=1}^M\beta_{\ell}(|u_m(t)|^{q_{\ell}-2}u_m(t),v_1(t)) +\langle f(t),v_1(t) \rangle.
\end{align*}
Using the Cauchy-Schwarz inequality, we have 
\begin{align}\label{damped derivative estimate - 01}
|\langle \partial_t u_m(t),v_1(t) \rangle| &\leq \nu \|\nabla u_m(t)\|_{L^2} \|\nabla v_1(t)\|_{L^2}+ \alpha\|u_m(t)\|^{p-1}_{L^{p}} \|v_1(t)\|_{L^p}\nonumber \\&  \quad +\sum_{\ell=1}^M| \beta_{\ell}| \|u_m(t)\|^{q_{\ell}-1}_{L^{q_{\ell}}}\|v_1\|_{L^{q_{\ell}}}+\|f(t)\|_{H^{-1}} \|\nabla v_1(t)\|_{L^2}.
\end{align}
By an application of H\"{o}lder's and Young's inequalities, we attain 
\begin{align*}
\sum_{\ell=1}^M |\beta_{\ell}|\|u_m(t)\|^{q_{\ell}-1}_{L^{q_{\ell}}}&\leq \sum_{\ell=1}^M |\beta_{\ell}| \|u_m(t)\|^{q_{\ell}-1}_{L^p} |\Omega|^{\frac{(q_{\ell}-1)(p-q_{\ell})}{pq_{\ell}}}\leq \alpha\|u_m(t)\|^{p-1}_{L^p}+ C_1,
\end{align*}
where \begin{align*}
C_1= \sum_{\ell=1}^M |\beta_{\ell}|\bigg(\frac{p-q_{\ell}}{p-1}\bigg)\bigg(\frac{M|\beta_{\ell}|(q_{\ell}-1)}{\alpha  (p-1)}\bigg)^{\frac{q_{\ell}-1}{p-q_{\ell}}}|\Omega|^{\frac{(q_{\ell}-1)(p-1)}{pq_{\ell}}}.
\end{align*}
Hence, using this estimate in \eqref{damped derivative estimate - 01} and taking the supremum over all $v\in D(A^{\frac{s}{2}})$, we obtain
\begin{equation}\label{damped derivative estimate - 02}
\|\partial_t u_m(t)\|_{(D(A^{\frac{s}{2}}))^{*}} \leq  C\bigg(\nu \|\nabla u_m(t)\|_{L^2}+\|f(t)\|_{H^{-1}} + 2\alpha \|u_m(t)\|^{p-1}_{L^{p}}+C_1\bigg),
\end{equation}
for a.e. $t\in[0,T]$. Integrating from $0$ to $T$ and using Young's inequality, we end up with
\begin{align}\label{damped time derivative}
    \int_0^T \|\partial_t u_m(t)\|^{p^{\prime}}_{(D(A^{\frac{s}{2}}))^{*}}\, dt &\leq  C\bigg\{\nu ^{p^{\prime}} \int_0^T\|\nabla u_m(t)\|^{{p^{\prime}}}_{L^2}\, dt  +\int_0^T \|f(t)\|^{{p^{\prime}}}_{H^{-1}}\,dt \nonumber\\ &\quad+ (2\alpha)^{{p^{\prime}}}\int_0^T\|u_m(t)\|^p_{L^{p}}\, dt + C_1T\bigg\}, \\
    &\leq C(\|u_0\|_{L^2},\nu,\alpha,C_1,T,\|f(t)\|_{L^2(0,T;H^{-1})}),\nonumber
    \end{align}
where in the last inequality, we have used \eqref{damped L2-estimate} along with the fact
that
\begin{align*}
\int_{0}^T \|f(t)\|^{p^{\prime}}_{H^{-1}}dt \leq T^{\frac{(p-2)}{2(p-1)}} \bigg( \int_0^T \|f(t)\|_{H^{-1}}^2 dt \bigg)^{\frac{p}{2(p-1)}}.
\end{align*}
\vspace{0.2 cm}\\
\textbf{Step 4:}\textit{ Passing to limit.} Since the energy estimates \eqref{damped L2-estimate} and \eqref{damped time derivative} are uniformly bounded and do not depend on $m$, we can use the Banach-Alaoglu Theorem \cite[Theorem 3.25, Chapter 3]{MR2759829} to extract a subsequence $\{u_{m_j}\}_{j=1}^\infty$ of $\{u_m\}_{m=1}^\infty$ with the property that 
\begin{equation}
\begin{cases}
u_{m_j} \overset{w^*}{\to} u \quad \text{in }\  L^\infty(0,T; L^2(\Omega)), \\
u_{m_j} \overset{w}{\to} u \quad \text{in }\  L^2(0,T; H^1_0(\Omega)), \\
u_{m_j} \overset{w}{\to} u \quad \text{in } \ L^{p}(0,T; L^{p}(\Omega)), \\
\partial_t u_{m_j} \overset{w}{\to} \partial_t u \quad \text{in }\  \partial_t u \in L^{p^{\prime}}(0,T;(D(A^{\frac{s}{2}}))^{*}),\\
|u_m|^{p-2}u_m \overset{w}{\to} \xi \text{ in }\  L^{p^{\prime}}(0,T;L^{p^{\prime}}(\Omega)),
\end{cases}
\end{equation}
as $j \to \infty$ and for every $p \geq 2$. Using the fact that $(D(A^{\frac{s}{2}}))^{*}$ is a Banach space such that $D(A^{\frac{s}{2}}) \hookrightarrow H_0^1(\Omega) \cap L^p(\Omega)\hookrightarrow H_0^1(\Omega)\hookrightarrow L^2(\Omega)\hookrightarrow H^{-1}(\Omega)\hookrightarrow (D(A^{\frac{s}{2}}))^{*}$, and applying the Aubin-Lions compactness lemma \cite{MR916688}, we have the following strong convergence:
\begin{align*}
u_{m_j} \to u  \quad \text{ in } \ L^2(0,T;L^2(\Omega)),
\end{align*}
as $j \to \infty $. The strong convergence further implies that $u_{m_j}(t) \to u(t)$ for a.e. $t\in [0,T]$ in $L^2(\Omega)$ as $j \to \infty$ and $u_{m_{k_{j}}}(x,t) \to u(x,t)$, for a.e. $(x,t) \in [0,T] \times \Omega$. 
By an application of \cite[Lemma 1.3]{MR259693}, we can prove $\xi = |u|^{p-2}u$. Hence, we have 
\begin{equation}\label{eqn-wean-nonlinear}
|u_m|^{p-2}u_m \overset{w}{\to} |u|^{p-2}u \text{ in } L^{p^{\prime}}(0,T;L^{p^{\prime}}(\Omega)).
\end{equation}
\vspace{0.2cm}
\textbf{Step 5:}\textit{ Initial condition.} Let us choose a function $v \in C^1([0,T];H_0^1(\Omega)\cap L^p(\Omega))$ of the form \begin{align}\label{eqn-dense}\displaystyle v(t)=\sum_{k=1}^Nd_m^k(t)w_k,\end{align} where $\{d_m^k(t)\}_{k=1}^N$ are given smooth functions and for fixed integer $N$. For $m \geq N$, multiplying \eqref{damped Galerkin projection} by $d_m^k(\cdot)$ then summing over all $k$ such that $\lambda_k<2^{m+1}$ and integrating from $0 $ to $T$, we can see that
\begin{align}\label{damped weak converge}
\int_0^T \bigg[ \langle \partial_t u_m, v(t)\rangle+ &\nu (\nabla u_m(t), \nabla v(t))+ \alpha (|u_m(t)|^{p-2}u_m(t), v(t)) \nonumber\\ &-\sum_{\ell=1}^M \beta_{\ell} (|u_m(t)|^{q_{\ell}-2}u_m(t), v(t))\bigg]\,dt = \int_0^T \langle f(t),v(t)\rangle dt .
\end{align}
By setting $m=m_j$ and applying the weak limits, we conclude that 
\begin{align}\label{damped convergence}
    \int_0^T \bigg[ \langle \partial_t u(t), v(t)\rangle+ &\nu (\nabla u(t), \nabla v(t))+ \alpha (|u(t)|^{p-2}u(t), v(t))\nonumber \\&-\sum_{\ell=1}^M \beta_{\ell} (|u(t)|^{q_{\ell}-2}u(t), v(t)) \bigg]dt = \int_0^T \langle f(t),v(t)\rangle \,dt .
    \end{align}
This equality then holds true for all functions $v\in L^2(0,T;H_0^1(\Omega)\cap L^p(\Omega))$, as functions of the form \eqref{eqn-dense} are dense in this space. This completes the proof of convergence.

 It remains to show that $u(0)=u_0$. By \eqref{damped convergence}, we find  
\begin{align}\label{initial data - 01}
 \int_0^T &\bigg[ -\langle  u(t), \partial_t v(t)\rangle + \nu (\nabla u(t), \nabla v(t))  + \alpha (|u(t)|^{p-2}u(t), v(t)) \nonumber\\&-\sum_{\ell=1}^M \beta_{\ell} (|u(t)|^{q_{\ell}-2}u(t), v(t)) \bigg]dt = \int_0^T \langle f(t),v(t)\rangle\, dt + (u(0),v(0)),
 \end{align}
for each $v \in C^1([0,T];H_0^1(\Omega)\cap L^p(\Omega)) $ with $v(T)=0$. Similarly, from \eqref{damped weak converge}, we can infer that
\begin{align*}
 & \int_0^T \bigg[ -\langle  u_m(t), \partial_t v(t)\rangle+ \nu (\nabla u_m(t), \nabla v(t))  + \alpha (|u_m(t)|^{p-2}u_m(t), v(t)) \\&\quad-\sum_{\ell=1}^M \beta_{\ell} (|u_m(t)|^{q_{\ell}-2}u_m(t), v(t))\bigg]\,dt= \int_0^T \langle f(t),v(t)\rangle\, dt + (u_m(0),v(0)).
 \end{align*}
Now, setting $m=m_j$, applying the convergence and passing to limits, we get 
\begin{align}\label{initial data - 02}
 \int_0^T &\bigg[ -\langle  u(t), \partial_t v(t)\rangle+ \nu (\nabla u(t), \nabla v(t))+ \alpha (|u(t)|^{p-2}u(t), v(t)) \nonumber\\&-\sum_{\ell=1}^M \beta_{\ell} (|u_m(t)|^{q_{\ell}-2}u_m(t), v(t)) \bigg]\,dt = \int_0^T \langle f(t),v(t)\rangle\, dt + (u_0,v(0)).
 \end{align}
Since $u_m(0) \to u_0$ in $L^2(\Omega)$ as $m \to \infty $. So, from \eqref{initial data - 01} and \eqref{initial data - 02}, we have $(u(0),v(0)) =(u_0,v(0))$ for all arbitrary $v(0)$, therefore $u(0) = u_0\in L^2(\Omega).$
\vspace{0.2 cm}\\
\textbf{Step 6:} Let us now prove that $\partial_t u \in L^2(0,T;H^{-1}(\Omega))\cap L^{p^\prime}(0,T;L^{p^\prime}(\Omega))$. For any $v \in L^2(0,T;H_0^1(\Omega))\cap L^p(0,T;L^p(\Omega))$, using the Cauchy-Schwarz inequality, we infer 
\begin{align*}
|\langle \partial_t u(t),v(t) \rangle |&\leq  \nu| \langle \Delta u(t),v(t) \rangle| - \alpha |\langle|u(t)|^{p-2}u(t),v(t)\rangle|+\sum_{\ell=1}^M| \beta_{\ell}| | \langle |u(t)|^{q_{\ell}-2}u(t),v(t)\rangle| \\&\quad+| \langle f(t) , v(t) \rangle|\\
& \leq \|\nabla u(t)\|_{L^2}\|\nabla v(t)\|_{L^2}+ \alpha \|u(t)\|^{p-1}_{L^p}\|v(t)\|_{L^p} + \sum_{\ell=1}^M |\beta_{\ell}|\|u(t)\|^{q_{\ell}-1}_{L^{q_{\ell}}}\|v(t)\|_{L^{q_{\ell}}}\\&\quad +\|f(t)\|_{H^{-1}}\|\nabla v(t)\|_{L^2}.
\end{align*}
Integrating from $0$ to $T$, the above expression implies
\begin{align*}
\int_0^T |\langle \partial_t u(t), v(t) \rangle|\,dt &\leq \int_0^T   \|\nabla u(t)\|_{L^2}\|\nabla v(t)\|_{L^2} \,dt+\alpha \int_0^T \|u(t)\|^{p-1}_{L^p}\|v(t)\|_{L^p}\, dt \\& \quad + \sum_{\ell=1}^M |\beta_{\ell}|\int_0^T \|u(t)\|^{q_{\ell}-1}_{L^{q_{\ell}}}\|v(t)\|_{L^{q_{\ell}}}\, dt+\int_0^T \|f(t)\|_{H^{-1}}\|\nabla v(t)\|_{L^2} \, dt , 
\end{align*}
where we can bound the term involving $q_{\ell}$ by using \eqref{u^q to u^p journey} and H\"{o}lder's inequality as follows:
\begin{align*}
&\sum_{\ell=1}^M |\beta_{\ell}| \int_0^T \|u(t)\|^{q_{\ell}-1}_{L^{q_{\ell}}}\|v(t)\|_{L^{q_{\ell}}}\, dt\\ & \leq \sum_{\ell=1}^M |\beta_{\ell}| \int_0^T \|u(t)\|^{q_{\ell}-1}_{L^p} \|v(t)\|_{L^p} |\Omega|^{\frac{p-q_{\ell}}{p}} \, dt \leq \sum_{\ell=1}^M |\beta_{\ell}| |\Omega|^{\frac{p-q_{\ell}}{p}}\int_0^T \|u(t)\|^{q_{\ell}-1}_{L^p} \|v(t)\|_{L^p} \, dt \\& \leq  \sum_{\ell=1}^M |\beta_{\ell}| |\Omega|^{\frac{p-q_{\ell}}{p}} \bigg(\int_0^T \|u(t)\|_{L^p}^{\frac{p(q_{\ell}-1)}{p-1}}\, dt\bigg)^{\frac{p-1}{p}}\bigg(\int_0^T \|v(t)\|^p_{L^p} \, dt \bigg)^{\frac{1}{p}}\\& \leq \sum_{\ell=1}^M |\beta_{\ell}| |\Omega|^{\frac{p-q_{\ell}}{p}} \Bigg[\bigg(\int_0^T\|u(t)\|^p_{L^p}\, dt \bigg)^{\frac{q_{\ell}-1}{p-1}}T^{\frac{p-q_{\ell}}{p-1}}\Bigg]^{\frac{p-1}{p}}\|v\|_{L^p(0,T;L^p)}\\& \leq 
\sum_{\ell=1}^M |\beta_{\ell}| |\Omega|^{\frac{p-q_{\ell}}{p}} T^{\frac{p-q_{\ell}}{p}}\|u\|^{q_{\ell}-1}_{L^p(0,T;L^p)}\|v\|_{L^p(0,T;L^p)}.
\end{align*}
Thus, we arrive at
{\small\begin{align*}
&	\int_0^T |\langle \partial_t u(t), v(t) \rangle|\,dt 
\nonumber\\& \leq \bigg(\int_0^T \| \nabla u(t)\|^2_{L^2} \, dt\bigg)^{\frac{1}{2}}
	\bigg(\int_0^T \| \nabla v(t)\|^2_{L^2} \, dt\bigg)^{\frac{1}{2}} + \alpha \bigg(\int_0^T \|  u(t)\|^p_{L^p} \, dt\bigg)^{\frac{p-1}{p}} 
	\bigg(\int_0^T \|  v(t)\|^p_{L^p} \, dt\bigg)^{\frac{1}{p}} \\
	&\quad + \sum_{\ell=1}^M |\beta_{\ell}| |\Omega|^{\frac{p-q_{\ell}}{p}} 
	T^{\frac{p-q_{\ell}}{p}}
	\|u\|^{q_{\ell}-1}_{L^p(0,T;L^p)} \|v\|_{L^p(0,T;L^p)}+ \|f\|_{L^2(0,T;H^{-1})}\|v\|_{L^2(0,t;H_0^1)} \\
	& \leq \bigg( \|u\|_{L^2(0,T;H_0^1)}
	+ \|u\|^{p-1}_{L^p(0,T;L^p)}+\|f\|_{L^2(0,T;H^{-1})}  + \sum_{\ell=1}^M |\beta_{\ell}| |\Omega|^{\frac{p-q_{\ell}}{p}}
	T^{\frac{p-q_{\ell}}{p}} \|u\|^{q_{\ell}-1}_{L^p(0,T;L^p)} \bigg) \\
	&\quad \times \big(\|v\|_{L^2(0,T;H_0^1)}+\|v\|_{L^p(0,T;L^p)}\big).
\end{align*}}
Hence, from the above estimate, it follows that
\begin{align*}
\|\partial_t u \|_{L^2(0,T;H^{-1})\cap L^{p^{\prime}}(0,T;L^{p^{\prime}})}&\leq \|u\|_{L^2(0,T;H_0^1)}+\|u\|^{p-1}_{L^p(0,T;L^p)}+\|f\|_{L^2(0,T;H^{-1})} \\&\quad + \sum_{\ell=1}^M \beta_{\ell} |\Omega|^{\frac{p-q_{\ell}}{p}}T^{\frac{p-q_{\ell}}{p}}\|u\|^{q_{\ell}-1}_{L^p(0,T;L^p)}.
\end{align*}
Therefore, $\
\partial_t u \in L^2(0,T;H^{-1}(\Omega))\cap L^{p^{\prime}}(0,T;L^{p^{\prime}}(\Omega)).$
\vspace{0.2 cm}\\
\textbf{Step 7:} Using \eqref{damped L2-estimate} and weakly lower semicontinuity property of norms, we have 
\begin{align*}
&\sup_{0\leq t \leq T} \|u(t)\|^2_{L^2}+ \int_0^T \|u(t)\|^2_{H_0^1}\, dt + \int_0^T \|u(t)\|^p_{L^p}\, dt\\ &\leq \liminf_{m \to \infty} \sup_{0\leq t\leq T}\|u_m(t)\|^2_{L^2}+\liminf_{m \to \infty} \int_0^T \|u_m(t)\|^2_{H_0^1}\, dt + \liminf_{m \to \infty} \int_0^T \|u_m(t)\|^p_{L^p}\, dt\\&\leq \liminf_{m \to \infty} \bigg(\sup_{0\leq t\leq T}\|u_m(t)\|^2_{L^2}+\int_0^T \|u_m(t)\|^2_{H_0^1}\, dt +  \int_0^T \|u_m(t)\|^p_{L^p}\, dt\bigg) \\ & \leq  \|u_0\|^2_{L^2}+ \frac{1}{\nu} \int_0^T \|f(t)\|^2_{H^{-1}}\, dt+C^*|\Omega|T,
\end{align*}
where $C^*$ is defined in \eqref{C^*}. The following energy equality follows as a direct consequence of the absolute continuity lemma:

\noindent \textit{Energy equality:} As $ u \in L^\infty(0,T;L^2(\Omega))\cap L^{2}(0,T;H_0^1(\Omega)) \cap L^{p}(0,T;L^{p}(\Omega))$ with its weak derivative $\partial_t u \in L^2(0,T;H^{-1}(\Omega))\cap L^{p^\prime}(0,T;L^{p^{\prime}}(\Omega))$, using an argument similar to \cite[Theorem 7.2, Exercise 8.2]{MR1881888}, we have that the mapping $[0,T]\ni t\mapsto\|u(t)\|_{L^2}^2\in\mathbb{R}$ is absolutely continuous, $u \in C([0,T];L^2(\Omega))$ and the following energy equality holds true: 
\begin{align*}
&\|u(t) \|^2_{L^2} + 2\nu \int_0^T \|\nabla u(t)\|^2_{L^2}\,dt + 2\alpha \int_0^T \|u(t)\|^p_{L^p}\, dt \\&= \|u_0\|^2_{L^2} +2\sum_{\ell=1}^M \beta_{\ell} \int_0^T\|u(t)\|^{q_{\ell}}_{L^{q_{\ell}}}\, dt+ 2\int_0^T \langle f(t), u(t) \rangle\, dt,
\end{align*}
for all $ t \in [0,T]$. 
This completes the proof of the existence of  weak solution.
\end{proof}

\begin{Theorem}[Uniqueness of weak solution]\label{damped uniqueness}For $u_0 \in L^2(\Omega)$ and $f \in L^2(0,T;H^{-1}(\Omega))$, the weak solution of the system \eqref{Damped Heat} as discussed in Theorem \ref{existenceweak} is unique.
\end{Theorem}
\begin{proof}
    Let $u_1(\cdot) \text{ and } u_2(\cdot)$ bet two weak solutions of \eqref{Damped Heat} both subject to same forcing $f(\cdot)$ and initial data $u_0$. Then, $w:=u_1 - u_2 $ satisfies the following in the weak sense:
    \begin{equation}\label{damped unique}
        \begin{cases}
        \begin{aligned}
        &\frac{\partial w(x,t)}{\partial t} - \nu \Delta w(x,t)+\alpha (|u_1(x,t)|^{p-2}u_1(x,t)-|u_2(x,t)|^{p-2}u_2(x,t))\\ &\quad-\sum_{\ell=1}^M\beta_{\ell} (|u_1(x,t)|^{q_{\ell}-2}u_1(x,t)-|u_2(x,t)|^{q_{\ell}-2}u_2(x,t))=0 \text{ for all }(x,t) \in \Omega \times (0,T),  
        \end{aligned}\\
        w(x,t)=0, \, \text{ for all } x \in \partial \Omega, \, \, t \in (0,T),\\
        w(x,0) = 0 \text{ for all }x\in \Omega.
        \end{cases}
    \end{equation}
Using the absolute continuity of the mapping $[0,T]\ni t\mapsto\|u(t)\|_{L^2}^2\in\mathbb{R}$, taking the inner product with $w$, we infer 
    \begin{align}\label{Damped unique - 01}
        \frac{1}{2} \frac{d}{dt} \|w(t)\|^2_{L^2}+ \nu \|\nabla w(t)\|^2_{L^2} &= -\alpha (|u_1(t)|^{p-2}u_1(t)-|u_2(t)|^{p-2}u_2(t), w(t)) \nonumber\\
        &\quad +\sum_{\ell=1}^M \beta_{\ell} (|u_1(t)|^{q_{\ell}-2}u_1(t)-|u_2(t)|^{q_{\ell}-2}u_2(t),w(t)),
        \end{align}
for a.e. $t\in[0,T]$.We begin by estimating the nonlinear term $-\alpha(|u_1|^{p-2}u_1-|u_2|^{p-2}u_2,w)$, following the approach in \cite[Section 3.1]{MR4251864} as
    \begin{align}\label{eqn-ces}
            -\alpha (|u_1|^{p-2}u_1-|u_2|^{p-2}u_2,w)
            &\leq -\frac{\alpha}{2}\||u_1|^{\frac{p-2}{2}}w\|^2_{L^2}-\frac{\alpha}{2}\||u_2|^{\frac{p-2}{2}}w\|^2_{L^2}.
        \end{align}
    We also estimate the pumping term as
    \begin{align*}
    (|u_1|^{q_{\ell}-2}u_1-|u_2|^{q_{\ell}-2}u_2,w)&\leq (q_{\ell}-1)(|u_1|+|u_2|)^{q_{\ell}-2}|u_1-u_2|, |u_1-u_2|) \\
    & \leq 2^{q_{\ell}-3}(q_{\ell}-1)((|u_1|^{q_{\ell}-2}+|u_2|^{q_{\ell}-2})|u_1-u_2|, |u_1-u_2|)\\
    &\leq 2^{q_{\ell}-3}(q_{\ell}-1)(|u_1|^{q_{\ell}-2}|u_1-u_2|, |u_1-u_2|)\\
    &\quad+2^{q_{\ell}-3}(q_{\ell}-1)(|u_2|^{q_{\ell}-2}|u_1-u_2|, |u_1-u_2|)\\
    &=  \tilde{C}(I_1+ I_2),
    \end{align*}
    where 
    \begin{align*}
    	\mbox{$\tilde{C}=2^{q_{\ell}-3}(q_{\ell}-1)$, $I_1=(|u_1|^{q_{\ell}-2}|u_1-u_2|, |u_1-u_2|)$ and $I_2=(|u_2|^{q_{\ell}-2}|u_1-u_2|, |u_1-u_2|)$.}
    \end{align*}
     Using the fact that $\frac{q_{\ell}-2}{p-2}+\frac{p-q_{\ell}}{p-2}=1$ and by an application of H\"{o}lder's and Young's inequalities, we can estimate $I_1$ and $I_2$ as
    \begin{align*}
   |I_1|&=|(|u_1|^{q_{\ell}-2}|u_1-u_2|^{\frac{2(q_{\ell}-2)}{p-2}}, |u_1-u_2|^{\frac{2(p-q_{\ell})}{p-2}})|\\&\leq (\||u_1|^\frac{p-2}{2}|u_1-u_2|\|^2_{L^2})^{\frac{q_{\ell}-2}{p-2}}(\|u_1-u_2\|^2_{L^2})^{\frac{p-q_{\ell}}{p-2}}\\
    &\leq \frac{\alpha}{2^{q_{\ell}-1}M|\beta_{\ell}|(q_{\ell}-1)}\||u_1|^\frac{p-2}{2}|u_1-u_2|\|^2_{L^2}\\&\quad + \bigg(\frac{2^{q_{\ell}-1}M|\beta_{\ell}|(q_{\ell}-1)(q_{\ell}-2)}{\alpha(p-2)}\bigg)^{\frac{q_{\ell}-2}{p-q_{\ell}}}\bigg(\frac{p-q_{\ell}}{p-2}\bigg)\|u_1-u_2\|^2_{L^2}\\
    |I_2| &\leq \frac{\alpha}{2^{q_{\ell}-1}M|\beta_{\ell}|(q_{\ell}-1)}\||u_2|^\frac{p-2}{2}|u_1-u_2|\|^2_{L^2}\\&\quad + \bigg(\frac{2^{q_{\ell}-1}M|\beta_{\ell}|(q_{\ell}-1)(q_{\ell}-2)}{\alpha(p-2)}\bigg)^{\frac{q_{\ell}-2}{p-q_{\ell}}}\bigg(\frac{p-q_{\ell}}{p-2}\bigg)\|u_1-u_2\|^2_{L^2}.
    \end{align*}
    Combining the above estimates, we arrive at
    \begin{align*}
      &   \frac{1}{2} \frac{d}{dt} \|w(t)\|^2_{L^2}+ \nu \|\nabla w(t)\|^2_{L^2} \\&\leq -\frac{\alpha}{4}\||u_1(t)|^{\frac{p-2}{2}}w(t)\|^2_{L^2}-\frac{\alpha}{4}\||u_2(t)\|^{\frac{p-2}{2}}w(t)\|^2_{L^2}\\&\quad +2\bigg[\sum_{\ell=1}^M \tilde{C}|\beta_{\ell}| \bigg(\frac{2^{q_{\ell}-1}M|\beta_{\ell}|(q_{\ell}-1)(q_{\ell}-2)}{\alpha(p-2)}\bigg)^{\frac{q_{\ell}-2}{p-q_{\ell}}}\bigg(\frac{p-q_{\ell}}{p-2}\bigg)\bigg]\|w\|^2_{L^2},
    \end{align*}
for a.e. $t\in[0,T]$.     For simplicity, let us assume \begin{align}\label{C_2}
    	\sum_{\ell=1}^M \tilde{C}|\beta_{\ell}| \bigg(\frac{2^{q_{\ell}-1}M|\beta_{\ell}|(q_{\ell}-1)(q_{\ell}-2)}{\alpha(p-2)}\bigg)^{\frac{q_{\ell}-2}{p-q_{\ell}}}\bigg(\frac{p-q_{\ell}}{p-2}\bigg)=C_2.\end{align} Therefore, the above estimate can be simplified as
    \begin{equation*}
     \frac{d}{dt} \|w(t)\|^2_{L^2}+ 2\nu \|\nabla w(t)\|^2_{L^2} \leq 4C_2\|(t)w\|^2_{L^2}, \ \text{ for a.e. }\ t\in[0,T].
    \end{equation*}
    Using Gronwall's inequality \cite[p. 624]{MR1625845}, we see that
    \begin{equation*}
    \|w(t)\|^2_{L^2} \leq \|w(0)\|^2_{L^2}\exp(4C_2 T)=0.
    \end{equation*}
    Since, the exponential term is bounded, therefore, $w(t)=0$ for a.e. $t\in [0,T]$ and hence $ w= u_1-u_2 = 0 $ so $u_1=u_2$ in $L^2(\Omega)$. Therefore, the  weak solution of the system \eqref{Damped Heat} is unique.
\end{proof}

\subsection{Regularity results}\label{regularity results}
In this section, we characterize the solution space under suitable assumptions on the regularity of initial condition and the external force.
\begin{Theorem}\label{regularity f in L2}
For the  initial data $u_0 \in H_0^1(\Omega)\cap L^p(\Omega),$ $ p\geq 2$  and $f \in L^2(0,T;L^2(\Omega))$, the weak solution of the system \eqref{Damped Heat}, $u$ in $\Omega$ (either convex, or a domain with $C^2$-boundary), is a \emph{strong solution} satisfying
\begin{equation*}
u \in C([0,T];H_0^1(\Omega)\cap L^p(\Omega)) \cap L^2(0,T;H^2(\Omega)) \cap L^{2p-2}(0,T;L^{2p-2}(\Omega)),
\end{equation*}
with $\partial _tu \in L^2(0,T;L^2(\Omega))$ and 
\begin{equation}
    \begin{cases}
    	\displaystyle
        \partial_t u(t) - \nu \Delta u(t)+\alpha |u(t)|^{p-2}u(t)-\sum_{\ell=1}^M\beta_{\ell} |u(t)|^{q_{\ell}-2}u(t) = f(t), \quad \text{ in }\  L^2(\Omega),\\
        u(0)=u_0 \in H_0^1(\Omega),\\
    \end{cases}
\end{equation}
for a.e. $t \in [0,T].$
\end{Theorem}
\begin{proof}
    \textbf{Step-1: }\textit{$H_0^1$- energy estimate}. Multiplying by $\lambda_kd_m^k(\cdot)$ in \eqref{damped Galerkin projection} and summing over all $k$ such that $\lambda_k<2^{m+1}$, we have 
    \begin{align*}
  -( \partial_t u_m(t),\Delta u_m(t))+&\nu(\Delta u_m(t), \Delta u_m(t))  +\alpha(|u_m(t)|^{p-2}u_m(t),-\Delta u_m(t))\\ &= 
  \sum_{\ell=1}^M \beta_{\ell} (|u_m(t)|^{q_{\ell}-2}u_m(t), -\Delta u_m(t))+(f(t), -\Delta u_m(t)),
\end{align*}
for a.e. $t\in[0,T]$. Therefore, it is immediate that 
\begin{align*}
  \frac{1}{2}\frac{d}{dt}\|\nabla u_m(t)\|^2_{L^2}+&\nu\|\Delta u_m(t)\|^2_{L^2}  +\alpha(p-1)(|u_m(t)|^{p-2}\nabla u_m(t),\nabla u_m(t))\\ & =
  \sum_{\ell=1}^M \beta_{\ell}(q_{\ell}-1) (|u_m(t)|^{q_{\ell}-2}\nabla u_m(t), \nabla u_m(t))+(f(t), -\Delta u_m(t))\nonumber\\
         & \leq
  \bigg|\sum_{\ell=1}^M \beta_{\ell}(q_{\ell}-1) (|u_m(t)|^{q_{\ell}-2}\nabla u_m(t), \nabla u_m(t))\bigg|+\|f(t)\|_{L^2} \|\Delta u_m(t)\|_{L^2},
    \end{align*}
where we have used the Cauchy-Schwarz inequality.     The term $(|u_m|^{q_{\ell}-2}\nabla u_m, \nabla u_m)$ can be estimated using the same approach as in Theorem \ref{damped uniqueness} as
    \begin{align}\label{ql2,grad transfomration}
   &\bigg| \sum_{\ell=1}^M \beta_{\ell}(q_{\ell}-1) (|u_m|^{q_{\ell}-2}\nabla u_m, \nabla u_m)\bigg| \nonumber \\&\leq \bigg| \sum_{\ell=1}^M \beta_{\ell}(q_{\ell}-1)\bigg(\||u_m|^{\frac{p-2}{2}}\nabla u_m\|^2_{L^2}\bigg)^{\frac{q_{\ell}-2}{p-2}}\bigg(\|\nabla u_m\|^2_{L^2}\bigg)^{\frac{p-q_{\ell}}{p-2}}\bigg| \nonumber \\
    & \leq \frac{\alpha(p-1)}{2} \||u_m|^{\frac{p-2}{2}}\nabla u_m\|^2_{L^2} \nonumber\\&\quad+ \sum_{\ell=1}^M |\beta_{\ell}|(q_{\ell}-1)\bigg(\frac{2M|\beta_{\ell}|(q_{\ell}-1)(q_{\ell}-2)}{\alpha(p-1)(p-2)}\bigg)^\frac{q_{\ell}-2}{p-q_{\ell}}\bigg(\frac{p-q_{\ell}}{p-2}\bigg)\|\nabla u_m\|^2_{L^2}.
    \end{align}
    Using the above estimate and Young's inequality, we arrive at the following:
\begin{align*}
    &\frac{1}{2}\frac{d}{dt}\|\nabla u_m(t)\|^2_{L^2}+\nu\|\Delta u_m(t)\|^2_{L^2}
    + \frac{\alpha(p-1)}{2}\||u_m(t)|^{\frac{p-2}{2}} \nabla u_m(t)\|^2_{L^2} \\
    &\leq  \frac{1}{2\nu} \|f(t)\|^2_{L^2} 
   + \frac{\nu}{2} \|\Delta u_m(t)\|^2_{L^2} \\&\quad+\bigg[\sum_{\ell=1}^M |\beta_{\ell}|\bigg(\frac{2M|\beta_{\ell}|(q_{\ell}-1)(q_{\ell}-2)}{\alpha(p-1)(p-2)}\bigg)^\frac{q_{\ell}-2}{p-q_{\ell}}\bigg(\frac{p-q_{\ell}}{p-2}\bigg)\bigg]\|\nabla u_m(t)\|^2_{L^2},
\end{align*}
for a.e. $t\in[0,T]$. For the sake of simplicity, let us take 
\begin{align}\label{eqn-C3}
    \sum_{\ell=1}^M |\beta_{\ell}|\bigg(\frac{2M|\beta_{\ell}|(q_{\ell}-1)(q_{\ell}-2)}{\alpha(p-1)(p-2)}\bigg)^\frac{q_{\ell}-2}{p-q_{\ell}}\bigg(\frac{p-q_{\ell}}{p-2}\bigg)=C_3.
\end{align}
 Hence, we deduce for a.e. $t\in[0,T]$ that
    \begin{align*}
    &  \frac{d}{dt}\|\nabla u_m(t)\|^2_{L^2}+ \nu\|\Delta u_m(t)\|^2_{L^2} +\alpha(p-1)\||u_m(t)|^{\frac{p-2}{2}} \nabla u_m(t)\|^2_{L^2}\\& \leq \frac{1}{\nu} \|f(t)\|^2_{L^2}+2C_3\|\nabla u_m(t)\|^2_{L^2} .
    \end{align*}
    Integrating from $0$ to $t$ and applying Gronwall's inequality, we conclude that
    \begin{align}\label{regularity-err}
    \|\nabla u_m(t)\|^2_{L^2}&+ \nu \int_0^t \|\Delta u_m(s)\|^2_{L^2}\,ds +\alpha(p-1)\int_0^t\||u_m(s)|^{\frac{p-2}{2}}\nabla u_m(s)\|^2_{L^2}ds \nonumber \\& \leq \bigg(\frac{1}{\nu} \int_0^t \|f(s)\|^2_{L^2} ds +\|\nabla u_0\|^2_{L^2}\bigg)\exp\bigg( 2C_3t\bigg),
    \end{align}
    for all $t\in [0,T]$. By taking supremum over all $t,$ we confirm that \begin{equation*}u_m \in L^\infty(0,T;H_0^1(\Omega)).
    \end{equation*}
    Moreover, we have $\Delta u_m\in L^2(0,T;L^2(\Omega))$ with $u_m(t)\big|_{\partial\Omega}=0$ for all $t\in[0,T]$. By the elliptic regularity result \cite[Theorem 3.1.2.1]{MR3396210}, it follows that
    \[
    u_m \in L^2(0,T;H^2(\Omega)).
    \]

\noindent     \textbf{Step-2: }\textit{Time derivative estimate}. For $m\geq 1$, multiplying \eqref{damped Galerkin projection} by $d_m^k{'}(\cdot)$ and summing over all $k$ such that $\lambda_k<2^{m+1}$, we have
    \begin{align*}
    \|\partial_t u_m(t) \|^2_{L^2} &+ \frac{\nu}{2}\frac{d}{dt} \|\nabla u_m(t)\|^2_{L^2}  +\alpha(|u_m(t)|^{p-2}u_m(t), \partial_t u_m(t))\\&=(f(t),\partial_tu_m(t)) +\sum_{\ell=1}^M \beta_{\ell} (|u_m(t)|^{q_{\ell}-2}u_m(t), \partial_t u_m(t)),
    \end{align*}
for a.e. $t\in[0,T]$.    Using the Cauchy-Schwarz and Young's inequalities, we attain
    \begin{align*}
     \|\partial_t u_m(t) \|^2_{L^2} &+ \frac{\nu}{2}\frac{d}{dt} \|\nabla u_m(t)\|^2_{L^2}+\frac{\alpha}{p}\frac{d}{dt}\|u_m(t)\|^p_{L^p}\\&\leq \frac{1}{2}\|f(t)\|^2_{L^2}+ \frac{1}{2}\|\partial_t u_m(t)\|^2_{L^2} + \sum_{\ell=1}^M \frac{\beta_{\ell}}{q_{\ell}}\frac{d}{dt}\|u_m(t)\|^{q_{\ell}}_{L^{q_{\ell}}},
    \end{align*}
that is,
	\begin{align*}
     \|\partial_t u_m (t)\|^2_{L^2} + \nu \frac{d}{dt} \|\nabla u_m(t)\|^2_{L^2}+\frac{2\alpha}{p}\frac{d}{dt}\|u_m(t)\|^p_{L^p}&\leq \|f(t)\|^2_{L^2}+2\sum_{\ell=1}^M \frac{\beta_{\ell}}{q_{\ell}}\frac{d}{dt}\|u_m(t)\|^{q_{\ell}}_{L^{q_{\ell}}},
      \end{align*}
 for a.e. $t\in[0,T]$. Integrating from $0$ to $t$, we achieve
    \begin{align}\label{regularity time}
         &\nu \|\nabla u_m(t)\|^2_{L^2}+\frac{2\alpha}{p} \|u_m(t)\|^p_{L^p}+\int_0^t \|\partial_s u_m(s) \|^2_{L^2} \, ds + 2\sum_{\ell=1}^M \frac{\beta_{\ell}}{q_{\ell}}\|u_m(0)\|^{q_{\ell}}_{L^{q_{\ell}}}\nonumber \\& \qquad  \leq\nu \|\nabla  u_m(0)\|^2_{L^2}+\frac{2\alpha}{p}\|u_m(0)\|^p_{L^p}+ \int_0^t \|f(s)\|^2_{L^2}\, ds+ \sum_{\ell=1}^M \frac{\beta_{\ell}}{q_{\ell}}\|u_m(t)\|^{q_{\ell}}_{L^{q_{\ell}}}.
         \end{align}
The term $\|u_m(\cdot)\|_{L^{q_{\ell}}}^{q_{\ell}}$ can be estimated following the calculations in \eqref{u^q to u^p journey}. Using the fact that (see Proposition \ref{Prop-Sm})$$\|u_m(0)\|_{L^p} = \|S_{m-1} u_0\|_{L^p} \le \|u_0\|_{L^p} < \infty\ \text{ and }\ \|\nabla  u_m(0)\|_{L^2}\leq\|\nabla u_0\|_{L^2}<\infty$$ (and similarly for $\|u_m(0)\|_{L^{q_{\ell}}}$), it follows that the right-hand side of \eqref{regularity time} is independent of $m$.
This leads to 
\begin{align*}
   u_m \in L^{\infty}(0,T;L^p(\Omega)),\,  \partial_t u_m \in L^2(0,T;L^2(\Omega)).
\end{align*}

\noindent
\textbf{Step-3:} \emph{Existence of strong solutions.}
 Applications of the Banach-Alaoglu theorem and Aubin-Lions compactness Lemma as performed  in Steps 4 and 5 in the proof of Theorem \ref{existenceweak}, lead to
 \begin{equation}\label{eqn-conv-1}
 	\left\{
 \begin{aligned}
 	u_m &\xrightarrow{w^*} \tilde{u} \in L^{\infty}(0,T;H_0^1(\Omega)\cap L^p(\Omega)),\\  
 	u_m &\xrightarrow{w} \tilde{u} \in L^2(0,T;H^2(\Omega)),\\ 
 	\partial_t u_m&\xrightarrow{w}  \partial_t \tilde{u} \in L^2(0,T;L^2(\Omega)),\\
 	u_m&\to \tilde{u}\in C([0,T];L^2(\Omega))\cap L^2(0,T;L^2(\Omega)),
 \end{aligned}
 \right.
 \end{equation}
 along a further subsequence still denoted  by the same symbol. By the uniqueness of weak limits, we obtain $\tilde{u}=u$ and 
    \begin{align*}
    u \in L^\infty(0,T;H_0^1(\Omega)\cap L^p(\Omega)) \cap L^2(0,T;H^2(\Omega)) \ \text{ and } \
    { \partial_t u} \in L^2(0,T;L^2(\Omega)).
    \end{align*}
    Therefore, an application of the Lions-Magenes lemma (\cite[Lemma 1.2, p. 176]{Te_1997}) yields 
     $ u \in C([0,T];H_0^1(\Omega))$. The weakly lowersemicontinuity property  of norms yields
     \begin{align*}
     	&\sup_{0\leq t\leq T}\left(\|u(t)\|_{H_0^1}^2+\|u(t)\|_{L^p}^p\right)+\int_0^T\|\Delta u(t)\|_{L^2}^2dt+\int_0^T\|\partial_t u(t)\|_{L^2}^2dt \nonumber\\&\leq \liminf_{m\to\infty}\left(\sup_{0\leq t\leq T}\left(\|u_m(t)\|_{H_0^1}^2+\|u_m(t)\|_{L^p}^p\right)+\int_0^T\|\Delta u_m(t)\|_{L^2}^2dt+\int_0^T\|\partial_t u_m(t)\|_{L^2}^2dt\right)\nonumber\\&\leq C.
     \end{align*}
     We infer from  \eqref{damped abstract formulation} that for a.e. $t\in[0,T]$
     \begin{align*}
     	\alpha |u(t)|^{p-2}u(t)=-\partial_tu(t) +\nu \Delta u(t)+\sum_{\ell=1}^M\beta_{\ell} |u(t)|^{q_{\ell}-2}u(t) +f(t)\ \text{ in }\  H^{-1}(\Omega),
     \end{align*}
     so that 
     \begin{align*}
     	\alpha^2\int_0^T\||u(t)|^{p-2}u(t)\|_{L^2}^2dt&\leq C\bigg(\int_0^T\left\|\partial_tu(t)\right\|_{L^2}^2dt+\nu^2\int_0^T\|\Delta u(t)\|_{L^2}^2dt\nonumber\\&\quad+\sum_{\ell=1}^M|\beta_{\ell}|^2\int_0^T\| |u(t)|^{q_{\ell}-2}u(t) \|_{L^2}^2dt+\int_0^T\|f(t)\|_{L^2}^2dt\bigg).
     \end{align*}
     An application of H\"older's and Young's inequalities yield 
     \begin{align*}
     	\frac{\alpha^2}{2}\int_0^T\|u(t)\|_{L^{2p-2}}^{2p-2}dt&\leq C\bigg(\int_0^T\left\|\partial_tu(t)\right\|_{L^2}^2dt+\nu^2\int_0^T\|\Delta u(t)\|_{L^2}^2dt\nonumber\\&\quad+\sum_{\ell=1}^M\left(\frac{p-q_{\ell}}{p-1}\right)\left(\frac{2M(q_{\ell}-1)}{\alpha(p-1)}\right)^{\frac{q_{\ell}-1}{p-q_{\ell}}}|\beta_{\ell}|^{\frac{p-1}{p-q_{\ell}}}|\Omega|+\int_0^T\|f(t)\|_{L^2}^2dt\bigg)\nonumber\\&\leq C,
     \end{align*}
     so that $u\in L^{2p-2}(0,T;L^{2p-2}(\Omega))$. This guarantees that  the equation \eqref{eqn-strong-form} is satisfied for all $v\in L^2(0,T;L^2(\Omega))$. Moreover, the above estimate enables us to show that the mapping $[0,T]\ni t\mapsto\|u(t)\|_{L^p(\Omega)}^p\in\mathbb{R}$ is absolutely continuous, $u\in C([0,T];L^p(\Omega))$  and for a.e. $t\in[0,T]$  (\cite[Step 9, p. 30.]{AB+ZB+MTM-2025})
     \begin{align}\label{eqn-Lp-0}
     	\frac{d}{d t}\|u(t)\|_{L^p}^p=p\left\langle\frac{\partial u(t)}{\partial t},|u(t)|^{p-2}u(t)\right\rangle. 
     \end{align}

   Note that, by the Sobolev embedding Theorem, we have
   \begin{align*}
   	H_0^1(\Omega) \hookrightarrow L^p(\Omega) & \ \text{ for }\ d \geq 2 \ \text{ and }\ 2 \leq p \leq\frac{2d}{d-2};\\
   	H_0^1(\Omega) \hookrightarrow L^p(\Omega) & \ \text{ for }\ d =1,2  \ \text{ and }\ 2 \leq p < \infty.
   \end{align*}
   For the above two cases, since $u\in L^{\infty}(0,T; H_0^1(\Omega))\cap L^2(0,T;D(A))$ and $\frac{\partial u}{\partial t}\in L^{2}(0,T; L^{2}(\Omega))$, an application of Lions-Magenes lemma yields $u\in C([0,T];H_0^1(\Omega))\hookrightarrow C([0,T];L^p(\Omega))$.
   
   Therefore, for $d\geq 2,$ it is enough to show that $u \in \; C([0,T]; L^p(\Omega) \cap H_0^1(\Omega))$ for $\frac{2d}{d-2} <p<\infty$. Let us first fix $\frac{2d}{d-2} < p < \infty$ and define $z := |u|^{\frac{p}{2}}$. Then,  we observe that
   \begin{align*}
   	\|z\|_{L^2(\Omega)}^2 = \||u|^{\frac{p}{2}}\|_{L^2(\Omega)}^2 = \|u\|_{L^p(\Omega)}^p.
   \end{align*}
   By using the fact that $u\in L^{\infty}(0,T; L^p(\Omega) )$, we have
  $
   	z \in L^{\infty}(0,T; L^2(\Omega) ).
$
  We infer from \eqref{regularity-err} that 
  \begin{align*}
  	\int_0^T\||u(t)|^{\frac{p-2}{2}}\nabla u(t)\|^2_{L^2}dt\leq\liminf_{m\to\infty}	\int_0^T\||u_m(t)|^{\frac{p-2}{2}}\nabla u_m(t)\|^2_{L^2}dt\leq C.
  \end{align*}
  Since 
  \begin{align*}
  	\int_0^T\|\nabla z(t)\|_{L^2}^2dt=\left(\frac{p}{2}\right)^2\int_0^T\||u(t)|^{\frac{p-2}{2}}\nabla u(t)\|^2_{L^2}dt\leq C,
  \end{align*}
  we infer $z\in L^2(0,T;H_0^1(\Omega))$. 
   Since $\frac{\partial u}{\partial t} \in L^2(0,T; L^2(\Omega))$ and $u\in L^{2p-2}(0,T; L^{2p-2}(\Omega))$, it implies
   \begin{align*}
   	\int_0^T\left\|\frac{\partial z(t)}{\partial t}\right\|_{L^{\frac{4p-4}{3p-4}}(\Omega)}^{\frac{4p-4}{3p-4}}dt
   	&\leq \left(\frac{p}{2}\right)^{\frac{4p-4}{3p-4}}\bigg(\int_0^T \left\|\frac{\partial u(t)}{\partial t}\right\|_{L^2(\Omega)}^2dt\bigg)^{\frac{2p-2}{3p-4}}\left(\int_0^T\|u(t) \|_{L^{2p-2}(\Omega)}^{2p-2}dt\right)^{\frac{p-2}{3p-4}}\\
   	& = \left(\frac{p}{2}\right)^{\frac{4p-4}{3p-4}} \left\|\frac{\partial u}{\partial t}\right\|_{L^2(0,T; L^2(\Omega))}^{\frac{4p-4}{3p-4}}\|u \|_{L^{2p-2}(0,T; L^{2p-2}(\Omega))}^{\frac{(p-2)(2p-2)}{3p-4}}  < \infty.
   \end{align*}
 Therefore, we deduce $$z \in L^{\frac{4p-4}{p}}(0,T; L^{\frac{4p-4}{p}}(\Omega))  \cap L^2(0,T; H_0^1(\Omega)),$$   and since $f\in L^2(0,T; H^{-1}(\Omega))$, the above estimate implies   $$\frac{\partial z}{\partial t} \in L^{\frac{4p-4}{3p-4}}(0,T; L^{\frac{4p-4}{3p-4}}(\Omega))+ L^2(0,T; H^{-1}(\Omega)),$$
   As a consequence of \cite[Theorem 7.2, Exercise 8.2]{MR1881888}, we deduce that
   \begin{itemize}
   	\item[(i)] $z = |u|^{\frac{p}{2}}$ is in the space $C([0,T];L^2(\Omega))$, i.e., $|u| \in C([0,T];L^p(\Omega))$,
   	\item[(ii)] the function $[0,T] \ni t \mapsto  \|z(t)\|_{L^2}^2\in \mathbb{R}$ is absolutely continuous  and
   	\begin{align*}
   	\frac{1}{2}	\frac{d}{dt}\|z(t)\|_{L^2}^2
   		& = {}_{L^{\frac{4p-4}{p}} }{\left\langle z(t), z'(t)\right\rangle}_{L^{\frac{4p-4}{3p-4}}}.
   	\end{align*}
   \end{itemize}
   This implies that the function $[0,T] \ni t \mapsto \|u(t)\|_{L^p(\Omega)}^p \in \mathbb{R}$ is absolutely continuous and it satisfies
   \begin{align*}
   	\frac{1}{p}\frac{d}{dt}{\|u(t)\|_{L^p(\Omega)}^p}
   	& = \left( |u(t)|^{p-2}u(t), \frac{\partial u(t)}{\partial t}\right),
   \end{align*}
   for a.e. $t\in[0,T]$. Since $u\in L^{2p-2}(0,T;L^{2p-2}(\Omega))$, one can choose the test function  $v=|u|^{p-2}u$ in \eqref{eqn-strong-form} and use the above fact to obtain
   \begin{align*}
   	\|u(t)\|_{L^p(\Omega)}^p 
   	& + p(p-1)\nu\int_0^t\||u(s)|^{\frac{p-2}{2}}\nabla u(s)\|_{L^p(\Omega)}^pds+p\alpha\int_0^t\|u(s)\|_{L^{2p-2}(\Omega)}^{2p-2}ds\\
   	& =\|u_0\|_{L^p(\Omega)}^p +\sum_{\ell=1}^M\beta_{\ell}\int_0^t\|u(s)\|_{L^{p+q_{\ell}-2}}^{p+q_{\ell}-2}ds+\int_0^t(f(s),|u(s)|^{p-2}u(s))ds,
   \end{align*}
   for all $t\in[0,T]$. Note that from the above argument, for any $t\in[0,T]$, we have $\|u(t)\|_{L^p(\Omega)} \to \|u_0\|_{L^p(\Omega)}$, as $t\to 0$. Let us now show that $u\in C_{w}([0,T]; L^p(\Omega)\cap H_0^1(\Omega))$. From  \eqref{eqn-conv-1}, we observe that
   \begin{align*}
   	u\in C([0,T];L^2(\Omega)) \hookrightarrow C_{w}([0,T];L^2(\Omega)).
   \end{align*}
   On the other hand, we also have
   $
   u \in L^\infty\big(0,T; L^p(\Omega)\cap H_0^1(\Omega)\big).
   $
   By applying the Strauss Lemma (\cite[Lemma 1.4]{Te_1997}), with the choice $X = L^p(\Omega)\cap H_0^1(\Omega)$ and $Y = L^2(\Omega)$, it follows that
   $
   u \in C_w\big([0,T]; L^p(\Omega)\cap H_0^1(\Omega)\big).
   $
     In particular, since $u \in C_w([0,T]; L^p(\Omega))$, we deduce that
   $
   u(t) \xrightarrow{w} u_0 \quad \text{in } L^p(\Omega), \quad \text{as } t \to 0.
   $
   Moreover, as every Hilbert space and $L^p(\Omega)$ space for $p>1$ are uniformly convex, the Radon-Riesz property ensures that this weak convergence together with norm convergence yields strong convergence. Hence,
   $
   u(t) \to u_0 \quad \text{in } L^p(\Omega), \ \text{ as } \ t \to 0,
$
     which shows that
   $
   u \in C([0,T]; L^p(\Omega)).
  $

       Therefore,   $u$ is a strong solution of the system \eqref{damped weak} and the uniqueness follows from Theorem \ref{damped uniqueness}. 
\end{proof}

	\begin{Remark}
	It has been noted in the literature (\cite[Theorem 2.3]{MR4288303}, \cite[Theorem 2.12]{MR4797426}) that, when establishing the $L^p$-norm of the Faedo-Galerkin approximations, the function $|u_m|^{p-2}u_m$ is often chosen as a test function. However, we emphasize that $|u_m|^{p-2}u_m \notin V_m$, and therefore it cannot be directly used as a test function. 
	Instead, one must take $P_m(|u_m|^{p-2}u_m) \in V_m$ as the test function, or alternatively, proceed as in Step 3 of the proof of Theorem \ref{regularity f in L2} to derive the desired regularity results.
\end{Remark}

\begin{Remark}
	The regularity results obtained in Theorem \ref{regularity f in L2} help us to show the energy dissipation given in \eqref{eqn-enery-est}. 
\end{Remark}

We now establish further regularity results under the additional assumptions $u_0 \in D(A)$, $f \in H^1(0,T;L^2(\Omega))$, and, in the stronger case, $u_0 \in D(A^{3/2})$, $f \in H^1(0,T;H^1(\Omega))$. Such regularity is essential for the subsequent numerical analysis of the problem. It is important to emphasize that obtaining this higher regularity requires certain compatibility conditions between the initial data $u_0$ and the forcing term $f$, which we specify below.

\begin{Remark}[Compatibility condition]\label{rem-comp} We need the following compatibility conditions for rest of the regularity results. 
	\vskip 0.1 cm 
\noindent
(1) For $u_0 \in D(A)$ and $f \in H^1(0,T;L^2(\Omega))$, we define
	\begin{align*}
	g_0 := \nu \Delta u_0 - \alpha |u_0|^{p-2}u_0 +\sum_{\ell=1}^M \beta_{\ell}|u_0|^{q_{\ell}-2}u_0 + f(0).
	\end{align*}
	In particular, the condition $g_0 \in L^2(\Omega)$ make sense if \begin{align}\label{compatibility condition}
		\big(\nu \Delta u_0 - \alpha |u_0|^{p-2}u_0 +\sum_{\ell=1}^M \beta_{\ell}|u_0|^{q_{\ell}-2}u_0 + f(0)\big) \in L^2(\Omega).
		\end{align}  
		Note that $f \in H^1(0,T;L^2(\Omega))$ implies $f \in C([0,T];L^2(\Omega))$ also (\cite[Theorem 2, Section 5.9.2]{MR1625845}).  Since $\Delta u_0 \in L^2(\Omega)$ and $f(0)\in L^2(\Omega)$, it is enough to show that $|u_0|^{p-2}u_0 \in L^2(\Omega)$. It is clear that 
$
		|u_0|^{p-2}u_0 \in L^2(\Omega) \text{ if } u_0 \in L^{2p-2}(\Omega).
	$
			Using the fact that $D(A)\hookrightarrow L^p(\Omega)$ for all $2\leq p<\infty$ whenever $1\leq d\leq 4$ and $p\leq \frac{2d}{d-4}$ if $d\geq 5$, we conclude that 
			\begin{align}\label{eqn-embedding}
				D(A) \hookrightarrow L^{2p-2}(\Omega)\  \text{ for }  2\leq p<\infty\ \text{ for }\ 1\leq d\leq 4\ \text{ and } \ \ 2\leq p\leq \frac{2d-4}{d-4}\ \text{ for }\ d\geq 5. 
			\end{align}
		Hence, \eqref{compatibility condition} makes sense for the values of $p$ given in \eqref{eqn-embedding}, then the condition of $|u_0|^{p-2}u_0\in L^2(\Omega)$ is automatically satisfied. 
		
			\vskip 0.1 cm 
		\noindent
		(2) For $u_0 \in D(A^{\frac{3}{2}})$ and $f \in H^1(0,T;H^1(\Omega))$, we define 
		\begin{align*}
			g_1=\nabla(\nu \Delta u_0)-\alpha(p-1)|u_0|^{p-2}\nabla u_0 + \sum_{\ell=1}^M \beta_{\ell}(q_{\ell}-1)|u_0|^{q_{\ell}-2}\nabla u_0 + \nabla f(0).
		\end{align*}
		In particular, the condition $g_1 \in L^2(\Omega)$ make sense if the right hand side belongs to $L^2(\Omega)$. Since, $u_0 \in D(A^{\frac{3}{2}})$ implies $\nabla(\Delta u_0)\in L^2(\Omega)$. Since $f \in H^1(0,T;H^1(\Omega))$, we infer $f\in C([0,T];H^1(\Omega))$, so that   $\nabla f(0)\in L^2(\Omega)$. Hence, it is enough to show that $|u_0|^{p-2}\nabla u_0 \in L^2(\Omega)$. For the values of $p$ given in \eqref{eqn-embedding}, using H\"older's inequality and  Sobolev's embedding, we estimate 
		\begin{align*}
			\||u_0|^{p-2}\nabla u_0\|^2_{L^2} \leq \||u_0|^{p-2}\|^2_{L^{\frac{d}{2}}}\|\nabla u_0\|^2_{L^{\frac{2d}{d-4}}}\leq \|u_0\|^{2(p-2)}_{L^{\frac{d}{2}(p-2)}}\|A^{\frac{3}{2}} u_0\|^2_{L^2}\leq \|u_0\|^{2(p-2)}_{D(A)}\|A^{\frac{3}{2}} u_0\|^2_{L^2},
		\end{align*}
		we conclude that $|u_0|^{p-2}\nabla u_0 \in L^2(\Omega)$.
		Therefore, $g_1 \in L^2(\Omega)$. 
\end{Remark}

\begin{Theorem}\label{regularity-f- H1} 
Let the initial data $u_0  \in D(A)$ be fixed.	Then, the following results are valid:  
	\begin{enumerate}
		\item [(a)] \label{regularity for <6} For the values of $p$ given in \eqref{eqn-embedding},  if the external force  $f \in H^1(0,T;L^2(\Omega))$, then the solution $u$ of \eqref{Damped Heat} belongs to $ L^\infty(0,T;D(A))$ and $\partial_t u \in L^\infty(0,T;L^2(\Omega))\cap L^2(0,T;H_0^1(\Omega))$.
		\item  [(b)] \label{regularity for >6} For the values of $p$ given in \eqref{eqn-values of p}, if the   forcing term satisfies the extra regularity $f\in L^2(0,T;H^1(\Omega))$, then the solution $u$ of \eqref{Damped Heat} belongs to $L^\infty(0,T;D(A))\cap L^2(0,T;D(A^{\frac{3}{2}}))$.
	\end{enumerate}
\end{Theorem}
\begin{proof}
	We begin by considering the case $1 \leq d \leq 4$ with $2 \leq p < \infty$, and for higher dimensions $d \geq 5$, the range $2 \leq p \leq \tfrac{2d-4}{d-4}$.
	
(a) We divide the proof into two steps.

\noindent
\textbf{Step-1.} To show $\partial_t u \in L^\infty(0,T;L^2(\Omega))\cap L^2(0,T;H_0^1(\Omega))$, we differentiate the equation \eqref{damped Galerkin projection} with respect to $t$ and obtain
\begin{align}\label{differentiation w.r. to t}
(\partial_{tt}u_m(t),w_k)&+\nu(\partial_t \nabla u_m(t),\nabla w_k)+\alpha(\partial_t(|u_m(t)|^{p-2}u_m(t)),w_k)\nonumber \\&=(\partial_t f(t),w_k)+\sum_{\ell=1}^M \beta_{\ell}(\partial_t(|u_m(t)|^{q_{\ell}-2}u_m(t)), w_k),
\end{align}
for a.e. $t\in[0,T]$. Multiplying above equation by $d_m^k{'}(\cdot)$ and summing over all $k$ such that $\lambda_k<2^{m+1}$, we derive
\begin{align*}
\frac{1}{2} \frac{d}{dt} \|\partial_t u_m(t)\|^2_{L^2}
+ &\nu \|\partial_t \nabla u_m(t)\|^2_{L^2}  
+ \alpha (p-1) \left( |u_m(t)|^{p-2} \partial_t u_m(t),\ \partial_t u_m(t) \right)\\&= \left( \partial_t f(t),\ \partial_t u_m(t) \right)+\sum_{\ell=1}^M \beta_{\ell}(q_{\ell}-1)(|u_m(t)|^{q_{\ell}-2}\partial_tu_m(t),\partial_t u_m(t)),
\end{align*} 
for a.e. $t\in[0,T]$. Proceeding as in the estimate for $(|u_m|^{q_{\ell}-2}\partial_t u_m,\partial_t u_m)$ given in \eqref{ql2,grad transfomration}, applying the Cauchy-Schwarz and Young's inequalities, we obtain
\begin{align*}
 \frac{1}{2} \frac{d}{dt} \|\partial_t u_m(t)\|^2_{L^2}&
+ \nu \|\partial_t \nabla u_m(t)\|^2_{L^2} + \frac{\alpha (p-1)}{2} \| |u_m(t)|^{\frac{p-2}{2}} \partial_t u_m(t) \|^2_{L^2}\\& \leq \frac{1}{2\alpha}\|\partial_t f(t)\|^2_{L^2}+\frac{\alpha}{2} \|\partial_t u_m(t)\|^2_{L^2}+C_3\|\partial_t u_m(t)\|^2_{L^2},
\end{align*}
where $C_3$ is defined in \eqref{eqn-C3}. Integrating from $0$ to $t$, we have
\begin{align*} 
\|\partial_t u_m(t)\|^2_{L^2}&+2\nu \int_0^t \|\partial_s \nabla u_m(s)\|^2_{L^2}\,ds+ \alpha (p-1) \int_0^t \||u_m(s)|^{\frac{p-2}{2}}\partial_s u_m(s)\|^2_{L^2}\, ds\\
& \leq \|\partial_t u_m(0)\|^2_{L^2}+\frac{1}{\alpha} \int_0^t \|\partial_s f(s)\|^2_{L^2}\,ds + (\alpha +2C_3)\int_0^t \|\partial_s u_m(s)\|^2_{L^2}\,ds.
\end{align*}
Using Gronwall's inequality, we conclude
\begin{align}\label{needed for fully discrete}
\|\partial_t u_m(t)\|^2_{L^2}&+2\nu \int_0^t \|\partial_s \nabla u_m(s)\|^2_{L^2}\,ds+ \alpha (p-1) \int_0^t \||u_m(s)|^{\frac{p-2}{2}}\partial_s u_m(s)\|^2_{L^2}\, ds\\
& \leq \bigg(\|\partial_t u_m(0)\|^2_{L^2}+\frac{1}{\alpha} \int_0^t \|\partial_s f(s)\|^2_{L^2}\,ds\bigg)\exp\bigg((\alpha +C_3)t\bigg),
\end{align}
for all $t\in[0,T]$. We estimate $\|\partial_t u_m(0)\|^2_{L^2}$ by taking $t=0$, $w_k=\partial_t u_m(0)$ and applying the Cauchy-Schwarz and Young's inequalities as
\begin{align*}
	\|\partial_t u_m(0)\|^2_{L^2}\leq C(\|u_0\|^2_{H^2}+\|u_0\|^{2(p-1)}_{H^2}+\|u_0\|^{2(q_{\ell}-1)}_{H^2}+\|f(0)\|^2_{L^2}),
\end{align*} 
where we have used the fact that $D(A)\hookrightarrow L^{2p-2}(\Omega)$. Hence, we infer the following:
\begin{align}\label{regularity fully}
\|\partial_t u_m(t)\|^2_{L^2}&+2\nu \int_0^t \|\nabla \partial_s u_m(s)\|^2_{L^2}\,ds+ \alpha (p-1) \int_0^t \||u_m(s)|^{\frac{p-2}{2}}\partial_s u_m(s)\|^2_{L^2}\, ds \nonumber \\& \leq C(\|u_0\|_{H^2},\|f(0)\|_{L^2} ,\|f\|_{H^1(0,T;L^2)}).
\end{align}
Taking supremum over time $0\leq t \leq T$, we therefore have 
\begin{align}\label{u_t estimate}
\sup_{0\leq t \leq T} \|\partial_t u_m(t)\|^2_{L^2}&+2\nu \int_0^T \|\nabla\partial_s  u_m(s)\|^2_{L^2}\,ds +\alpha (p-1) \int_0^T \||u_m(s)|^{\frac{p-2}{2}}\partial_s u_m(s)\|^2_{L^2}\, ds\nonumber \\& \leq C(\|u_0\|_{H^2},\|f(0)\|_{L^2} ,\|f\|_{H^1(0,T;L^2)}).
\end{align}
An argument similar to Steps 4 and 5 of Theorem \ref{existenceweak} implies that
\begin{align*}\partial_tu \in L^\infty(0,T;L^2(\Omega)) \cap L^2(0,T;H_0^1(\Omega)).\end{align*}

\noindent
\textbf{Step-2. }Next we prove $u \in L^\infty(0,T;H^2(\Omega))$. Multiplying \eqref{damped Galerkin projection} by $\lambda_k d_m^k(\cdot)$ and summing over all $k$ such that $\lambda_k<2^{m+1}$, we have
\begin{align}\label{used for H^2}
 -(\partial_t u_m(t), \Delta u_m)&+\nu(\Delta u_m(t), \Delta u_m(t))- \alpha (|u_m(t)|^{p-2}u_m(t),\Delta u_m(t))\nonumber \\&=-(f(t),\Delta u_m(t))+\sum_{\ell=1}^M\beta_{\ell} (|u_m(t)|^{q_{\ell}-2}u_m(t),- \Delta u_m(t)),
\end{align}
for a.e. $t\in[0,T]$. On rearranging and using integration by parts, we get for a.e. $t\in[0,T]$
\begin{align}\label{eqn-higher}
   & \nu \|\Delta u_m(t)\|^2_{L^2}+ \alpha(p-1) \||u_m|^{\frac{p-2}{2}}\nabla u_m(t)\|^2_{L^2}\nonumber\\&=-(f(t),\Delta u_m(t))+(\partial_t u_m(t), \Delta u_m(t)) + \sum_{\ell=1}^M \beta_{\ell} (q_{\ell}-1)(|u_m(t)|^{q_{\ell}-2}\nabla u_m(t),\nabla u_m(t)).
\end{align}
Using Young's inequality, we can estimate the terms as
\begin{align*}
|(f,\Delta u_m)|&\leq  \frac{1}{\nu}\|f\|^2_{L^2}+ \frac{\nu}{4}\|\Delta u_m\|^2_{L^2},\\
|(\partial_t u_m, \Delta u_m)| &\leq  \frac{1}{\nu}\|\partial_t u_m\|^2_{L^2}+ \frac{\nu}{4}\|\Delta u_m\|^2_{L^2}.
\end{align*}
Using these estimates and \eqref{ql2,grad transfomration} in \eqref{eqn-higher}, we reach at following estimate:
\begin{align}\label{improved reg H^2}
&\nu\|\Delta u_m(t)\|^2_{L^2}+\alpha (p-1)\||u_m(t)|^{\frac{p-2}{2}}\nabla u_m(t)\|^2_{L^2}\nonumber\\& \leq \frac{2}{\nu}\|f(t)\|^2_{L^2}+ \frac{2}{\nu}\|\partial_t u_m(t)\|^2_{L^2} + 2C_3\|\nabla u_m(t)\|^2_{L^2},
\end{align}
for a.e. $t \in [0,T]$. Hence, taking supremum over all $t$ and using \eqref{u_t estimate}, we conclude
\begin{align}\label{H^2 estimate complete}
\sup_{0\leq t< T} \nu \|\Delta u_m(t)\|^2_{L^2}+\sup_{0\leq t\leq T}\alpha (p-1)\||u_m(t)|^{\frac{p-2}{2}}\nabla u_m(t)\|^2_{L^2} \leq C(\|u_0\|_{H^2}, \|f\|_{H^1(0,T;L^2)},C_3).
\end{align}
Therefore, by elliptic regularity and employing weak limits together with the compactness lemma, as in Theorem \ref{existenceweak}, we conclude that
\begin{align*}
u \in L^\infty(0,T;H^2(\Omega)) ,
\end{align*}
which completes the proof. 

\vskip 0.1 cm
\noindent 
Next, we consider  the case $1 \leq d \leq 4$ with $2 \leq p < \infty$, and for higher dimensions $d \geq 5$, the range $2 \leq p \leq \tfrac{2d-6}{d-4}$.

\noindent
(b) By multiplying \eqref{damped Galerkin projection} with $\lambda_k^2d_m^k(\cdot)$ and summing over all $k$ such that $\lambda_k<2^{m+1}$, we have 
\begin{align*}
	(\partial_t u_m(t),A^2u_m(t))&+\nu (Au_m(t),A^2u_m(t))+\alpha (|u_m(t)|^{p-2}u_m(t),A^2 u_m(t))\\&=(f(t),A^2u_m(t))+\sum_{\ell=1}^M \beta_{\ell} (|u_m(t)|^{q_{\ell}-2}u_m(t),A^2u_m(t)),
\end{align*}
for a.e. $t\in[0,T]$. Using the self-adjointness of the operator $A$, we get for a.e. $t\in[0,T]$
\begin{align*}
	\frac{1}{2}\frac{d}{dt}\|Au_m(t)\|^2_{L^2}+\nu\|A^{\frac{3}{2}}u_m(t)\|^2_{L^2}&=(A^{\frac{1}{2}}f(t),A^{\frac{3}{2}}u_m(t))-\alpha(A^{\frac{1}{2}}(|u_m(t)|^{p-2}u_m(t)),A^{\frac{3}{2}}u_m(t))\\& \quad + \sum_{\ell=1}^M \beta_{\ell} (A^{\frac{1}{2}}(|u_m(t)|^{q_{\ell}-2}u_m(t)),A^{\frac{3}{2}}u_m(t)).
\end{align*}
By an application of the Cauchy-Schwarz and Young's inequalities, we obtain  
\begin{align*}
		\frac{1}{2}\frac{d}{dt}\|Au_m(t)\|^2_{L^2}+\frac{\nu}{2}\|A^{\frac{3}{2}}u_m(t)\|^2_{L^2} &\leq \frac{3}{2\nu}\|A^{\frac{1}{2}}f(t)\|^2_{L^2}+\frac{3 \alpha^2(p-1)}{2\nu}\||u_m(t)|^{p-2}\nabla u_m(t)\|^2_{L^2}\\&\quad + \frac{3M}{2\nu}\sum_{\ell=1}^M \beta_{\ell}^2(q_{\ell}-1)\||u_m(t)|^{q_{\ell}-2}\nabla u_m(t)\|^2_{L^2},
\end{align*}
for a.e. $t\in[0,T]$. Integrating from $0$ to $t$ and taking supremum over all $t \in [0,T]$, we conclude 
\begin{align*}
&	\sup_{0\leq t\leq T}\|Au_m(t)\|^2_{L^2}+\nu \int_0^T \|A^{\frac{3}{2}}u_m(t)\|^2_{L^2} \, dt\\ &\leq \frac{3}{\nu}\int_0^T \|A^{\frac{1}{2}}f(t)\|^2_{L^2}\,dt + \frac{3\alpha^2(p-1)}{\nu}\int_0^T \||u_m(t)|^{p-2}\nabla u_m(t)\|^2_{L^2}\, dt  \\&\quad +\frac{3M}{\nu}\sum_{\ell=1}^M \beta_{\ell}^2 (q_{\ell}-1)\int_0^T \||u_m(t)|^{q_{\ell}-2}\nabla u_m(t)\|^2_{L^2}\,dt.
\end{align*}
Next, we can estimate $\int_0^T \||u_m(t)|^{p-2}\nabla u_m(t)\|^2_{L^2}\, dt$ using the H\"older's and Sobolev's inequalities as 
\begin{align}\label{eqn-p-2-est}
	\int_0^T \||u_m(t)|^{p-2}\nabla u_m(t)\|^2_{L^2}\, dt &\leq \int_0^T \||u_m(t)|^{(p-2)}\|^2_{L^d} \|\nabla u_m(t)\|^2_{L^{\frac{2d}{d-2}}}\,dt\nonumber\\&\leq \int_0^T \|u_m(t)\|^{2(p-2)}_{L^{d(p-2)}}\|\nabla u_m(t)\|^2_{L^{\frac{2d}{d-2}}} \, dt.
\end{align}
As we know that $D(A) \hookrightarrow L^{d(p-2)}$ for $2\leq p<\infty$ for $1\leq d\leq 4$ and $2\leq p\leq  \frac{2d-6}{d-4}$ for $d\geq 5$, we achieve the following: 
\begin{align}\label{eqn-method-avoid}
	\int_0^T \||u_m(t)|^{p-2}\nabla u_m(t)\|^2_{L^2}\, dt &\leq C\int_0^T  \|Au_m(t)\|^{2p-2}_{L^2}\,dt \leq C T \sup_{0\leq t \leq T} \|Au_m(t)\|^{2p-2}_{L^2}\nonumber \\& \leq CT(\|u_0\|_{H^2},\|f\|_{L^2(0,T;H^1)}).
\end{align}
Observe that the term $\int_0^T \||u_m(t)|^{q_{\ell}-2}\nabla u_m(t)\|^2_{L^2}\, dt $ can be estimated using \eqref{ql2,grad transfomration} together with the preceding calculation. Therefore, we arrive at 
\begin{align}\label{eqn-H2-est}
		&\sup_{0\leq t\leq T}\|Au_m(t)\|^2_{L^2}+\nu \int_0^T \|A^{\frac{3}{2}}u_m(t)\|^2_{L^2} \, dt \leq C(\|u_0\|_{H^2},\nu,\beta_{\ell},\alpha,M, \|f\|_{L^2(0,T;H^1)}).
\end{align}
Hence, by an application of the elliptic regularity, we identify
\begin{align*}
	u_m \in L^\infty(0,T;H^2(\Omega))\cap L^2(0,T;D(A^\frac{3}{2})).
\end{align*}
An argument similar to Steps 4 and 5 of Theorem \ref{existenceweak} implies that 
\[
	u \in 	L^\infty(0,T;H^2(\Omega))\cap L^2(0,T;D(A^\frac{3}{2})),
\]
and  the proof is completed.
%
\end{proof}

\begin{Remark}
	The estimates $u\in L^\infty(0,T;D(A))\cap L^2(0,T;D(A^{\frac{3}{2}}))$ and $\partial_t u \in L^2(0,T;H_0^1(\Omega))$ imply $Au\in L^{\infty}(0,T;L^2(\Omega))\cap L^2(0,T;D(A^{\frac{1}{2}}))$ and $\partial_tAu\in L^2(0,T;D(A^{\frac{1}{2}}))$. An application of the Lions-Magenes lemma yields that $Au\in C([0,T];L^2(\Omega))$ and the mapping $[0,T]\ni t\mapsto\|Au(t)\|_{L^2}^2\in\mathbb{R}$ is absolutely continuous such that $\frac{1}{2}\frac{d}{dt}\|Au(t)\|_{L^2}^2=\langle \partial_tAu(t),Au(t)\rangle$ for a.e. $t\in[0,T]$. Moreover, for the values of $p$ given in \eqref{eqn-values of p}, $D(A)\hookrightarrow L^{2p-2}(\Omega)$ implies that $u\in C([0,T];L^{2p-2}(\Omega))$. 
\end{Remark}

\begin{Theorem}\label{thm-more-regular}
For the values of $p$ given in \eqref{eqn-values of p}, 	if the initial data $u_0\in D(A^{\frac{3}{2}})$ and forcing $f\in H^1(0,T;H^1(\Omega))$, then  $\partial_t  u \in L^\infty(0,T;H^1_0(\Omega))$ and $\partial_{tt} u \in L^2(0,T,L^2(\Omega))$.
\end{Theorem}
\begin{proof}
	Multiplying \eqref{differentiation w.r. to t} with $d_{m}^k{''}(\cdot)$ and summing over all $k$ such that $\lambda_k < 2^{m+1}$, we obtain 
	\begin{align*}
		(\partial_{tt} u_m(t), \partial_{tt}u_m(t))+&\nu(\partial_t \nabla u_m(t), \partial_{tt}\nabla u_m(t))+\alpha (\partial_t(|u_m(t)|^{p-2}u_m(t)), \partial_{tt}u_m(t))\\& = (\partial_t f(t), \partial_{tt}u_m(t))+\sum_{\ell=1}^M \beta_{\ell}(\partial_t(|u_m(t)|^{q_{\ell}-2}u_m(t)), \partial_{tt}u_m(t)),
	\end{align*}
for a.e. $t\in[0,T]$ 	which can be further rearranged as
	\begin{align}\label{eqn-utt}
		\|\partial_{tt}{u_m(t)}\|^2_{L^2}+\frac{\nu}{2}\frac{d}{dt}\|\partial_t\nabla u_m(t)\|^2_{L^2}&=(\partial_t f(t), \partial_{tt}u_m(t))+\sum_{\ell=1}^M \beta_{\ell}(\partial_t(|u_m(t)|^{q_{\ell}-2}u_m(t)), \partial_{tt}u_m(t))\nonumber\\ &\quad - \alpha (\partial_t(|u_m(t)|^{p-2}u_m(t)), \partial_{tt}u_m(t)),
	\end{align} 
for a.e. $t\in[0,T]$. 	By an application of the Cauchy-Schwarz and Young's inequalities, we have the following estimates:
	\begin{align*}
	|(\partial_t f, \partial_{tt}u_m)| &\leq \frac{3}{2}\|\partial_t f\|^2_{L^2}+\frac{1}{6}\|\partial_{tt}u_m\|^2_{L^2},\\
	|- \alpha (\partial_t(|u_m|^{p-2}u_m), \partial_{tt}u_m)|&\leq \frac{3\alpha^2(p-1)^2}{2}\||u_m|^{p-2}\partial_t u_m\|^2_{L^2}+\frac{1}{6}\|\partial_{tt}u_m\|^2_{L^2}, \\
	\bigg|\sum_{\ell=1}^M \beta_{\ell}(\partial_t(|u_m|^{q_{\ell}-2}u_m), \partial_{tt}u_m)\bigg|&\leq \sum_{\ell=1}^M\frac{3\beta_{\ell}^2M(q_{\ell}-1)^2}{2}\||u_m|^{q_{\ell}-2}\partial_t u_m\|^2_{L^2}+\frac{1}{6}\|\partial_{tt}u_m\|^2_{L^2}.
	\end{align*}
Combining these estimates and substituting in \eqref{eqn-utt}, we deduce 
\begin{align*}
	\|\partial_{tt}{u_m(t)}\|^2_{L^2}+\nu\frac{d}{dt}\|\partial_t\nabla u_m(t)\|^2_{L^2}&\leq 3\|\partial_t f(t)\|^2_{L^2}+3\alpha^2(p-1)^2\||u_m(t)|^{p-2}\partial_t u_m(t)\|^2_{L^2}\\&\quad + \sum_{\ell=1}^M 3\beta_{\ell}^2M(q_{\ell}-1)^2\||u_m(t)|^{q_{\ell}-2}\partial_t u_m(t)\|^2_{L^2},
\end{align*}
for a.e. $t\in[0,T]$. 	Integrating the above equation from $0$ to $t$, we get 
	\begin{align*}
		&\int_0^t 	\|\partial_{ss}{u_m(s)}\|^2_{L^2}\, ds +\nu \|\partial_t\nabla u_m(t)\|^2_{L^2}\\&\leq \nu\|\partial_t\nabla u_m(0)\|^2_{L^2}+3\int_0^t\|\partial_s f(s)\|^2_{L^2}\, ds +3\alpha^2(p-1)^2\int_0^t \||u_m(s)|^{p-2}\partial_s u_m(s)\|^2_{L^2}\, ds\\& \quad + \sum_{\ell=1}^M 3\beta_{\ell}^2M(q_{\ell}-1)^2\int_0^t \||u_m(s)|^{q_{\ell}-2}\partial_s u_m(s)\|^2_{L^2}\, ds,
	\end{align*}
for all $t\in[0,T]$. 	Using \eqref{eqn-embedding},  we can estimate $\|\partial_t\nabla u_m(0)\|^2_{L^2}$ as
	\begin{align*}
		\|\partial_t\nabla u_m(0)\|^2_{L^2}&\leq C(\|A^{\frac{3}{2}}u_m(0)\|^2_{L^2}+\||u_m(0)|^{p-2}\nabla u_m(0)\|^2_{L^2}+\||u_m(0)|^{q_{\ell}-2}\nabla u_m(0)\|^2_{L^2}\\&\quad+\|A^{\frac{1}{2}}f(0)\|^2_{L^2})\\& \leq C(\|u_0\|^2_{D(A^{\frac{3}{2}})}+\|u_0\|^{2p-2}_{D(A)}+\|u_0\|^{2q_{\ell}-2}_{D(A)}+\|A^{\frac{1}{2}}f(0)\|^2_{L^2}).
	\end{align*}
	Also, by an application of H\"{o}lder's inequality and Sobolev's embedding (see \eqref{eqn-p-2-est} and \eqref{eqn-method-avoid}), we achieve  for all $t\in[0,T]$ that 
	\begin{align*}
	\int_0^t	\||u_m(s)|^{p-2}\partial_s u_m(s)\|^2_{L^2}ds&\leq C\sup_{t\in[0,T]}\|Au_m(t)\|_{L^2}^{2(p-2)}\int_0^T\|\partial_t \nabla u_m(t)\|^2_{L^2}dt\leq C,\\
	\int_0^t	\||u_m(s)|^{q_{\ell}-2}\partial_s u_m(s)\|^2_{L^2}d s&\leq C\sup_{t\in[0,T]}\|Au_m(t)\|_{L^2}^{2(q_{\ell}-2)}\int_0^T\|\partial_t \nabla u_m(t)\|^2_{L^2}dt\leq C,
	\end{align*}
by using \eqref{regularity fully} and \eqref{H^2 estimate complete}. 	Hence, using all these estimates and then  taking supremum  over all $t\in [0,T]$, we conclude
	\begin{align*}
		\int_0^T \|\partial_{tt}u_m(t)\|^2_{L^2}\, dt + \nu \sup_{0\leq t\leq T}\|\partial_t \nabla u_m(t)\|^2_{L^2} \leq C\big(\|u_0\|_{D(A^{\frac{3}{2}})},\|A^{\frac{1}{2}}f(0)\|_{L^2},\|f\|_{H^1(0,T;H^1)}\big).
	\end{align*}
	Therefore, using similar arguments as performed in Steps 4 and 5 in Theorem \ref{existence and uniquenss}, one obtains 
	\begin{align*}
	 \partial_t u \in L^\infty(0,T;H_0^1(\Omega)),\, 	\partial_{tt}u \in L^2(0,T;L^2(\Omega)),
	\end{align*}
	which completes the proof.
\end{proof}
\begin{Remark}
	Note that Theorem \ref{regularity-f- H1} (a) holds for all $2\leq p<\infty$ when $1\leq d\leq 4$ and for $2\leq p \leq \frac{2d-4}{d-4}$ when $d\geq 5$. Moreover, if $u_0\in D(A)\cap L^{2p-2}(\Omega)$, Theorem \ref{regularity-f- H1} (a) holds true for all $2\leq p<\infty$ and $d\in\mathbb{N}$. On the other hand, Theorem \ref{regularity-f- H1} (b) holds for all $p$ given in \eqref{eqn-values of p}. Hence, in general, Theorem \ref{regularity-f- H1} holds true for all values of $p$ given in \eqref{eqn-values of p}.
\end{Remark}

\section{Conforming finite element method}\label{section CFEM}
The conforming finite element method is a Galerkin-based spatial discretization technique in which the finite-dimensional trial and test spaces are subspaces of the continuous energy space associated with the problem, typically $H^1_0(\Omega)$ for second-order elliptic and for parabolic PDEs with homogeneous Dirichlet boundary conditions. By ensuring conformity, the discrete functions inherit the essential continuity and boundary conditions of the exact solution space, which allows direct application of the variational formulation without modification of the bilinear form. Standard choices include piecewise polynomial spaces, such as continuous piecewise linear ($P_1$) or higher-order Lagrange elements, defined over a shape-regular triangulation of the domain. The method naturally preserves the symmetry, coercivity, and continuity properties of the continuous problem, leading to stable and convergent approximations under standard mesh refinement. This section focuses on proving the optimal convergence rates for semidiscrete and fully discrete Galerkin approximations in CFEM. The error estimates are derived using a standard approach, similar to that in \cite{MR1479170}. The elliptic projection only provide the error estimate for $2\leq p\leq \frac{2d}{d-2}$, so, we introduce Scott-Zhang quasi-interpolation operator  \cite{MR2050138,MR2373954} to prove optimal order of convergence for $2\leq p<\infty $ when $1\leq d\leq 4$ and when $d \geq 5$, it is obtained for $p\leq \frac{2d-4}{d-4}$.
We begin by presenting the semidiscrete conforming scheme, including its formulation, energy estimates, error analysis, and numerical experiments.

\subsection{Semidiscrete conforming FEM}\label{semi-discrete Galerkin}

This section focuses on establishing the error estimates for the semidiscrete Galerkin method. We divide the domain $\Omega$ into a regular shape mesh, denoted by $\mathcal{T}_h$, consisting of triangles or rectangles. Corresponding to this mesh, we define a finite-dimensional subspace $V_h$ as follows:
\begin{equation}\label{fem polynomial damped}
V_h = \left\{ v_h : v_h \in C^0(\overline{\Omega}),\ v_h|_T \in \mathbb{P}_1(T) \ \text{ for all } \ T \in \mathcal{T}_h \right\},
\end{equation}
where $\mathbb{P}_1$ represents the space of polynomials of degree at most 1. Clearly, $V_h$ is a subspace of $H^1_0(\Omega)$ with finite dimension, and the mesh parameter $h>0$. The goal is to derive an estimate of the form (refer to \cite[Chapter 1]{MR1479170}):
\begin{equation*}
\inf_{\chi \in V_h} \left\{ \|u - \chi\|_{L^2} + h \|\nabla (u - \chi)\|_{L^2} \right\} \leq C h^2 \|u\|_{H^2},
\end{equation*}
for all $u \in H^2(\Omega) \cap H^1_0(\Omega),$ and for some constant $C>0$. For the wellposedness of the semi-discrete solution and error estimates, we define a  projection as well as a quasi-interpolation operator   as:
\begin{Definition}[Ritz (or elliptic) projection {\cite[Chapter 1]{MR1479170}}]\label{Ritz def projection}
	The Ritz projection $R_h:H_0^1(\Omega)\to V_h\subset H_0^1(\Omega)$ is defined by 
	\begin{align*}
		a(R_h u - u, v_h) = 0 \ \text{ for all } \ v_h \in V_h,
	\end{align*}
	where $a(\cdot,\cdot)$ is the bilinear form associated with the elliptic operator 
	(e.g., $a(w,v) = \int_\Omega \nabla w \cdot \nabla v \, dx$ for the Dirichlet Laplacian). 
\end{Definition}
Note that, unlike the $L^2$-projection, the Ritz projection preserves the elliptic structure of the problem and requires $u\in H_0^1(\Omega)$. It satisfies the following estimates:
\begin{equation}\label{Ritz projection}
	(\nabla R_h u, \nabla \chi) = (\nabla u, \nabla \chi) \  \text{ for all } \ \chi \in V_h.
\end{equation}
By testing the above relation with $\chi = R_h u$, we observe the stability of the Ritz projection: $\| \nabla R_h u \|_{L^2} \leq \| \nabla u \|_{L^2}$, for all $u \in H^1_0(\Omega)$. Furthermore, the following approximation estimate holds:
\begin{equation}\label{4.5.3}
	\|R_h u - u\|_{L^2} + h \|\nabla (R_h u - u)\|_{L^2} \leq C h^s \|u\|_{H^s}, \quad \text{for} \quad s = 1, 2.
\end{equation}

\begin{Definition}[Scott-Zhang interpolation{ \cite[Section 4.8]{MR2373954}}]\label{Scott-Zhang}
	Let $\mathcal{T}_h$ be a shape-regular triangulation of $\Omega \subset \mathbb{R}^d$, and $V_h$ be the associated finite element space of degree $m$.  
	The Scott-Zhang quasi-interpolation operator 
	\[
	S_h : W^{l,p}(\Omega) \to V_h,
	\]
where $l\geq 1$ if $p=1$ and $l>\frac{1}{p}$ otherwise, is defined nodewise by suitable local averages over $(d-1)$--simplices adjacent to each node, and preserves homogeneous boundary conditions. 
\end{Definition}
	From \cite[Lemma 1.130]{MR2050138}, for $1\leq p<\infty$ and $l>\frac{1}{p}$, we have 
	\begin{equation}\label{scott zhang L^p estimate}
		\|S_hv\|_{W^{k,p}}\leq C\|v\|_{W^{l,p}} \  \text{ for all } \ 0\leq k\leq \min\{1,l\}.
	\end{equation}
	Also from \cite[Section 3]{MR2975554}, we have
	\begin{equation}\label{scott zhang approximation result}
		\|v-S_hv\|_{W^{k,p}}\leq Ch^{l-k}\|v\|_{W^{l,p}}, \ \text{ for all }\  0\leq k\leq l \leq r+1,
	\end{equation}
	where $r$ is the degree of polynomial. By an application of interpolation inequality \cite[Section 8.6]{MR2759829} and \eqref{4.5.3}, we have
	\begin{align}\label{to bound u-W in Lp}
			\|A^{\theta}(S_hu-u)\|_{L^2} &\leq \|S_hu-u\|^{1-\frac{2\theta}{s}}_{L^2}\|A^{\frac{s}{2}}(S_hu-u)\|^{\frac{2\theta}{s}}_{L^2}\nonumber \\
			&\leq Ch^{s-2\theta}\|u\|_{H^s}\  \text{ for all } \ 0\leq s \leq r+1,
		\end{align}
	for any $\theta \in [0,1]$.

The semidiscrete Galerkin approximation of problem \eqref{damped abstract formulation} is formulated in the following way: 

Find $u_h(\cdot, t) \in V_h$ such that, for every $t \in [0, T]$,
\begin{equation}\label{damped fem}
\begin{cases}
	\displaystyle
\langle \partial_t u_h(t), \chi \rangle + \nu(\nabla u_h(t), \nabla \chi) + \alpha (|u_h(t)|^{p-2} u_h(t), \chi)-\sum_{\ell=1}^M\beta_{\ell} (|u_h(t)|^{q_{\ell}-2}u_h(t), \chi) \\ \quad = \langle f(t), \chi \rangle, \\
(u_h(0), \chi) =
\begin{cases}
	(R_h u_0,\chi) &\text{ if } 2\leq p\leq \frac{2d}{d-2},\\[6pt]
	(S_h u_0,\chi) &\text{ if } \frac{2d}{d-2}< p \leq \frac{2d-6}{d-4},
\end{cases}
\end{cases}
\end{equation}
for all $\chi \in V_h$. More precisely, we consider the following values of $p$ separately. 

	\begin{table}[ht!]
		\centering
		\begin{tabular}{|c|c|c|}
		\hline
\text{Interpolation operator}&	\text{Values of }$p$&\text{Dimension} \\  \hline
\text{Ritz (elliptic) projection}&	$2\leq p<\infty$&$d=1,2$ \\ \hline
$"$&$2\leq p\leq \frac{2d}{d-2}$ & $d \geq 3$\\ \hline
\text{Scott-Zhang interpolation}&$\frac{2d}{d-2}<p<\infty$&$d=3,4$\\ \hline
$"$&$\frac{2d}{d-2} < p \leq  \frac{2d-6}{d-4}$ &$d\geq 5$\\ \hline 
	\end{tabular}
		\caption{Values of $p$}\label{values of p}
	\end{table}


\begin{Lemma}\label{energy estimate of CFEM}
    (a) For $f\in L^2(0,T;H^{-1}(\Omega))$ and $u_0 \in L^2(\Omega)$, the semidiscrete system \eqref{damped fem} admits a unique solution $u_h \in V_h$ and the following estimate is satisfied:
     \begin{align}\label{eqn-bound}
    	\sup_{0\leq t \leq T}\|u_h(t)\|^2_{L^2}&+\nu \int_0^T \|\nabla u_h(t)\|^2_{L^2} \, dt + \alpha \int_0^T \|u_h(t)\|^p_{L^p}\, dt \nonumber\\&\leq \|u_0\|^2_{L^2}+\frac{1}{\nu} \int_0^T \|f(t)\|^2_{H^{-1}}\, dt + C^{*}T|\Omega|,
    \end{align}
    where $C^{*}$ is defined in \eqref{C^*}.
    
    \noindent
    (b) Furthemore, if $f\in L^2(0,T;L^2(\Omega))$ and $u_0 \in D(A)$, then, we obtain the following improved estimate: 
     \begin{align*}
    	&\int_0^T \|\partial_t u_h(t)\|^2_{L^2}\, dt + \sup_{0\leq t \leq T}\bigg(\nu\|\nabla u_h(t)\|^2_{L^2}+\frac{2\alpha}{p}\|u_h(t)\|^p_{L^p}\bigg) \nonumber\\&\quad\leq \nu \|\nabla u_0\|^2_{L^2}+\frac{2\alpha}{p}\|u_0\|^p_{D(A)} + \int_0^T \|f(t)\|^2_{L^2}\, dt.
    \end{align*}
    \end{Lemma}
    \begin{proof}
        (a) For $f \in L^2(0,T;H^{-1}(\Omega))$ and $u_0 \in L^2(\Omega)$, the proof of existence follows by arguments similar to those used in Theorem \ref{existenceweak}. Now, for proving the estimate \eqref{eqn-bound}, taking $\chi = u_h$ in \eqref{damped fem}, we infer for a.e. $t\in[0,T]$ 
        \begin{align*}
   		\frac{1}{2}\frac{d}{dt}\|u_h(t)\|^2_{L^2}+\nu \|\nabla u_h(t)\|^2_{L^2}+\alpha \|u_h(t)\|^p_{L^p} &= (f(t),u_h(t))+\sum_{\ell=1}^M \beta_{\ell} \|u_h(t)\|^{q_{\ell}}_{L^{q_{\ell}}}.
        \end{align*} 
        By an application of the Cauchy-Schwarz and Young's inequalities, along with \eqref{u^q to u^p journey}, we obtain
        \begin{align*}
        	\frac{1}{2}\frac{d}{dt}\|u_h(t)\|^2_{L^2}+\frac{\nu}{2} \|\nabla u_h(t)\|^2_{L^2}+\frac{\alpha}{2} \|u_h(t)\|^p_{L^p} \leq \frac{1}{2\nu}\|f(t)\|_{H^{-1}} + C^{*}|\Omega|,
        \end{align*}
   for a.e. $t\in[0,T]$.     Integrating from $0 \text{ to } t$, we achieve
        \begin{align*}
        &	\|u_h(t)\|^2_{L^2}+\nu \int_0^t \|\nabla u_h(s)\|^2_{L^2} \, ds + \alpha \int_0^t \|u_h(s)\|^p_{L^p}\, ds \nonumber\\&\leq \|u_h(0)\|_{L^2}+\frac{1}{\nu} \int_0^t \|f(s)\|^2_{H^{-1}}\, ds+ C^{*}t|\Omega|,
        \end{align*}
   for all $t\in[0,T]$.      Taking supremum over all $t \in [0,T]$, we conclude that
        \begin{align*}
        &	\sup_{0\leq t \leq T}\|u_h(t)\|^2_{L^2}+\nu \int_0^T \|\nabla u_h(t)\|^2_{L^2} \, dt + \alpha \int_0^T \|u_h(t)\|^p_{L^p}\, dt \nonumber\\&\leq \|u_0\|^2_{L^2}+\frac{1}{\nu} \int_0^T \|f(t)\|^2_{H^{-1}}\, dt + C^{*}T|\Omega|.
        \end{align*}
        Hence, we see that
        \[u_h \in L^\infty(0,T;L^2(\Omega))\cap L^2(0,T,H_0^1(\Omega))\cap L^p(0,T,L^p(\Omega)).\]\\
        \noindent
        (b) For $f \in L^2(0,T;L^2(\Omega))$ and $u_0 \in D(A)$, taking $\chi = \partial_t u_h$ in \eqref{damped fem} yields
        \begin{align*}
        	\|\partial_t u_h(t)\|^2_{L^2} + \frac{\nu}{2}\frac{d}{dt}\|\nabla u_h(t)\|^2_{L^2}+\frac{\alpha}{p}\frac{d}{dt}\|u_h(t)\|^p_{L^p} = (f(t),\partial_t u_h(t))+\sum_{\ell=1}^M  \frac{\beta_{\ell}}{q_{\ell}}\frac{d}{dt}\|u_h(t)\|^{q_{\ell}}_{L^{q_{\ell}}},
        \end{align*}
    for a.e. $t\in[0,T]$.    Using the Cauchy-Schwarz and Young's inequalities, we deduce
        \begin{align*}
        	\frac{1}{2}\|\partial_t u_h(t)\|^2_{L^2} + \frac{\nu}{2}\frac{d}{dt}\|\nabla u_h(t)\|^2_{L^2}+\frac{\alpha}{p}\frac{d}{dt}\|u_h(t)\|^p_{L^p} \leq  \frac{1}{2}\|f(t)\|^2_{L^2}+\sum_{\ell=1}^M  \frac{\beta_{\ell}}{q_{\ell}}\frac{d}{dt}\|u_h(t)\|^{q_{\ell}}_{L^{q_{\ell}}},
        \end{align*}
    for a.e. $t\in[0,T]$.     Now, integrating from $0 \text{ to } t$, we infer the following:
        \begin{align*}
        &\int_0^t \|\partial_s u_h(s)\|^2_{L^2}\, ds + \nu\|\nabla u_h(t)\|^2_{L^2}+\frac{2\alpha}{p}\|u_h(t)\|^p_{L^p} \\&\leq \nu \|\nabla u_h(0)\|^2_{L^2}+\frac{2\alpha}{p}\|u_h(0)\|^p_{L^p}+ \int_0^t \|f(s)\|^2_{L^2}\, ds +\sum_{\ell=1}^M  \frac{2\beta_{\ell}}{q_{\ell}}\|u_h(t)\|^{q_{\ell}}_{L^{q_{\ell}}}-\sum_{\ell=1}^M \frac{2\beta_{\ell}}{q_{\ell}}\|u_h(0)\|^{q_{\ell}}_{L^{q_{\ell}}},
        \end{align*}
    for all $t\in[0,T]$.    The term $\|u_h(\cdot)\|^{q_{l}}_{L^{q_{l}}}$ can be estimated in the same way as in \eqref{u^q to u^p journey}. Also using the Sobolev's embedding, we have 
    \begin{align*}
    \mbox{
    $\|u_h(0)\|_{L^p}=\|R_h u_0\|_{L^p}\leq C \|\nabla R_h u_0\|_{L^2}$\  for \ $2\leq p\leq \frac{2d}{d-2}$.
   }
   \end{align*} By exploiting  \eqref{scott zhang L^p estimate} and Sobolev's embedding, we infer 
   \begin{align*}
   	\mbox{$\|u_h(0)\|_{L^p}=\|S_h u_0\|_{L^p}\leq C \|u_0\|_{W^{\frac{d-2}{2d},p}}\leq \|u_0\|_{D(A)}$\  for all \ $\frac{2d}{d-2}< p\leq \frac{2d-6}{d-4}$. }
   \end{align*}
   Hence, by taking supremum over all $t\in [0,T]$, we conclude
        \begin{align*}
        	u_h\in L^\infty(0,T;H_0^1(\Omega))\cap L^{\infty}(0,T;L^p(\Omega)) \ \text{ and }\  \partial_t u_h \in L^2(0,T;L^2(\Omega)),
        \end{align*}
      which completes the proof. 
    \end{proof}

We proceed to present and derive the following theorem concerning the error analysis of the semi-discrete scheme.
\begin{Theorem}\label{semidiscrete error analysis}
    Let $V_h$ be the finite-dimensional subspace of $H_0^1(\Omega)$ as defined in \eqref{fem polynomial damped}. Assume that $u_0\in D(A)$ and $u(\cdot)$ is the solution of \eqref{Damped Heat}.
    Then\\ 
   \noindent
   (a) for $f \in H^1(0,T;L^2(\Omega))$ and  $2\leq p\leq \frac{2d}{d-2}$, we have
   \begin{align*}
   	&\|u_h-u\|^2_{L^\infty(0,T;L^2)}+\nu\|u_h-u\|^2_{L^2(0,T;H_0^1)}+\frac{\alpha}{2^{p-3}}\|u_h-u\|^p_{L^p(0,T;L^p)} \\
   	&\leq Ch^2\bigg\{\|u_0\|^2_{H_0^1}+\|u\|^2_{L^{\infty}(0,T;H_0^1)}+\int_0^T\|\partial_t u(t)\|_{H_0^1}^2\, dt\\& \quad +\left(\|u_0\|_{D(A)}^p+ \|f\|^2_{H^1(0,T;L^2)}+\|u\|^{2(p-2)}_{L^\infty(0,T;L^p)}\right)\int_0^T \|u(t)\|^2_{H^2}\,dt \bigg\},
   \end{align*}
   \noindent
   (b) for $f \in H^1(0,T;H^1(\Omega))$ and  $\frac{2d}{d-2}<p\leq \frac{2d-6}{d-4}$, we obtain 
   
   \noindent
  (i) if $2\leq q_{\ell}<1+\frac{p}{2}$, the following estimate holds:
   \begin{align*}
   	&\|u_h-u\|^2_{L^\infty(0,T;L^2)}+\nu \|u_h-u\|^2_{L^2(0,T;H_0^1)}+\frac{\alpha}{2^{p-2}}\|u_h-u\|^p_{L^p(0,T;L^p)}\\&\leq Ch^2\bigg\{\|u_0\|^2_{H_0^1}+\|u\|^2_{L^\infty(0,T;H_0^1)}+\int_0^T \|\partial_tu(t)\|^2_{H_0^1}\, dt +\int_0^T \|u(t)\|^2_{H^2}\,dt + \int_0^T \|u(t)\|^p_{H^2}\,dt\\& \quad  + \int_0^T \|u(t)\|^2_{D(A^{\frac{3}{2}})}\,dt \bigg\},
   \end{align*}
   
   \noindent
   (ii) if $1+\frac{p}{2}\leq q_{\ell}<p$, the following estimate holds:
   \begin{align*}
   	&\|u_h-u\|^2_{L^\infty(0,T;L^2)}+\nu \|u_h-u\|^2_{L^2(0,T;H_0^1)}+\frac{\alpha}{2^{p-2}}\|u_h-u\|^p_{L^p(0,T;L^p)}\\&\leq Ch^2\bigg\{\|u_0\|^2_{H_0^1}+\|u\|^2_{L^\infty(0,T;H_0^1)}+\int_0^T \|\partial_tu(t)\|^2_{H_0^1}\, dt +\int_0^T \|u(t)\|^2_{H^2}\,dt + \int_0^T \|u(t)\|^p_{H^2}\,dt\\& \quad  +\|u\|^{\frac{2q_{\ell}-p-2}{p-q_{\ell}+1}}_{L^\infty(0,T;H^2)} \int_0^T \|u(t)\|^2_{D(A^{\frac{3}{2}})}\,dt \bigg\}.
   \end{align*}
\end{Theorem}
\begin{proof}
    Using Theorem \ref{damped weak} and equation \eqref{damped fem}, we know that $(u_h-u)(\cdot)$ satisfies
    \begin{align}\label{SFEM err-01}
    \langle \partial_t u_h(t)-\partial_t u(t), \chi \rangle &+ \nu(\nabla (u_h(t)-u(t)), \nabla \chi) + \alpha (|u_h(t)|^{p-2}u_h(t)- |u(t)|^{p-2}u(t), \chi)\nonumber \\& = \sum_{\ell=1}^M\beta_{\ell} (|u_h(t)|^{q_{\ell}-2}u_h(t) -|u(t)|^{q_{\ell}-2}u(t), \chi ),
    \end{align}
    for all $\chi \in V_h$ and a.e. $t \in [0,T]$. Choosing $\chi = u_h -u+u-W \in V_h $, $W\in V_h,$ we obtain
    \begin{align}\label{eqn-diff}
        & \langle \partial_t u_h(t)-\partial_t u(t), u_h(t)-u(t) \rangle + \nu(\nabla (u_h(t)-u(t)), \nabla (u_h(t)-u(t))) \nonumber \\
         &\quad   +\alpha (|u_h(t)|^{p-2}u_h(t)- |u(t)|^{p-2}u(t), u_h(t) - u(t)) \nonumber \\&= - \langle \partial_t u_h(t)-\partial_t u(t), u(t)-W(t) \rangle -\nu(\nabla (u_h(t)-u(t)), \nabla (u(t)-W(t)))  \nonumber\\
         & \quad- \alpha (|u_h(t)|^{p-2}u_h(t)- |u(t)|^{p-2}u(t), u(t) - W(t))\\&\quad+ \sum_{\ell=1}^M\beta_{\ell}(|u_h(t)|^{q_{\ell}-2}u_h(t)-|u(t)|^{q_{\ell}-2}u(t), u_h(t)-u(t))  \nonumber\\&\quad + \sum_{\ell=1}^M\beta_{\ell}(|u_h(t)|^{q_{\ell}-2}u_h(t)-|u(t)|^{q_{\ell}-2}u(t), u(t)-W(t)),\nonumber
    \end{align}
for a.e. $t\in[0,T]$.    By performing a similar calculation for the nonlinear terms, as was done in Theorem \ref{damped uniqueness}, we find
    \begin{align*}
 	\alpha (|u_h|^{p-2}u_h- |u|^{p-2}u, u_h - u) \geq \frac{\alpha}{2}\||u_h|^{\frac{p-2}{2}}(u_h-u)\|^2_{L^2}+\frac{\alpha}{2}\||u|^{\frac{p-2}{2}}(u_h-u)\|^2_{L^2},
 \end{align*}
 and
 \begin{align*}
 	&	\sum_{\ell=1}^M\beta_{\ell}(|u_h|^{q_{\ell}-2}u_h-|u|^{q_{\ell}-2}u, u_h-u)\\&\leq \frac{\alpha}{4}\||u_h|^{\frac{p-2}{2}}(u_h-u)\|^2_{L^2}+\frac{\alpha}{4}\||u|^{\frac{p-2}{2}}(u_h-u)\|^2_{L^2} + 2C_2\|u_h-u\|^2_{L^2},
 \end{align*}
    where $C_2$ given in \eqref{C_2}. Therefore, using the above estimates in \eqref{eqn-diff}, we deduce 
    \begin{align}\label{SFEM err-02}
           \frac{1}{2} &\frac{d}{dt}\|u_h(t)-u(t)\|^2_{L^2}+ \nu \|\nabla (u_h(t)-u(t)\|^2_{L^2(\Omega)}+\frac{\alpha}{4} \||u_h(t)|^{\frac{p-2}{2}}(u_h(t)-u(t))\|^2_{L^2} \nonumber\\
            & \quad+ \frac{\alpha}{4} \||u(t)|^{\frac{p-2}{2}}(u_h(t)-u(t)\|^2_{L^2} \nonumber\\
            &\leq - \partial_t(u_h(t)-u(t),u(t)-W(t)) +(u_h(t)-u(t), \partial_t(u(t)-W(t))) \nonumber \\& \quad-  \nu(\nabla (u_h(t)-u(t)), \nabla (u(t)-W(t)))- \alpha (|u_h(t)|^{p-2}u_h(t)- |u|^{p-2}u, u(t) - W(t))\nonumber\\
            &\quad +\sum_{\ell=1}^M\beta_{\ell}(|u_h(t)|^{q_{\ell}-2}u_h(t)-|u(t)|^{q_{\ell}-2}u(t), u(t)-W(t)) +2C_2\|u_h(t)-u(t)\|^2_{L^2},
        \end{align}
for a.e. $t\in[0,T]$.    From (\cite[Proposition 2.2]{AB+ZB+MTM-2025}), we have the estimate
    \begin{align*}
    	\frac{\alpha}{4} \||u_h|^{\frac{p-2}{2}}(u_h-u)\|^2_{L^2}
    	 + \frac{\alpha}{4} \||u|^{\frac{p-2}{2}}(u_h-u)\|^2_{L^2} \geq \frac{\alpha}{2^{p-1}}\|u_h-u\|^p_{L^p}.
    \end{align*}
   Therefore, we arrive at
    \begin{align}\label{SFEM err-03}
           &\frac{1}{2}\frac{d}{dt}\|u_h(t)-u(t)\|^2_{L^2}+ \nu \|\nabla (u_h(t)-u(t))\|^2_{L^2}+\frac{\alpha}{2^{p-1}}\|u_h(t)-u(t)\|^p_{L^p}
            \nonumber \\&\leq - \partial_t(u_h(t)-u(t),u(t)-W(t)) + \sum_{j=1}^4I_j + 2C_2\|u_h(t)-u(t)\|^2_{L^2}.
    \end{align}
   We now need to estimate the above terms separately for 
  $2\leq p\leq \frac{2d}{d-2}$ and $\frac{2d}{d-2}<p\leq \frac{2d-6}{d-4}$, since the nonlinear terms must be treated differently in each case.
    \vskip 0.1 cm
    \noindent
 	(a) \textbf{For $2\leq p \leq \frac{2d}{d-2}$:}
   We estimate $I_1$ and $I_2$ by using the Cauchy-Schwarz and Young's inequalities as 
   \begin{align*}
   	|I_1|&= |(u_h-u, \partial_t(u-W))|&\leq \|u_h-u\|_{L^2} \|\partial_t(u-W)\|_{L^2} \leq \frac{1}{2}\|u_h-u\|^2_{L^2} +\frac{1}{2}\|\partial_t(u-W)\|^2_{L^2},
   \end{align*}
   and
   \begin{align*}
   	|I_2|&= | \nu(\nabla (u_h-u), \nabla (u-W))|\\
   	&\leq \nu \|\nabla (u_h-u)\|_{L^2} \|\nabla (u - W) \|_{L^2} \leq \frac{\nu}{4}\|\nabla(u_h-u)\|^2_{L^2} +\nu \|\nabla (u-W)\|^2_{L^2}.
   \end{align*}
   Using Taylor's formula, H\"{o}lder's, Young's, and Gagliardo-Nirenberg's inequalities \cite[Theorem 1.5.2, Chapter 1]{MR2962068}, we estimate $I_3$ for $p\geq 2$ as 
   \begin{align*}
   	|I_3|&=|- \alpha (|u_h|^{p-2}u_h- |u|^{p-2}u, u - W)| \\
   	&\leq (p-1)\left\vert\alpha((\theta u_h +(1-\theta)u)^{p-2}(u_h-u), u-W )\right\vert\\
   	&\leq \alpha (p-1)2^{p-3}\big|\big( (|u_h|^{p-2}+|u|^{p-2})(u_h-u),u-W\big)\big|\\&\leq \alpha (p-1)2^{p-3}\|u_h-u\|_{L^{\frac{2d}{d-2}}}(\|u_h\|^{p-2}_{L^p}+\|u\|^{p-2}_{L^p})\|u-W\|_{L^{\frac{2dp}{(2-d)p+4d}}}\\
   	&\leq \frac{\nu}{4}\|\nabla(u_h-u)\|^2_{L^2}+\frac{\alpha^2 (p-1)^2 2^{2p-5}}{\nu}(\|u_h\|^{2(p-2)}_{L^p}+\|u\|^{2(p-2)}_{L^p})\|u-W\|^2_{L^p}.
   \end{align*}
   We Estimate $I_4$ using a similar  calculation as
   \begin{align*}
   	|I_4|&=\bigg|\sum_{\ell=1}^M \beta_{\ell}(|u_h(t)|^{q_{\ell}-2}u_h(t)-|u(t)|^{q_{\ell}-2}u(t), u(t)-W(t))\bigg|\\
   	&\leq \frac{\nu}{4}\|\nabla(u_h-u)\|^2_{L^2}+\sum_{\ell=1}^M \frac{M|\beta_{\ell}|^2(q_{\ell}-1)^2 2^{2q_{\ell}-5}}{\nu}(\|u_h\|^{2(q_{\ell}-2)}_{L^{q_{\ell}}}+\|u\|^{2(q_{\ell}-2)}_{L^{q_{\ell}}})\|u-W\|^2_{L^{q_{\ell}}}.
   \end{align*}
   Using all these estimates in \eqref{SFEM err-03}, we establish the following:
   {\small
   \begin{align*}
   &\frac{1}{2}\frac{d}{dt}\|u_h(t)-u(t)\|^2_{L^2}+ \frac{\nu}{4} \|\nabla (u_h(t)-u(t)\|^2_{L^2} + \frac{\alpha}{2^{p-1}}\|u_h(t)-u(t)\|^p_{L^p}\\&\leq -\partial_t(u_h(t)-u(t),u(t)-W(t))+\frac{1}{2}\|u_h(t)-u(t)\|^2_{L^2}+\frac{1}{2}\|\partial_t(u(t)-W(t))\|^2_{L^2}\\&\quad + \nu\|\nabla(u(t)-W(t))\|^2_{L^2}+\frac{\alpha^2 (p-1)^2 2^{2p-5}}{\nu}(\|u_h(t)\|^{2(p-2)}_{L^p}+\|u(t)\|^{2(p-2)}_{L^p})\|u(t)-W(t)\|^2_{L^p}\\&\quad + \sum_{\ell=1}^M \frac{M|\beta_{\ell}|^2(q_{\ell}-1)^2 2^{2q_{\ell}-5}}{\nu}(\|u_h(t)\|^{2(q_{\ell}-2)}_{L^{q_{\ell}}}+\|u(t)\|^{2(q_{\ell}-2)}_{L^{q_{\ell}}})\|u(t)-W(t)\|^2_{L^{q_{\ell}}}+2C_2\|u_h(t)-u(t)\|^2_{L^2},
   \end{align*}}
   for a.e. $t\in[0,T]$. 
   Integrating from $0$ to $t$, we arrive at
  { \small
   \begin{align*}
   	&\|u_h(t)-u(t)\|^2_{L^2}+\frac{\nu}{2}\int_0^t \|\nabla(u_h(s)-u(s))\|^2_{L^2}\, ds+\frac{\alpha}{2^{p-2}}\int_0^t\|u_h(s)-u(s)\|^p_{L^p}\, ds \\ 
   	&\leq \|u_h(0)-u_0\|^2_{L^2}-2(u_h(t)-u(t),u(t)-W(t))+ 2(u_h(0)-u_0, u_0-W(0))\\ & \quad + \int_0^t \|\partial_s(u(s)-W(s))\|^2_{L^2}\,ds + 2\nu\int_0^t \|\nabla(u(s)-W(s))\|^2_{L^2}\,ds+(1+4C_2)\int_0^t \|u_h(s)-u(s)\|^2_{L^2}\,ds\\&\quad + \frac{\alpha^2 (p-1)^2 2^{2p-5}}{\nu}\int_0^t \big(\|u_h(s)\|^{2(p-2)}_{L^p}+\|u(s)\|^{2(p-2)}_{L^p}\big)\|u(s)-W(s)\|^2_{L^p}\, ds \\&\quad +  \sum_{\ell=1}^M \frac{M|\beta_{\ell}|^2(q_{\ell}-1)^2 2^{2q_{\ell}-5}}{\nu}\int_0^t (\|u_h(s)\|^{2(q_{\ell}-2)}_{L^{q_{\ell}}}+\|u(s)\|^{2(q_{\ell}-2)}_{L^{q_{\ell}}})\|u(s)-W(s)\|^2_{L^{q_{\ell}}}\,ds.
   \end{align*}}
   For $2\leq p \leq \frac{2d}{d-2}$, using Sobolev embedding and \eqref{4.5.3}, we can estimate the terms in above equation by taking $W=R_h u$ as
   \begin{align*}
   	\|u-W\|^2_{L^p}&\leq \|\nabla(u-R_hu)\|^2_{L^2}\leq Ch^2\|u\|^2_{H^2},\\
   	\|u_0-W(0)\|^2_{L^2}&=\|u_0-R_hu_0\|^2_{L^2}\leq Ch^2\|u_0\|^2_{H_0^1},\\
   	\|u-W\|^2_{L^2}&= \|u-R_hu\|^2_{L^2}\leq Ch^2\|u\|_{H_0^1},\\
   	\|\partial_t(u-W)\|^2_{L^2}&=\|\partial_t(u-R_hu)\|^2_{L^2}\leq Ch^2\|\partial_tu\|^2_{H_0^1}.
   \end{align*}
     Using the Cauchy-Schwarz and Young's inequalities, we estimate $- 2(u_h-u, u-W)$ and $2 (u_h(0)-u_0, u_0-W(0)) $ as
   \begin{align*}
   	-2(u_h - u, u - W) &\leq \frac{1}{2} \|u_h - u\|_{L^2}^2 +  2\|u - W\|_{L^2}^2\leq \frac{1}{2} \|u_h - u\|_{L^2}^2 +  2Ch^2\|u\|_{H_0^1}^2 , \\
   	2(u_h(0) - u_0, u_0 - W(0)) &\leq \|u_h(0) - u_0\|_{L^2}^2 + \|u_0 - W(0)\|_{L^2}^2 \leq 2Ch^2\|u_0\|^2_{H_0^1}.
   \end{align*}
   Using these estimates, taking supremum over all $t\in [0,T]$ and by an application of Gronwall's inequality, we see that 
   \begin{align*}
   	&\sup_{0\leq t\leq T}\|u_h(t)-u(t)\|^2_{L^2}+ \nu \int_0^T\|\nabla(u_h(t)-u(t))\|^2_{L^2}\,dt+ \frac{\alpha}{2^{p-3}}\int_0^T\|u_h(t)-u(t)\|^p_{L^p}\, dt\\ 
   	&\leq\bigg[ 4Ch^2\|u_0\|^2_{H_0^1}+4Ch^2\|u\|_{L^{\infty}(0,T;H_0^1)}^2+2Ch^2\int_0^T\|\partial_t u(t)\|^2_{H_0^1}\, dt \\&\quad + Ch^2\frac{\alpha^2 (p-1)^2 2^{2p-4}}{\nu} \big(\|u_h\|^{2(p-2)}_{L^\infty(0,T;L^p)}+\|u\|^{2(p-2)}_{L^\infty(0,T;L^p)}\big)\int_0^T \|u(t)\|^2_{H^2}\, dt\\&\quad + Ch^2\sum_{\ell=1}^M \frac{M|\beta_{\ell}|^2(q_{\ell}-1)^2 2^{2q_{\ell}-4}}{\nu} \big(\|u_h\|^{2(q_{\ell}-2)}_{L^\infty(0,T;L^{q_{\ell}})}+\|u\|^{2(q_{\ell}-2)}_{L^\infty(0,T;L^{q_{\ell}})}\big)\int_0^T\|u(t)\|^2_{H^2}\,dt\\ & \quad +  4Ch^2\nu \int_0^T \| u(t)\|^2_{H^2}\, dt\bigg] \exp( (2+8C_2)T).
   \end{align*}
  We already have $u_h,u\in L^\infty(0,T;H_0^1(\Omega))\hookrightarrow L^\infty(0,T;L^p(\Omega))$ and bounded by the given data, we can use \eqref{u^q to u^p journey} to conclude
   \begin{align*}
   &	\|u_h(t)-u(t)\|^2_{L^\infty(0,T;L^2)}+\nu\|u_h(t)-u(t)\|^2_{L^2(0,T;H_0^1)}+\frac{\alpha}{2^{p-3}}\|u_h(t)-u(t)\|^p_{L^p(0,T;L^p)} \\
   	&\leq Ch^2\bigg\{\|u_0\|^2_{H_0^1}+\|u\|^2_{L^{\infty}(0,T;H_0^1)}+\int_0^T\|\partial_t u(t)\|_{H_0^1}\, dt\\& \quad +(\|u_h\|^{2(p-2)}_{L^\infty(0,T;L^p)}+\|u\|^{2(p-2)}_{L^\infty(0,T;L^p)})\int_0^T \|u(t)\|^2_{H^2}\,dt \bigg\},
   \end{align*}
   that leads to desired result using Theorem \ref{energy estimate of CFEM}(a).
   
   \noindent
   (b) \textbf{For $\frac{2d}{d-2}<p\leq \frac{2d-6}{d-4}$:}
    We estimate $I_1$ and $I_2$ by using the Cauchy-Schwarz and Young's inequalities as 
    \begin{align*}
    |I_1|= |(u_h-u, \partial_t(u-W))|\leq \|u_h-u\|_{L^2} \|\partial_t(u-W)\|_{L^2} \leq \frac{1}{2}\|u_h-u\|^2_{L^2} +\frac{1}{2}\|\partial_t(u-W)\|^2_{L^2},
    \end{align*}
    and
    \begin{align*}
        |I_2|&= | \nu(\nabla (u_h-u), \nabla (u-W))|\\
        &\leq \nu \|\nabla (u_h-u)\|_{L^2} \|\nabla (u - W) \|_{L^2} \leq \frac{3\nu}{4}\|\nabla(u_h-u)\|^2_{L^2} + \frac{\nu}{3} \|\nabla (u-W)\|^2_{L^2}.
    \end{align*}
    To obtain the optimal convergence order under minimal regularity, we estimate $I_3$ and $I_4$
    separately for $2\leq q_{\ell}<1+\frac{p}{2}$ and $1+\frac{p}{2}\leq q_{\ell}<p$  as follows:
     
     (i) \textbf{For $2\leq q_{\ell}<1+\frac{p}{2}$ :}
     
      Using Taylor's formula, H\"{o}lder's, Young's, and Gagliardo-Nirenberg inequalities \cite[Theorem 1.5.2, Chapter 1]{MR2962068}, we estimate $I_3$ for $p\geq 2$ as 
     \begin{align*}
     |I_3|&=|- \alpha (|u_h|^{p-2}u_h- |u|^{p-2}u, u - W)| \\
     &\leq (p-1)\left \vert \alpha ((\theta(u_h-u)+u)^{p-2}(u_h-u),u-W)\right \vert \\
     &\leq 2^{p-3}\alpha(p-1)((|u_h-u|^{p-2}+|u|^{p-2})|u_h-u|,|u-W|)\\
     &\leq 2^{p-3}\alpha(p-1)\big((|u_h-u|^{p-1},|u-W|)+(|u|^{p-2}|u_h-u|,|u-W|)\big)\\
     &\leq 2^{p-3}\alpha (p-1)\|u_h-u\|^{p-1}_{L^p}\|u-W\|_{L^p}+2^{p-3}\alpha(p-1)\|u\|^{p-2}_{L^{2(p-1)}} \|u_h-u\|_{L^2}\|u-W\|_{L^{2(p-1)}}\\&\leq \frac{\alpha}{2^{p}}\|u_h-u\|^{p}_{L^p}+\bigg(\frac{2^{p-3}\alpha(p-1)}{p}\bigg)\bigg(\frac{2^{2p-3}(p-1)^2}{p}\bigg)^{p-1}\|u-W\|^p_{L^p}\\&\quad +2^{p-2}\alpha(p-1)\bigg(\|u\|^{2(p-2)}_{L^{2(p-1)}}\|u_h-u\|^2_{L^2}+\|u-W\|^2_{L^{2(p-1)}}\bigg).
    \end{align*}
    Let us now estimate $I_4$. Note that, we have $2\leq 2(q_{\ell}-1)<p$ from the range of $q_{\ell}$. Therefore, we infer 
    \begin{align*}
    	|I_4| &= \bigg|\sum_{\ell=1}^M \beta_{\ell} (|u_h|^{q_{\ell}-2}u_h-|u|^{q_{\ell}-2}u, u-W)\bigg| \\
    	&\leq \bigg|\sum_{\ell=1}^M \beta_{\ell}(q_{\ell}-1)((\theta u_h + (1-\theta)u)^{q_{\ell}-2}(u_h-u),u-W)\bigg|\\
    	&\leq \sum_{\ell=1}^M |\beta_{\ell}|(q_{\ell}-1)((|u_h|+|u|)^{q_{\ell}-2}|u_h-u|,|u-W|)\\
    	&\leq  \sum_{\ell=1}^M |\beta_{\ell}|(q_{\ell}-1)2^{q_{\ell}-3}((|u_h|^{q_{\ell}-2}+|u|^{q_{\ell}-2})|u_h-u|,|u-W|)\\
    	&\leq \sum_{\ell=1}^M |\beta_{\ell}|(q_{\ell}-1)2^{q_{\ell}-3}(\|u_h\|_{L^{2(q_{\ell}-1)}}^{q_{\ell}-2}+\|u\|_{L^{2(q_{\ell}-1)}}^{q_{\ell}-2})\|u_h-u\|_{L^2}\|u-W\|_{L^{2(q_{\ell}-1)}}\\&\leq \sum_{\ell=1}^M |\beta_{\ell}|(q_{\ell}-1)2^{q_{\ell}-3} |\Omega|^{\frac{(p-2q_{\ell}+2)(q_{\ell}-2)}{2p(q_{\ell}-1)}}(\|u_h\|_{L^p}^{q_{\ell}-2}+\|u\|^{q_{\ell}-2}_{L^p})\|u_h-u\|_{L^2} |\Omega|^{\frac{p-2q_{\ell}+2}{2p(q_{\ell}-1)}}\|u-W\|_{L^p}\\& \leq \sum_{\ell=1}^M |\beta_{\ell}|(q_{\ell}-1)2^{q_{\ell}-3} |\Omega|^{\frac{p-2q_{\ell}+2}{2p}}(\|u_h\|_{L^p}^{q_{\ell}-2}+\|u\|^{q_{\ell}-2}_{L^p})\|u_h-u\|_{L^2}\|u-W\|_{L^p}\\&\leq 2\bigg[\sum_{\ell=1}^M |\beta_{\ell}|(q_{\ell}-1)2^{q_{\ell}-3} |\Omega|^{\frac{p-2q_{\ell}+2}{2p}}\bigg]^2 (\|u_h\|_{L^p}^{2q_{\ell}-4}+\|u\|^{2q_{\ell}-4}_{L^p})\|u_h-u\|^2_{L^2} + \|u-W\|^2_{L^p}.
    \end{align*}
    Using these estimates in \eqref{SFEM err-03}, we arrive at 
    \begin{align*}
    &\frac{1}{2}\frac{d}{dt}\|u_h(t)-u(t)\|^2_{L^2}+\frac{\nu}{4}\|\nabla(u_h(t)-u(t))\|^2_{L^2}+\frac{\alpha}{2^{p}}\|u_h(t)-u(t)\|^p_{L^p}\\&\leq -\partial_t(u_h(t)-u(t),u(t)-W(t))+\bigg(\frac{1}{2}+2C_2\bigg)\|u_h(t)-u(t)\|^2_{L^2}+\frac{1}{2}\|\partial_t(u(t)-W(t))\|^2_{L^2}\\&\quad +\frac{\nu}{3}\|\nabla(u(t)-W(t))\|^2_{L^2}+\bigg(\frac{2^{p-3}\alpha(p-1)}{p}\bigg)\bigg(\frac{2^{2p-3}(p-1)^2}{p}\bigg)^{p-1}\|u(t)-W(t)\|^p_{L^p}\\&\quad + 2^{p-2}\alpha(p-1)\bigg(\|u(t)\|^{2(p-2)}_{L^{2(p-1)}}\|u_h(t)-u(t)\|^2_{L^2}+\|u(t)-W(t)\|^2_{L^{2(p-1)}}\bigg)+\|u(t)-W(t)\|^2_{L^p}\\&\quad + 2\bigg[\sum_{\ell=1}^M |\beta_{\ell}|(q_{\ell}-1)2^{q_{\ell}-3} |\Omega|^{\frac{p-2q_{\ell}+2}{2p}}\bigg]^2 (\|u_h(t)\|_{L^p}^{2q_{\ell}-4}+\|u(t)\|^{2q_{\ell}-4}_{L^p})\|u_h(t)-u(t)\|^2_{L^2},
	\end{align*}
 for  a.e. $t\in[0,T]$.   Integrating from $0$ to $t$, we obtain
   { \small
    \begin{align*}
    	&\|u_h(t)-u(t)\|^2_{L^2}+\frac{\nu}{2}\int_0^t\|\nabla(u_h(s)-u(s))\|^2_{L^2}\, ds+\frac{\alpha}{2^{p-1}}\int_0^t \|u_h(s)-u(s)\|^p_{L^p}\,ds \\&\leq \|u_h(0)-u_0\|^2_{L^2}-2(u_h(t)-u(t),u(t)-W(t))+2(u_h(0)-u_0,u_0-W(0))\\&\quad + (1+4C_2)\int_0^t \|u_h(s)-u(s)\|^2_{L^2} \, ds+\int_0^t \|\partial_s(u(s)-W(s))\|^2_{L^2}\, ds+\frac{2\nu}{3}\int_0^t \|\nabla(u(s)-W(s))\|^2_{L^2}\, ds\\ &\quad + \bigg(\frac{2^{p-2}\alpha(p-1)}{p}\bigg)\bigg(\frac{2^{2p-3}(p-1)^2}{p}\bigg)^{p-1}\int_0^t \|u(s)-W(s)\|^p_{L^p}+2\int_0^t \|u(s)-W(s)\|^2_{L^p}\, ds\\&\quad + 2^{p-1}\alpha(p-1)\bigg(\int_0^t \|u(s)\|^{2(p-2)}_{L^{2(p-1)}}\|u_h(s)-u(s)\|^2_{L^2}\, ds+\int_0^t \|u(s)-W(s)\|^2_{L^{2(p-1)}}\, ds\bigg)\\&\quad + 4\bigg[\sum_{\ell=1}^M |\beta_{\ell}|(q_{\ell}-1)2^{q_{\ell}-3} |\Omega|^{\frac{p-2q_{\ell}+2}{2p}}\bigg]^2 \int_0^t (\|u_h(s)\|_{L^p}^{2q_{\ell}-4}+\|u(s)\|^{2q_{\ell}-4}_{L^p})\|u_h(s)-u(s)\|^2_{L^2}\, ds,
    \end{align*}}
 for all $t\in[0,T]$.    Using the Cauchy-Schwarz and Young's inequalities, we estimate $- 2(u_h-u, u-W)$ and $2 (u_h(0)-u_0, u_0-W(0)) $ as,
    \begin{align*}
    	-2(u_h - u, u - W) & \leq \frac{1}{2} \|u_h - u\|_{L^2}^2 +  2\|u - W\|_{L^2}^2, \\
    	2(u_h(0) - u_0, u_0 - W(0)) &\leq \|u_h(0) - u_0\|_{L^2}^2 + \|u_0 - W(0)\|_{L^2}^2.
    \end{align*}
    Using these estimates in the above equation, applying Gronwall's inequality, we reach at the relation
    \begin{align*}
    &	\sup_{0\leq t\leq T}\|u_h(t)-u(t)\|^2_{L^2}+\nu \int_0^T\|\nabla(u_h(t)-u(t))\|^2_{L^2}\, dt+\frac{\alpha}{2^{p-2}}\int_0^T \|u_h(t)-u(t)\|^p_{L^p}\, dt\\&\leq \bigg[4\|u_h(0)-u(0)\|^2_{L^2}+2\|u_0-W(0)\|^2_{L^2}+4\|u(t)-W(t)\|^2_{L^\infty(0,T;L^2)}\\&\quad +2\int_0^T \|\partial_t(u(t)-W(t))\|^2_{L^2}\, dt+\frac{4\nu}{3}\int_0^T \|\nabla(u(t)-W(t))\|^2_{L^2}\, dt+4\int_0^T \|u(t)-W(t)\|^2_{L^p}\, dt\\ &\quad + 2\bigg(\frac{2^{p-2}\alpha(p-1)}{p}\bigg)\bigg(\frac{2^{2p-3}(p-1)^2}{p}\bigg)^{p-1}\int_0^T \|u(t)-W(t)\|^p_{L^p}\, dt\\&\quad +2^{p}\alpha(p-1)\int_0^T\|u(t)-W(t)\|^2_{L^{2(p-1)}}\, dt\bigg]\exp\bigg(2^{p-2}\alpha(p-1)\int_0^T\|u(t)\|^{2p-4}_{L^{2(p-1)}}\, dt\\&\quad + 4\bigg[\sum_{\ell=1}^M |\beta_{\ell}|(q_{\ell}-1)2^{q_{\ell}-3} |\Omega|^{\frac{p-2q_{\ell}+2}{2p}}\bigg]^2 \int_0^T (\|u_h(t)\|_{L^p}^{2q_{\ell}-4}+\|u(t)\|^{2q_{\ell}-4}_{L^p})\, dt+ (1+4C_2)T\bigg).
    \end{align*}
    Since, $u_h \in L^\infty(0,T;L^p)$ and $u\in L^\infty(0,T;H^2(\Omega))$, the terms appearing in exponential is bounded. Also, we can estimate the terms in the above expression by taking $W=S_hu$ and using \eqref{scott zhang approximation result} as
    \begin{align*}
    	\|u_0-S_hu_0\|^2_{L^2}&\leq Ch^2\|u_0\|_{H_0^1}, \ \ 
    	\|u-S_hu\|^2_{L^2}\leq Ch^2\|u\|_{H_0^1}, \\
    	\|\partial_t(u-S_hu)\|_{L^2}^2&\leq Ch^2\|\partial_tu\|_{H_0^1},\ \ 
    	\|\nabla(u-S_hu)\|^2_{L^2}\leq Ch^2\|u\|^2_{H^2}.
    \end{align*}    
     We can estimate $\|u-W\|_{L^p}$ by using Gagliardo-Nirenberg's interpolation inequality and \eqref{to bound u-W in Lp} as
    \begin{align*}
    	\|u-W\|^2_{L^p} &\leq C\|A^{\theta}(u-W)\|^2_{L^2}\leq Ch^{2s-4\theta}\|u\|^2_{H^s} \ \text{ for }\ \frac{2d}{d-2}<p\leq \frac{2d}{d-4},
    \end{align*}
    where $\theta = \frac{d(p-2)}{4p}.$ To obtain optimal order, we need $2s-4\theta=2,$ which implies $s=1+d(\frac{1}{2}-\frac{1}{p})$. It means $u$ should belong to $L^2(0,T;H^{1+d(\frac{1}{2}-\frac{1}{p})}(\Omega)).$ Note that, for  $\frac{2d}{d-2}<p \leq \frac{2d}{d-4}$, we conclude $1+d(\frac{1}{2}-\frac{1}{p}) \leq 3$. Hence, we can rewrite the above estimate as
    \begin{align}\label{H3 regularity optimal order-q-01}
    	\|u-W\|^2_{L^p} \leq Ch^2 \|u\|^2_{D(A^{\frac{3}{2}})}.
    \end{align}
    Since, $\frac{2d-6}{d-4}<\frac{2d}{d-4}$, the above estimates holds  true for all $p\leq \frac{2d-6}{d-4}.$
    Again from \eqref{to bound u-W in Lp}, we also have
    \begin{align*}
    	\|u-W\|^p_{L^p}\leq Ch^{p(s-2\theta)}\|u\|^p_{H^s},
    \end{align*}
    where $\theta=\frac{d(p-2)}{4p}$. To obtain optimal order, we need $p(s-2\theta)=2,$ that is, $s=\frac{2}{p}+d\big(\frac{1}{2}-\frac{1}{p}\big),$ which requires $u \in L^\infty(0,T;H^{\frac{2}{p}+d\big(\frac{1}{2}-\frac{1}{p}\big)}(\Omega))$. Therefore, for $\frac{2}{p}+d\big(\frac{1}{2}-\frac{1}{p}\big)\leq 2$, we conclude that $p \leq \frac{2d-4}{d-4}$ which is also true for $p\leq \frac{2d-6}{d-4}$ and
    \begin{align}\label{u-w Lp in H2}
    	\|u-W\|^p_{L^p}\leq Ch^2\|u\|^p_{H^2}.
    \end{align}
    Similarly, one can show that 
    \begin{align*}
    	\|u-W\|^2_{L^{2(p-1)}} \leq Ch^2\|u\|^2_{D(A^{\frac{3}{2}})} \quad \text{ for } \frac{2d}{d-2}<p\leq \frac{2d-6}{d-4}.
    \end{align*}
    Hence, combining all these estimates, we arrive at 
    \begin{align*}
    	 &	\sup_{0\leq t\leq T}\|u_h(t)-u(t)\|^2_{L^2}+\nu \int_0^T\|\nabla(u_h(t)-u(t))\|^2_{L^2}\, dt+\frac{\alpha}{2^{p-2}}\int_0^T \|u_h(t)-u(t)\|^p_{L^p}\, dt\\&\leq Ch^2\bigg(\|u_0\|^2_{H_0^1}+\|u\|^2_{L^\infty(0,T;H_0^1)}+\int_0^T \|\partial_t u(t)\|^2_{H_0^1}\, dt + \int_0^T \|u(t)\|^2_{H^2}\,dt+\int_0^T \|u(t)\|^p_{H^2}\, dt\\&\quad +\int_0^T \|u(t)\|^2_{D(A^{\frac{3}{2}})}\, dt\bigg),
    \end{align*}
    which completes the proof of Theorem \ref{semidiscrete error analysis}(b)(i). 
   
   (ii) \textbf{ For $1+\frac{p}{2}\leq q_{\ell}<p$ :}
   
   Estimating $I_3$ in a similar way as previously done.
    \begin{align*}
     |I_3|&\leq \frac{\alpha}{2^{p+1}}\|u_h-u\|^{p}_{L^p}+\bigg(\frac{2^{p-3}\alpha(p-1)}{p}\bigg)\bigg(\frac{2^{2p-2}(p-1)^2}{p}\bigg)^{p-1}\|u-W\|^p_{L^p}\\&\quad +2^{p-2}\alpha(p-1)\bigg(\|u\|^{2(p-2)}_{L^{2(p-1)}}\|u_h-u\|^2_{L^2}+\|u-W\|^2_{L^{2(p-1)}}\bigg).
      \end{align*}
    Next, we estimate $I_4$ in the following way:
       \small
     \begin{align*}
     |I_4| &= \bigg|\sum_{\ell=1}^M \beta_{\ell} (|u_h|^{q_{\ell}-2}u_h-|u|^{q_{\ell}-2}u, u-W)\bigg| \\
     &\leq \bigg|\sum_{\ell=1}^M |\beta_{\ell}| (q_{\ell}-1)((\theta (u_h-u) +u)^{q_{\ell}-2}(u_h-u), u-W )\bigg|\\
     &\leq  \sum_{\ell=1}^M |\beta_{\ell}|(q_{\ell}-1)2^{q_{\ell}-3}\big((|u_h-u|^{q_{\ell}-2}+|u|^{q_{\ell}-2})|u_h-u|,|u-W|\big)\\
     &\leq  \sum_{\ell=1}^M |\beta_{\ell}|(q_{\ell}-1)2^{q_{\ell}-3}\big((|u_h-u|^{q_{\ell}-1},|u-W|)+(|u|^{q_{\ell}-2}|u_h-u|,|u-W|)\big)\\
     &\leq \sum_{\ell=1}^M |\beta_{\ell}|(q_{\ell}-1)2^{q_{\ell}-3}\bigg(\|u_h-u\|^{q_{\ell}-1}_{L^{q_{\ell}}}\|u-W\|_{L^{q_{\ell}}}+ \|u\|^{q_{\ell}-2}_{L^{2(q_{\ell}-1)}} \|u_h-u\|_{L^2}\|u-W\|_{L^{2(q_{\ell}-1)}}\bigg)\\
     &\leq \sum_{\ell=1}^M |\beta_{\ell}|(q_{\ell}-1)2^{q_{\ell}-3} |\Omega|^{\frac{p-q_{\ell}}{p}}\|u_h-u\|^{q_{\ell}-1}_{L^p}\|u-W\|_{L^{q_{\ell}}}\\& \quad + \sum_{\ell=1}^M |\beta_{\ell}|(q_{\ell}-1)2^{q_{\ell}-2}\bigg(\|u\|^{2(q_{\ell}-2)}_{L^{2(q_{\ell}-1)}}\|u_h-u\|^2_{L^2}+\|u-W\|^2_{L^{2(q_{\ell}-1)}}\bigg)\\
     &\leq \frac{\alpha}{2^{p+1}}\|u_h-u\|^p_{L^p}+\sum_{\ell=1}^M |\beta_{\ell}|\bigg(\frac{M|\beta_{\ell}|2^{p+1}(q_{\ell}-1)}{\alpha p}\bigg)^{\frac{q_{\ell}-1}{p-q_{\ell}+1}}\bigg(\frac{p-q_{\ell}+1}{p}\bigg)|\Omega|^{\frac{p-q_{\ell}}{p-q_{\ell}+1}}\|u-W\|^{\frac{p}{p-q_{\ell}+1}}_{L^{p}}\\&\quad + \sum_{\ell=1}^M |\beta_{\ell}|(q_{\ell}-1)2^{q_{\ell}-2}\bigg(\|u\|^{2(q_{\ell}-2)}_{L^{2(q_{\ell}-1)}}\|u_h-u\|^2_{L^2}+|\Omega|^{\frac{p-2q_{\ell}+2}{p(q_{\ell}-1)}}\|u-W\|^2_{L^p}\bigg).
     \end{align*}
     \normalsize
     Using all these estimates in \eqref{SFEM err-03}, we obtain the following:
     \small
         \begin{align*}
              &\frac{1}{2}\frac{d}{dt}\|u_h(t)-u(t)\|^2_{L^2}+ \frac{\nu}{4} \|\nabla (u_h(t)-u(t)\|^2_{L^2}+\frac{\alpha}{2^p}\|u_h(t)-u(t)\|^p_{L^p}\\
            &\leq - \partial_t(u_h(t)-u(t),u(t)-W(t))+\bigg(\frac{1}{2}+2C_2\bigg) \|u_h(t)-u(t)\|^2_{L^2} +\frac{1}{2}\|\partial_t(u(t)-W(t))\|^2_{L^2}\\&\quad+\frac{\nu}{3}\|\nabla(u(t)-W(t))\|^2_{L^2} + \bigg(\frac{2^{p-3}\alpha(p-1)}{p}\bigg)\bigg(\frac{2^{2p-2}(p-1)^2}{p}\bigg)^{p-1}\|u(t)-W(t)\|^p_{L^p}\\&\quad+2^{p-2}\alpha(p-1)\|u(t)-W(t)\|^2_{L^{2(p-1)}}+\sum_{\ell=1}^M|\beta_{\ell}|(q_{\ell}-1)2^{q_{\ell}-2}|\Omega|^{\frac{p-2q_{\ell}+2}{p(q_{\ell}-1)}}\|u(t)-W(t)\|^2_{L^p}\\&\quad + \sum_{\ell=1}^M |\beta_{\ell}|\bigg(\frac{M|\beta_{\ell}|2^{p+1}(q_{\ell}-1)}{\alpha p}\bigg)^{\frac{q_{\ell}-1}{p-q_{\ell}+1}}\bigg(\frac{p-q_{\ell}+1}{p}\bigg)|\Omega|^{\frac{p-q_{\ell}}{p-q_{\ell}+1}}\|u(t)-W(t)\|^{\frac{p}{p-q_{\ell}+1}}_{L^{p}} \\&\quad +\bigg(2^{p-2}\alpha(p-1)\|u(t)\|^{2(p-2)}_{L^{2(p-1)}}+ \sum_{\ell=1}^M|\beta_{\ell}|(q_{\ell}-1)2^{q_{\ell}-2}\|u(t)\|^{2(q_{\ell}-2)}_{L^{2(q_{\ell}-1)}}\bigg)\|u_h(t)-u(t)\|^2_{L^2},
         \end{align*}
         \normalsize
for a.e. $t\in[0,T]$.      Integrating the above equation from $0$ to $t$, we get
     \small
        \begin{align*}
        &\|u_h(t)-u(t)\|^2_{L^2}+\frac{\nu}{2}\int_0^t \|\nabla(u_h(s)-u(s))\|^2_{L^2}\, ds + \frac{\alpha}{2^{p-1}}\int_0^t \|u_h(s)-u(s)\|^p_{L^p}\, ds \\&\leq \|u_h(0)-u_0\|^2_{L^2}-2(u_h(t)-u(t), u(t)-W(t))+2(u_h(0)-u_0,u_0-W(0))\\&\quad +(1+4C_2) \int_0^t \|u_h(s)-u(s)\|^2_{L^2}\, ds + \int_0^t \|\partial_s(u(s)-W(s))\|^2_{L^2} \, ds \\&\quad + \frac{2\nu}{3}\int_0^t \|\nabla(u(s)-W(s))\|^2_{L^2}\, ds +  \bigg(\frac{2^{p-2}\alpha(p-1)}{p}\bigg)\bigg(\frac{2^{2p-2}(p-1)^2}{p}\bigg)^{p-1}\int_0^t\|u(s)-W(s)\|^p_{L^p} \, ds \\&\quad + 2^{p-1}\alpha(p-1)\int_0^t\|u(s)-W(s)\|^2_{L^{2(p-1)}}\, ds+\sum_{\ell=1}^M|\beta_{\ell}|(q_{\ell}-1)2^{q_{\ell}-1}|\Omega|^{\frac{p-2q_{\ell}+2}{p(q_{\ell}-1)}}\int_0^t\|u(s)-W(s)\|^2_{L^p}\, ds\\&\quad + 2\sum_{\ell=1}^M |\beta_{\ell}|\bigg(\frac{M|\beta_{\ell}|2^{p+1}(q_{\ell}-1)}{\alpha p}\bigg)^{\frac{q_{\ell}-1}{p-q_{\ell}+1}}\bigg(\frac{p-q_{\ell}+1}{p}\bigg)|\Omega|^{\frac{p-q_{\ell}}{p-q_{\ell}+1}}\int_0^t \|u(s)-W(s)\|^{\frac{p}{p-q_{\ell}+1}}_{L^{p}} \, ds\\&\quad + \int_0^t \bigg(2^{p-1}\alpha(p-1)\|u(s)\|^{2(p-2)}_{L^{2(p-1)}}+ \sum_{\ell=1}^M|\beta_{\ell}|(q_{\ell}-1)2^{q_{\ell}-1}\|u(s)\|^{2(q_{\ell}-2)}_{L^{2(q_{\ell}-1)}}\bigg)\|u_h(s)-u(s)\|^2_{L^2}\, ds,
        \end{align*}
        \normalsize
 for all $t\in[0,T]$.    Using the Cauchy-Schwarz and Young's inequalities, we estimate $- 2(u_h-u, u-W)$ and $2 (u_h(0)-u_0, u_0-W(0)) $ as
     \begin{align*}
-2(u_h - u, u - W) &\leq 2\|u_h - u\|_{L^2} \|u - W\|_{L^2} \leq \frac{1}{2} \|u_h - u\|_{L^2}^2 +  2\|u - W\|_{L^2}^2, \\
2(u_h(0) - u_0, u_0 - W(0)) &\leq 2\|u_h(0) - u_0\|_{L^2} \|u_0 - W(0)\|_{L^2} \nonumber\\&\leq \|u_h(0) - u_0\|_{L^2}^2 + \|u_0 - W(0)\|_{L^2}^2.
\end{align*}
    Using the estimates in the above equation and applying Gronwall's inequality, we arrive at 
     \small
     \begin{align*}
            &\sup_{0\leq t\leq T}\|u_h(t)-u(t)\|^2_{L^2}+\nu \int_0^T \|\nabla(u_h(t)-u(t))\|^2_{L^2}\, dt + \frac{\alpha}{2^{p-2}}\int_0^T \|u_h(t)-u(t)\|^p_{L^p}\, dt \\&\leq \bigg[4\|u_h(0)-u_0\|^2_{L^2}+2\|u_0-W(0)\|^2_{L^2}+4\|u(t)-W(t)\|^2_{L^\infty(0,T;L^2)} + 2\int_0^T \|\partial_t(u(t)-W(t))\|^2_{L^2} \, dt \\&\quad+ \frac{4\nu}{3}\int_0^T \|\nabla(u(t)-W(t))\|^2_{L^2}\, dt +  \bigg(\frac{2^{p-1}\alpha(p-1)}{p}\bigg)\bigg(\frac{2^{2p-2}(p-1)^2}{p}\bigg)^{p-1}\int_0^T\|u(t)-W(t)\|^p_{L^p} \, dt \\&\quad + 2^{p}\alpha(p-1)\int_0^T\|u(t)-W(t)\|^2_{L^{2(p-1)}}\, dt+\sum_{\ell=1}^M|\beta_{\ell}|(q_{\ell}-1)2^{q_{\ell}}|\Omega|^{\frac{p-2q_{\ell}+2}{p(q_{\ell}-1)}}\int_0^T\|u(t)-W(t)\|^2_{L^p}\, dt\\&\quad + 4\sum_{\ell=1}^M |\beta_{\ell}|\bigg(\frac{M|\beta_{\ell}|2^{p+1}(q_{\ell}-1)}{\alpha p}\bigg)^{\frac{q_{\ell}-1}{p-q_{\ell}+1}}\bigg(\frac{p-q_{\ell}+1}{p}\bigg)|\Omega|^{\frac{p-q_{\ell}}{p-q_{\ell}+1}}\int_0^T \|u(t)-W(t)\|^{\frac{p}{p-q_{\ell}+1}}_{L^{p}} \, dt\bigg]\\&\quad \times\exp \bigg( \int_0^T \bigg(2^{p-1}\alpha(p-1)\|u(t)\|^{2(p-2)}_{L^{2(p-1)}}+ \sum_{\ell=1}^M|\beta_{\ell}|(q_{\ell}-1)2^{q_{\ell}-1}\|u(t)\|^{2(q_{\ell}-2)}_{L^{2(q_{\ell}-1)}}+2+8C_2\bigg)\, dt\bigg).
     \end{align*}
     \normalsize
     Since $u \in L^\infty(0,T;H^2(\Omega))$, therefore, the terms appearing in the exponential are uniformly bounded and independent of $h$ for $u_0 \in D(A)$. Setting  $W=S_hu$ and applying the same estimates as in the previous case, we estimate $\|u-W\|^{\frac{p}{p-q_{\ell}+1}}_{L^p}$. To achieve the optimal convergence order, the condition ${\frac{p}{p-q_{\ell}+1}}\geq 2$ requires $q_{\ell}\geq 1+\frac{p}{2}$. This motivates splitting the range of $q_{\ell}$, 
     into two cases. Hence, proceeding with calculations analogous to the previous case with $W=S_h u$,  we obtain 
     \begin{align*}
     	\|u-W\|^{\frac{p}{p-q_{\ell}+1}}_{L^p}\leq 	\|u-W\|^{2}_{L^p}\|u-W\|^{\frac{2q_{\ell}-p-2}{p-q_{\ell}+1}}_{L^p}\leq Ch^2\|u\|^2_{D(A^{\frac{3}{2}})}\|u-W\|^{\frac{2q_{\ell}-p-2}{p-q_{\ell}+1}}_{H^2}.
     \end{align*}
       Using \eqref{scott zhang approximation result}, we get
       \begin{align*}
       	\|u-W\|^{\frac{p}{p-q_{\ell}+1}}_{L^p}\leq Ch^2\|u\|^2_{D(A^{\frac{3}{2}})}\|u\|_{H^2}^{\frac{2q_{\ell}-p-2}{p-q_{\ell}+1}},
       \end{align*}
       which implies
       \begin{align*}
       	\int_0^T \|u(t)-W(t)\|^{\frac{p}{p-q_{\ell}+1}}_{L^p}\, dt \leq Ch^2 \|u\|^{\frac{2q_{\ell}-p-2}{p-q_{\ell}+1}}_{L^\infty(0,T;H^2)}\int_0^T \|u(t)\|^2_{D(A^{\frac{3}{2}})}\, dt.
       \end{align*}
        Hence, by combining all these estimates we arrive to the following error estimate for $1+\frac{p}{2}\leq q_{\ell}<p$:
     \begin{align*}
     	&	\sup_{0\leq t\leq T}\|u_h(t)-u(t)\|^2_{L^2}+\nu \int_0^T\|\nabla(u_h(t)-u(t))\|^2_{L^2}\, dt+\frac{\alpha}{2^{p-2}}\int_0^T \|u_h(t)-u(t)\|^p_{L^p}\, dt\\&\leq Ch^2\bigg(\|u_0\|^2_{H_0^1}+\|u\|^2_{L^\infty(0,T;H_0^1)}+\int_0^T \|\partial_t u(t)\|^2_{H_0^1}\, dt + \int_0^T \|u(t)\|^2_{H^2}\,dt+\int_0^T \|u(t)\|^p_{H^2}\, dt\\&\quad +\|u\|^{\frac{2q_{\ell}-p-2}{p-q_{\ell}+1}}_{L^\infty(0,T;H^2)}\int_0^T \|u(t)\|^2_{D(A^{\frac{3}{2}})}\, dt\bigg).
     \end{align*}
This completes the proof and yields the desired error estimate. For $\frac{2d}{d-2}<p\leq \frac{2d-6}{d-4}$, the argument was divided into two ranges of $q_{\ell}$
because of the difficulty in estimating 
$I_4$. Nevertheless, both cases lead to the optimal order of convergence.
\end{proof}
\begin{Remark}
	For $p>\tfrac{2d-6}{d-4}$, error estimates can be derived under the assumption of additional regularity on the initial data and the forcing term. In this case, the data must also satisfy certain compatibility conditions, similar to those described in \cite[Section 7.1.3, Chapter 7]{MR1625845}. Since numerical studies are generally restricted to one, two, or three spatial dimensions, we do not provide further details here.
\end{Remark}
\subsection{Fully-discrete conforming FEM}\label{fully discrete section}
In this section, we introduce a fully-discrete conforming finite element scheme. We partition the time interval $[0,T]$ into $0 = t_0 < t_1 < t_2 < \dots < t_N = T$ with uniform time stepping $\Delta t$, that is, $t_k = k \Delta t$. We replace the time derivative with a backward Euler discretization. We denote $C$ as a generic constant which is independent of $h$ and $\Delta t$. We then define a generic finite element approximate solution $u_{kh}$ by 
\begin{equation}\label{generic fully-discrete solution}
u_{kh}|_{[t_{k-1},t_k]} =u_h^{k-1} + \bigg(\frac{t-t_{k-1}}{\Delta t}\bigg)(u_h^k-u_h^{k-1}), \quad 1\leq k\leq N,\  \text{  for } \ t\in [t_{k-1},t_k]. 
\end{equation}
Note that $\partial_t u_{kh}=\frac{u_h^k-u_h^{k-1}}{\Delta t}$. The continuous Galerkin approximation of the problem \eqref{damped weak} is defined in the following way: Find $u_h^k \in V_h$, for $k=1,2,\ldots,N,$ such that for $\chi \in V_h$:
\begin{equation}\label{4.13}
\begin{aligned}
	\begin{cases}
		\displaystyle
    \left( \frac{u_h^k - u_h^{k-1}}{\Delta t}, \chi \right) 
    - \nu (\Delta u_h^k, \chi) 
    + \alpha (|u_h^k|^{p-2} u_h^k, \chi)-\sum_{\ell=1}^M \beta_{\ell}(|u_h^k|^{q_{\ell}-2}u_h^k, \chi)
    = \langle f^k, \chi \rangle, \\
    
    (u_h^0(x_i), \chi) = 
    \begin{cases}
    (R_hu_0(x_i), \chi),& \text{ for } 2\leq p\leq \frac{2d}{d-2},\\
    (S_hu_0(x_i),\chi),& \text{ for } \frac{2d}{d-2}<p\leq \frac{2d-6}{d-4},
	\end{cases}
	\end{cases}
\end{aligned}
\end{equation}
where $i=1,2,\ldots, N$, $f^k= (\Delta t)^{-1}\int_{t_{k-1}}^{t_k}f(s)\,ds$ for $f \in L^2(0,T;H^{-1}(\Omega))$, $u_h^0$ approximates $u_0$ in $V_h$ and $x_i, i=1,\ldots,N,$ represents the nodes of triangulation $\mathcal{T}_h$ of $\Omega$. Note that the initial data is specified according to the values of 
$p$ given in Table \ref{values of p}.

The existence of a discrete solution can be established by applying Brouwer's fixed point theorem \cite[Lemma 1.4, p. 110]{MR609732}, as in \cite[Theorem 3.1]{MR4288303}. Indeed, we have the following result:
\begin{Lemma}
For a given $ {u_h^{k-1}}\in V_h,$ and $f\in  H^{-1}(\Omega)$, 	there exists a solution $u_h^k\in V_h$ for the problem \eqref{4.13}. 
\end{Lemma}
\begin{proof}
 For a given $ {u_h^{k-1}}\in V_h$  and $f\in  H^{-1}(\Omega)$,  we consider 
\begin{align}
	 \left( \frac{u_h^k }{\Delta t}, \chi \right) 
	- \nu (\Delta u_h^k, \chi) 
	+ \alpha (|u_h^k|^{p-2} u_h^k, \chi)-\sum_{\ell=1}^M \beta_{\ell}(|u_h^k|^{q_{\ell}-2}u_h^k, \chi)
	= 
	\left\langle f^k+\frac{u_h^{k-1}}{\Delta t}\, , \chi \right\rangle,
\end{align}
for any $\chi \in V_h$. Let the inherited  scalar product and norm from $H_0^1(\Omega)$ in $V_h$ be denoted by $[\cdot,\cdot]$ and $[\cdot]$. Then for the projection $P=P_h:H_0^1(\Omega)\to V_h$, we consider 
\begin{align}
	[P(u),v]= \left( \frac{u }{\Delta t}, v \right) 
	+ \nu (\nabla u, \nabla v) 
	+ \alpha (|u|^{p-2} u, v)-\sum_{\ell=1}^M \beta_{\ell}(|u|^{q_{\ell}-2}u, v)-
	\left\langle g , v \right\rangle,
\end{align}
for all $u,v\in V_h$, where $g= f^k+\frac{u_h^{k-1}}{\Delta t}\in H^{-1}(\Omega)$. The continuity of the mapping $P$ is clear. Let us now consider
\begin{align}
	[P(u),u]&= \left( \frac{u }{\Delta t}, u \right) 
	+\nu (\nabla u, \nabla u) 
	+ \alpha (|u|^{p-2} u, u)-\sum_{\ell=1}^M \beta_{\ell}(|u|^{q_{\ell}-2}u, u)-
	\left\langle g , u \right\rangle\nonumber\\&=\frac{1}{\Delta t}\|u\|_{L^2}^2+\nu\|\nabla u\|_{L^2}^2+\alpha\|u\|_{L^p}^p-\sum_{\ell=1}^M \beta_{\ell}\|u\|_{L^{q_{\ell}}}^{q_{\ell}}-\langle g,u\rangle\nonumber\\&\geq \frac{1}{\Delta t}\|u\|_{L^2}^2+\nu\|\nabla u\|_{L^2}^2+\frac{\alpha}{2}\|u\|_{L^p}^p-C^*|\Omega|-\|g\|_{H^{-1}}\|\nabla u\|_{L^2}\nonumber\\&\geq \frac{\nu}{2}\|\nabla u\|_{L^2}^2-C^*|\Omega|-\frac{1}{2\nu}\|g\|_{H^{-1}}^2=\frac{\nu}{2}[u]^2-C^*|\Omega|-\frac{1}{2\nu}\|g\|_{H^{-1}}^2,
\end{align}
where $C^*$ is defined in \eqref{C^*}. It follows that $[P(u),u]>0$ for $[u]=k$, and $k$ sufficiently large: more precisely, $k>\sqrt{\frac{2}{\nu}\left(C^*|\Omega|+\frac{1}{2\nu}\|g\|_{H^{-1}}^2\right)}$.  Therefore, hypothesis of \cite[Lemma 1.4, p. 110]{MR609732} (an application of Brouwer’s fixed point theorem) are satisfied and there exists a solution $u_h^k\in V_h$ for the problem \eqref{4.13}. 
\end{proof}

To show error estimates, we need the following stability result:
\begin{Lemma}\label{for fully fk}
    Assume $f\in L^2(0,T;H^{-1}(\Omega))$. Then the set $\{f^k\}_{k=1}^N$ defined by \begin{equation*}f^k = (\Delta t)^{-1}\int_{t_{k-1}}^{t_k}f(s) ds,\end{equation*} satisfies 
    \begin{equation*}
    \Delta t \sum_{k=1}^N \|f^k\|^2_{H^{-1}} \leq C\|f\|^2_{L^2(0,T;H^{-1})},\ \text{ and }\  \sum_{k=1}^N \int_{t_{k-1}}^{t_k}\|f^k -f(t)\|^2_{H^{-1}}dt \to 0  \quad \text{as } \Delta t \to 0.
    \end{equation*}

    If $f \in H^{\gamma}(0,T;H^{-1}(\Omega))$ for some $\gamma \in[0,1],$ then 
    \begin{align*}
     \sum_{k=1}^N \int_{t_{k-1}}^{t_k}\|f^k -f(t)\|^2_{H^{-1}}dt \leq C(\Delta t)^{2\gamma} \|f\|^2_{H^{\gamma}(0,T;H^{-1})}.
    \end{align*}
\end{Lemma}
\begin{proof}
The proof of this lemma proceeds along the same lines as \cite[Lemma 3.2]{MR2238170}.
\end{proof}
\begin{Lemma}[Stability]\label{Stability fully}
Assume $f \in L^2(0,T;H^{-1}(\Omega))$ and $u_h^0\in L^2(\Omega)$. Let $\{u_h^k\}_{k=1}^N \subset V_h$ be defined by \eqref{fem polynomial damped}. Then 
\begin{align}\label{stability fully ineq}
\nu \sum_{k=1}^N \Delta t \|\nabla u_h^k\|^2_{L^2} \leq \frac{1}{\nu}\|f\|^2_{L^2(0,T;H^{-1})}+\|u_h^0\|^2_{L^2}+2C^*|\Omega|,
\end{align}
where $C^*$ is defined in \eqref{C^*}.
\end{Lemma}
\begin{proof}
Taking $\chi=u_h^k$ in \eqref{4.13} and using the Cauchy-Schwarz and Young's inequalities along with \eqref{u^q to u^p journey}, we obtain
\begin{align*}
\frac{1}{2\Delta t}\|u_h^k\|^2_{L^2}-\frac{1}{2\Delta t}\|u_h^{k-1}\|^2_{L^2}&+\frac{1}{2\Delta t}\|u_h^k-u_h^{k-1}\|^2_{L^2} + \nu \|\nabla u_h^k\|^2_{L^2}+ \alpha\|u_h^k\|^p_{L^p}\\ & \leq \langle f^k, u_h^k \rangle+\sum_{\ell=1}^M \beta_{\ell} \|u_h^k\|^{q_{\ell}}_{L^{q_{\ell}}},
\end{align*}
so that 
\begin{align*}
\frac{1}{2\Delta t}\|u_h^k\|^2_{L^2}-\frac{1}{2\Delta t}\|u_h^{k-1}\|^2_{L^2}&+\frac{1}{2\Delta t}\|u_h^k-u_h^{k-1}\|^2_{L^2} + \nu \|\nabla u_h^k\|^2_{L^2}+ \alpha\|u_h^k\|^p_{L^p}\\&\leq \frac{1}{2\nu}\|f^k\|^2_{H^{-1}}+\frac{\nu}{2}\|\nabla u_h^k\|^2_{L^2}+\frac{\alpha}{2}\|u_h^k\|^{p}_{L^p}+C^*|\Omega|.
\end{align*}
Using the positivity of the third and fifth terms of the left-hand side, we get
\begin{align*}
\frac{1}{\Delta t}\|u_h^k\|^2_{L^2}-\frac{1}{\Delta t}\|u_h^{k-1}\|^2_{L^2}+ \nu \|\nabla u_h^k\|^2_{L^2} \leq \frac{1}{\nu}\|f^k\|^2_{H^{-1}}+2C^*|\Omega|,
\end{align*}
where $C^*$ is defined in \eqref{C^*}. Summing over $k$ for $k=1,2,\ldots, N$, we have
\begin{align*}
\frac{1}{\Delta t} \|u_h^N\|^2_{L^2}-\frac{1}{\Delta t}\|u_h^0\|^2_{L^2}+ \nu \sum_{k=1}^N \|\nabla u_h^k\|^2_{L^2} \leq \frac{1}{\nu \Delta t} \Delta t \sum_{k=1}^N \|f^k\|^2_{H^{-1}}+2C^*|\Omega|.
\end{align*}
Using Lemma \ref{for fully fk}, we attain the following:
\begin{align*}
\frac{1}{\Delta t} \|u_h^N\|^2_{L^2}+ \nu \sum_{k=1}^N \|\nabla u_h^k\|^2_{L^2} \leq \frac{1}{\nu \Delta t} C\|f\|^2_{L^2(0,T;H^{-1})}+\frac{1}{\Delta t}\|u_h^0\|^2_{L^2}+2C^*|\Omega|,
\end{align*}
which further implies
\begin{align*}
\nu \sum_{k=1}^N \Delta t \|\nabla u_h^k\|^2_{L^2} \leq \frac{1}{\nu}\|f\|^2_{L^2(0,T;H^{-1})}+\|u_0\|^2_{L^2}+2C^*|\Omega|,
\end{align*}
which completes the proof. 
\end{proof}

%
%

The following result provides the error between semidiscrete and fully-discrete solutions. 
\begin{Lemma}\label{lem-fully-bound}
	For initial data $u_0 \in D(A^\frac{3}{2})$ and $f\in H^1(0,T;H^1)$, we have the following energy estimate:
	\begin{align}\label{eqn-fully-bound}
	\|\partial_t u_h\|^2_{L^{\infty}(0,T;L^2)}&+\nu \|\partial_t u_h\|^2_{L^2(0,T;H_0^1)}+\alpha(p-1)\||u_h|^{\frac{p-2}{2}}\partial_tu_h\|^2_{L^2(0,T;L^2)}\nonumber\\& \leq C\big(\|u_0\|_{D(A^{\frac{3}{2}})},\|f(0)\|_{H^1},\|f\|_{H^1(0,T;H^1)}\big).
	\end{align}
\end{Lemma}
\begin{proof}
	From \eqref{needed for fully discrete}, we infer
	\begin{align*}
		\|\partial_t u_h\|^2_{L^{\infty}(0,T;L^2)}&+\nu \|\partial_t u_h\|^2_{L^2(0,T;H_0^1)}+\alpha(p-1)\||u_h|^{\frac{p-2}{2}}\partial_tu_h\|^2_{L^2(0,T;L^2)}\\& \leq\bigg(\|\partial_t u_h(0)\|^2_{L^2}+\frac{1}{\alpha}\int_0^T\|\partial_tf(t)\|^2_{L^2}\,dt\bigg)\exp((\alpha+C_3)T).
	\end{align*}
	We estimate the term $\|\partial_t u_h(0)\|^2_{L^2}$ as (see Remark \ref{rem-comp})
	\begin{align*}
		\|\partial_tu_h(0)\|^2_{L^2}&\leq 2(\|\partial_t(u_h(0)-u(0))\|^2_{L^2}+ \|\partial_t u(0)\|^2_{L^2})\nonumber\\&\leq Ch^2\|\partial_t \nabla u(0)\|^2_{L^2}+\|\partial_t u(0)\|^2_{L^2} \leq C(\|u_0\|_{D(A^{\frac{3}{2}})},\|f(0)\|_{H^1}).
	\end{align*}
	Hence, by combining these estimates, we obtain the required result \eqref{eqn-fully-bound}. 
\end{proof}
\begin{Proposition}\label{CFEM uk-ukh error}
For all values of $p$ given in \eqref{eqn-values of p}, if $u_0\in D(A^{\frac{3}{2}})$ and $f \in H^1(0,T;H^1(\Omega))$, let $u_h$ be the semidiscrete solution \eqref{damped fem} and $u_{kh}$ be the generic fully-discrete solution of \eqref{generic fully-discrete solution}. Then, we have 
\begin{equation}
\|u_h-u_{kh}\|^2_{L^{\infty}(0,T;L^2)}+\nu \|u_h-u_{kh}\|^2_{L^2(0,T;H_0^1)}\leq c(\Delta t)^2(\|u_0\|^2_{H^2}+\|f\|^2_{H^1(0,T;H^1)}+C').
\end{equation}
\end{Proposition}
\begin{proof}
We divide our proof into two steps: In Step 1, we prove the error estimates on the nodal values and in Step 2, we prove the final estimate.\\
\noindent
\textbf{Step 1:} \textit{Estimates at $t_k$.} Integrating the semi-discrete scheme \eqref{damped fem} from $t_{k-1}$ to $t_k$, we obtain
\begin{align}\label{4.16}
(u_h(t_k)-u_h(t_{k-1}),\chi) &+ \nu \bigg(\int_{t_{k-1}}^{t_k} \nabla u_h(t)\, dt, \nabla \chi \bigg)+ \alpha \bigg( \int_{t_{k-1}}^{t_k}|u_h(t)|^{p-2}u_h(t)\,dt, \chi \bigg)\nonumber \\&-\sum_{\ell=1}^M \beta_{\ell}\bigg( \int_{t_{k-1}}^{t_k} |u_h(t)|^{q_{\ell}-2}u_h(t)\,dt, \chi \bigg) = \bigg( \int_{t_{k-1}}^{t_k} f(t)\,dt , \chi\bigg),
\end{align}
for all $\chi\in V_h$. Also, we have the fully-discrete scheme \eqref{4.13} as 
\begin{align}\label{4.17}
 \left( \frac{u_h^k - u_h^{k-1}}{\Delta t}, \chi \right) 
    - \nu (\Delta u_h^k, \chi) 
    + \alpha (|u_h^k|^{p-2} u_h^k, \chi)-\sum_{\ell=1}^M \beta_{\ell} (|u_h^k|^{q_{\ell}-2} u_h^k, \chi) 
    = \langle f^k, \chi \rangle,
\end{align}
for all $\chi\in V_h$.  Subtracting \eqref{4.16} and \eqref{4.17}, we infer  for all $\chi\in V_h$ that 
\begin{align*}
(u_h(t_k)-u_h^k,\chi)&-(u_h(t_{k-1})-u_h^{k-1},\chi)+ \nu \bigg( \int_{t_{k-1}}^{t_k} \nabla u_h(t) \, dt - \Delta t \nabla u_h^k, \nabla \chi \bigg) \\&
+\alpha \bigg( \int_{t_{k-1}}^{t_k} |u_h(t)|^{p-2}u_h(t)\,dt-\Delta t |u_h^k|^{p-2} u_h^k,\chi\bigg)\\&-\sum_{\ell=1}^M \beta_{\ell} \bigg(\int_{t_{k-1}}^{t_k} |u_h(t)|^{q_{\ell}-2}u_h(t)\, dt - \Delta t |u_h^k|^{q_{\ell}-2}u_h^k, \chi \bigg) = 0.
\end{align*}
Taking $\chi = u_h(t_k)-u_h^k\in V_h$, we get
\begin{align*}
&\|u_h(t_k) - u_h^k\|^2_{L^2} 
+ \nu \Delta t \|\nabla(u_h(t_k) - u_h^k)\|^2_{L^2} \\
&= (u_h(t_{k-1}) - u_h^{k-1}, u_h(t_k) - u_h^k) - \alpha  
\left( \int_{t_{k-1}}^{t_k} |u_h(t)|^{p-2} u_h(t)\, dt- \Delta t\, |u_h^k|^{p-2} u_h^k, u_h(t_k)-u_h^k\right) \\
&\quad - \nu
\left( \int_{t_{k-1}}^{t_k} \nabla u_h(t)\, dt- \Delta t\, \nabla u_h(t_k),\ \nabla(u_h(t_k) - u_h^k) \right)\\ & \quad +\sum_{\ell=1}^M \beta_{\ell} \bigg(\int_{t_{k-1}}^{t_k} |u_h(t)|^{q_{\ell}-2}u_h(t)\, dt - \Delta t |u_h^k|^{q_{\ell}-2}u_h^k, u_h(t_k)-u_h^k \bigg) .
\end{align*}
Rearranging the terms, we infer the following:
\begin{align*}
&\|u_h(t_k) - u_h^k\|^2_{L^2} 
+ \nu \Delta t \|\nabla(u_h(t_k) - u_h^k)\|^2_{L^2} \\
&= (u_h(t_{k-1}) - u_h^{k-1}, u_h(t_k) - u_h^k)- \nu
\left( \int_{t_{k-1}}^{t_k} \nabla u_h(t)\, dt- \Delta t\, \nabla u_h(t_k),\ \nabla(u_h(t_k) - u_h^k) \right) \\
&\quad - \alpha \Delta t(|u_h(t_k)|^{p-2}u_h(t_k)- |u_h^k|^{p-2}u_h^k, u_h(t_k)-u_h^k)\\&\quad
- \alpha 
\left( \int_{t_{k-1}}^{t_k} |u_h(t)|^{p-2} u_h(t)\, dt- \Delta t\, |u_h(t_k)|^{p-2} u_h^k,  u_h(t_k)-u_h^k \right) \\ & \quad +\sum_{\ell=1}^M \beta_{\ell} \bigg(\int_{t_{k-1}}^{t_k} |u_h(t)|^{q_{\ell}-2}u_h(t)\, dt - \Delta t |u_h(t_k)|^{q_{\ell}-2}u_h(t_k),  u_h(t_k)-u_h^k \bigg)\\&\quad 
+\sum_{\ell=1}^M \beta_{\ell}\Delta t (|u_h(t_k)|^{q_{\ell}-2}u_h(t_k)-|u_h^k|^{q_{\ell}-2}u_h^k, u_h(t_k)-u_h^k)) \\ &
= \sum_{i=1}^6 I_i.
\end{align*}
Let us first estimate each term on the right hand side. Using the Cauchy-Schwarz and Young's inequalities, we can estimate the $I_1$ and $I_2$ as 
\begin{align*}
 I_1=(u_h(t_{k-1}) - u_h^{k-1}, u_h(t_k) - u_h^k) &\leq \frac{1}{2} \|u_h(t_{k-1})-u_h^{k-1}\|^2_{L^2} + \frac{1}{2}\|u_h(t_k)-u_h^k\|^2_{L^2}
\end{align*}
and
\begin{align*}
I_2&=\nu \bigg( \int_{t_{k-1}}^{t_k} \nabla u_h(t) \, dt - \Delta t \nabla u_h(t_k), \nabla (u_h(t_k)-u_h^k) \bigg)\\ &\leq \nu \bigg\|\int_{t_{k-1}}^{t_k} \nabla u_h(t) \, dt - \Delta t \nabla u_h(t_k)\bigg\|_{L^2} \|\nabla (u_h(t_k)-u_h^k)\|_{L^2}\\
&\leq \frac{1}{2\nu \Delta t} \bigg\|\int_{t_{k-1}}^{t_k} \nabla u_h(t) \, dt - \Delta t \nabla u_h(t_k)\bigg\|^2_{L^2} +\frac{\nu \Delta t}{2} \|\nabla(u_h(t_k)-u_h^k)\|^2_{L^2}.
\end{align*}
One can calculate the first term of above estimate as
\begin{align*}
\frac{1}{\Delta t} \bigg\| \int_{t_{k-1}}^{t_k} \nabla u_h(t)\,dt-\Delta t \nabla u_h(t_k)\bigg\|^2_{L^2} &= \frac{1}{\Delta t} \bigg\|\int_{t_{k-1}}^{t_k}(\nabla u_h(t)-\nabla u_h(t_k))\,dt\bigg\|^2_{L^2}\\
&=\frac{1}{\Delta t} \bigg\|\int_{t_{k-1}}^{t_k} \int_{t_k}^{t}  \partial_s \nabla u_h(s) \,ds\,dt\bigg\|^2_{L^2}\\
&\leq \frac{1}{\Delta t}\bigg( \int_{t_{k-1}}^{t_k} \int_{t_{k-1}}^{t_k} \| \partial_s \nabla u_h(s)\|_{L^2}\,ds\,dt\bigg)^2\\
&\leq \Delta t \bigg( \int_{t_{k-1}}^{t_k}\|\nabla \partial_s u_h(s)\|_{L^2}\,ds\bigg)^2\\
&\leq (\Delta t)^2\int_{t_{k-1}}^{t_k}\|\nabla \partial_s u_h(s) \|^2_{L^2} \,ds.
\end{align*}
Now, $I_3$ can be estimated in the same way as in the proof of Theorem \ref{damped uniqueness} as (see \eqref{eqn-ces})
\begin{align*}
I_3&=- \alpha \Delta t (|u_h(t_k)|^{p-2}u_h(t_k)- |u_h^k|^{p-2}u_h^k, u_h(t_k)-u_h^k)\\
&\leq -\frac{\alpha \Delta t}{2}\||u_h(t_k)|^{\frac{p-2}{2}}(u_h(t_k)-u_h^k\|^2_{L^2} -\frac{\alpha \Delta t}{2}\||u_h^k|^{\frac{p-2}{2}}(u_h(t_k)-u_h^k)\|^2_{L^2}.
\end{align*}
The damping term $I_4$ can be estimated by using Young's inequality as 
\begin{align*}
&\left| \left(\alpha \int_{t_{k-1}}^{t_k} |u_h(t)|^{p-2}u_h(t)\,dt
- \Delta t\, |u_h(t_k)|^{p-2}u_h(t_k),\ 
u_h(t_k)-u_h^k\right) \right| \\
&\quad = \left| \alpha \int_{t_{k-1}}^{t_k} 
\left( |u_h(t)|^{p-2}u_h(t) - |u_h(t_k)|^{p-2}u_h(t_k),\ 
u_h(t_k)-u_h^k \right) \,dt \right|\\
& \quad = \bigg| \alpha \int_{t_{k-1}}^{t_k} \bigg( \int_{t_k}^t \partial_s (|u_h(s)|^{p-2}u_h(s))ds, u_h(t_k)-u_h^k\bigg)\,dt \bigg|\\
&\quad \leq \alpha (p-1) \int_{t_{k-1}}^{t_k} \int_{t_{k-1}}^{t_k}\big| (|u_h(s)|^{p-2}\partial_s u_h(s) , u_h(t_k)-u_h^k)\big| \, ds \, dt\\
& \quad \leq \alpha (p-1) \Delta t \int_{t_{k-1}}^{t_k}|(|u_h(t)|^{\frac{p-2}{2}}|u_h(t)|^{\frac{p-2}{2}}\partial_t u_h(t), u_h(t_k)-u_h^k)|\,dt\\
& \quad \leq \alpha (p-1) \Delta t \int_{t_{k-1}}^{t_k} \||u_h(t)|^{\frac{p-2}{2}}\|_{L^{\frac{2p}{p-2}}(\Omega)} \||u_h(t)|^{\frac{p-2}{2}}\partial_t u_h(t)\|_{L^2} \|u_h(t_k)-u_h^k\|_{L^p}\,dt\\
&\quad \leq \alpha (p-1) \Delta t \int_{t_{k-1}}^{t_k} \|u_h(t)\|^{\frac{p-2}{2}}_{L^{p}} \||u_h(t)|^{\frac{p-2}{2}}\partial_t u_h(t)\|_{L^2}\| u_h(t_k)-u_h^k\|_{L^p}\,dt\\
&\quad \leq \frac{\alpha (p-1)^2 \Delta t}{p} \bigg(\frac{2^{p+1}(p-1)\Delta t}{p}\bigg)^{\frac{1}{p-1}} \int_{t_{k-1}}^{t_k} \|u_h(t)\|^{\frac{p(p-2)}{2(p-1)}}_{L^p} \||u_h(t)|^{\frac{p-2}{2}} \partial_t u_h(t)\|^{\frac{p}{p-1}}_{L^2}\,dt\\&\qquad + \frac{\alpha \Delta t}{2^{p+1}}\|u_h(t_k)-u_h^k\|^p_{L^p}\\&\quad \leq  \frac{\alpha (p-1)(p-2) (\Delta t)^2}{2p} \bigg(\frac{2^{p+1}(p-1)\Delta t}{p}\bigg)^{\frac{2}{p-2}}\bigg(\frac{p}{2(p-1)}\bigg)^{\frac{p}{p-2}}\|u_h\|^{p}_{L^\infty(0,T;L^p)} \\ &\qquad + \frac{\alpha(p-1)^2(\Delta t)^2}{p}\int_{t_{k-1}}^{t_k} \||u_h(t)|^{\frac{p-2}{2}}\partial_t u_h(t)\|^2_{L^2}\, dt+ \frac{\alpha \Delta t }{2^{p+1}}\|u_h(t_k)-u_h^k\|^p_{L^p}.
\end{align*}
We now estimate the pumping term 
$I_5$, proceeding analogously to the treatment of 
$I_4$, in order to obtain
\begin{align*}
&\bigg| \sum_{\ell=1}^M \beta_{\ell} \bigg(\int_{t_{k-1}}^{t_k} |u_h(t)|^{q_{\ell}-2}u_h(t)\, dt - \Delta t |u_h(t_k)|^{q_{\ell}-2}u_h(t_k), u_h(t_k)-u_h^k \bigg) \bigg|\\ &\leq \sum_{\ell=1}^M \frac{|\beta_{\ell}| (q_{\ell}-1)(q_{\ell}-2)(\Delta t)^2}{2q_{\ell}} \bigg(\frac{M|\beta_{\ell}| (q_{\ell}-1)\Delta t 2^{p-1}}{q_{\ell}}\bigg)^\frac{2}{q_{\ell}-2} \bigg(\frac{q_{\ell}}{2(q_{\ell}-1)}\bigg)^\frac{q_{\ell}}{q_{\ell}-2}\|u_h\|^{q_{\ell}}_{L^\infty(0,T;L^{q_{\ell}})}\\&\quad + \sum_{\ell=1}^M \frac{|\beta_{\ell}| (q_{\ell}-1)^2(\Delta t )^2}{q_{\ell}}\int_{t_{k-1}}^{t_k} \||u_h(t)|^{\frac{q_{\ell}-2}{2}} \partial_t u_h(t)\|^2_{L^2}\, dt + \frac{\Delta t}{2^{p-1}}\|u_h(t_k)-u_h^k\|^{q_{\ell}}_{L^{q_{\ell}}}.
\end{align*}
In the above estimate, the final term can be further bounded analogously to the argument in \eqref{u^q to u^p journey}, as
\begin{align*}
\frac{\Delta t}{2^{p-1}}\|u_h(t_k)-u_h^k\|^{q_{\ell}}_{L^{q_{\ell}}} \leq \frac{\alpha \Delta t}{2^{p+1}}\|u_h(t_k)-u_h^k\|^{p}_{L^p}+ \frac{(\Delta t)^2}{2^{p}}\bigg(\frac{p-q_{\ell}}{p}\bigg)\bigg(\frac{2q_{\ell}}{\alpha p}\bigg)^\frac{q_{\ell}}{p-q_{\ell}}.
\end{align*}
Finally, the term $I_6$ can be estimated in the same way as in Theorem \ref{damped uniqueness}:
\begin{align*}
& \bigg|\sum_{\ell=1}^M \beta_{\ell}\Delta t (|u_h(t_k)|^{q_{\ell}-2}u_h(t_k)-|u_h^k|^{q_{\ell}-2}u_h^k, u_h(t_k)-u_h^k))\bigg|\\&  \leq \frac{\alpha \Delta t}{4}\||u_h(t_k)|^{\frac{p-2}{2}}(u_h(t_k)-u_h^k)\|^2_{L^2}+\frac{\alpha \Delta t}{4}\||u_h^k|^{\frac{p-2}{2}}(u_h(t_k)-u_h^k)\|^2_{L^2} + 2C_2\Delta t\|u_h(t_k)-u_h^k\|^2_{L^2}.
\end{align*}
Before bringing together all the estimates,  using \eqref{eqn-nonlinear-est}, we observe that the following bound also holds:
\begin{align*}
\frac{\alpha \Delta t}{4}\||u_h(t_k)|^{\frac{p-2}{2}}(u_h(t_k)-u_h^k)\|_{L^2}^2 +\frac{\alpha \Delta t}{4}\||u_h^k|^{\frac{p-2}{2}}(u_h(t_k)-u_h^k)\|^2_{L^2} \geq \frac{\alpha \Delta t}{2^{p-1}}\|u_h(t_k)-u_h^k\|^p_{L^p}.
\end{align*}
 Combining all these estimates, we arrive at 
\begin{align*}
&\frac{1}{2} \|u_h(t_k)-u_h^k\|^2_{L^2}-\frac{1}{2}\|u_h(t_{k-1})-u_h^{k-1}\|^2_{L^2}+\frac{\nu \Delta t}{2} \|\nabla(u_h(t_k)-u_h^k\|^2_{L^2}+\frac{\alpha \Delta t}{2^{p}}\|u_h(t_k)-u_h^k\|^{p}_{L^p}\\
& \leq  \frac{(\Delta t)^2}{2\nu} \int_{t_{k-1}}^{t_k}\|\partial_t\nabla u_h(t)\|^2_{L^2}\,dt +\frac{\alpha(p-1)^2(\Delta t)^2}{p}\int_{t_{k-1}}^{t_k} \||u_h(t)|^{\frac{p-2}{2}}\partial_tu_h(t)\|^2_{L^2}\,dt\\&\quad+ \frac{\alpha (p-1)(p-2) (\Delta t)^2}{2p} \bigg(\frac{2^p(p-1)\Delta t}{p}\bigg)^{\frac{2}{p-2}}\bigg(\frac{p}{2(p-1)}\bigg)^{\frac{p}{p-2}}\|u_h\|^{p}_{L^\infty(0,T;L^p)} \\&\quad 
+\sum_{\ell=1}^M \frac{|\beta_{\ell}| (q_{\ell}-1)(q_{\ell}-2)(\Delta t)^2}{2q_{\ell}} \bigg(\frac{M|\beta_{\ell}| (q_{\ell}-1)\Delta t 2^{p-1}}{q_{\ell}}\bigg)^\frac{2}{q_{\ell}-2} \bigg(\frac{q_{\ell}}{2(q_{\ell}-1)}\bigg)^\frac{q_{\ell}}{q_{\ell}-2}\|u_h\|^{q_{\ell}}_{L^\infty(0,T;L^{q_{\ell}})}\\&\quad + \sum_{\ell=1}^M \frac{|\beta_{\ell}| (q_{\ell}-1)^2(\Delta t )^2}{q_{\ell}}\int_{t_{k-1}}^{t_k} \||u_h(t)|^{\frac{q_{\ell}-2}{2}} \partial_t u_h(t)\|^2_{L^2}\, dt+ \frac{(\Delta t)^2}{2^{p-1}}\bigg(\frac{p-q_{\ell}}{p}\bigg)\bigg(\frac{2q_{\ell}}{\alpha  p}\bigg)^{\frac{q_{\ell}}{p-q_{\ell}}}\\&\quad + 2C_2\Delta t\|u_h(t_k)-u_h^k\|^2_{L^2}.
\end{align*}
Now, summing over $k=1,\ldots, N$, we deduce 
\begin{align*}
&\|u_h(t_N)-u_h^N\|^2_{L^2}+ \sum_{k=1}^N \nu \Delta t \|\nabla (u_h(t_k)-u_h^k)\|^2_{L^2}+\sum_{k=1}^N\frac{\alpha \Delta t}{2^{p-1}}\|u_h(t_k)-u_h^k\|^{p}_{L^p}\\
&\quad \leq   \frac{(\Delta t)^2}{\nu} \int_{0}^{T}\|\partial_t\nabla u_h(t)\|^2_{L^2}\,dt +\frac{2\alpha(p-1)^2(\Delta t)^2}{p}\int_{0}^{T} \||u_h(t)|^{\frac{p-2}{2}}\partial_tu_h(t)\|^2_{L^2}\,dt\\&\quad+ \frac{2\alpha (p-1)(p-2) (\Delta t)^2}{2p} \bigg(\frac{2^p(p-1)\Delta t}{p}\bigg)^{\frac{2}{p-2}}\bigg(\frac{p}{2(p-1)}\bigg)^{\frac{p}{p-2}}\|u_h\|^{p}_{L^\infty(0,T;L^p)} \\&\quad 
+\sum_{\ell=1}^M \frac{2|\beta_{\ell}| (q_{\ell}-1)(q_{\ell}-2)(\Delta t)^2}{2q_{\ell}} \bigg(\frac{M|\beta_{\ell}| (q_{\ell}-1)\Delta t 2^{p-1}}{q_{\ell}}\bigg)^\frac{2}{q_{\ell}-2} \bigg(\frac{q_{\ell}}{2(q_{\ell}-1)}\bigg)^\frac{q_{\ell}}{q_{\ell}-2}\|u_h\|^{q_{\ell}}_{L^\infty(0,T;L^p)}\\&\quad + \sum_{\ell=1}^M \frac{2|\beta_{\ell}| (q_{\ell}-1)^2(\Delta t )^2}{q_{\ell}}\int_{0}^{T} \||u_h(t)|^{\frac{q_{\ell}-2}{2}} \partial_t u_h(t)\|^2_{L^2}\, dt+ \frac{2N(\Delta t)^2}{2^{p-1}}\bigg(\frac{p-q_{\ell}}{p}\bigg)\bigg(\frac{2q_{\ell}}{\alpha  p}\bigg)^{\frac{q_{\ell}}{p-q_{\ell}}}\\&\quad + \Delta t \sum_{k=1}^N 4C_2 \|u_h(t_k)-u_h^k\|^2_{L^2}.
\end{align*}
By an application of Gronwall's inequality, we conclude
\begin{align}\label{u(t_N)-uhN}
&\|u_h(t_N)-u_h^N\|^2_{L^2}+ \sum_{k=1}^N \nu \Delta t \|\nabla (u_h(t_k)-u_h^k)\|^2_{L^2}+ \sum_{k=1}^N\frac{\alpha \Delta t}{2^{p-1}}\|u_h(t_k)-u_h^k\|^{p}_{L^p} \nonumber\\ &  \leq C(\Delta t)^2 \bigg( \int_0^T \|\partial_t \nabla u_h(t)\|^2_{L^2}\, dt + \int_0^T \||u_h(t)|^{\frac{p-2}{2}}\partial_t u_h(t)\|^2_{L^2} \, dt + C' \bigg)\exp( 4C_2\Delta t ),
\end{align}
where 
\begin{align}\label{C'}
C' &= \frac{2\alpha (p-1)(p-2) (\Delta t)}{2p} \bigg(\frac{2^p(p-1)\Delta t}{p}\bigg)^{\frac{2}{p-2}}\bigg(\frac{p}{2(p-1)}\bigg)^{\frac{p}{p-2}}\|u_h\|^{p}_{L^\infty(0,T;L^p)} \nonumber \\&\,\,
+\sum_{\ell=1}^M \frac{2|\beta_{\ell}| (q_{\ell}-1)(q_{\ell}-2)(\Delta t)}{2q_{\ell}} \bigg(\frac{M|\beta_{\ell}| (q_{\ell}-1)\Delta t 2^{p-1}}{q_{\ell}}\bigg)^\frac{2}{q_{\ell}-2} \bigg(\frac{q_{\ell}}{2(q_{\ell}-1)}\bigg)^\frac{q_{\ell}}{q_{\ell}-2}\|u_h\|^{q_{\ell}}_{L^\infty(0,T;L^{q_{\ell}})}\nonumber\\& \quad 
+ \frac{2N}{2^{p-1}}\bigg(\frac{p-q_{\ell}}{p}\bigg)\bigg(\frac{2q_{\ell}}{\alpha  p}\bigg)^{\frac{q_{\ell}}{p-q_{\ell}}}. 
\end{align}
Since $u_h \in L^\infty(0,T;L^p(\Omega))$, $C'$ and the exponential term are bounded by given date. Hence, by an application of the  \eqref{eqn-fully-bound} for semidiscrete case, the right hand side of \eqref{u(t_N)-uhN} is bounded.

\vskip 0.2cm 
\noindent
\textbf{Step 2:} \textit{Estimate for any $t\in[t_{k-1},t_k]$.}
Let us introduce a linear interpolation $\mathcal{I}u_h$ for the semidiscrete solution $u_h$ by 
\begin{align*}
 \mathcal{I}u_h(t)=u_h(t_{k-1})+\bigg(\frac{t-t_{k-1}}{\Delta t}\bigg)(u_h(t_k)-u_h(t_{k-1})) \ \text{ for } \ t \in [t_{k-1},t_k].
\end{align*}
Let us write $u_h-u_{kh}=u_h-\mathcal{I}u_h+\mathcal{I}u_h-u_{kh}$, where $\mathcal{I}u_h$ is the linear interpolation as defined above. Using the triangle inequality, we obtain
\begin{align*}
\|u_h-u_{kh}\|^2_{L^2(0,T;H_0^1)}\leq 2 \|u_h- \mathcal{I}u_h\|^2_{L^2(0,T;H_0^1)}+2\| \mathcal{I}u_h-u_{kh}\|^2_{L^2(0,T;H_0^1)}.
\end{align*}
The first term can be estimated using \cite[Lemma 3.2]{MR3432852} and the final term using \eqref{u(t_N)-uhN} as
\begin{align}
& \|u_h - \mathcal{I}u_{h}\|_{L^2(0,T; H^1_0)}^2 
= \sum_{i=1}^N \int_{t_{i-1}}^{t_i} \|u_h(t) - \mathcal{I}u_{h}(t)\|_{H_0^1}^2 \, dt 
\leq C(\Delta t)^2 \int_0^T \|\partial_t u_h(t)\|_{H_0^1}^2 \, dt, \nonumber \\
& \|\mathcal{I}u_h - u_{kh}\|_{L^2(0,T; H^1_0)}^2 
= \sum_{i=1}^N \int_{t_{i-1}}^{t_i} \|\mathcal{I}u_h(t) - u_{kh}(t)\|_{H_0^1}^2 \, dt \leq C \sum_{i=1}^N \Delta t \|u_h(t_i) - u_h^i\|_{H_0^1}^2. \label{step2 -fully}
\end{align}
Once again using the  triangle inequality, we have
\begin{align*}
\|u_h- u_{kh}\|_{L^\infty(0,T; L^2)}^2 
&\leq 2 \|u_h - \mathcal{I}u_h\|_{L^\infty(0,T; L^2)}^2 
+ 2 \|\mathcal{I}u_h - u_{kh}\|_{L^\infty(0,T; L^2)}^2.
\end{align*}
Similarly as in \eqref{step2 -fully}, by using \cite[Corollary 3.1]{MR3432852} and the estimate \eqref{u(t_N)-uhN}, we deduce
\begin{align*}
&\|u_h - \mathcal{I}u_h\|_{L^\infty(0,T; L^2)}^2 
\leq \sup_{1 \leq i \leq N} \left( \sup_{t_{i-1} \leq t \leq t_i} \|u_h - \mathcal{I}u_h\|_{L^2}^2 \right) \leq C(\Delta t)^2 \|u_h\|_{W^{1,\infty}(0,T;L^2)}^2,
\end{align*}and 
\begin{align*}
&\|\mathcal{I}u_h - u_{kh}\|_{L^\infty(0,T; L^2)}^2 
\leq \sup_{1 \leq i \leq N} \left( \sup_{t_{i-1} \leq t \leq t_i} \|\mathcal{I}u_h - u_{kh}\|_{L^2}^2 \right) 
\leq C \sup_{1 \leq i \leq N} \left( \|u_h(t_i) - u_h^i\|_{L^2}^2 \right),
\end{align*}
and the required result follows.
\end{proof}
\begin{Theorem}\label{fullydiscrete error}
For all values of $p$ given in \eqref{eqn-values of p}, if $u_0\in D(A^{\frac{3}{2}})$ and $f\in H^1(0,T;H^1(\Omega))$, the finite element approximation $u_{kh}$ converges to $u$ as $\Delta t, h  \to 0$. In addition, there exists a constant $C> 0$ such that the approximation $u_{kh}$ satisfies the following error estimate:
\begin{align}\label{eqn-error-cfem}
&\|u-u_{kh}\|^2_{L^\infty(0,T;L^2)}+\nu \|u-u_{kh}\|^2_{L^2(0,T;H_0^1)}\nonumber\\& \leq C((\Delta t)^2+h^2) ( \|u_0\|^2_{D(A^{\frac{3}{2}})}+\|f\|^2_{H^1(0,T;H^1)}+C'),
\end{align}
where $C'$ is given in \eqref{C'}.
\end{Theorem}
\begin{proof}
	We can prove this estimate in the following way:
	\begin{align*}
		\|u-u_{kh}\|^2_{L^\infty(0,T;L^2)}&\leq 2\|u-u_h\|^2_{L^\infty(0,T;L^2)}+2\|u_h-u_{kh}\|^2_{L^\infty(0,T;L^2)},
	\end{align*}
	and
	\begin{align*}
		\|u-u_{kh}\|^2_{L^2(0,T;H_0^1)}&\leq 2\|u-u_h\|^2_{L^2(0,T;H_0^1)}+2\|u_h-u_{kh}\|^2_{L^2(0,T;H_0^1)}.
	\end{align*}
	Hence, combining above terms and using Theorem \ref{semidiscrete error analysis} and Proposition \ref{CFEM uk-ukh error}, we arrive at the desired estimate \eqref{eqn-error-cfem}.
\end{proof}

\begin{Remark}
	In the literature, results analogous to Proposition \ref{CFEM uk-ukh error} and Theorem \ref{fullydiscrete error} are typically derived under the additional assumption $u_{tt} \in L^2(0,T;L^2(\Omega))$ (see, for instance, \cite[Theorem 1.5, Chapter 1]{MR1479170} for the heat equation). In contrast, our analysis relies on the bound \eqref{eqn-fully-bound} established in Lemma \ref{lem-fully-bound}. Nevertheless, both approaches require the regularity conditions $u_0 \in D(A^{3/2})$ and $f \in H^1(0,T;H^1(\Omega))$ (see also Theorem \ref{thm-more-regular}), and we emphasize that neither assumption leads to any additional advantage in the final estimates.
	\end{Remark}

\subsection{Numerical studies}\label{Numerical studies for CFEM} 
The numerical simulations presented in this section are conducted to validate the theoretical findings established in Theorem \ref{fullydiscrete error}. All computations have been carried out using the open-source finite element package FEniCS, version $2019.1.0$, with Python $3.11.12$. 
The experiments were performed on a Windows $10$ machine running Ubuntu $22.04$ through the Windows Subsystem for Linux (WSL), configured with a $12$th Generation Intel Core i$5-1235$U processor ($1.30$ GHz), featuring $10$ physical cores and $12$ logical threads, and $8$GB of RAM.

In this section, our aim is to understand how pumping term affects the solution. Therefore, first we present the numerical experiment for the equation
\begin{equation}\label{without pumping}
\frac{\partial u}{\partial t} - \nu \Delta u + \alpha |u|^{p-2} u = f.
\end{equation}
Then, using the same initial condition and forcing term, we present the numerical experiment for the entire system \eqref{Damped Heat}. 
\begin{equation}\label{pumped}
\frac{\partial u}{\partial t} - \nu \Delta u + \alpha |u|^{p-2} u-\sum_{\ell=1}^M\beta_{\ell} |u|^{q_{\ell}-2}u = f.
\end{equation}

For all numerical experiments, we employ a Newton solver. To conduct the experiment corresponding to \eqref{without pumping}, we prescribe a manufactured solution and compute the corresponding forcing term using FEniCS. The resulting numerical approximation is then used to evaluate the errors (with respect to the norm defined below) and the corresponding convergence rates.
For the problem \eqref{pumped}, we use the same forcing term and initial condition as in the experiment for \eqref{without pumping}. In this case, however, deriving an exact analytical solution for error calculation is not straightforward. Therefore, we take the numerical solution obtained on a highly refined mesh (e.g., of size $256 \times 256$) as a reference solution, denoted by $u_{\mathrm{ref}}$, and compute the errors and convergence rates with respect to $u_{\mathrm{ref}}$.

The error tables below report the computed errors and the observed orders of convergence for both systems, measured with respect to the following norm:
\begin{equation}\label{norm for computation CFEM}
\triplenorm{u^k - u_h^k}^2  := \| u_h(t_N) - u_h^N \|_{L^2}^2 + \nu \sum_{k=1}^N \Delta t \| \nabla(u_h(t_k) - u_h^k) \|_{L^2}^2.
\end{equation}

In all of the experiments, we employ the backward Euler method for time discretization and finite element methods for spatial discretization. 
The time interval $[0,T]$ is discretized as $t_k = k\,\Delta t$ for $k = 0,1,\dots,M$, where $\Delta t = \frac{T}{M}$ is the time step size, and $h$ denotes the spatial discretization parameter.
We choose $\Delta t \propto h$, and the rate of convergence is computed by
\begin{equation*}
r = \frac{\log(\triplenorm{u - u_{h_1}} / \triplenorm{u - u_{h_1}})}{\log(h_1 / h_2)}.
\end{equation*}

\begin{Remark}
As demonstrated in Theorem \ref{fullydiscrete error}, the order of convergence in the $\triplenorm{\cdot}$ norm is $\mathcal{O}(\Delta t + h)$. Given that we employ the backward Euler method for time discretization (which is a first-order method), and conforming FEM with piecewise linear elements ($\mathbb{P}_1$) for spatial discretization (which is also first-order), the combined method is expected to converge with first-order accuracy. Therefore, by choosing $\Delta t \propto h$, the overall scheme exhibits first-order convergence.
\end{Remark}

\noindent \textbf{Example 1:} 
As discussed above, we first consider \eqref{without pumping} by choosing the parameters $\nu=1,\alpha=1,$ and $p=5$ on the unit square domain \(\Omega = (0,1)^2\) over the time interval \(t \in [0,1]\).
The exact solution of \eqref{without pumping} is  
\begin{equation*}
u(x, y, t) = e^{-t}  \sin(\pi x)(x - \tfrac{1}{2})  \sin(\pi y)\left(y - \frac{2}{3}\right),
\end{equation*}
for the given forcing term and initial condition
\begin{align*}
	f:= \frac{\partial u}{\partial t} - \nu \Delta u + \alpha |u|^{p-2} u \text{ and } u_0=\sin(\pi x)(x - \tfrac{1}{2})  \sin(\pi y)\left(y - \frac{2}{3}\right).
\end{align*}
To validate our numerical method, we compute the numerical solution using uniform triangular meshes of size \(N \times N\) and time step \(\Delta t = 0.01\). The numerical solution is then compared to the exact solution using the norm defined in \eqref{norm for computation CFEM}. We also compute the convergence rates.

The computed errors together with the corresponding convergence rates for \eqref{without pumping} are presented in Table \ref{tab: convergence for damped CFEM}, where the expected rates of convergence are observed. In addition, Figure \ref{fig:CFEM damped in row} provides a visual comparison between the exact and numerical solutions for mesh sizes $32 \times 32$, $64 \times 64$, and $128 \times 128$. The first row shows the exact solution, while the second row displays the corresponding numerical approximations.

\begin{table}[ht!]
	\centering
	\begin{tabular}{|c|c|c|c|}
		\hline
		\text{Grid Size} & $h$ & $\triplenorm{u-u_h^N}$ & \text{Convergence Rate} \\
		\hline
		4 $\times$ 4     & 3.54e{-01} & 7.8785e{-02} & N/A \\
		8 $\times$ 8     & 1.77e{-01} & 4.3238e{-02} & 0.87 \\
		16 $\times$ 16   & 8.84e{-02} & 2.2165e{-02} & 0.96 \\
		32 $\times$ 32   & 4.42e{-02} & 1.1153e{-02} & 0.99 \\
		64 $\times$ 64   & 2.21e{-02} & 5.5854e{-03} & 1.00 \\
		128 $\times$ 128 & 1.10e{-02} & 2.7938e{-03} & 1.00 \\
		\hline
	\end{tabular}
	\caption{Computed errors and rates of convergence for \eqref{without pumping} in CFEM.}
	\label{tab: convergence for damped CFEM}
\end{table}

\begin{figure}[ht!]
    \centering
    \begin{subfigure}{0.32\textwidth}
        \centering
        \includegraphics[width=\linewidth]{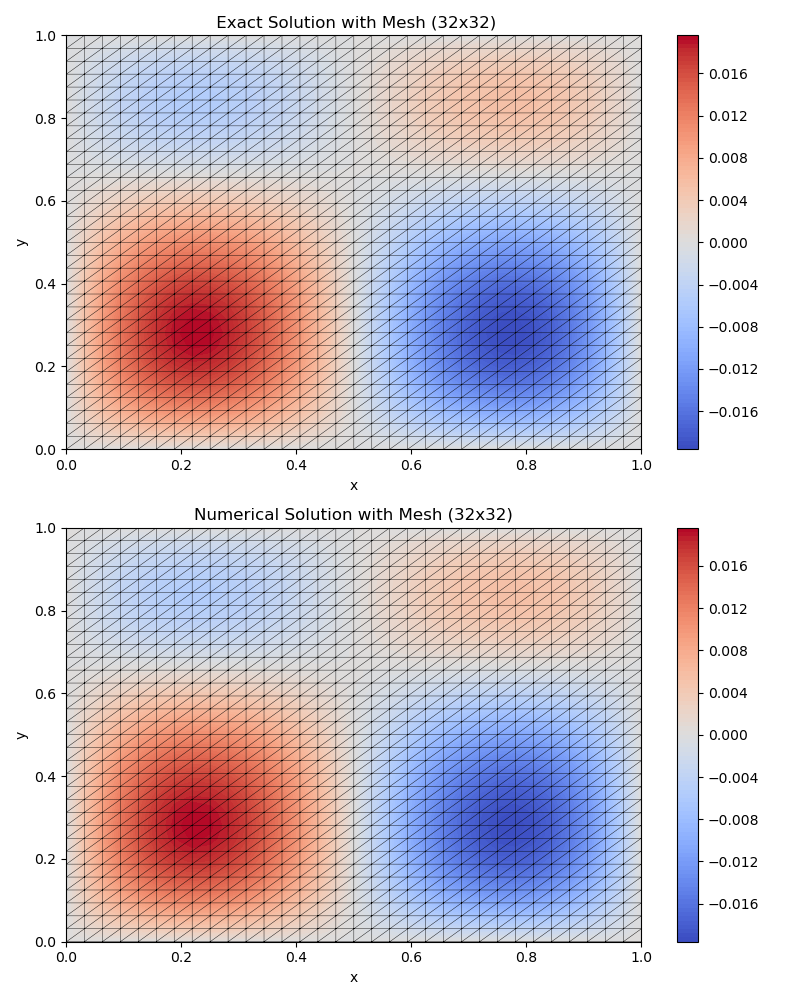}
    \end{subfigure}
    \begin{subfigure}{0.32\textwidth}
        \centering
        \includegraphics[width=\linewidth]{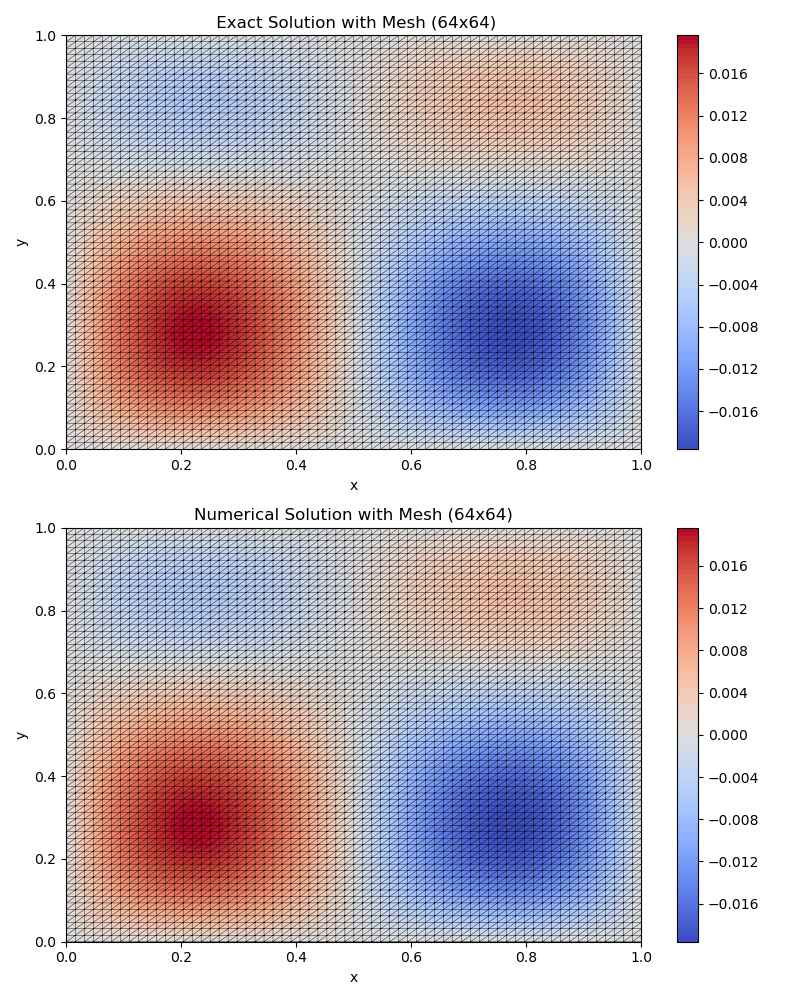}
    \end{subfigure}
    \begin{subfigure}{0.32\textwidth}
        \centering
        \includegraphics[width=\linewidth]{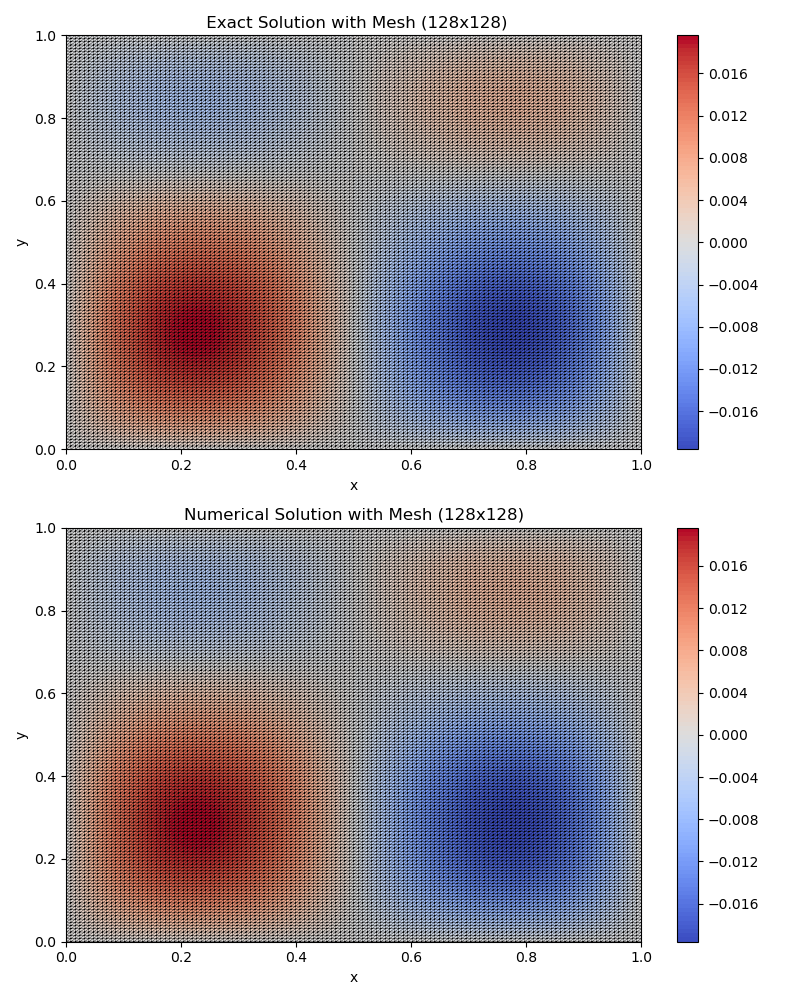}
    \end{subfigure}
    \caption{Exact solution and approximated solution of \eqref{without pumping} on different grid sizes.}
    \label{fig:CFEM damped in row}
\end{figure}
Next, we turn to the full system, namely \eqref{Damped Heat} (or equivalently \eqref{pumped}), with parameters chosen as $M=2$, $q_1=3$, $q_2=4$, $\beta_1=2$, $\beta_2=4$, while keeping all other parameters the same as before. In this setting, we compute the numerical solution on different mesh sizes. Since an analytical solution is not available, we take the solution obtained on the refined mesh $256 \times 256$ as the reference solution, for computing the errors. The corresponding errors and convergence rates for various mesh sizes are reported in Table \ref{tab:pumped-CFEM}, where the expected convergence behavior is observed. Furthermore, Figure \ref{fig:CFEM full pumped} illustrates visual comparisons between the reference solution $u_{\text{ref}}$ and the numerical solutions on meshes of size $32 \times 32$, $64 \times 64$, and $128 \times 128$.

\begin{table}[ht!]
	\centering
	\begin{tabular}{|c|c|c|c|}
		\hline
		\text{Grid Size} & $h$ & $\triplenorm{u_{\text{ref}}-u_{h}^N}$& \text{Convergence Rate} \\
		\hline
		4 $\times$ 4     & 3.54e{-01} & 7.884355e{-02} & N/A \\
		8 $\times$ 8     & 1.77e{-01} & 4.326943e{-02} & 0.87 \\
		16 $\times$ 16   & 8.84e{-02} & 2.218086e{-02} & 0.96 \\
		32 $\times$ 32   & 4.42e{-02} & 1.116111e{-02} & 0.99 \\
		64 $\times$ 64   & 2.21e{-02} & 5.589469e{-03} & 1.00 \\
		128 $\times$ 128 & 1.10e{-02} & 2.831218e{-03} & 0.98 \\
		\hline
	\end{tabular}
	\caption{Computed errors and rates of convergence for \eqref{pumped} in CFEM.}
	\label{tab:pumped-CFEM}
\end{table}

\begin{figure}[ht!]
    \centering
    \begin{subfigure}{0.32\textwidth}
        \centering
        \includegraphics[width=\linewidth]{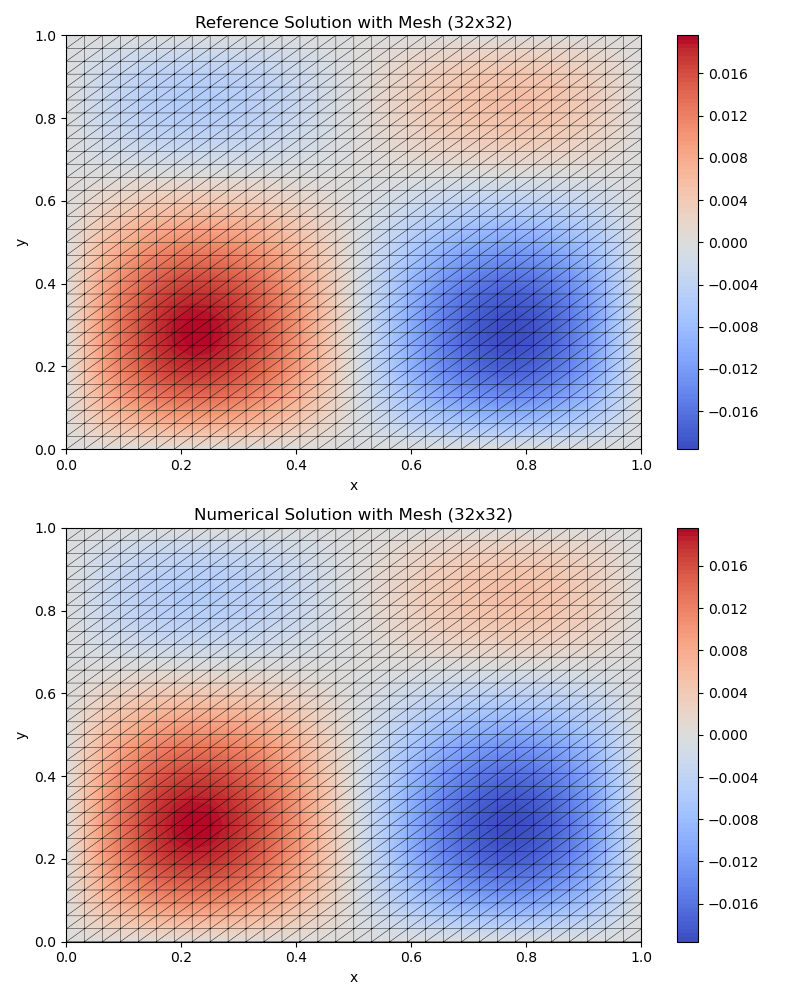}
    \end{subfigure}
    \begin{subfigure}{0.32\textwidth}
        \centering
        \includegraphics[width=\linewidth]{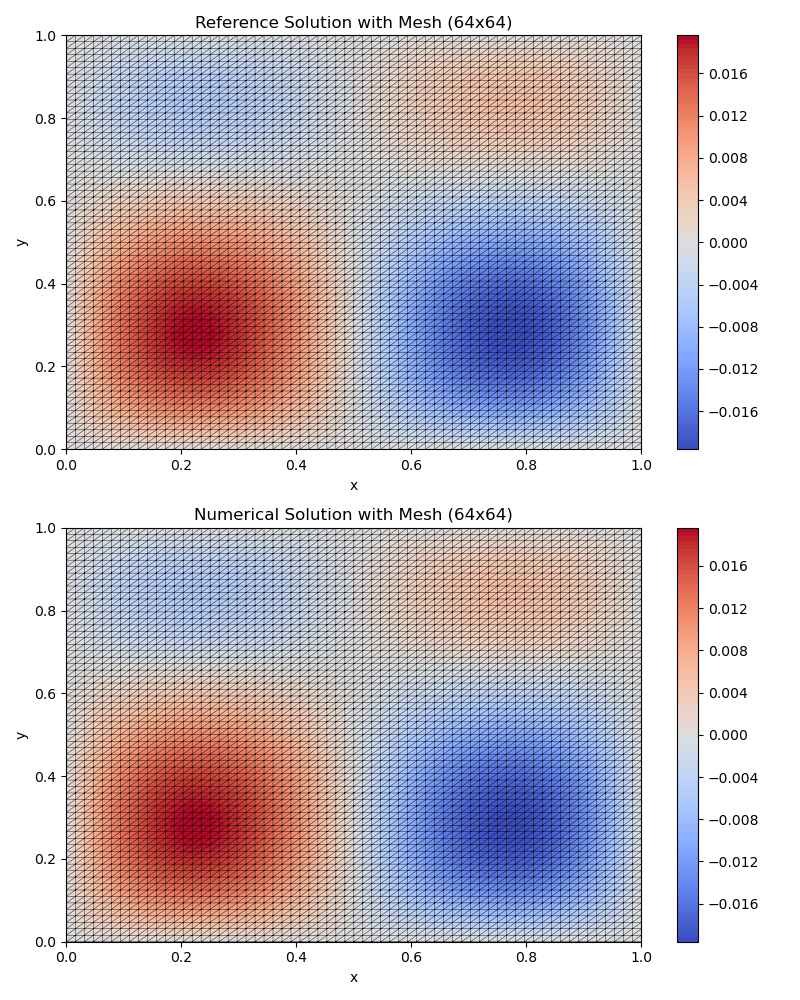}
    \end{subfigure}
    \begin{subfigure}{0.32\textwidth}
        \centering
        \includegraphics[width=\linewidth]{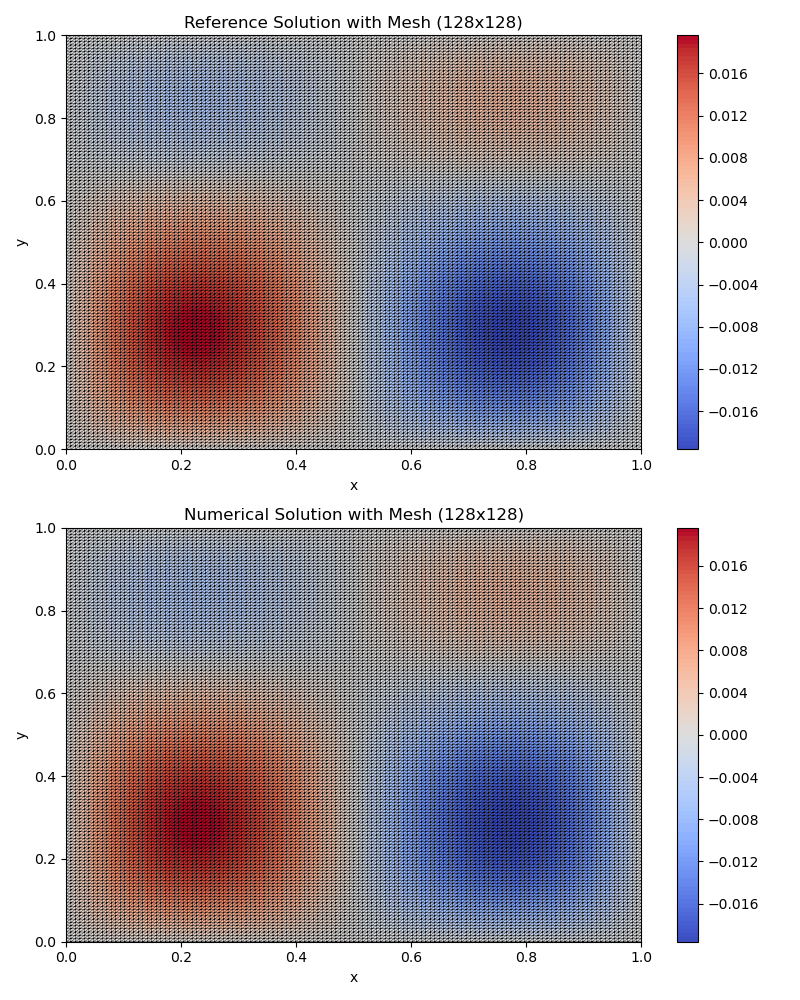}
    \end{subfigure}
    \caption{Reference and approximated solution of \eqref{pumped} on different grid sizes.}
    \label{fig:CFEM full pumped}
\end{figure}
\noindent 
\textbf{Example 2:} 
In this example, we consider \eqref{Damped Heat} with parameters $\nu = 1$, $\alpha = 1$, $M = 3$, $q_1 = 3$, $q_2 = 6$, $q_3 = 9$, $\beta_1 = 2.5$, $\beta_2 = 2$, $\beta_3 = 3$, and $p = 11$, on the unit cube domain $\Omega = (0,1)^3$ over the time interval $t \in [0,1]$. We prescribe the exact solution of \eqref{Damped Heat} as
$$
u(x,y,z,t) = e^{-t}\,\sin(\pi x)(x-0.50)\,\sin(\pi y)(y-0.56)\,\sin(\pi z)(z-0.48).
$$
To validate the proposed numerical method, we compute the solution using uniform triangular meshes of size $N \times N \times N$ with a time step of $\Delta t = 0.1$. The numerical solution is compared with the exact solution using the norm defined in \eqref{norm for computation CFEM}. The computed errors and corresponding convergence rates are presented in Table \ref{tab:pumped-3D}, which clearly demonstrate the optimal order of convergence. In addition, Figure \ref{fig:3DCFEM full} provides visual comparisons of the exact and numerical solutions for mesh sizes $15 \times 15 \times 15$, $20 \times 20 \times 20$, and $25 \times 25 \times 25$.

\begin{table}[ht!]
	\centering
	\begin{tabular}{|c|c|c|c|}
		\hline
		\text{Grid Size} & $h$ & $\triplenorm{u-u_h^N}$ & \text{Convergence Rate} \\
		\hline
		$5 \times 5 \times 5$     & 3.46e{-01} & 1.0531e{-02} & N/A \\
		$10 \times 10 \times 10$  & 1.73e{-01} & 5.9324e{-03} & 0.83 \\
		$15 \times 15 \times 15$  & 1.15e{-01} & 4.0511e{-03} & 0.94 \\
		$20 \times 20 \times 20$  & 8.66e{-02} & 3.0644e{-03} & 0.97 \\
		$25 \times 25 \times 25$  & 6.93e{-02} & 2.4614e{-03} & 0.98 \\
		\hline
	\end{tabular}
	\caption{Computed errors and rates of convergence for \eqref{pumped} in CFEM of 3D case.}
	\label{tab:pumped-3D}
\end{table}

\begin{figure}[ht!]
	\centering
	\begin{subfigure}{0.32\textwidth}
		\centering
		\includegraphics[width=\linewidth]{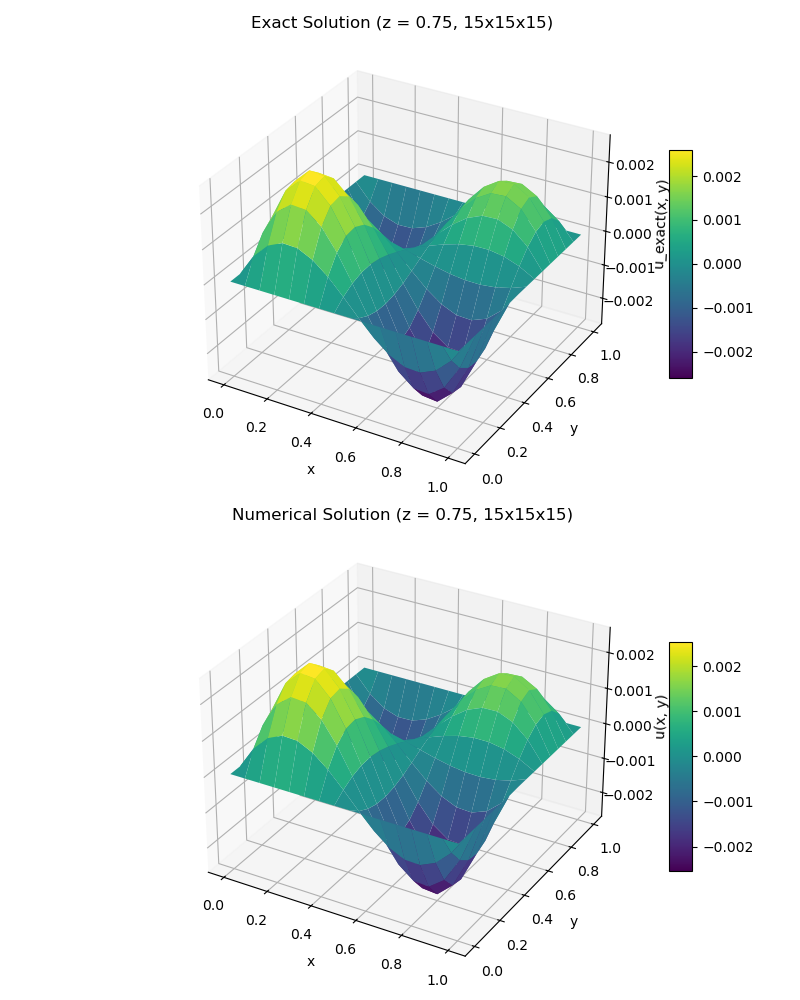}
	\end{subfigure}
	\begin{subfigure}{0.32\textwidth}
		\centering
		\includegraphics[width=\linewidth]{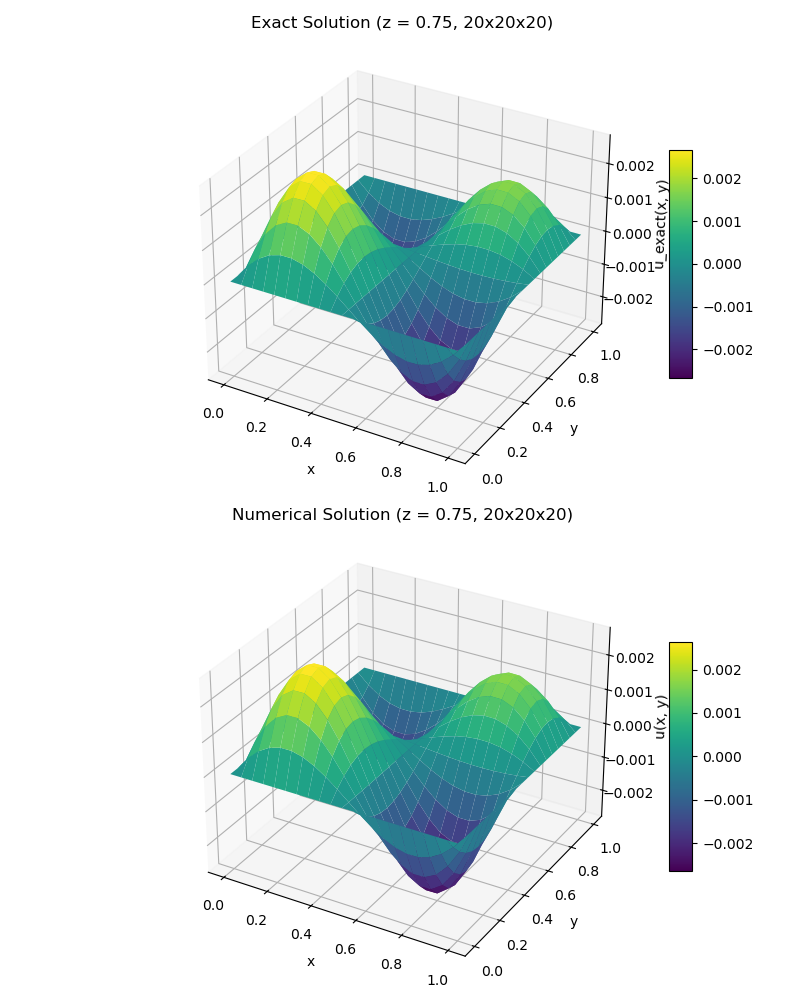}
	\end{subfigure}
	\begin{subfigure}{0.32\textwidth}
		\centering
		\includegraphics[width=\linewidth]{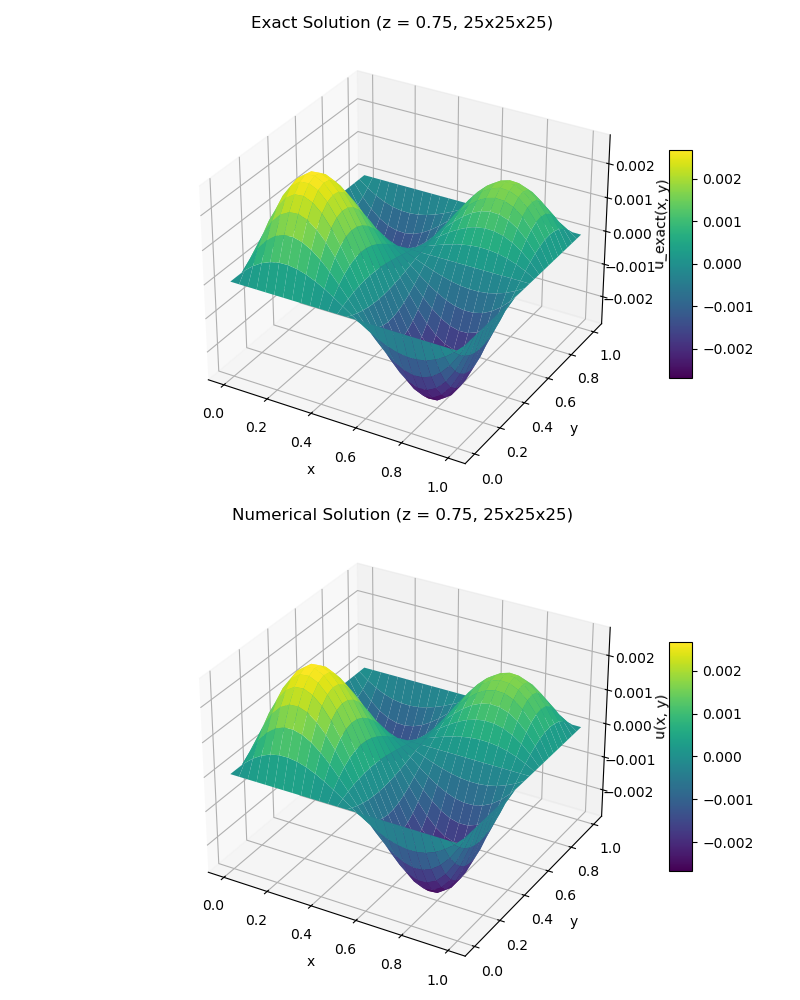}
	\end{subfigure}
	\caption{Exact and approximated solution of \eqref{pumped} on different grid sizes in 3D}
	\label{fig:3DCFEM full}
\end{figure}
\section{Non-conforming finite element method}\label{NCFEM section}

The previous section focused on the conforming finite element approximation. While this approach provides a robust framework with solid theoretical foundations, it can prove restrictive in practice, especially for problems involving low-regularity solutions, complex geometries, or irregular meshes. These challenges motivate the development of nonconforming finite element methods (NCFEM), which offer a more flexible alternative.

In contrast to conforming methods, which require the trial and test spaces to be strict subspaces of $H^1_0(\Omega)$, nonconforming methods permit finite element spaces that need not satisfy global continuity. Instead, continuity is imposed in a weaker sense, usually across element interfaces, thereby relaxing the constraints on the construction of approximation spaces and allowing the use of elements with only minimal inter-element continuity. This property is particularly advantageous when dealing with nonlinear PDEs, where the regularity conditions required by conforming elements may fail to hold.

Specifically, we employ the Crouzeix-Raviart (CR) element, a widely used nonconforming element based on piecewise linear functions that are continuous only at edge midpoints. This element combines computational simplicity with the ability to satisfy essential stability and convergence requirements. Furthermore, it is particularly well-suited for problems involving nonlinearities and discontinuities, such as those arising from damping, memory effects, or degenerate diffusion.

In the sections that follow, we develop both the semidiscrete and fully discrete formulations of \eqref{Damped Heat} using NCFEM. We also derive error estimates that highlight the effectiveness of this method and compare the results with the conforming formulation analyzed previously.

\subsection{Semidiscrete non-conforming FEM}\label{Semi NCFEM}
In this subsection, we develop the semidiscrete Galerkin formulation of \eqref{Damped Heat} on the NCFEM approach. The domain $\Omega$ is subdivided into regular-shape meshes. These meshes may consist of triangles or rectangles in two dimensions, or a tetrahedron in three dimensions, and are denoted by $\mathcal{T}_h$. The set of edges of the mesh, including the interior and boundary edges, is represented by $\mathcal{E}_h$, $\mathcal{E}_h^i$, and $\mathcal{E}_h^b$, respectively. For a given mesh $\mathcal{T}_h$, the spaces $C^0(\mathcal{T}_h)$ and $H^s(\mathcal{T}_h)$ denote the broken function spaces associated with continuous and differentiable functions, respectively.

The edge shared by two adjacent mesh cells $K_+$ and $K_-$ is denoted by $E = K_+ \cap K_- \in \mathcal{E}_h^i$. For a function $w \in C^0(\mathcal{T}_h)$, its traces on $E$ from $K_+$ and $K_-$ are denoted by $w_+$ and $w_-$, respectively. The average and jump operators on edge $E$ are given by
\begin{equation*}
\{\!\{ w \}\!\} = \frac{1}{2}(w_+ + w_-) \quad \text{and} \quad \llbracket w \rrbracket = w_+ \mathbf{n}_+ + w_- \mathbf{n}_-,
\end{equation*}
respectively. Here, $\mathbf{n}_\pm$ are the outward unit normals for the respective elements $K_\pm$. When $w \in C^1(\mathcal{T}_h)$, the jump in the normal derivative across $E$ is defined by
\begin{equation*}
\left\llbracket \frac{\partial w}{\partial n} \right\rrbracket = \nabla(w_+ - w_-) \cdot \mathbf{n}_{+}.
\end{equation*}
 On boundary edges $E \in \mathcal{E}_h \cap \partial \Omega$, we take $\llbracket w \rrbracket = w_+ \mathbf{n}_+$ and $\{\!\{ w \}\!\} = w_+$. The exterior trace of a function $u$ is denoted by $u^e$, and on the boundary we set $u^e = 0$.

Let $\mathbb{P}_1$ denote the space of polynomials of degree at most one. The Crouzeix-Raviart (CR) finite element space is defined by
\begin{equation}\label{NCFEM FEM space}
{V}^{CR}_{h} = \left\{ v \in L^2(\Omega) : \ \text{for all } K \in \mathcal{T}_h;\ v|_K \in \mathbb{P}_1 \ \text{and} \ \int_E \llbracket v \rrbracket = 0,\  E \in \mathcal{E}_h \right\}.
\end{equation}

For every mesh, the broken (piecewise) gradient operator $\nabla_h : H^1(\mathcal{T}_h) \to L^2(\Omega; \mathbb{R}^d)$ is defined by setting
\begin{equation*}
(\nabla_h v)|_K = \nabla (v|_K), \  \text{ for all } K \in \mathcal{T}_h.
\end{equation*}
Next, we introduce the Cl\'ement interpolation operator, following the definition in \cite{MR400739}:
\begin{Definition}[Modified Cl\'ement interpolation operator for the linear Crouzeix-Raviart space {\cite[Lemma 1.127]{MR2050138}}]\label{Clement interpolation}
	Let $\mathcal{T}_h$ be a shape-regular simplicial triangulation of $\Omega \subset \mathbb{R}^d$, and let $\mathcal{E}_h$ denote the set of all edges (faces if $d \geq 3$).  
	The Crouzeix--Raviart non-conforming finite element space of piecewise linear functions is denoted by $V_h^{CR}$.  
	
	The modified Cl\'ement interpolation operator is the linear mapping
	\[
	C_h : W^{s,p}(\Omega) \to V_h^{CR}, 
	\ (C_h v)|_{E} :=
	\begin{cases}
		\displaystyle \frac{1}{|E|}\int_E v \, dS, & E \in \mathcal{E}_h \text{ interior}, \\[1.2ex]
		0, & E \subset \partial \Omega \text{ (boundary edge)} ,
	\end{cases}
	\]
	so that $C_h v$ is the unique CR function whose edge averages coincide with those of $v$ on interior edges and vanish on boundary edges.  
	
	For an element $K \in \mathcal{T}_h$, let $\omega_K$ denote the patch consisting of all 
	elements sharing at least one vertex with $K$. Then $C_h$ satisfies:
	
	\begin{enumerate}
		\item \textbf{Stability.} For $1 \le p < \infty$ and $0 \le m \le 1$, there 
		exists a constant $C>0$, depending only on the mesh regularity, such that
		\[
		\|C_h v\|_{W^{m,p}(\Omega)} \le C \, \|v\|_{W^{m,p}(\Omega)} 
		\  \text{ for all } \ v \in W^{m,p}(\Omega).
		\]
		
		\item \textbf{Local approximation.} For $1 \le p < \infty$, 
		$0\leq m \leq s\leq 2$, there exists $C>0$, depending only on 
		the mesh regularity, such that for every element $K \in \mathcal{T}_h$,
		\[
		\|v - C_h v\|_{W^{m,p}(K)} \le C \, h_K^{\,s-m}\,
		\|v\|_{W^{s,p}(\omega_K)} \ 
		\text{ for all } \ v \in W^{s,p}(\omega_K).
		\]
		
		\item \textbf{Global estimate.} Summing over all elements yields
		\[
		\|v - C_h v\|_{W^{m,p}(\mathcal{T}_h)} \le C \, h^{\,s-m}\, 
		\|v\|_{W^{s,p}(\Omega)} ,
		\]
		where $\|\cdot\|_{W^{m,p}(\mathcal{T}_h)}$ denotes the broken Sobolev norm.
	\end{enumerate}
\end{Definition}

With this setup, the semidiscrete weak formulation of \eqref{Damped Heat} is described as: for each $t \in (0, T)$, find $u_h^{CR} \in V_h^{CR}$ such that
{\small
\begin{equation}\label{NCFEM semidiscrete}
	\left\lbrace
\begin{aligned}
		\displaystyle
&(\partial_tu_h^{CR}(t), \xi)+\nu a_{CR}(u_h^{CR}(t),\xi)+\alpha(|u_h^{CR}(t)|^{p-2}u_h^{CR}(t),\xi)- \sum_{\ell=1}^M \beta_{\ell} (|u_h^{CR}(t)|^{q_{\ell}-2}u_h^{CR}(t), \xi)\\&\quad=(f(t),\xi), \\
&(u_h^{CR}(0), \xi) = (C_hu_0,\xi),
\end{aligned}\right.
\end{equation} }
for all $\xi \in V_h^{CR}$ where $a_{CR}(u,v)=(\nabla_h u,\nabla_h v)$.

The discrete energy norm for the CR approximation is defined as 
\begin{equation}
\triplenorm{v}^2_{CR} := \int_0^T \|\nabla_h v(t)\|^2_{L^2(\mathcal{T}_h)}\,dt.
\end{equation}
The stability estimate for the semidiscrete system is proved in the following lemma:

\begin{Theorem}\label{ncfem existence}
	The semidiscrete system \eqref{NCFEM semidiscrete} has a unique solution $u_h^{CR} \in V_h^{CR}$.
\end{Theorem}
\begin{proof} 
	The existence of a solution can be shown using an argument similar to that in \cite[Theorem 3.1]{MR2495062}. Let the basis functions in $V_h^{CR}$ be denoted by $\phi_i^{CR}(\cdot)$, for $i=1,2,\ldots,m$. Then $u_h^{CR}$ can be expressed as
	\begin{align}\label{NCFFEM  basis}
		u_h^{CR}(x,t)=\sum_{i=1}^m h_i(t)\phi_i^{CR}(x), \quad (x,t)\in \Omega \times [0,T].
	\end{align}
	For $j = 1,2,\ldots,m$, we take $\xi=\phi_j^{CR}$ in \eqref{NCFEM semidiscrete} and utilize \eqref{NCFFEM  basis} to obtain, for all $t\in [0,T]$,
	\begin{align}\label{NCFEM ODE}
		\begin{cases}
			A\dfrac{dH(t)}{dt}+B(t)H(t) = F(t),\\[4pt]
			H(0)= H_0,
		\end{cases}
	\end{align}
	where
	\begin{align*}
		H(t)&=(h_1(t),h_2(t),\ldots,h_m(t))^T, \quad 
		A=(\phi_i^{CR},\phi_j^{CR})_{m\times m}, \quad
		F(t) = (f(t),\phi_j^{CR})_{m\times 1},\\
		B(t) &= (\nabla_h \phi_i^{CR},\nabla_h \phi_j^{CR})_{m \times m}
		+ \bigg(\bigg|\sum_{i=1}^m h_i(t)\phi_i^{CR}(x)\bigg|^{p-2}\phi_i^{CR}(x), \phi_j^{CR}(x)\bigg)_{m \times m}\\
		&\quad -\sum_{\ell=1}^M \beta_{\ell}\bigg(\bigg|\sum_{i=1}^m h_i(t)\phi_i^{CR}(x)\bigg|^{p-2}\phi_i^{CR}(x), \phi_j^{CR}(x)\bigg)_{m \times m}.
	\end{align*}
	
	The matrix $A$ is symmetric positive definite since $(\phi_i^{CR},\phi_j^{CR})$ defines an inner product on $V_h^{CR}$; hence $A$ is invertible. The initial condition $H(0)=H_0$ follows from the discrete initial condition $u_h^{CR}(x,0)=C_h u_0(x)$. Writing both in the basis $\{\phi_i^{CR}\}_{i=1}^m$ and using the linear independence of the basis functions, we obtain $h_i(0)=(H_0)_i$, that is, $H(0)=H_0$.
	
	Since the system \eqref{NCFEM ODE} is finite-dimensional and the nonlinear operator is locally Lipschitz in \(H\), the standard Picard-Lindel\"of theorem (\cite[Theorem 3.1, p. 12]{MR69338}) that guarantees the existence and uniqueness of a local solution. Furthermore, energy estimates for the semidiscrete system ensure that the solution cannot blow up in finite time, so the local solution extends to a global solution for all \(t\in [0,T]\). Hence, there exists a unique solution \(u_h^{CR} \in V_h^{CR}\) to \eqref{NCFEM semidiscrete}. 
	
	To prove uniqueness, let $u_{h,1}^{CR}$ and $u_{h,2}^{CR}$ be two solutions of \eqref{NCFEM semidiscrete}. By using an argument analogous to that in Theorem~\ref{damped uniqueness}, we obtain $u_{h,1}^{CR}=u_{h,2}^{CR}$, that is, the solution is unique.
\end{proof}

\begin{Theorem}[Energy estimates]\label{energy estimate for NCFEM}
	(a) If $f\in L^2(0,T;H^{-1}(\Omega))$ and $u_0 \in L^2(\Omega)$, then the solution $u_h^{CR}$ of \eqref{NCFEM semidiscrete} is unique and satisfies the following estimate:
	\begin{align*}
		\sup_{0\leq t \leq T}\|u_h^{CR}(t)\|^2_{L^2}&+\nu \int_0^T \|\nabla_h u_h^{CR}(t)\|^2_{L^2(\mathcal{T}_h)} \, ds + \alpha \int_0^T \|u_h^{CR}(t)\|^p_{L^p}\, dt \\&\leq \|u_0\|^2_{L^2}+\frac{1}{\nu} \int_0^T \|f(t)\|^2_{H^{-1}}\, dt + C^{*}T|\Omega|,
	\end{align*}
where $C^*$ is defined in \eqref{C^*}.

	\noindent
	(b) If $f\in L^2(0,T;L^2(\Omega))$ and $u_0 \in D(A)$, then 
	\begin{align*}
		\int_0^T \|\partial_t u_h^{CR}(t)\|^2_{L^2}\, dt &+ \sup_{0\leq t \leq T}\bigg(\nu\|\nabla_h u_h^{CR}(t)\|^2_{L^2(\mathcal{T}_h)}+\frac{2\alpha}{p}\|u_h^{CR}(t)\|^p_{L^p}\bigg) \\&\leq \nu \|\nabla u_0\|^2_{L^2}+\frac{2\alpha}{p}\|u_0\|^p_{D(A)} + \int_0^T \|f(t)\|^2_{L^2}\, dt.
	\end{align*}
	
\end{Theorem}
\begin{proof}
The proof can be derived using a similar reasoning to that employed in establishing Theorem \ref{energy estimate of CFEM}.
\end{proof}
We proceed to derive the optimal error estimates for the semidiscrete case under minimal regularity assumptions and for the values of  $p$ given in Table \ref{values of p}.
\begin{Theorem}\label{NCFEM Semiestimate error estimate}
Let $V_h^{CR}$ be the space defined in \eqref{NCFEM FEM space}. Assume that $u_0\in D(A)$ and $u(\cdot)$ is the solution of problem \eqref{Damped Heat}.
Then $u_h^{CR}$, the solution of \eqref{NCFEM semidiscrete}, satisfies the following estimate:

\noindent
(a) \textbf{For $f \in H^1(0,T,L^2(\Omega))$ and $2\leq p\leq \frac{2d}{d-2}$:}
\begin{align*}
	&\|u_h^{CR}-u\|^2_{L^\infty(0,T;L^2)}+\nu \triplenorm{u_h^{CR}-u}^2_{CR}+\frac{\alpha}{2^{p-3}}\|u_h^{CR}-u\|^p_{L^p(0,T;L^p)} \\
&\leq Ch^2\bigg\{\|u\|^2_{L^{\infty}(0,T;H_0^1)}+\|u\|^2_{L^\infty(0,T;H^2)}+\int_0^T\|\partial_t u(t)\|_{H_0^1}^2\, dt+\|u\|^p_{L^{\infty}(0,T;H^2)}\\& \quad +\left(\|u_0\|_{D(A)}^p+ \|f\|^2_{H^1(0,T;L^2)}+\|u\|^{2(p-2)}_{L^\infty(0,T;L^p)}\right)\int_0^T \|u(t)\|^2_{H^2}\,dt \bigg\},
\end{align*}
for some positive constant $C.$\\
(b)\textbf{For $f \in H^1(0,T;H^1(\Omega))$ and $\frac{2d}{d-2}<p\leq \frac{2d-6}{d-4}$:}\\
(i) if $2\leq q_{\ell}<1+\frac{p}{2}$, the following estimate holds:
\begin{align*}
&\|u_h^{CR}-u\|^2_{L^\infty(0,T;L^2)}+\nu \triplenorm{u_h^{CR}-u}^2_{CR} + \frac{\alpha}{2^{p-2}}\|u_h^{CR}-u\|^p_{L^p(0,T;L^p)} \\& \leq Ch^2\bigg( \|u\|^2_{L^\infty(0,T;H_0^1)}+\|u\|^2_{L^\infty(0,T;H^2)}+\|u\|^p_{L^\infty(0,T;H^2)}+ \int_0^T \|u(t)\|^2_{H^2}\,dt \\& \quad  +\int_0^T \|\partial_t u(t)\|^2_{H_0^1}\,dt+\int_0^T \|u(t)\|^2_{D(A^{\frac{3}{2}})}\,dt \bigg).
\end{align*}
 (ii) if $1+\frac{p}{2}\leq q_{\ell}<p$, the following estimate holds:
\begin{align*}
	&\|u_h^{CR}-u\|^2_{L^\infty(0,T;L^2)}+\nu \triplenorm{u_h^{CR}-u}_{CR}+\frac{\alpha}{2^{p-2}}\|u_h^{CR}-u\|^p_{L^p(0,T;L^p)}\\&\leq Ch^2\bigg\{\|u\|^2_{L^\infty(0,T;H_0^1)}+\|u\|^2_{L^\infty(0,T;H^2)}+\|u\|^p_{L^\infty(0,T;H^2)}+\int_0^T \|\partial_tu(t)\|^2_{H_0^1}\, dt \\& \quad + \int_0^T \|u(t)\|^2_{H^2}\,dt   +\|u\|^{\frac{p(q_{\ell}-3)+2}{p-1}}_{L^\infty(0,T;H^2)} \int_0^T \|u(t)\|^2_{D(A^{\frac{3}{2}})}\,dt \bigg\}.
\end{align*}
\end{Theorem}
\begin{proof}
By an application of the triangle inequality, we have the following:
\begin{align*}
\triplenorm{u_h^{CR}-u}_{CR} \leq \triplenorm{u_h^{CR}-W}_{CR}+ \triplenorm{W-u}_{CR}.
\end{align*}
Note that $\triplenorm{W-u}_{CR}\leq Ch\|u\|_{H^2}$ follows from the result $$\displaystyle \|u-W\|_{H^1(K)}\leq \inf_{w \in V_h^{CR}}\|u-w\|_{H^1(K)}$$ from Definition \ref{Clement interpolation}. Also, we have $\|W-u\|^p_{L^p}\leq Ch^2\|u\|^p_{H^2}$ (see \eqref{u-w Lp in H2}). So, our aim to prove the estimate for $\triplenorm{u_h^{CR}-W}_{CR}$. Using the regularity result (see Theorem \ref{regularity-f- H1}), the following holds for a.e. $t\in[0,T]$:
\begin{align}\label{NCFEM SDE-01}
(\partial_t u(t), \xi)&+\nu a_{CR}(u(t),\xi)+\alpha(|u(t)|^{p-2}u(t),\xi)-\sum_{\ell=1}^M \beta_{\ell} (|u(t)|^{q_{\ell}-2}u(t), \xi)\nonumber \\& =(f,\xi)+\sum_{K \in \mathcal{T}_h}\int_{\partial K}\nu \frac{\partial u(s)}{\partial n_K}\xi\, dS ,
\end{align}
for all $\xi \in V_h^{CR}$. So from \eqref{NCFEM semidiscrete} and \eqref{NCFEM SDE-01}, we infer 
\begin{align*}
&(\partial_t (u_h^{CR}(t)-u(t)), \xi)+\nu a_{CR}(u_h^{CR}(t)-u(t), \xi)+\alpha (|u_h^{CR}(t)|^{p-2}u_h^{CR}(t)-|u(t)|^{p-2}u(t), \xi)\\&+ \sum_{\ell=1}^M \beta_{\ell} (|u_h^{CR}(t)|^{q_{\ell}-2}u_h^{CR}(t)-|u(t)|^{q_{\ell}-2}u(t), \xi) =-\sum_{K \in \mathcal{T}_h} \int_{\partial K} \nu \frac{\partial u(s)}{\partial n_K} \xi \, dS,
\end{align*}
for a.e. $t\in[0,T]$. Rearranging the equation by taking $u_h^{CR}-u = u_h^{CR}-W+W-u$, we obtain
\begin{align*}
&(\partial_t (u_h^{CR}(t)-W(t)), \xi)+\nu a_{CR}(u_h^{CR}(t)-W(t), \xi)+\alpha (|u_h^{CR}(t)|^{p-2}u_h^{CR}(t)-|W(t)|^{p-2}W(t), \xi)\\&=-(\partial_t (W(t)-u(t)), \xi)-\nu a_{CR}(W(t)-u(t),\xi)-\sum_{K\in \mathcal{T}_h}\int_{\partial K}\nu \frac{\partial u(s)}{\partial n_K}\xi \, dS\\&
\quad -\alpha (|W(t)|^{p-2}W(t)-|u(t)|^{p-2}u(t), \xi )+\sum_{\ell=1}^M \beta_{\ell} (|u_h^{CR}(t)|^{q_{\ell}-2}u_h^{CR}(t)-|W(t)|^{q_{\ell}-2}W(t),\xi)\\&\quad + \sum_{\ell=1}^M \beta_{\ell} (|W(t)|^{q_{\ell}-2}W(t)- |u(t)|^{q_{\ell}-2}u(t), \xi),
\end{align*}
for a.e. $t\in[0,T]$. Let us choose $\xi=u_h^{CR}-W$ in the above equation to find for a.e. $t\in[0,T]$ that 
\begin{align*}
&\frac{1}{2}\frac{d}{dt}\|u_h^{CR}(t)-W(t)\|^2_{L^2}+\nu \|\nabla_h (u_h^{CR}(t)-W(t))\|^2_{L^2(\mathcal{T}_h)}+J_1 = \sum_{i=2}^6 J_i - \sum_{K \in \mathcal{T}_h} \int_{\partial K} \nu \frac{\partial u(s)}{\partial n_K} \xi \, dS,
\end{align*}
where 
\begin{align*}
&J_1 = \alpha (|u_h^{CR}|^{p-2}u_h^{CR}-|W|^{p-2}W, u_h^{CR}-W), \, J_2 = -(\partial_t (W-u), u_h^{CR}-W), \\& J_3= -\nu a_{CR}(W-u,u_h^{CR}-W),\,
J_4=-\alpha (|W|^{p-2}W-|u|^{p-2}u,u_h^{CR}-W),\\&
J_5=\sum_{\ell=1}^M \beta_{\ell} (|u_h^{CR}|^{q_{\ell}-2}u_h^{CR}-|W|^{q_{\ell}-2}W,u_h^{CR}-W),\\
J_6&=\sum_{\ell=1}^M \beta_{\ell} (|W|^{q_{\ell}-2}W- |u|^{q_{\ell}-2}u,u_h^{CR}-W).
\end{align*}
The terms 
$J_1$
and 
$J_5$ can be estimated using calculations analogous to those in the proof of Theorem \ref{semidiscrete error analysis}, yielding
\begin{align*}
	|J_1|&\geq \frac{\alpha}{2}\||u_h^{CR}|^{\frac{p-2}{2}}(u_h^{CR}-W)\|^2_{L^2}+\frac{\alpha}{2}\||W|^{\frac{p-2}{2}}(u_h^{CR}-W)\|^2_{L^2},\\
	|J_5|&\leq \frac{\alpha}{4}\||u_h^{CR}|^{\frac{p-2}{2}}(u_h^{CR}-W)\|^2_{L^2}+\frac{\alpha}{4}\||W|^{\frac{p-2}{2}}(u_h^{CR}-W)\|^2_{L^2}+2C_2\|u_h^{CR}-W\|^2_{L^2}.
\end{align*}
Also, we use the fact that  (see \eqref{eqn-nonlinear-est} )
\begin{align*}
	\frac{\alpha}{4}\||u_h^{CR}|^{\frac{p-2}{2}}(u_h^{CR}-W)\|^2_{L^2}+\frac{\alpha}{4}\||W|^{\frac{p-2}{2}}(u_h^{CR}-W)\|^2\geq \frac{1}{2^{p-1}}\|u_h^{CR}-W\|^p_{L^p}.
\end{align*}
Combining these estimates, we get
\begin{align}\label{NCFEM error main step}
&	\frac{1}{2}\frac{d}{dt}\|u_h^{CR}(t)-W(t)\|^2_{L^2}+\nu \|\nabla_h (u_h^{CR}(t)-W(t))\|^2_{L^2(\mathcal{T}_h)}+\frac{\alpha}{2^{p-1}}\|u_h^{CR}(t)-W(t)\|^p_{L^p}\nonumber \\&\qquad\leq J_2+J_3+J_4+J_6-\sum_{K \in \mathcal{T}_h} \int_{\partial K} \nu \frac{\partial u(s)}{\partial n_K} \xi \, dS+2C_2\|u_h^{CR}(t)-W(t)\|^2_{L^2}, 
\end{align}
for a.e. $t\in[0,T]$. 

\vskip 0.1 cm
\noindent 
(a)\textbf{ For $2\leq p \leq \frac{2d}{d-2}$:} Estimating 
$J_2,J_3,J_4$ and $J_6$
by the same calculations as in the proof of Theorem \ref{semidiscrete error analysis}, we obtain
\begin{align*}
	|J_2|&\leq \frac{1}{2}\|\partial_t(W-u)\|^2_{L^2}+\frac{1}{2}\|u_h^{CR}-W\|^2_{L^2},\\
	|J_3|&\leq 2\nu \|\nabla_h(W-u)\|^2_{L^2(\mathcal{T}_h)}+\frac{\nu}{8}\|\nabla_h(u_h^{CR}-W)\|^2_{L^2({\mathcal{T}_h})},\\
	|J_4|&\leq \frac{\nu}{4}\|\nabla_h(u_h^{CR}-W)\|^2_{L^2(\mathcal{T}_h)}+\frac{\alpha^2(p-1)^2 2^{2p-5}}{\nu}(\|W\|^{2(p-2)}_{L^p}+\|u\|^{2(p-2)}_{L^p})\|W-u\|^2_{L^p},\\
	|J_6|&\leq \frac{\nu}{4}\|\nabla_h(u_h^{CR}-W)\|^2_{L^2(\mathcal{T}_h)}\nonumber\\&\quad+\sum_{\ell=1}^M \frac{M|\beta_{\ell}|^2(q_{\ell}-1)^2 2^{2q_{\ell}-5}}{\nu}(\|W\|^{2(q_{\ell}-2)}_{L^{q_{\ell}}}+\|u\|^{2(q_{\ell}-2)}_{L^{q_{\ell}}})\|W-u\|^2_{L^{q_{\ell}}}.
\end{align*}
Consequently, we may apply the estimate from \cite[Section 10.3]{MR2373954}, which gives
\begin{align*}
\left| \sum_{K \in \mathcal{T}_h} \int_{\partial K} \nu \, \frac{\partial u}{\partial n_K} \, \xi dS \right|
&\leq C \left( \sum_{K \in \mathcal{T}_h} \nu h_K^2 \|u\|_{H^2(K)}^2 \right)^{1/2} \|\nabla_h(u_h^{CR}(t)-W(t))\|_{L^2(\mathcal{T}_h)},\\
\left| \sum_{K \in \mathcal{T}_h} \int_{\partial K} \nu \, \frac{\partial u}{\partial n_K} \, \xi dS\right| & \leq 2C\bigg(\sum_{K \in \mathcal{T}_h}\nu h_K^2 \|u\|_{H^2(K)}^2\bigg)+ \frac{\nu}{8}\|\nabla_h(u_h^{CR}(t)-W(t))\|^2_{L^2(\mathcal{T}_h)}.
\end{align*}
By combining all these estimates, applying arguments similar to those in Theorem \ref{semidiscrete error analysis}, and using Definition \ref{Clement interpolation}, we obtain the desired error estimate. Note that, as in Theorem \ref{semidiscrete error analysis}, it is necessary to consider separate cases and perform analogous calculations for all terms. For the sake of brevity and to avoid repetition, we omit the detailed steps here.
\end{proof}
\subsection{Fully-discrete non-conforming FEM}\label{fully NCFEM}
In this section, we introduce a fully-discrete non-conforming finite element scheme. We partition the time interval $[0,T]$ into $0 = t_0 < t_1 < t_2 < \dots < t_N = T$ with uniform time stepping $\Delta t$, that is, $t_k = k \Delta t$. We replace the time derivative with a backward Euler discretization. We denote $C$ as a generic constant which is independent of $h$ and $\Delta t$. Also note that ${\partial_t} u^{CR}_{kh} \approx \frac{u_h^k - u_h^{k-1}}{\Delta t}$ (where for simplicity we are taking $(u_h^{CR})^{k}=u_h^k$). The fully-discrete non-conforming Galerkin approximation of the problem \eqref{damped weak} is defined in the following way: 

Find $u_h^k \in V_h^{CR}$, for $k = 1,2,\ldots,N$, such that for $\xi \in V_h^{CR}$:
\begin{equation}\label{NCFEM fully discrete formulation}
\begin{aligned}
	\begin{cases}
		\displaystyle
	 \left( \frac{u_h^k - u_h^{k-1}}{\Delta t}, \xi \right) 
    + \nu (\nabla_h u_h^k, \nabla_h \xi)
    + \alpha (|u_h^k|^{p-2} u_h^k, \xi)-\sum_{\ell=1}^M \beta_{\ell} (|u_h^k|^{q_{\ell}-2}u_h^k,\xi) 
    = (f^k, \xi), \\
    (u_h^0(x_i), \xi) =(C_hu_0(x_i),\xi),
\end{cases}
\end{aligned}
\end{equation}
where $f^k = (\Delta t)^{-1} \int_{t_{k-1}}^{t_k} f(s)\, ds$ for $f \in L^2(0,T;L^2(\Omega))$, and $u_h^{0}$ is the approximation of $u_0$ in $V_h^{CR}$.
We then define a generic non-conforming finite element approximate solution $u_{kh}^{CR}$ by 
\begin{equation}
u^{CR}_{kh}|_{[t_{k-1},t_k]} = u_h^{k-1} + \bigg(\frac{t - t_{k-1}}{\Delta t}\bigg)(u_h^k - u_h^{k-1}), \quad 1 \leq k \leq N, \ \text{ for }\  t \in [t_{k-1}, t_k].
\end{equation}
 The discrete bilinear form is given by $a_{CR}(u,v) = (\nabla_h u, \nabla_h v)$, and the discrete energy norm is defined as
\[
\|u_h^k\|_{CR}^2 := \Delta t \sum_{k=1}^N \|\nabla_h u_h^k\|_{L^2(\mathcal{T}_h)}^2.
\]
\begin{Lemma}\label{NCFEM f^k}
Let us define the set $\{f^k\}_{k=1}^N$ by $f^k = (\Delta t)^{-1} \int_{t_{k-1}}^{t_k} f(s)\, ds$. If the external forcing $f \in L^2(0,T;L^2(\Omega))$, then we have
\begin{align*}
\Delta t \sum_{k=1}^N \|f^k\|_{L^2}^2 \leq C \|f\|_{L^2(0,T;L^2)}^2\ 
\text{ and }\  \sum_{k=1}^N \int_{t_{k-1}}^{t_k} \|f^k - f(t)\|_{L^2}^2 dt \to 0, \quad \Delta t \to 0.
\end{align*}
 Further, if $f \in H^\epsilon(0,T;L^2(\Omega))$ for some $\epsilon \in [0,1]$, then
\begin{align*}
\sum_{k=1}^N \int_{t_{k-1}}^{t_k} \|f^k - f(t)\|_{L^2}^2 dt \leq C (\Delta t)^{2\epsilon} \|f\|_{H^\epsilon(0,T;L^2(\Omega))}^2.
\end{align*}
\end{Lemma}
\begin{proof}
The proof of this lemma follows the same arguments as \cite[Lemma 3.2]{MR2238170} and is therefore omitted.
\end{proof}
\begin{Lemma}[stability]
Assume $f \in L^2(0,T;L^{2}(\Omega))$ and $u_0\in L^2(\Omega)$. Let $\{u_h^k\}_{k=1}^N \subset V_h^{CR}$ be defined by \eqref{NCFEM FEM space}. Then 
\begin{align}\label{NCFEM stability fully ineq}
\triplenorm{u_h^k}_{CR}^2 \leq \frac{1}{\nu}\|f\|^2_{L^2(0,T;L^{2})}+\|u_0\|^2_{L^2}+2C^*|\Omega|,
\end{align}
where $C^*$ is defined in \eqref{C^*}.
\end{Lemma}
\begin{proof}
Taking $\xi=u_h^k$ in \eqref{NCFEM semidiscrete}, we get
\begin{align*}
\frac{1}{2\Delta t}\|u_h^k\|^2_{L^2}-\frac{1}{2\Delta t}\|u_h^{k-1}\|^2_{L^2}&+\frac{1}{2\Delta t}\|u_h^k-u_h^{k-1}\|^2_{L^2} + \nu \|\nabla_h u_h^k\|^2_{L^2(\mathcal{T}_h)}+ \alpha\|u_h^k\|^p_{L^p}\\& \, \leq \langle f^k, u_h^k \rangle + \sum_{\ell=1}^M \beta_{\ell} \|u_h^k\|^{q_{\ell}}_{L^{q_{\ell}}}.
\end{align*}
Applying \eqref{u^q to u^p journey} along with Cauchy’s, Poincar\'e’s, and Young’s inequalities, we obtain
\begin{align*}
\frac{1}{2\Delta t}\|u_h^k\|^2_{L^2}-\frac{1}{2\Delta t}\|u_h^{k-1}\|^2_{L^2}+\frac{1}{2\Delta t}\|u_h^k-u_h^{k-1}\|^2_{L^2} + \nu \|\nabla_h u_h^k\|^2_{L^2(\mathcal{T}_h)}+ \frac{\alpha}{2}\|u_h^k\|^p_{L^p}&\\ \, \leq \frac{C_{\Omega}}{2\nu}\|f^k\|^2_{L^{2}}+\frac{\nu}{2}\|\nabla_h u_h^k\|^2_{L^2(\mathcal{T}_h)}+C^*|\Omega|.
\end{align*}
Using the positivity of the third and fifth terms of the left-hand side, we see that
\begin{align*}
\frac{1}{\Delta t}\|u_h^k\|^2_{L^2}-\frac{1}{\Delta t}\|u_h^{k-1}\|^2_{L^2}+ \nu \|\nabla_h u_h^k\|^2_{L^2(\mathcal{T}_h)} \leq \frac{C_{\Omega}}{\nu}\|f^k\|_{L^{2}}+2C^*|\Omega|.
\end{align*}
Summing over $k$ for $k=1,2,3,\ldots, N$, we have
\begin{align*}
\frac{1}{\Delta t} \|u_h^N\|^2_{L^2}-\frac{1}{\Delta t}\|u_h^0\|^2_{L^2}+ \nu \sum_{k=1}^N \|\nabla_h u_h^k\|^2_{L^2(\mathcal{T}_h)} \leq \frac{C_{\Omega}}{\nu \Delta t} \Delta t \sum_{k=1}^N \|f^k\|^2_{L^{2}}+2C^*|\Omega|.
\end{align*}
Lemma \ref{NCFEM f^k} implies that
\begin{align*}
\frac{1}{\Delta t} \|u_h^N\|^2_{L^2}+ \nu \sum_{k=1}^N \|\nabla_h u_h^k\|^2_{L^2(\mathcal{T}_h)} \leq \frac{C_{\Omega}}{\nu \Delta t} \|f\|^2_{L^2(0,T;L^{2})}+\frac{1}{\Delta t}\|u_h^0\|^2_{L^2}+2C^*|\Omega|.
\end{align*}
Since,$ \|u_h^0\|_{L^2}\leq \|u_0\|_{L^2}$, we further conclude
\begin{align*}
\nu \sum_{k=1}^N \Delta t \|\nabla_h u_h^k\|^2_{L^2(\mathcal{T}_h)} \leq \frac{C_\Omega}{\nu}\|f\|^2_{L^2(0,T;L^{2})}+\|u^0\|^2_{L^2}+2C^*|\Omega|,
\end{align*}
so that the proof is completed.
\end{proof}
\begin{Proposition}
For all values of $p$ given in \eqref{eqn-values of p}, if $u_0\in D(A^{\frac{3}{2}})$, $f \in H^1(0,T;H^1(\Omega))$, and $u^{CR}_h$ be the semidiscrete solution of \eqref{NCFEM semidiscrete} and $u^{CR}_{kh}$ be the generic fully-discrete solution \eqref{NCFEM fully discrete formulation}. Then
\begin{align*}
	&\|u_h^{CR}-u_{kh}^{CR}\|^2_{L^\infty(0,T;L^2)}+\nu\int_0^T \triplenorm{u_h^{CR}(t)-u_{kh}^{CR}(t)}^2_{CR}\, dt\\& \quad \leq c(\Delta t)^2(\|u_0\|^2_{D(A^{\frac{3}{2}})}+\|f\|^2_{H^1(0,T;H^1)}+C'),
\end{align*}
where $C'$ is the same as \eqref{C'} by replacing $u_h$ by $u_h^{CR}$.
\end{Proposition}
\begin{proof} 
The proof of this lemma can be carried out in the same manner as that of Lemma \ref{CFEM uk-ukh error}.
\end{proof}
\begin{Theorem}\label{NCFEM fully error estimate}
For all values of $p$ given in \eqref{eqn-values of p}, if $u_0 \in D(A^{\frac{3}{2}})$ and $f\in H^1(0,T;H^1(\Omega))$, the non-conforming finite element approximation $u^{CR}_{kh}$ converges to $u$ as $\Delta t, h  \to 0$. In addition, there exists a constant $C\geq 0$ such that the approximation $u^{CR}_{kh}$ satisfies the following error estimates:
\begin{align*}
	&	\|u-u^{CR}_{kh}\|^2_{L^\infty(0,T;L^2)}+\nu\int_0^T \triplenorm{u(t)-u^{CR}_{kh}(t)}^2_{CR}\, dt \\& \quad \leq C((\Delta t)^2+h^2) (\|f\|^2_{H^1(0,T;H^1)}+ \|u_0\|^2_{D(A^{\frac{3}{2}})}+C'),
	\end{align*}
	where $C'$ is same as \eqref{C'} by replacing $u_h$ by $u_h^{CR}$.
\end{Theorem}
\begin{proof}
The proof of this theorem is similar to that of Theorem \ref{fullydiscrete error} and is therefore omitted to avoid repetition.
\end{proof}
\subsection{Numerical studies}\label{Numerical Studies for NCFEM}
In this section, we present a numerical example using the Crouzeix-Raviart (CR) finite element method. The error table below reports the error estimates and observed convergence rates for problems \eqref{without pumping} and \eqref{pumped} with parameters $\alpha=1$, $\nu=1$, $M=2$, $q_1=3$, $q_2=4$, $\beta_1=1$, $\beta_2=3$, and $p=5$ on the unit square domain $\Omega=(0,1)^2$ over the time interval $t \in [0,1]$, measured in the following norm:
\begin{align*}
\triplenorm{u_k-u_h^k}^2_{NC}=\| u_h(t_N) - u_h^N \|_{L^2(\Omega)}^2 + \nu   \triplenorm{u_h(t_k)-u_h^k}^2_{CR}.
\end{align*}
In this example, the exact solution of \eqref{without pumping} is
\begin{align*}
u(x,y,t)= (t-t^2+1)\cos(\pi x)\cos(\pi y)y(1-y)x(1-x),
\end{align*}
for the forcing term and initial condition
\begin{align*}
	f:=\partial_t u - \Delta u + |u|^3 u, \text{ and } u_0 = \cos(\pi x)\cos(\pi y)y(1-y)x(1-x).
\end{align*}
To validate the numerical method, we compute the solution on uniform triangular meshes of size 
$N\times N$ 
with a time step of  $\Delta t = 0.01$.

First, Table \ref{tab:NCFEM-damped} presents the errors and observed convergence rates for different grid sizes for problem \eqref{without pumping}, showing that the method achieves the expected convergence rates. We also provide visual comparisons between the exact and numerical solutions for mesh sizes $16 \times 16$, $32 \times 32$, and $64 \times 64$ in Figure \ref{fig:NCFEM damped}. From these visualizations, it is evident that the numerical solution closely approximates the exact solution, particularly on finer meshes.

\begin{table}[ht!]
	\centering
	\begin{tabular}{|c|c|c|c|}
		\hline
		\text{Grid Size} & $h$ & $\triplenorm{u-u_h^{CR}} $ & \text{Convergence Rate} \\
		\hline
		$4 \times 4$   & 3.54e{-01} & 1.7225e{-02} & N/A \\
		$8 \times 8$   & 1.77e{-01} & 7.2360e{-03} & 1.25 \\
		$16 \times 16$ & 8.84e{-02} & 3.4279e{-03} & 1.08 \\
		$32 \times 32$ & 4.42e{-02} & 1.6898e{-03} & 1.02 \\
		$64 \times 64$ & 2.21e{-02} & 8.4191e{-04} & 1.01 \\
		\hline
	\end{tabular}
	\caption{Computed errors and rates of convergence for \eqref{without pumping} in NCFEM.}
	\label{tab:NCFEM-damped}
\end{table}
\begin{figure}[ht!]
    \centering
    \begin{subfigure}{0.32\textwidth}
        \centering
        \includegraphics[width=\linewidth]{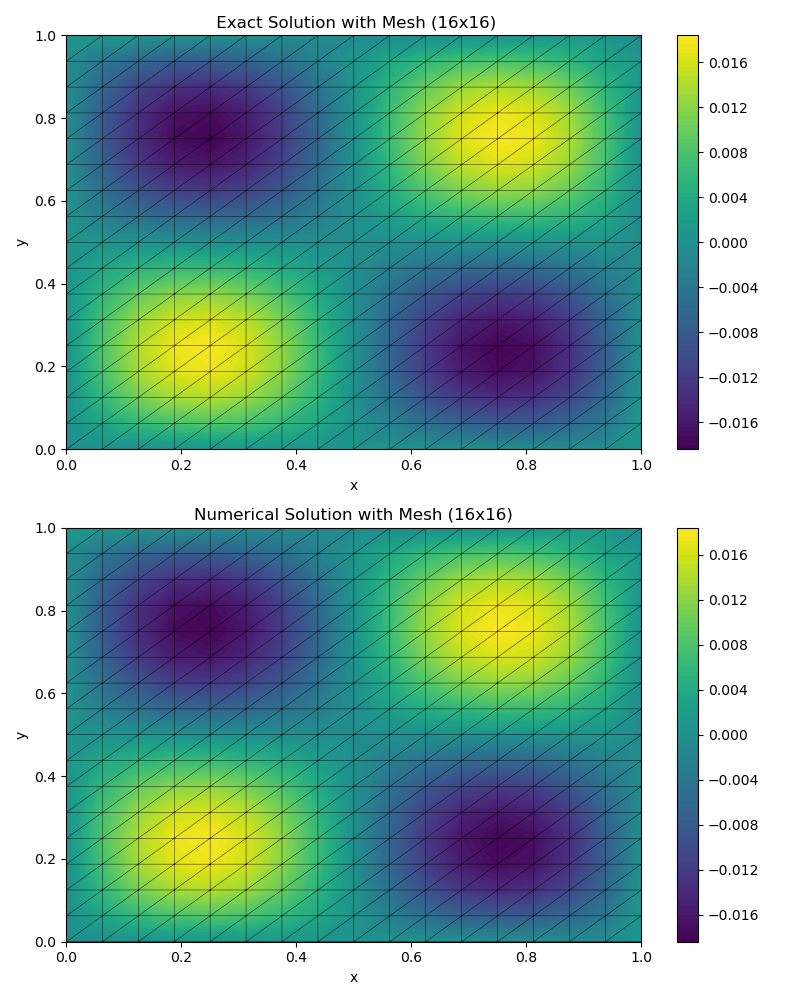}
    \end{subfigure}
    \begin{subfigure}{0.32\textwidth}
        \centering
        \includegraphics[width=\linewidth]{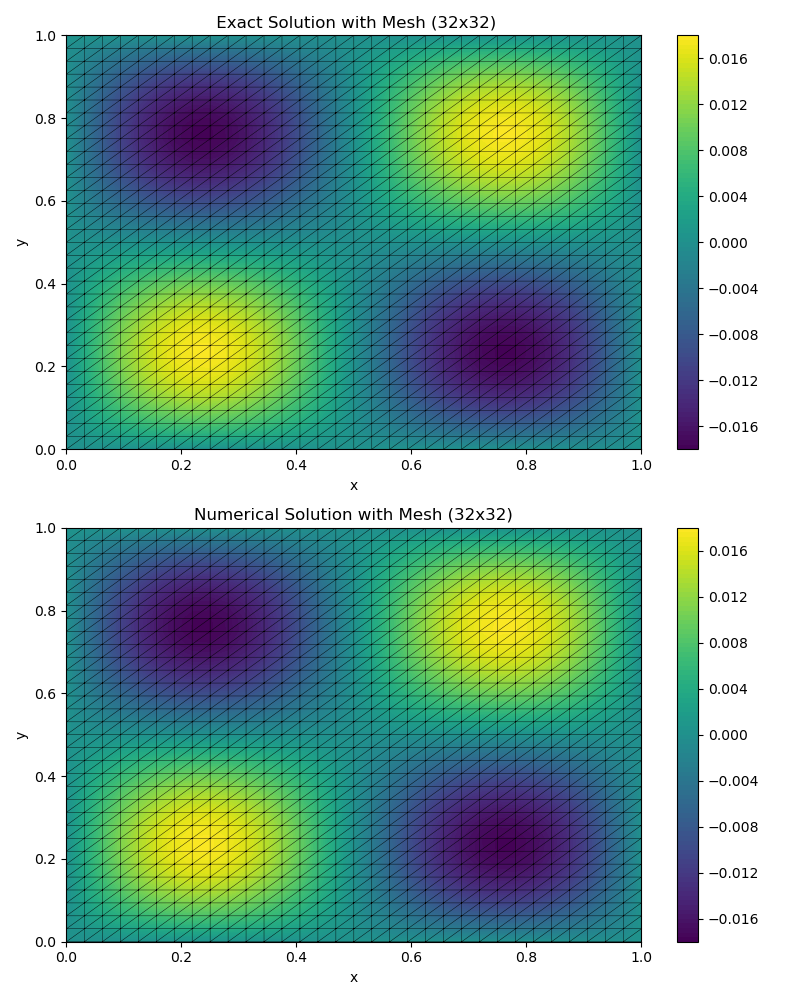}
    \end{subfigure}
    \begin{subfigure}{0.32\textwidth}
        \centering
        \includegraphics[width=\linewidth]{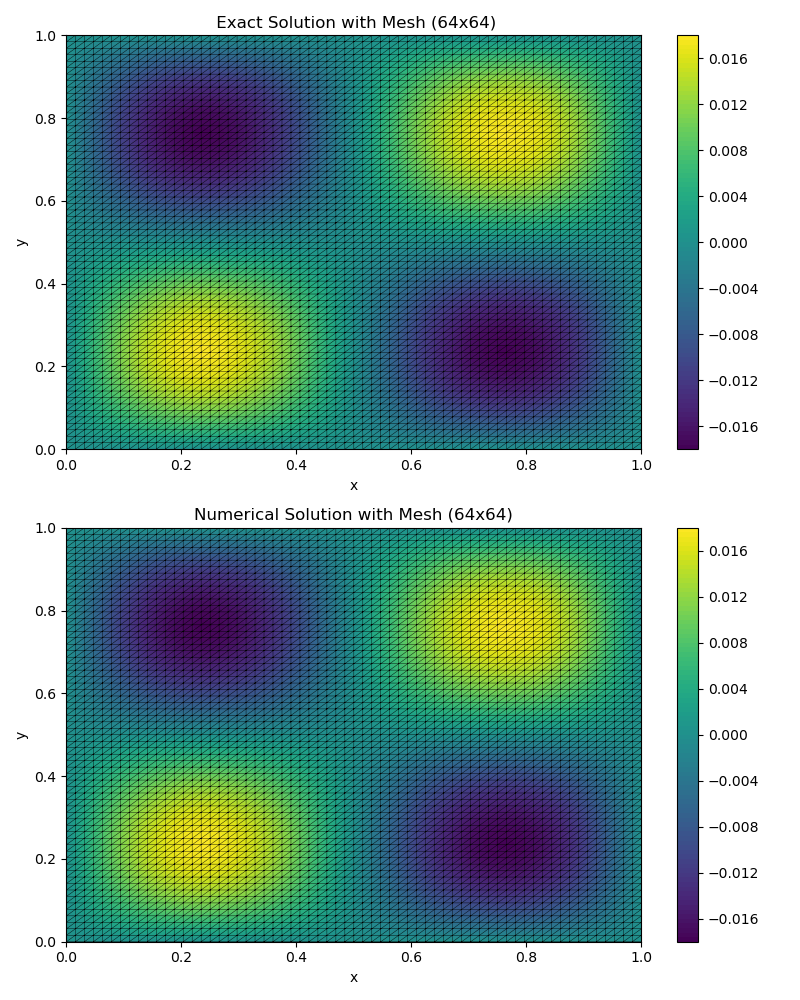}
    \end{subfigure}
    \caption{Exact and approximated solution of \eqref{without pumping} on different grid sizes.}
    \label{fig:NCFEM damped}
\end{figure}

Next, Table \ref{tab:NCFEM-pumped} presents the computed errors and corresponding convergence rates for problem \eqref{pumped}, using the solution on a $256 \times 256$ mesh as the reference. The expected rates of convergence are clearly observed. Additionally, Figure \ref{fig:NCFEM full} provides visual comparisons between the reference solution and numerical solutions on mesh sizes $16 \times 16$, $32 \times 32$, and $64 \times 64$.

\begin{table}[ht!]
	\centering
	\begin{tabular}{|c|c|c|c|}
		\hline
		\text{Grid Size} & $h$ & $\triplenorm{u_{ref}-u_h^{CR}}$ & \text{Convergence Rate} \\
		\hline
		$4 \times 4$   & 3.54e{-01} & 1.722952e{-02} & N/A \\
		$8 \times 8$   & 1.77e{-01} & 7.236985e{-03} & 1.25 \\
		$16 \times 16$ & 8.84e{-02} & 3.431414e{-03} & 1.08 \\
		$32 \times 32$ & 4.42e{-02} & 1.692934e{-03} & 1.02 \\
		$64 \times 64$ & 2.21e{-02} & 8.535687e{-04} & 0.99 \\
		\hline
	\end{tabular}
	\caption{Computed errors and rates of convergence for \eqref{pumped} in NCFEM.}
	\label{tab:NCFEM-pumped}
\end{table}

\begin{figure}[ht!]
    \centering
    \begin{subfigure}{0.32\textwidth}
        \centering
        \includegraphics[width=\linewidth]{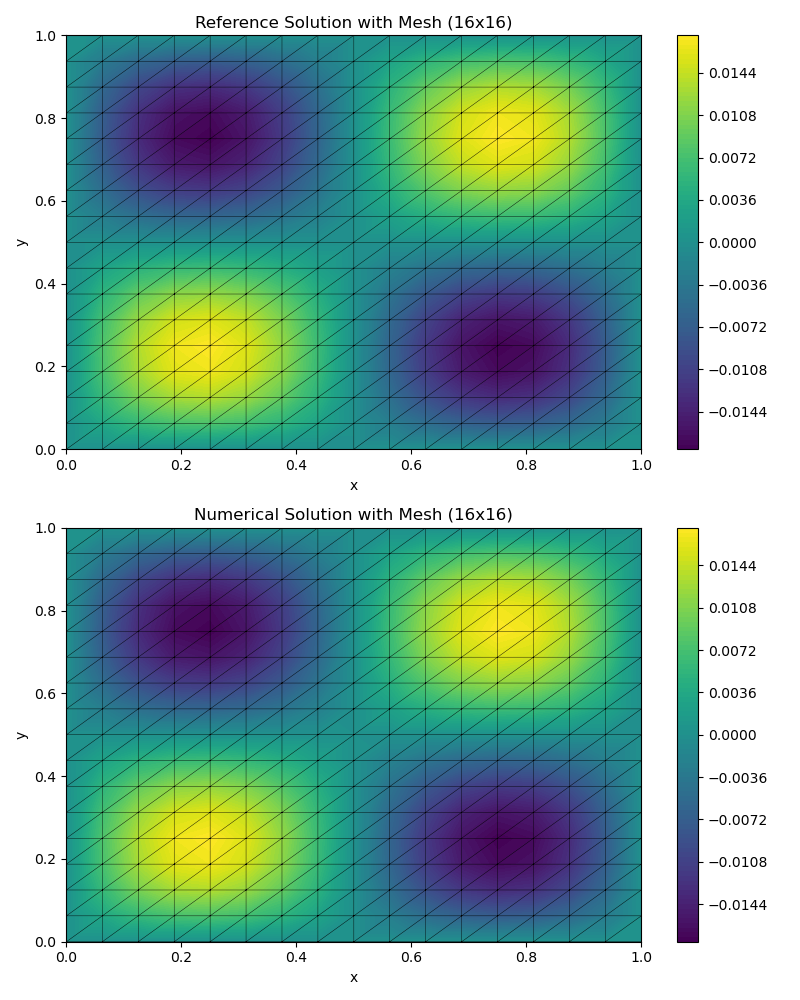}
    \end{subfigure}
    \begin{subfigure}{0.32\textwidth}
        \centering
        \includegraphics[width=\linewidth]{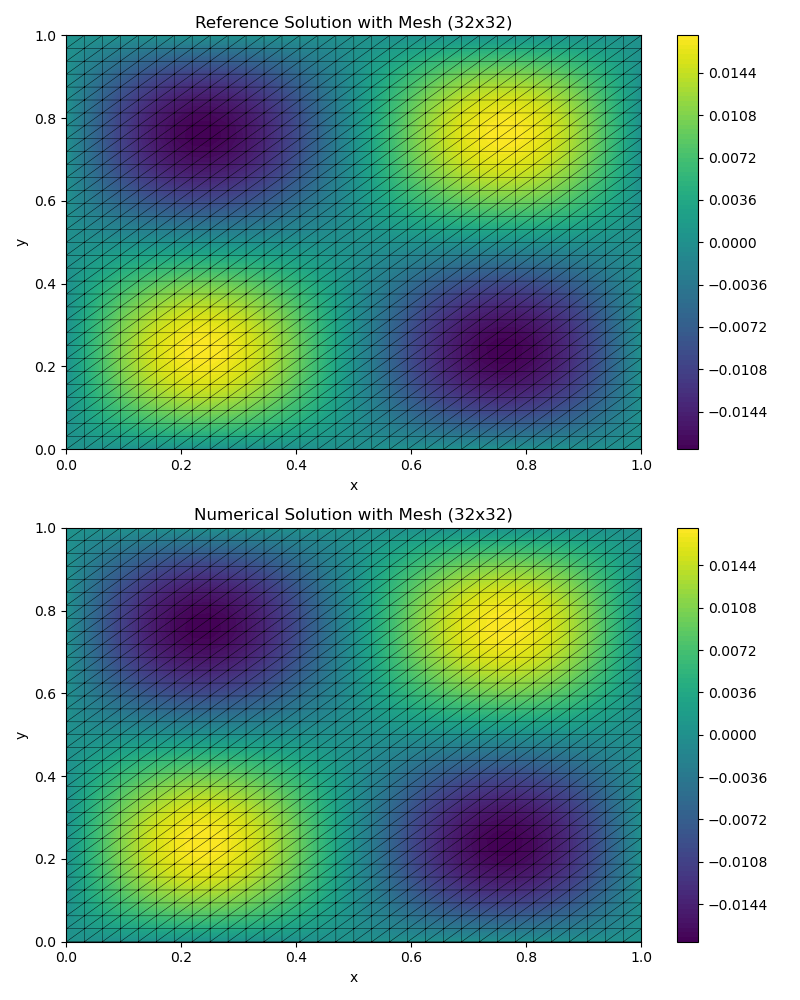}
    \end{subfigure}
    \begin{subfigure}{0.32\textwidth}
        \centering
        \includegraphics[width=\linewidth]{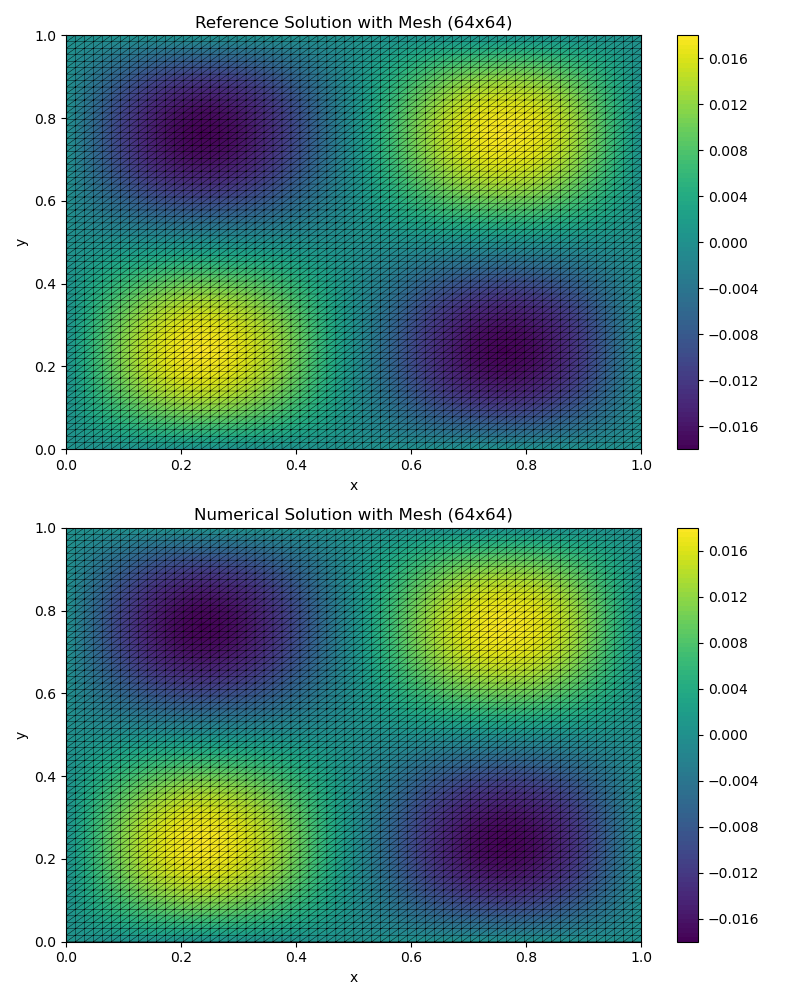}
    \end{subfigure}
    \caption{Reference and approximated solution of \eqref{pumped} on different grid sizes.}
    \label{fig:NCFEM full}
\end{figure}

\section{Discontinuous Galerkin method}\label{DG scheme}
The Discontinuous Galerkin (DG) method is a finite element approach in which the trial and test functions are permitted to be discontinuous across element boundaries. Unlike conforming methods, which enforce continuity, or nonconforming methods, which impose it weakly, DG does not require continuity between elements. Instead, inter-element interaction is managed through numerical fluxes and penalty terms, which ensure stability and consistency in a weak sense. DG methods are highly flexible, locally conservative, and well-suited for complex geometries as well as adaptive mesh refinement.

We define the discrete space for DG formulation as 
\begin{equation}\label{DG space}
V_h^{DG}=\left\{v \in L^2(\Omega)\big| \text{ for all } K \in \mathcal{T}_h: v|_K \in P_1(K)\right\},
\end{equation}
where $P_1(K)$ denotes the space of linear polynomials.
\subsection{Semidiscrete discontinuous Galerkin FEM}\label{Semi DGFEM}
We begin by defining the projection operator employed in the DG method.
\begin{Definition}[Elementwise $L^2$-projection onto discontinuous spaces {\cite[Section 1.6]{MR2050138}}]\label{DG projection}
	Let $\mathcal{T}_h$ be a partition of $\Omega$ into elements $K$, and let $P_1(K)$ denote the polynomials of degree $1$ on $K$. 
	Define the operator 
	\[
	\pi_h^{1} : L^2(\Omega) \to V_h^{DG},
	\]
	by requiring, for each $K \in \mathcal{T}_h$, 
	\begin{equation*}
		\int_K \big( \pi_h^{1} v - v \big) \, q \, dx = 0 
		\ \text{ for all }\ q \in P_1(K).
	\end{equation*}
	In other words, $\pi_h^1 v|_K$ is the local $L^2$-projection of $v|_K$ onto $P_1(K)$. 
	
	Moreover, there exists a constant $C>0$, independent of $h$, such that for all integers $0 \leq l \leq 2$, 
	for all $1 \leq p \leq \infty$, and for all $v \in W^{l,p}(\Omega)$, the following approximation estimate is satisfied:
	\begin{align*}
		\| v - \pi_h^1 v \|_{L^p(\Omega)} 
		&\leq C \, h^l \, | v |_{W^{l,p}(\Omega)}.
	\end{align*}
	\begin{Remark} 
	The above  estimate holds without assuming strong shape-regularity of the mesh. 
	The operator $\pi_h^1$ is the natural projection onto discontinuous finite element spaces, 
	and it is particularly useful in discontinuous Galerkin methods, e.g., for projecting initial data and in error analysis.
	\end{Remark}
\end{Definition}

We define the semidiscrete formulation of \eqref{Damped Heat} as follows: 
Find $u_h^{DG}(t) \in V_h^{DG}$ for $t\in [0,T]$ such that 
\small
\begin{equation}\label{DG semidiscrete}
\begin{aligned}
	\begin{cases}
		\displaystyle
&(\partial_t u_h^{DG}(t),\chi)+\nu a_{DG}(u_h^{DG}(t),\chi)+\alpha(|u_h^{DG}(t)|^{p-2}u_h^{DG}(t),\chi)-\sum_{\ell=1}^M \beta_{\ell} (|u_h^{DG}(t)|^{q_{\ell}-2}u_h^{DG}(t),\chi)\\&\quad=(f(t),\chi),\\
&(u_h^{DG}(0),\chi)=(\pi_h^1u_0,\chi),
\end{cases}
\end{aligned}
\end{equation}
for all $\chi \in V_h^{DG}$, where 
\begin{align*}
a_{DG}(u,v)&=(\nabla_h u,\nabla_h v)-\sum_{E \in \mathcal{E}_h}\int_E \{\!\{ \nabla_h u \}\!\} \llbracket v \rrbracket\ ds-\sum_{E \in \mathcal{E}_h}\int_E \{\!\{ \nabla_h v \}\!\} \llbracket u \rrbracket\, ds+\sum_{E \in \mathcal{E}_h} \int_E \gamma_h \llbracket u \rrbracket \llbracket v \rrbracket\, ds.
\end{align*}
Here $\gamma_h = \frac{\gamma}{h_E}$. The length of the edge $E$ is represented by the parameter $h_E$. In order to guarantee the stability of the formulation, the penalty parameter $\gamma$ is selected to be sufficiently large. The following discrete norm is used for further error analysis:
\begin{align*}
\triplenorm{v}^2_{DG}:=\sum_{K \in \mathcal{T}_h}\|\nabla_h v\|^2_{L^2(\mathcal{T}_h)}+ \sum_{E \in \mathcal{E}_h} \gamma_h \|\llbracket v\rrbracket\|^2_{L^2(E)}.
\end{align*}
Note that $a_{DG}$ can be bounded as (\cite[Section 2]{MR664882}):
\begin{align}\label{continuity of DG}
|a_{DG}(u,v)| \leq (1+\gamma)\triplenorm{u}_{DG}\triplenorm{v}_{DG}.
\end{align}
\begin{Lemma}[Coercivity]\label{Coercivity and stability}
For any \( v \in V_h^{DG} \), the bilinear form \( a_{DG} \) is \emph{coercive}, that is,
    \begin{equation}\label{coercivity}
    a_{DG}(v, v) \geq \alpha_a \triplenorm{v}_{DG}^2,
    \end{equation}
    for a positive constant \( \alpha_a \geq 0 \).
\end{Lemma}
\begin{proof}
The proof of coercitivity of $a_{DG}(\cdot,\cdot)$ follows directly from \cite[section 3]{MR664882}.
\end{proof}
\begin{Theorem}[Existence and energy estimates]\label{energy estimate of DGFEM}
	\textbf{(a)} If $f\in L^2(0,T;H^{-1}(\Omega))$ and $u_0 \in L^2(\Omega)$,  then the solution $u_h^{DG}$ of \eqref{DG semidiscrete} is unique and satisfies the following estimate:
	\begin{align*}
		\sup_{0\leq t \leq T}\|u_h^{DG}(t)\|^2_{L^2}&+\nu \int_0^T \triplenorm{u_h^{DG}(t)}^2_{DG} \, dt + \alpha \int_0^T \|u_h^{DG}(t)\|^p_{L^p}\, dt \\&\leq \|u_0\|^2_{L^2}+\frac{1}{\nu} \int_0^T \|f(t)\|^2_{H^{-1}}\, dt + C^{*}T|\Omega|,
	\end{align*}
	where $C^*$ is defined in \eqref{C^*}.
	
	\noindent
	\textbf{(b)}If $f\in L^2(0,T;L^2(\Omega))$ and $u_0 \in D(A)$, then 
	\begin{align*}
		\int_0^T \|\partial_t u_h^{DG}(t)\|^2_{L^2}\, dt &+ \sup_{0\leq t \leq T}\bigg(\nu\triplenorm{u_h^{DG}(t)}^2_{DG}+\frac{2\alpha}{p}\|u_h^{DG}(t)\|^p_{L^p}\bigg) \\&\leq \nu \|\nabla u_0\|^2_{L^2}+\frac{2\alpha}{p}\|u_0\|^p_{D(A)} + \int_0^T \|f(t)\|^2_{L^2}\, dt.
	\end{align*}
\end{Theorem}
\begin{proof}
	The proof of this theorem follows directly from arguments similar to those in Theorem \ref{energy estimate of CFEM}.
\end{proof}
\begin{Theorem}\label{DGFEM Semidiscrete error estimate}
Let $V_h^{DG}$ be the space as defined in \eqref{DG space}. Assume that $u_0\in D(A)$, and $u(\cdot)$ is the solution of \eqref{Damped Heat}.
Then the $u_h^{DG}$, the solution of \eqref{DG semidiscrete} satisfies the following estimate:

\noindent
(a) \textbf{For $f \in H^1(0,T,L^2(\Omega))$ and $2\leq p\leq \frac{2d}{d-2}$:}
\begin{align*}
	&\|u_h^{DG}-u\|^2_{L^\infty(0,T;L^2)}+\nu \int_0^T \triplenorm{u_h^{DG}(t)-u(t)}^2_{DG}\, dt+\frac{\alpha}{2^{p-3}}\|u_h^{CR}-u\|^p_{L^p(0,T;L^p)} \\
	&\leq Ch^2\bigg\{\|u\|^2_{L^{\infty}(0,T;H_0^1)}+\|u\|^2_{L^\infty(0,T;H^2)}+\int_0^T\|\partial_t u(t)\|_{H_0^1}^2\, dt+\|u\|^p_{L^{\infty}(0,T;H^2)}\\& \quad +(\|u_0\|_{D(A)}^p+ \|f\|^2_{H^1(0,T;L^2)}+\|u\|^{2(p-2)}_{L^\infty(0,T;L^p)})\int_0^T \|u(t)\|^2_{H^2}\,dt \bigg\},
\end{align*}
for some positive constant $C.$\\
(b)\textbf{For $f \in H^1(0,T;H^1(\Omega))$ and $\frac{2d}{d-2}<p\leq \frac{2d-6}{d-4}$:}\\
(i) if $2\leq q_{\ell}<1+\frac{p}{2}$, the following estimate holds:
\begin{align*}
	&\|u_h^{DG}-u\|^2_{L^\infty(0,T;L^2)}+\nu \int_0^T\triplenorm{u_h^{DG}(t)-u(t)}^2_{DG} \, dt + \frac{\alpha}{2^{p-2}}\|u_h^{DG}-u\|^p_{L^p(0,T;L^p)} \\& \leq Ch^2\bigg( \|u\|^2_{L^\infty(0,T;H_0^1)}+\|u\|^2_{L^\infty(0,T;H^2)}+\|u\|^p_{L^\infty(0,T;H^2)}+ \int_0^T \|u(t)\|^2_{H^2}\,dt \\& \quad  +\int_0^T \|\partial_t u(t)\|^2_{H_0^1}\,dt+\int_0^T \|u(t)\|^2_{D(A^{\frac{3}{2}})}\,dt \bigg),
\end{align*}
(ii) if $1+\frac{p}{2}\leq q_{\ell}<p$, the following estimate holds:
\begin{align*}
	&\|u_h^{DG}-u\|^2_{L^\infty(0,T;L^2)}+\nu \int_0^T\triplenorm{u_h^{DG}(t)-u(t)}_{DG} \, dt+\frac{\alpha}{2^{p-2}}\|u_h^{DG}-u\|^p_{L^p(0,T;L^p)}\\&\leq Ch^2\bigg\{\|u\|^2_{L^\infty(0,T;H_0^1)}+\|u\|^2_{L^\infty(0,T;H^2)}+\|u\|^p_{L^\infty(0,T;H^2)}+\int_0^T \|\partial_tu(t)\|^2_{H_0^1}\, dt \\& \quad + \int_0^T \|u(t)\|^2_{H^2}\,dt   +\|u\|^{\frac{p(q_{\ell}-3)+2}{p-1}}_{L^\infty(0,T;H^2)} \int_0^T \|u(t)\|^2_{D(A^{\frac{3}{2}})}\,dt \bigg\},
\end{align*}
where $C>0 $.
\end{Theorem}
\begin{proof}
An application of the triangle inequality yields
	\begin{equation*}
		\triplenorm{u_h^{DG}-u}_{DG}\leq \triplenorm{u_h^{DG}-W}_{DG}+\triplenorm{W-u}_{DG},
	\end{equation*}
	where $W = \pi_h^1 u$. Given that $\triplenorm{W-u}_{DG}\leq Ch|u|_{H^2}$ (see \cite[Section 3]{MR664882}), we are thus led to estimating $\triplenorm{u_h^{DG}-W}_{DG}$. Using the regularity result (see Theorem \ref{regularity-f- H1}), the following holds for a.e. $t\in[0,T]$:
	\begin{equation}\label{DG exact sol for error}
		(\partial_t u(t), \chi)+\nu a_{DG}(u(t),\chi)+\alpha (|u(t)|^{p-2}u(t),\chi)-\sum_{\ell=1}^M \beta_{\ell}(|u(t)|^{q_{\ell}-2}u(t),\chi)=(f(t),\chi),
	\end{equation}
for all $\chi \in V_h^{DG}$. Therefore, from \eqref{DG semidiscrete} and \eqref{DG exact sol for error}, we get 
\begin{align*}
	(\partial_t(u_h^{DG}(t)-u(t)),\chi)&+\nu a_{DG}(u_h^{DG}(t)-u(t),\chi)+\alpha(|u_h^{DG}(t)|^{p-2}u_h^{DG}(t)-|u(t)|^{p-2}u(t),\chi)\\&=\sum_{\ell=1}^M \beta_{\ell}(|u_h^{DG}(t)|^{q_{\ell}-2}u_h^{DG}(t)-|u(t)|^{q_{\ell}-2}u(t),\chi),
\end{align*}
for a.e. $t\in[0,T]$. Rearranging the equation by taking $u_h^{DG}-u=u_h^{DG}-W+W-u$, we obtain 
\begin{align*}
	&(\partial_t (u_h^{DG}(t)-W(t)),\chi)+\nu a_{DG}(u_h^{DG}(t)-W(t),\chi)+\alpha(|u_h^{DG}(t)|^{p-2}u_h^{DG}(t)-|W(t)|^{p-2}W(t),\chi)\\&
	=-(\partial_t(W(t)-u(t)),\chi)-\nu a_{DG}(W(t)-u(t),\chi)-\alpha(|W(t)|^{p-2}W(t)-|u(t)|^{p-2}u(t),\chi)\\&\quad + \sum_{\ell=1}^M \beta_{\ell}(|u_h^{DG}(t)|^{q_{\ell}-2}u_h^{DG}(t)-|W(t)|^{q_{\ell}-2}W(t),\chi)+\sum_{\ell=1}^M\beta_{\ell}(|W(t)|^{q_{\ell}-2}W(t)-|u(t)|^{q_{\ell}-2}u(t),\chi).
\end{align*}
Taking $\chi=u_h^{DG}-W$ and using  \eqref{coercivity}, we arrive at 
\begin{align*}
	&\frac{1}{2}\frac{d}{dt}\|u_h^{DG}(t)-W(t)\|^2_{L^2}+\nu \alpha_a \triplenorm{u_h^{DG}(t)-W(t)}_{DG}^2\\&\quad+\alpha(|u_h^{DG}(t)|^{p-2}u_h^{DG}(t)-|W(t)|^{p-2}W(t),u_h^{DG}(t)-W(t))\\&
	=-(\partial_t(W(t)-u(t)),u_h^{DG}(t)-W(t))-\nu a_{DG}(W(t)-u(t),u_h^{DG}(t)-W(t))\\&\quad-\alpha(|W(t)|^{p-2}W(t)-|u(t)|^{p-2}u(t),u_h^{DG}(t)-W(t))\\&\quad+ \sum_{\ell=1}^M \beta_{\ell}(|u_h^{DG}(t)|^{q_{\ell}-2}u_h^{DG}(t)-|W(t)|^{q_{\ell}-2}W(t),u_h^{DG}(t)-W(t))\\&\quad+\sum_{\ell=1}^M\beta_{\ell}(|W(t)|^{q_{\ell}-2}W(t)-|u(t)|^{q_{\ell}-2}u(t),u_h^{DG}(t)-W(t)),
\end{align*} 
for a.e. $t\in[0,T]$. Hence, by applying \eqref{continuity of DG} and proceeding with calculations analogous to those in Theorem \ref{NCFEM Semiestimate error estimate}, we obtain the desired estimate.
\end{proof}
\subsection{Fully-discrete discontinuous Galerkin FEM} \label{Fully DGFEM}
The fully discrete weak formulation of \eqref{Damped Heat} is given as follows: Find $(u_h^{DG})^k=u_h^k \in V_h^{DG}$(for simplicity of notation we have taken $(u_h^{DG})^k= u_h^k$), such that 
\begin{equation}\label{DG fully discrete formulation}
\begin{aligned}
	\begin{cases}
		\displaystyle
    \left( \frac{u_h^k - u_h^{k-1}}{\Delta t}, \chi \right) 
    + \nu a_{DG}(u_h^k, \chi)
    + \alpha (|u_h^k|^{p-2} u_h^k, \chi) -\sum_{\ell=1}^M \beta_{\ell}(|u_h^{k}|^{q_{\ell}-2}u_h^{k},\chi)
    = (f^k, \chi), \\
    (u_h^0(x_i), \chi) = (\pi_{h}^1 u_0(x_i),\chi),
    	\end{cases}
\end{aligned}
\end{equation}
for $i=1,2,\ldots, N$ and where $f^k = (\Delta t)^{-1} \int_{t_{k-1}}^{t_k} f(s)\, ds$ for $f \in L^2(0,T;L^2(\Omega))$, and $u_h^{0}$ is the approximation of $u_0$ in $V_h^{DG}$.

We then define DG approximated solution $u^{DG}_{kh}$ by 
\begin{equation}\label{DG generic fully}
u^{DG}_{kh}|_{[t_{k-1},t_k]} = u_h^{k-1} + \bigg(\frac{t - t_{k-1}}{\Delta t}\bigg)(u_h^k - u_h^{k-1}), \quad 1 \leq k \leq N\  \text{ for }\  t \in [t_{k-1}, t_k].
\end{equation}
The error estimates for the fully discrete case are discussed in the following results:
\begin{Lemma}
For all the values of $p$ given in \eqref{eqn-values of p}, if $u_0\in D(A^{\frac{3}{2}})$ and $f \in H^1(0,T;H^1(\Omega))$, let $u^{DG}_h$ be the semidiscrete solution \eqref{DG semidiscrete} and $u^{DG}_{kh}$ be the generic fully-discrete solution \eqref{DG generic fully}. Then
\begin{align*}
&\|u^{DG}_h-u^{DG}_{kh}\|^2_{L^{\infty}(0,T;L^2)}+ \nu\int_0^T \triplenorm{u^{DG}_h(t)-u^{DG}_{kh}(t)}^2_{DG} dt\nonumber\\&\leq c(\Delta t)^2(\|u_0\|^2_{D(A^{\frac{3}{2}})}+\|f\|^2_{H^1(0,T;H^1)}+C'),
\end{align*}
where $C'$ is same as \eqref{C'} by replacing $u_h$ by $u_h^{DG}$.
\end{Lemma}
\begin{proof} 
The proof of this lemma proceeds in the same manner as that of Lemma \ref{CFEM uk-ukh error}.
\end{proof}
\begin{Theorem}\label{DGFEM fully discrete error estimate}
For all values of $p$ given in \eqref{eqn-values of p}, if $u_0\in D(A)$ and $f\in H^1(0,T;L^2(\Omega))$, the discontinuous Galerkin finite element approximation $u^{DG}_{kh}$ converges to $u$ as $\Delta t, h  \to 0$. In addition, there exists a constant $C> 0$ such that the approximation $u^{DG}_{kh}$ satisfies the following error estimate:
\begin{align*}
	&	\|u-u^{DG}_{kh}\|^2_{L^\infty(0,T;L^2)}+\nu \int_0^T \triplenorm{u(t)-u^{DG}_{kh}(t)}^2_{DG}\, dt  \nonumber\\&\leq C((\Delta t)^2+h^2) (\|u_0\|^2_{D(A^{\frac{3}{2}})}+\|f\|^2_{H^1(0,T;H^1)}+C'),
	\end{align*}
	where $C'$ is same as \eqref{C'} by replacing $u_h$ by $u_h^{DG}$.
\end{Theorem}
\begin{proof}
The proof of this theorem can be carried out in the same manner as that of Theorem \ref{fullydiscrete error}.
\end{proof}
\subsection{Numerical studies}\label{Numerical Studies for DG method}
In this section, we consider an example solved using the DG finite element method. The error table below reports the error estimates and observed convergence rates for problems \eqref{without pumping} and \eqref{pumped}, with parameters $\alpha=1$, $\nu=1$, $M=2$, $q_1=3$, $q_2=4$, $\beta_1=1$, $\beta_2=3$, and $p=5$, on the unit square domain $\Omega=(0,1)^2$ over the time interval $t \in [0,1]$, measured in the following norm:
\begin{equation*}
\triplenorm{u_k-u_h^k}^2=\| u_h(t_N) - u_h^N \|_{L^2(\Omega)}^2 + \nu \Delta t \triplenorm{u_h(t_k)-u_h^k}^2_{DG}.
\end{equation*}
In this example, let the exact solution of \eqref{without pumping} to be 
\begin{equation*}
u(x,y,t)= (\cos(t)+1)\sin(3\pi x)\cos(2\pi y)y(1-y).
\end{equation*}
for the forcing term $
f= \partial_t u - \Delta u + |u|^3 u,
$. To validate our numerical method, we compute the numerical solution using uniform triangular meshes of size \(N \times N\) and time step \(\Delta t = 0.01\).

First, we present the error values together with the observed rates of convergence for different mesh sizes for problem \eqref{without pumping} in Table \ref{tab:DGFEM-damped}. Furthermore, Figure \ref{fig:DGFEM damped} provides visual comparisons between the exact and numerical solutions for mesh sizes 
$16\times 16, $
$32\times 32$, and 
$64\times 64.$
\begin{table}[ht!]
	\centering
	\begin{tabular}{|c|c|c|c|}
		\hline
		\text{Grid Size} & $h$ & $\triplenorm{u-u_{h}^{DG}}$ & \text{Convergence Rate} \\
		\hline
		$4 \times 4$   & 3.54e{-01} & 1.0407e{+00} & N/A \\
		$8 \times 8$   & 1.77e{-01} & 5.8507e{-01} & 0.83 \\
		$16 \times 16$ & 8.84e{-02} & 2.7907e{-01} & 1.07 \\
		$32 \times 32$ & 4.42e{-02} & 1.3394e{-01} & 1.06 \\
		$64 \times 64$ & 2.21e{-02} & 6.5800e{-02} & 1.03 \\
		\hline
	\end{tabular}
	\caption{Computed errors and rates of convergence for \eqref{without pumping} in DGFEM.}
	\label{tab:DGFEM-damped}
\end{table}

\begin{figure}[ht!]
    \centering
    \begin{subfigure}{0.32\textwidth}
        \centering
        \includegraphics[width=\linewidth]{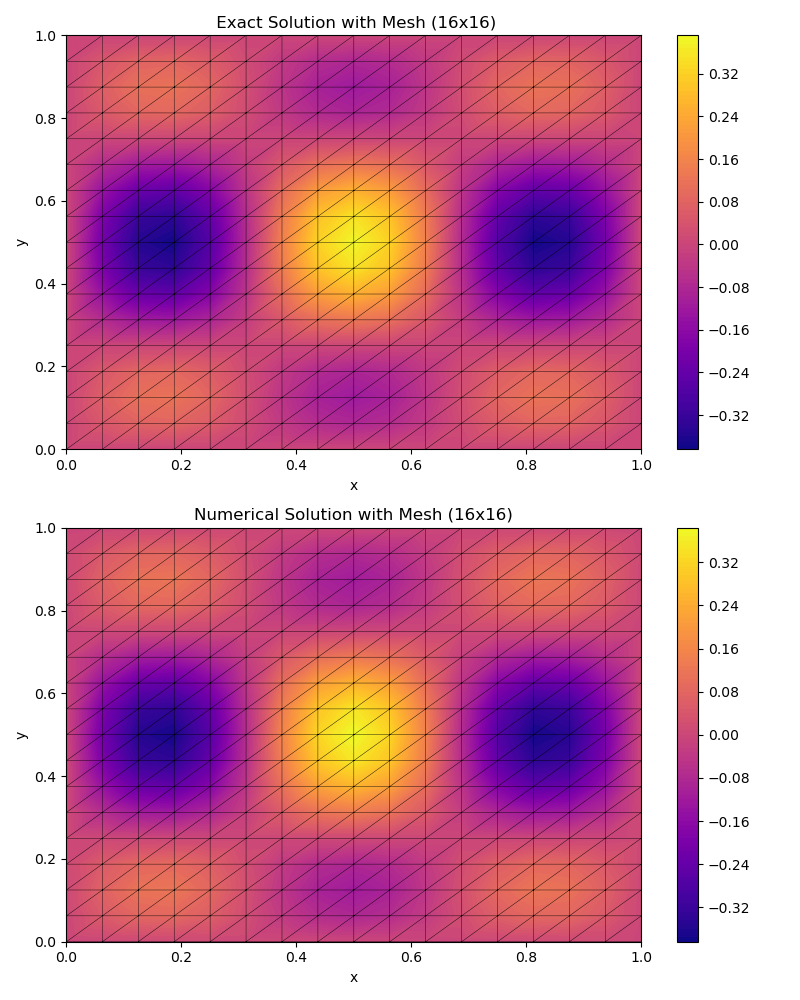}
    \end{subfigure}
    \begin{subfigure}{0.32\textwidth}
        \centering
        \includegraphics[width=\linewidth]{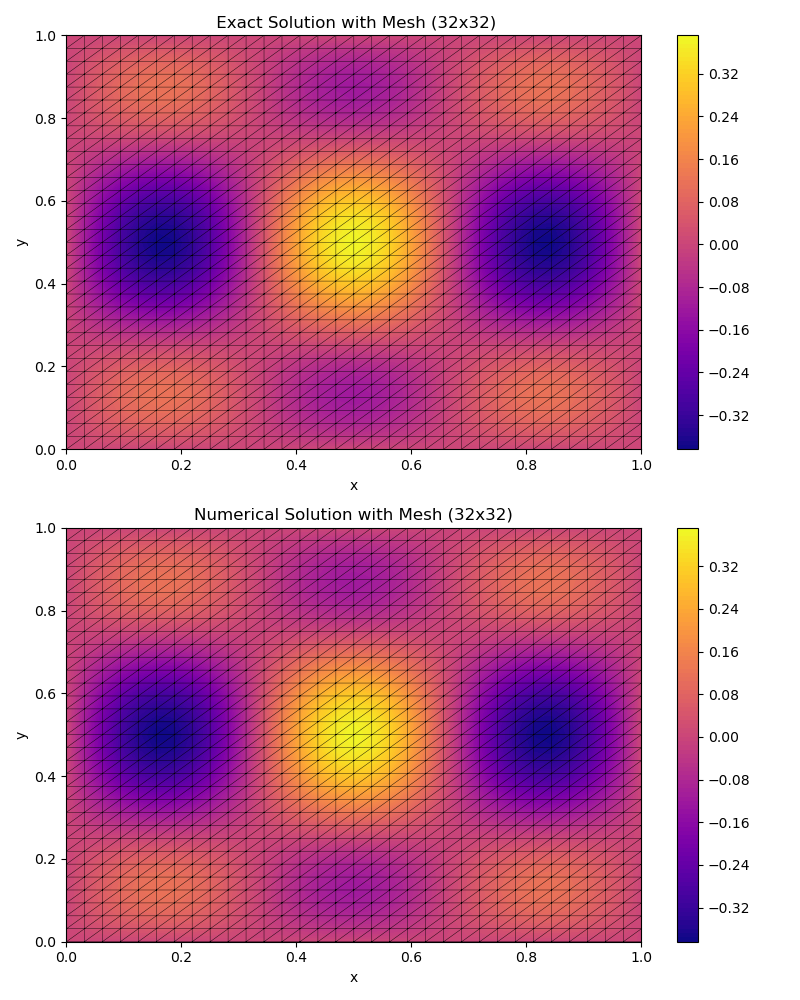}
    \end{subfigure}
    \begin{subfigure}{0.32\textwidth}
        \centering
        \includegraphics[width=\linewidth]{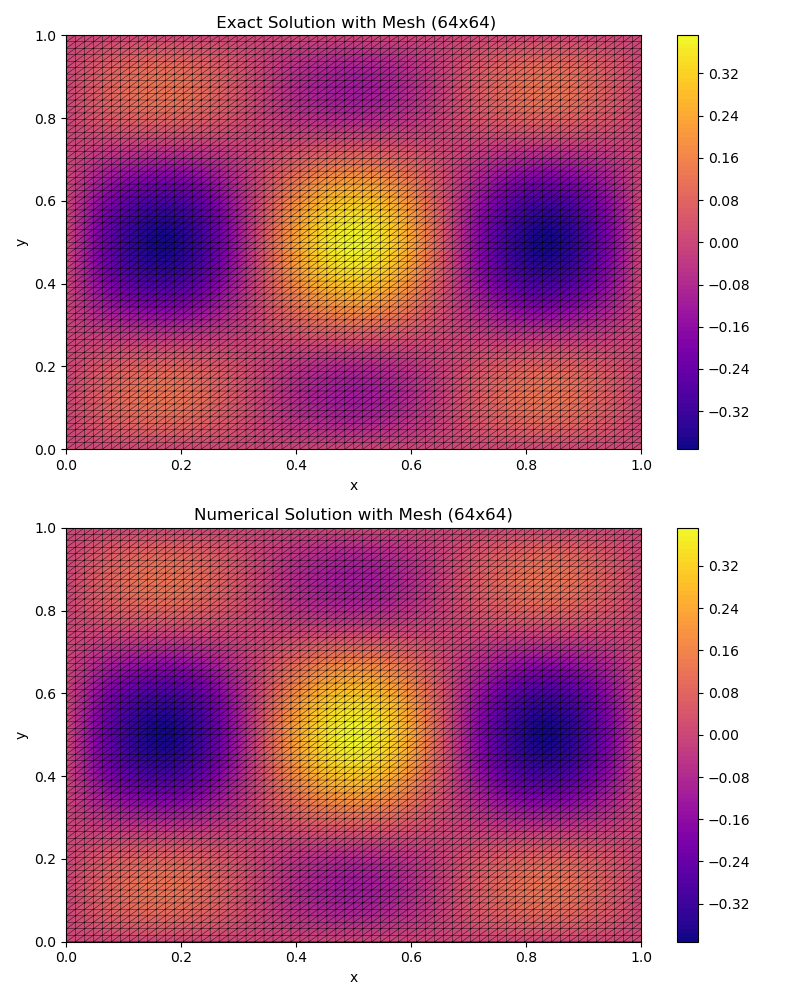}
    \end{subfigure}
    \caption{Exact and approximated solution of \eqref{without pumping} on different grid sizes.}
    \label{fig:DGFEM damped}
\end{figure}
Next, we compute and present the errors along with the corresponding rates of convergence for \eqref{pumped} in Table \ref{tab:DGFEM-full}, using the solution on a 
$256\times 256$ mesh as the reference for error computation. Visual comparisons between the reference solution and the numerical solutions for mesh sizes 
$16\times 16,$ 
$32\times 32$, and 
$64\times 64$ are shown in Figure \ref{fig:DGFEM full}.
\begin{table}[ht!]
	\centering
	\begin{tabular}{|c|c|c|c|}
		\hline
		\text{Grid Size} & $h$ & $\triplenorm{u_{ref}-u_h^{DG}}$ & \text{Convergence Rate} \\
		\hline
		$4 \times 4$   & 3.54e{-01} & 1.045534e{+00} & N/A \\
		$8 \times 8$   & 1.77e{-01} & 5.883905e{-01} & 0.83 \\
		$16 \times 16$ & 8.84e{-02} & 2.806583e{-01} & 1.07 \\
		$32 \times 32$ & 4.42e{-02} & 1.346609e{-01} & 1.06 \\
		$64 \times 64$ & 2.21e{-02} & 6.613109e{-02} & 1.03 \\
		\hline
	\end{tabular}
	\caption{Computed errors and rates of convergence for \eqref{pumped} in DGFEM.}
	\label{tab:DGFEM-full}
\end{table}

\begin{figure}[ht!]
    \centering
    \begin{subfigure}{0.32\textwidth}
        \centering
        \includegraphics[width=\linewidth]{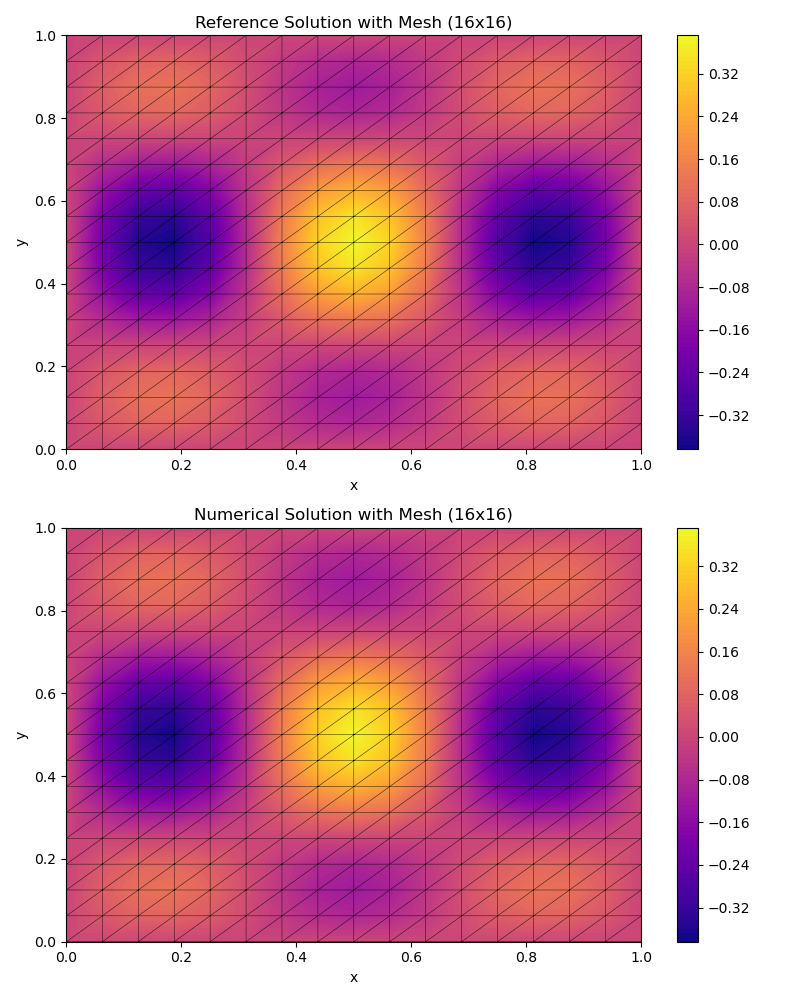}
    \end{subfigure}
    \begin{subfigure}{0.32\textwidth}
        \centering
        \includegraphics[width=\linewidth]{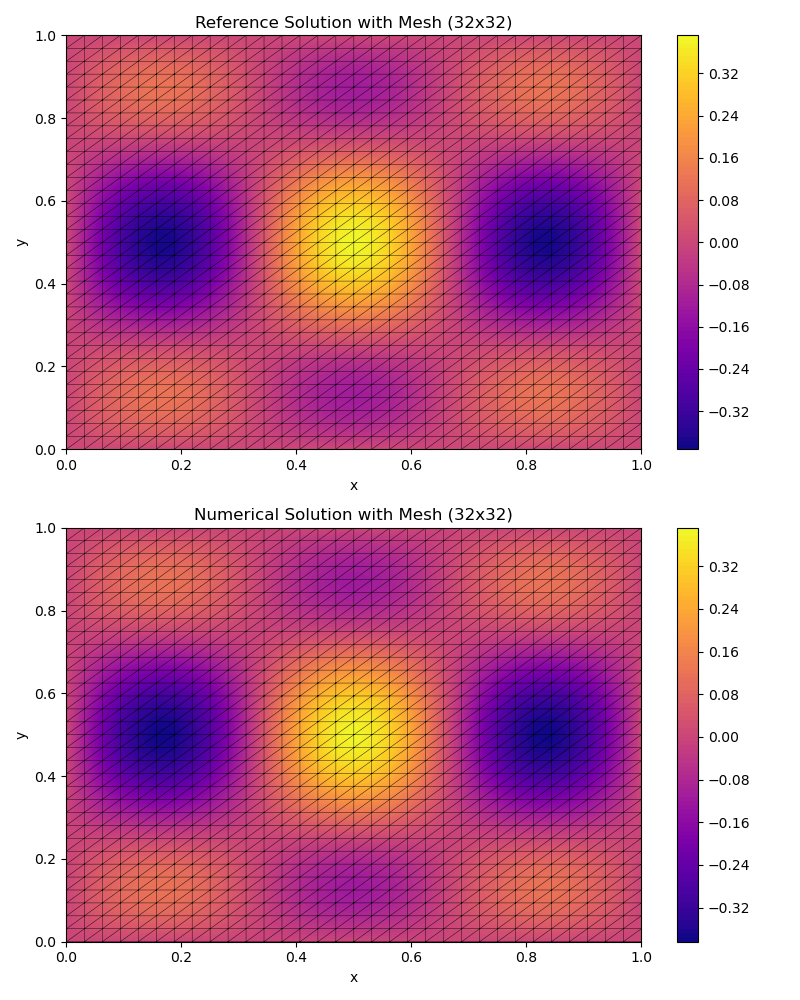}
    \end{subfigure}
    \begin{subfigure}{0.32\textwidth}
        \centering
        \includegraphics[width=\linewidth]{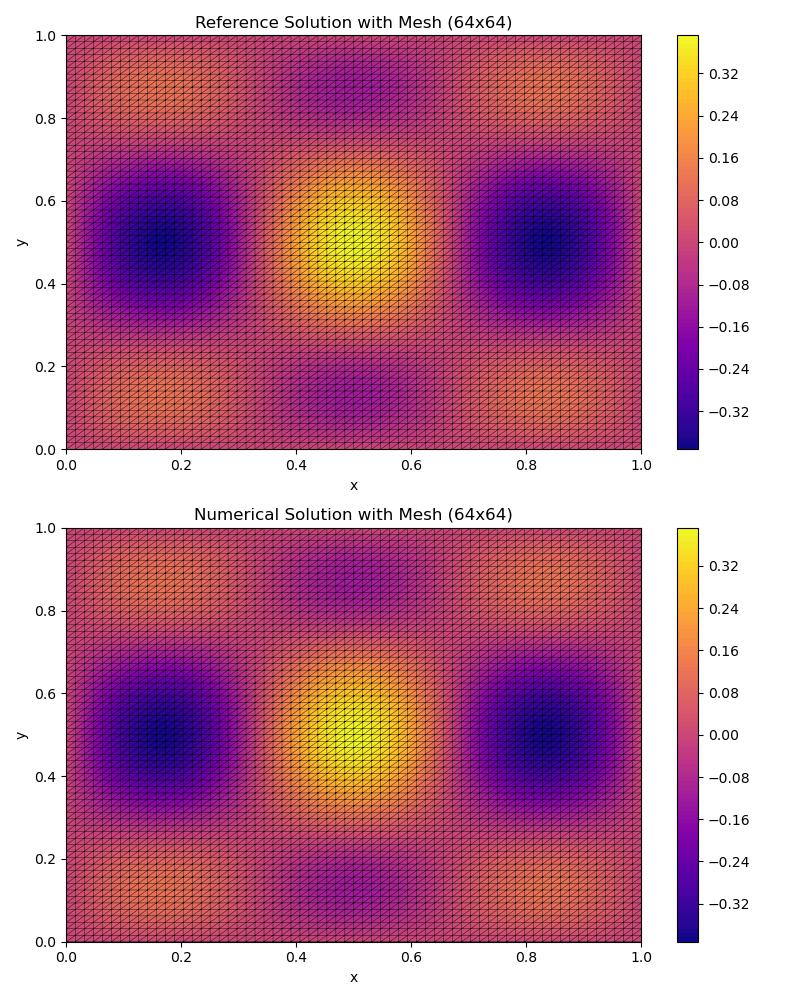}
    \end{subfigure}
    \caption{Reference and approximated solution of \eqref{pumped} on different grid sizes.}
    \label{fig:DGFEM full}
\end{figure}
\section{Comparison and 
Conclusion}\label{Conclusion and Comparison} 
\subsection{Fundamental differences between methods}\label{Theoretical differences}

In this subsection, we provide a brief overview of the principal theoretical differences among conforming, non-conforming, and discontinuous Galerkin (DG) finite element methods. These distinctions mainly stem from the selection of discrete function spaces, the continuity requirements across element interfaces, and the corresponding variational formulations.

\textbf{Conforming Finite Element Methods (CFEM):}
Conforming methods employ approximation spaces $V_h \subset H^1(\Omega)$, consisting of globally continuous functions with square-integrable derivatives. Basis functions are continuous across element boundaries, ensuring that integration by parts can be carried out over the entire domain without additional interface terms. This strong continuity leads to a simpler formulation and fewer degrees-of-freedom (DOFs) compared to DG methods of the same polynomial order, and the convergence theory is well established. However, the requirement of global continuity restricts the choice of approximation spaces and demands mesh conformity.

\textbf{Non-Conforming Finite Element Methods (NCFEM)}
Non-conforming methods employ approximation spaces that are not strict subsets of $H^1(\Omega)$ but satisfy certain weak continuity conditions. A notable example is the Crouzeix-Raviart (CR) element, where continuity is enforced only at edge midpoints rather than along the entire interface. In such methods, integration by parts is performed elementwise, and interface terms may vanish only under specific weak continuity constraints. These methods offer greater flexibility than conforming methods, particularly for problems with discontinuous coefficients, but the analysis is more complex and error estimates often require the use of broken Sobolev norms.

\textbf{Discontinuous Galerkin Methods (DG)}
DG methods use fully discontinuous polynomial spaces $V_h \subset L^2(\Omega)$ without any interelement continuity constraints. Continuity is, instead, enforced weakly through numerical fluxes and penalty terms. The variational formulation is constructed element-wise, and integration by parts produces explicit interface terms involving jumps and averages of the solution and its fluxes. DG methods are highly flexible, allowing non-matching meshes and local $hp$-refinement, and are locally conservative, making them attractive for convection-dominated problems. However, they introduce more degrees of freedom than conforming methods for the same polynomial order, increase computational cost, and require careful choice of penalty parameters for stability.

In summary, the discrete spaces satisfy the inclusion
\[
V_{\mathrm{DG}} \supset V_{\mathrm{NCFEM}} \supset V_{\mathrm{CFEM}},
\]
where each relaxation of continuity increases flexibility but also changes the structure of the variational formulation and typically increases the number of degrees of freedom.
\subsection{Numerical studies}\label{combined numerical}
In this section, we consider the Allen-Cahn equation solved using conforming, nonconforming, and discontinuous Galerkin methods, each measured under its respective norm. The table below reports the error estimates and observed convergence orders for problem \eqref{Damped Heat}, with parameters $\nu=1$, $\alpha=1$, $M=1$, $\beta_1=1$, $q_{\ell}=2$, $p=4$, and $\gamma=10$, on the unit square domain $\Omega = (0,1)^2$ over the time interval $t \in [0,1]$.

Let us assume the exact solution to be
\begin{align*}
u(x,y,t)=\exp(-t) \sin(2\pi x) \sin(3\pi y).
\end{align*}
Table \ref{tab:all-in-one} presents the errors and observed convergence rates for different grid sizes, with $\Delta t = 0.01$. The expected convergence rates are clearly observed for all three methods.
\begin{table}[ht!]
	\centering
	\begin{tabular}{|c|c|c|c|c|}
		\hline
		\text{Method} & \text{Grid Size} & $h$ & $\triplenorm{u-u_h^k}$  & \text{Convergence Rate} \\
		\hline
		\multirow{6}{*}{CFEM} 
		& $4 \times 4$     & 3.54e{-01} & 2.8717e{+00} & N/A \\
		& $8 \times 8$     & 1.77e{-01} & 1.6973e{+00} & 0.76 \\
		& $16 \times 16$   & 8.84e{-02} & 8.8968e{-01} & 0.93 \\
		& $32 \times 32$   & 4.42e{-02} & 4.5029e{-01} & 0.98 \\
		& $64 \times 64$   & 2.21e{-02} & 2.2584e{-01} & 1.00 \\
		& $128 \times 128$ & 1.10e{-02} & 1.1301e{-01} & 1.00 \\
		\hline
		\multirow{6}{*}{NCFEM} 
		& $4 \times 4$     & 3.54e{-01} & 9.7766e{-01} & N/A \\
		& $8 \times 8$     & 1.77e{-01} & 4.7234e{-01} & 1.05 \\
		& $16 \times 16$   & 8.84e{-02} & 2.3590e{-01} & 1.00 \\
		& $32 \times 32$   & 4.42e{-02} & 1.1800e{-01} & 1.00 \\
		& $64 \times 64$   & 2.21e{-02} & 5.9006e{-02} & 1.00 \\
		& $128 \times 128$ & 1.10e{-02} & 2.9504e{-02} & 1.00 \\
		\hline
		\multirow{6}{*}{DGFEM} 
		& $4 \times 4$     & 3.54e{-01} & 1.9102e{+00} & N/A \\
		& $8 \times 8$     & 1.77e{-01} & 1.0189e{+00} & 0.91 \\
		& $16 \times 16$   & 8.84e{-02} & 4.8240e{-01} & 1.08 \\
		& $32 \times 32$   & 4.42e{-02} & 2.3203e{-01} & 1.06 \\
		& $64 \times 64$   & 2.21e{-02} & 1.1405e{-01} & 1.02 \\
		& $128 \times 128$ & 1.10e{-02} & 5.6621e{-02} & 1.01 \\
		\hline
	\end{tabular}
	\caption{Computed errors and rates of convergence for \eqref{pumped} on different grid sizes and methods.}
	\label{tab:all-in-one}
\end{table}

\subsection{Conclusion}\label{conclusion}
In this study, we established the well-posedness of \eqref{Damped Heat} and derived corresponding regularity results. Subsequently, we developed a priori error estimates for conforming, nonconforming, and symmetric interior penalty discontinuous Galerkin (DG) methods, demonstrating that all three schemes attain optimal convergence rates in mesh-dependent norms. Numerical experiments conducted using FEniCS verified the theoretical convergence behavior and confirmed the stability of the computed solutions. The comparative assessment of the three discretizations provided valuable insights into how method-specific stability mechanisms interact with monotone nonlinearities in parabolic problems. Future work will aim to extend these findings by deriving a posteriori error estimates.

\medskip
\noindent
\textbf{Acknowledgments:}  Wasim Akram is supported by National Board of Higher Mathematics (NBHM)  postdoctoral fellowship, No. 0204/16(1)(2)/2024/R\&D-II/10823. Support for M. T. Mohan's research received from the National Board of Higher Mathematics (NBHM), Department of Atomic Energy, Government of India (Project No. 02011/13/2025/NBHM(R.P)/R\&D II/1137). 

\bibliographystyle{amsplain}
\bibliography{References} 	

\providecommand{\bysame}{\leavevmode\hbox to3em{\hrulefill}\thinspace}
\providecommand{\MR}{\relax\ifhmode\unskip\space\fi MR }
\providecommand{\MRhref}[2]{%
  \href{http://www.ams.org/mathscinet-getitem?mr=#1}{#2}
}
\providecommand{\href}[2]{#2}
\begin{thebibliography}{10}

\bibitem{MR1885308}
M.~Ainsworth and J.~T. Oden, \emph{A posteriori error estimation in finite
  element analysis}, Pure and Applied Mathematics (New York),
  Wiley-Interscience [John Wiley \& Sons], New York, 2000. \MR{1885308}

\bibitem{AllenCahn1979}
S.~M. Allen and J.~W. Cahn, \emph{A microscopic theory for antiphase boundary
  motion and its application to antiphase domain coarsening}, Acta Metallurgica
  \textbf{27} (1979), no.~6, 1085--1095.

\bibitem{Alnaes2015}
M.~S. Aln{\ae}s, J.~Blechta, J.~Hake, A.~Johansson, B.~Kehlet, A.~Logg,
  C.~Richardson, J.~Ring, M.~E. Rognes, and G.~N. Wells, \emph{The fenics
  project version 1.5}, Archive of Numerical Software \textbf{3} (2015),
  no.~100, 9--23.

\bibitem{MR664882}
D.~N. Arnold, \emph{An interior penalty finite element method with
  discontinuous elements}, SIAM J. Numer. Anal. \textbf{19} (1982), no.~4,
  742--760. \MR{664882}

\bibitem{MR2975554}
S.~Badia, \emph{On stabilized finite element methods based on the
  {S}cott-{Z}hang projector. {C}ircumventing the inf-sup condition for the
  {S}tokes problem}, Comput. Methods Appl. Mech. Engrg. \textbf{247/248}
  (2012), 65--72. \MR{2975554}

\bibitem{AB+ZB+MTM-2025}
A.~Bawalia, Z.~Brze\'zniak, and M.~T. Mohan, \emph{Global well-posedness and
  asymptotic analysis of a nonlinear heat equation with constraints of finite
  codimension}, https://arxiv.org/pdf/2507.00160 (2025).

\bibitem{MR3097958}
D.~Boffi, F.~Brezzi, and M.~Fortin, \emph{Mixed finite element methods and
  applications}, Springer Series in Computational Mathematics, vol.~44,
  Springer, Heidelberg, 2013. \MR{3097958}

\bibitem{MR2373954}
S.~C. Brenner and L.~R. Scott, \emph{The mathematical theory of finite element
  methods}, third ed., Texts in Applied Mathematics, vol. \textbf{15},
  Springer, New York, 2008. \MR{2373954}

\bibitem{MR2759829}
H.~Brezis, \emph{Functional analysis, sobolev spaces and partial differential
  equations}, Universitext, Springer, New York, 2011. \MR{2759829}

\bibitem{MR1115205}
F.~Brezzi and M.~Fortin, \emph{Mixed and hybrid finite element methods},
  Springer Series in Computational Mathematics, vol.~15, Springer-Verlag, New
  York, 1991. \MR{1115205}

\bibitem{ZB+BF+MZ-24}
Z.~Brze\'zniak, B.~Ferrario, and M.~Zanella, \emph{Invariant measures for a
  stochastic nonlinear and damped 2{D} {S}chr\"odinger equation}, Nonlinearity
  \textbf{37} (2024), no.~1, Paper No. 015001, 66. \MR{4678996}

\bibitem{ZB+FH+UM-20}
Z.~Brze\'zniak, F.~Hornung, and U.~Manna, \emph{Weak martingale solutions for
  the stochastic nonlinear {S}chr\"odinger equation driven by pure jump noise},
  Stoch. Partial Differ. Equ. Anal. Comput. \textbf{8} (2020), no.~1, 1--53.
  \MR{4058955}

\bibitem{ZB+FH+LW-19}
Z.~Brze\'zniak, F.~Hornung, and L.~Weis, \emph{Martingale solutions for the
  stochastic nonlinear {S}chr\"odinger equation in the energy space}, Probab.
  Theory Related Fields \textbf{174} (2019), no.~3-4, 1273--1338. \MR{3980316}

\bibitem{MR1041253}
J.~R. Cannon and Y.~P. Lin, \emph{A priori {$L^2$} error estimates for
  finite-element methods for nonlinear diffusion equations with memory}, SIAM
  J. Numer. Anal. \textbf{27} (1990), no.~3, 595--607. \MR{1041253}

\bibitem{MR2962068}
P.~Cherrier and A.~Milani, \emph{Linear and quasi-linear evolution equations in
  {H}ilbert spaces}, Graduate Studies in Mathematics, vol. 135, American
  Mathematical Society, Providence, RI, 2012. \MR{2962068}

\bibitem{MR1921920}
K.~Chrysafinos and L.~S. Hou, \emph{Error estimates for semidiscrete finite
  element approximations of linear and semilinear parabolic equations under
  minimal regularity assumptions}, SIAM J. Numer. Anal. \textbf{40} (2002),
  no.~1, 282--306. \MR{1921920}

\bibitem{MR520174}
P.~G. Ciarlet, \emph{The finite element method for elliptic problems}, Studies
  in Mathematics and its Applications, vol. \textbf{4}, North-Holland
  Publishing Co., Amsterdam-New York-Oxford, 1978. \MR{520174}

\bibitem{MR400739}
P.~Cl\'ement, \emph{Approximation by finite element functions using local
  regularization}, Rev. Fran\c caise Automat. Informat. Recherche
  Op\'erationnelle S\'er. Rouge Anal. Num\'er. \textbf{9} (1975).

\bibitem{MR69338}
E.~A. Coddington and N.~Levinson, \emph{Theory of ordinary differential
  equations}, McGraw-Hill Book Co., Inc., New York-Toronto-London, 1955.
  \MR{69338}

\bibitem{CrossHohenberg1993}
M.~C. Cross and P.~C. Hohenberg, \emph{Pattern formation outside of
  equilibrium}, Reviews of Modern Physics \textbf{65} (1993), no.~3, 851--1112.

\bibitem{MR343661}
M.~Crouzeix and P.~A. Raviart, \emph{Conforming and nonconforming finite
  element methods for solving the stationary {S}tokes equations. {I}}, Rev.
  Fran\c caise Automat. Informat. Recherche Op\'erationnelle S\'er. Rouge
  \textbf{7} (1973), 33--75. \MR{343661}

\bibitem{MR1156075}
R.~Dautray and J.~L. Lions, \emph{Mathematical analysis and numerical methods
  for science and technology. {V}ol. 5}, Springer-Verlag, Berlin, 1992,
  Evolution problems. I, With the collaboration of Michel Artola, Michel
  Cessenat and H\'el\`ene Lanchon, Translated from the French by Alan Craig.
  \MR{1156075}

\bibitem{DEVI2025274}
R.~Devi and D.~N. Pandey, \emph{Discontinuous galerkin time-stepping method for
  semilinear parabolic problems with mild growth condition}, Computers \&
  Mathematics with Applications \textbf{198} (2025), 274--292.

\bibitem{MR319379}
J.~Douglas, Jr. and T.~Dupont, \emph{Galerkin methods for parabolic equations
  with nonlinear boundary conditions}, Numer. Math. \textbf{20} (1972/73),
  213--237. \MR{319379}

\bibitem{MR2050138}
A.~Ern and J.~L. Guermond, \emph{Theory and practice of finite elements},
  Applied Mathematical Sciences, vol. 159, Springer-Verlag, New York, 2004.
  \MR{2050138}

\bibitem{MR1625845}
L.~C. Evans, \emph{Partial differential equations}, Graduate Studies in
  Mathematics, vol. \textbf{19}, American Mathematical Society, Providence, RI,
  1998. \MR{1625845}

\bibitem{FitzHugh1961}
R.~FitzHugh, \emph{Impulses and physiological states in theoretical models of
  nerve membrane}, Biophysical Journal \textbf{1} (1961), no.~6, 445--466.

\bibitem{SG+MTM-24+}
S.~Gautam and M.~T. Mohan, \emph{On the convective {B}rinkman-{F}orchheimer
  equations}, Dyn. Partial Differ. Equ. \textbf{22} (2025), no.~3, 191--233.

\bibitem{GB+KR=04}
B.~H. Gilding and R.~Kersner, \emph{Travelling waves in nonlinear
  diffusion-convection reaction}, Progress in Nonlinear Differential Equations
  and their Applications, vol.~60, Birkh\"{a}user Verlag, Basel, 2004.
  \MR{2081104}

\bibitem{MR851383}
V.~Girault and P.~A. Raviart, \emph{Finite element methods for
  {N}avier-{S}tokes equations}, Springer Series in Computational Mathematics,
  vol.~5, Springer-Verlag, Berlin, 1986, Theory and algorithms. \MR{851383}

\bibitem{MR3396210}
P.~Grisvard, \emph{Elliptic problems in nonsmooth domains}, Classics in Applied
  Mathematics, vol.~69, Society for Industrial and Applied Mathematics (SIAM),
  Philadelphia, PA, 2011, Reprint of the 1985 original [ MR0775683], With a
  foreword by Susanne C. Brenner. \MR{3396210}

\bibitem{HilhorstMatanoSakamoto1997}
D.~Hilhorst, H.~Matano, and N.~Sakamoto, \emph{Singular limit of the
  {Allen--Cahn} equation and the motion of interfaces}, Annales de l’Institut
  Henri Poincaré C \textbf{14} (1997), no.~5, 637--696.

\bibitem{LH-18}
L.~Hornung, \emph{Strong solutions to a nonlinear stochastic {M}axwell equation
  with a retarded material law}, J. Evol. Equ. \textbf{18} (2018), no.~3,
  1427--1469. \MR{3859455}

\bibitem{MR2238170}
L.~S. Hou and W.~Zhu, \emph{Error estimates under minimal regularity for single
  step finite element approximations of parabolic partial differential
  equations}, Int. J. Numer. Anal. Model. \textbf{3} (2006), no.~4, 504--524.
  \MR{2238170}

\bibitem{Huxley1959}
A.~F. Huxley, \emph{Ion movements during nerve activity}, Annals of the New
  York Academy of Sciences \textbf{81} (1959), no.~2, 221--246.

\bibitem{MR4288303}
A.~Khan, M.~T. Mohan, and R.~Ruiz-Baier, \emph{Conforming, nonconforming and
  {DG} methods for the stationary generalized {B}urgers-{H}uxley equation}, J.
  Sci. Comput. \textbf{88} (2021), no.~3, Paper No. 52, 26. \MR{4288303}

\bibitem{Kuehn2015}
C.~Kuehn, \emph{Numerical continuation and spde stability for the 2d
  cubic–quintic allen–cahn equation}, SIAM/ASA Journal on Uncertainty
  Quantification \textbf{3} (2015), no.~1, 762--789.

\bibitem{MR259693}
J.~L. Lions, \emph{Quelques m\'{e}thodes de r\'{e}solution des probl\`emes aux
  limites non lin\'{e}aires}, Dunod, Paris; Gauthier-Villars, Paris, 1969.
  \MR{259693}

\bibitem{MR3075806}
A.~Logg, K.A. Mardal, and G.~N. Wells (eds.), \emph{Automated solution of
  differential equations by the finite element method}, Lecture Notes in
  Computational Science and Engineering, vol.~84, Springer, Heidelberg, 2012,
  The FEniCS book. \MR{3075806}

\bibitem{MR923707}
C.~Lubich, \emph{Convolution quadrature and discretized operational calculus.
  {I}}, Numer. Math. \textbf{52} (1988), no.~2, 129--145. \MR{923707}

\bibitem{MR4798380}
S.~Mahahjan and A.~Khan, \emph{Finite element approximation for the delayed
  generalized {B}urgers-{H}uxley equation with weakly singular kernel: {P}art
  {II} {N}onconforming and {DG} approximation}, SIAM J. Sci. Comput.
  \textbf{46} (2024), no.~5, A2972--A2998. \MR{4798380}

\bibitem{MR4797426}
S.~Mahajan, A.~Khan, and M.~T. Mohan, \emph{Finite element approximation for a
  delayed generalized {B}urgers-{H}uxley equation with weakly singular kernels:
  {P}art {I} well-posedness, regularity and conforming approximation}, Comput.
  Math. Appl. \textbf{174} (2024), 261--286. \MR{4797426}

\bibitem{MR1225703}
W.~McLean and V.~Thom\'ee, \emph{Numerical solution of an evolution equation
  with a positive-type memory term}, J. Austral. Math. Soc. Ser. B \textbf{35}
  (1993), no.~1, 23--70. \MR{1225703}

\bibitem{Mirams2013}
G.~R. Mirams, C.~J. Arthurs, M.~O. Bernabeu, R.~Bordas, J.~Cooper, A.~Corrias,
  Y.~Davit, S.~Dunn, A.~G. Fletcher, D.~G. Harvey, M.~Marsh, J.~Osborne,
  P.~Pathmanathan, J.~Pitt-Francis, J.~Southern, N.~Zemzemi, and D.~J.
  Gavaghan, \emph{Chaste: An open source c++ library for computational
  physiology and biology}, PLoS Computational Biology \textbf{9} (2013), no.~3,
  e1002970.

\bibitem{MTM-22}
M.~T. Mohan, \emph{Well-posedness and asymptotic behavior of stochastic
  convective {B}rinkman-{F}orchheimer equations perturbed by pure jump noise},
  Stoch. Partial Differ. Equ. Anal. Comput. \textbf{10} (2022), no.~2,
  614--690. \MR{4439993}

\bibitem{MR4251864}
M.~T. Mohan and A.~Khan, \emph{On the generalized {B}urgers-{H}uxley equation:
  existence, uniqueness, regularity, global attractors and numerical studies},
  Discrete Contin. Dyn. Syst. Ser. B \textbf{26} (2021), no.~7, 3943--3988.
  \MR{4251864}

\bibitem{MR2755668}
K.~Mustapha and W.~McLean, \emph{Piecewise-linear, discontinuous {G}alerkin
  method for a fractional diffusion equation}, Numer. Algorithms \textbf{56}
  (2011), no.~2, 159--184. \MR{2755668}

\bibitem{Nagumo1962}
J.~Nagumo, S.~Arimoto, and S.~Yoshizawa, \emph{An active pulse transmission
  line simulating nerve axon}, Proceedings of the IRE \textbf{50} (1962),
  no.~10, 2061--2070.

\bibitem{NewellWhitehead1969}
A.~C. Newell and J.~A. Whitehead, \emph{Finite bandwidth, finite amplitude
  convection}, Journal of Fluid Mechanics \textbf{38} (1969), no.~2, 279--303.

\bibitem{MV-78}
M.~V. Polley, \emph{Grain boundary diffusion with radioactive decay}, J. Phys.
  D: Appl. Phys. \textbf{11} (1978), no.~9, 1287.

\bibitem{MR2431403}
B.~Rivi\`ere, \emph{Discontinuous {G}alerkin methods for solving elliptic and
  parabolic equations}, Frontiers in Applied Mathematics, vol.~35, Society for
  Industrial and Applied Mathematics (SIAM), Philadelphia, PA, 2008, Theory and
  implementation. \MR{2431403}

\bibitem{MR1881888}
J.~C. Robinson, \emph{Infinite-dimensional dynamical systems}, Cambridge Texts
  in Applied Mathematics, Cambridge University Press, Cambridge, 2001, An
  introduction to dissipative parabolic PDEs and the theory of global
  attractors. \MR{1881888}

\bibitem{MR4586337}
\bysame, \emph{An introduction to functional analysis}, Cambridge University
  Press, London, 2020. \MR{4586337}

\bibitem{MR1011446}
L.~Ridgway Scott and S.~Zhang, \emph{Finite element interpolation of nonsmooth
  functions satisfying boundary conditions}, Math. Comp. \textbf{54} (1990),
  no.~190, 483--493. \MR{1011446}

\bibitem{Segel1969}
L.~A. Segel, \emph{Distant side-walls cause slow amplitude modulation of
  cellular convection}, Journal of Fluid Mechanics \textbf{38} (1969), no.~1,
  203--224.

\bibitem{MR2495062}
D.~Shi, H.~Wang, and Y.~Du, \emph{An anisotropic nonconforming finite element
  method for approximating a class of nonlinear {S}obolev equations}, J.
  Comput. Math. \textbf{27} (2009), no.~2-3, 299--314. \MR{2495062}

\bibitem{MR916688}
J.~Simon, \emph{Compact sets in the space {$L^p(0,T;B)$}}, Ann. Mat. Pura Appl.
  (4) \textbf{\textbf{146}} (1987), 65--96. \MR{916688}

\bibitem{MR1948323}
E.~S\"uli and P.~Houston, \emph{Adaptive finite element approximation of
  hyperbolic problems}, Error estimation and adaptive discretization methods in
  computational fluid dynamics, Lect. Notes Comput. Sci. Eng., vol.~25,
  Springer, Berlin, 2003, pp.~269--344. \MR{1948323}

\bibitem{MR609732}
R.~Temam, \emph{Navier-stokes equations: Theory and numerical analysis},
  Studies in Mathematics and its Applications, Vol. 2, North-Holland Publishing
  Co., Amsterdam-New York-Oxford, 1977. \MR{609732}

\bibitem{Te_1997}
\bysame, \emph{Infinite-dimensional dynamical systems in mechanics and
  physics}, second ed., Applied Mathematical Sciences, vol.~68,
  Springer-Verlag, New York, 1997. \MR{1441312}

\bibitem{MR1479170}
V.~Thom\'ee, \emph{Galerkin finite element methods for parabolic problems},
  Springer Series in Computational Mathematics, vol.~25, Springer-Verlag,
  Berlin, 1997. \MR{1479170}

\bibitem{MR3432852}
L.~Yi and B.~Guo, \emph{An {$h$}-{$p$} version of the continuous
  {P}etrov-{G}alerkin finite element method for {V}olterra integro-differential
  equations with smooth and nonsmooth kernels}, SIAM J. Numer. Anal.
  \textbf{53} (2015), no.~6, 2677--2704. \MR{3432852}

\bibitem{EZ-12}
E.~Zeidler, \emph{Applied functional analysis: applications to mathematical
  physics}, vol. 108, Springer Science \& Business Media, 2012.

\bibitem{MR3292660}
O.~C. Zienkiewicz, R.~L. Taylor, and J.~Z. Zhu, \emph{The finite element
  method: Its basis and fundamentals}, seventh ed., Elsevier/Butterworth
  Heinemann, Amsterdam, 2013. \MR{3292660}

\end{thebibliography}

\end{document}